\documentclass{article}

\usepackage[numbers]{natbib}

\usepackage[final]{neurips_2022}

\usepackage[utf8]{inputenc} 
\usepackage[T1]{fontenc}    
\usepackage{hyperref}       
\usepackage{url}            
\usepackage{booktabs}       
\usepackage{amsfonts}       
\usepackage{nicefrac}       
\usepackage{microtype}      
\usepackage{xcolor}         
\usepackage{amsmath}
\usepackage{amssymb}
\usepackage{mathtools}
\usepackage{amsthm}
\usepackage{makecell}
\usepackage[font=small]{caption}
\usepackage{multirow}
\usepackage{makecell}
\usepackage{algorithm}
\usepackage{algorithmic}
\usepackage{lscape}
\usepackage{wrapfig}

\newcommand{\algname}[1]{{\tt \sf \footnotesize #1}}

\allowdisplaybreaks
\usepackage{colortbl}
\definecolor{bgcolor}{rgb}{0.8,1,1}
\definecolor{bgcolor2}{rgb}{0.8,1,0.8}
\usepackage{threeparttable}
\newcommand{\myred}[1]{{\color{red}#1}}
\newcommand{\myblue}[1]{{\color{blue}#1}}

\usepackage[capitalize,noabbrev]{cleveref}

\theoremstyle{plain}
\newtheorem{theorem}{Theorem}[section]

\newtheorem{lemma}[theorem]{Lemma}
\newtheorem{corollary}[theorem]{Corollary}
\theoremstyle{definition}
\newtheorem{definition}[theorem]{Definition}
\newtheorem{assumption}[theorem]{Assumption}
\theoremstyle{remark}

\usepackage[textsize=tiny]{todonotes}

\def\R{\mathbb{R}}

\def\R{\mathbb R}
\def\N{\mathbb N}
\def\EE{\mathbb E}

\def\la{\langle}
\def\ra{\rangle}

\def\y{\mathbf{y}}

\def\x{\mathbf{x}}
\def\w{\mathbf{w}}
\def\z{\mathbf{z}}

\newcommand{\E}[1]{\mathbb{E}\left[#1\right]}
\newcommand{\Ek}[1]{\mathbb{E}_k\left[#1\right]}
\newcommand{\mI}{\mathbf{I}}
\newcommand{\mA}{\mathbf{A}}
\newcommand{\mB}{\mathbf{B}}
\newcommand{\eqdef}{\vcentcolon=}

\newcommand{\cO}{\mathcal{O}}

\newcommand{\norm}[1]{\left\| #1 \right\|}

\DeclareMathOperator{\dom}{\mathrm{dom}}
\newcommand{\prox}{\mathrm{prox}}
\newcommand{\g}{\mathbf{g}}
\newcommand{\F}{\mathbf{F}}
\newcommand{\sqnw}[1]{\sqn{#1}_{(\mWp\otimes\mI_d)}}
\newcommand{\Lavg}{\overline{L}}
\newcommand{\mW}{\mathbf{W}}
\newcommand{\cL}{\mathcal{L}}
\newcommand{\mWp}{\mathbf{W}^{\dagger}}
\newcommand{\mP}{\mathbf{P}}
\newcommand{\ones}{\mathbf{1}}

\def\<#1,#2>{\langle #1,#2\rangle}
\newcommand{\sqn}[1]{\norm{#1}^2}
\newcommand{\sqN}[1]{\norm{#1}^2}

\title{Optimal Algorithms for Decentralized Stochastic Variational Inequalities}

\author{%
  Dmitry Kovalev\\
  KAUST\thanks{King Abdullah University of Science and Technology}, Saudi Arabia\\
  \texttt{dakovalev1@gmail.com}
  \And
  Aleksandr Beznosikov \\
  MIPT\thanks{Moscow Institute of Physics and Technology}, HSE University and Yandex, Russia\\
  \texttt{anbeznosikov@gmail.com}
  \And
  Abdurakhmon Sadiev\\
  MIPT, Russia\\
  \texttt{sadiev.aa@phystech.edu}
  \And
  Michael Persiianov\\
  MIPT, Russia\\
  \texttt{persiianov.mi@phystech.edu}
  \And
  Peter Richt\'{a}rik \\
  KAUST, Saudi Arabia \\
  \texttt{peter.richtarik@kaust.edu.sa}
  \And
  Alexander Gasnikov \\
  MIPT, HSE University and IITP RAS\thanks{Institute for Information Transmission Problems RAS}, Russia\\
  \texttt{gasnikov@yandex.ru} 
}

\begin{document}

\maketitle

\begin{abstract}
Variational inequalities are a formalism that includes games, minimization, saddle point, and equilibrium problems as special cases. Methods for variational inequalities are therefore universal approaches for many applied tasks, including machine learning problems. This work concentrates on the decentralized setting, which is increasingly important but not well understood. In particular, we consider decentralized stochastic (sum-type) variational inequalities over fixed and time-varying networks. We present lower complexity bounds for both communication and local iterations and construct optimal algorithms that match these lower bounds. Our algorithms are the best among the available literature not only in the decentralized stochastic case, but also in the decentralized deterministic and non-distributed stochastic cases. Experimental results confirm the effectiveness of the presented algorithms.
\end{abstract}

\section{Introduction}

Variational inequalities are a broad and flexible class of problems that includes  minimization, saddle point, Nash equilibrium, and fixed point problems as special cases; see \citep{VIbook2003,Heinz} for an introduction. 
Over the long history of modern research on variational inequalities spanning at least half a century, the community developed their own methods and theory, differing from the approaches in their sister field, optimization. The \algname{ExtraGradient} / \algname{MirrorProx} methods  due to \cite{Korpelevich1976TheEM,Nemirovski2004,juditsky2008solving} have a similar foundational standing in the variational inequalities field that gradient descent occupies in the optimization literature.
As in the case of gradient descent, many modifications~\citep{hsieh2019convergence} and variants~\cite{doi:10.1137/S0363012998338806} of these methods were proposed and studied in the variational inequalities literature.

\subsection{Applications of variational inequalities} 
In recent years, there has been a significant increase of research activity in the study of variational inequalities due to new connections to reinforcement learning~\citep{Omidshafiei2017:rl,Jin2020:mdp},  adversarial training \citep{Madry2017:adv}, and GANs~\citep{goodfellow2014generative}. In particular, ~\cite{daskalakis2017training,gidel2018variational,mertikopoulos2018optimistic,chavdarova2019reducing,pmlr-v89-liang19b,peng2020training} show that even if one considers the classical (in the variational inequalities literature) regime involving monotone and strongly monotone inequalities, it is possible to obtain  insights, methods and recommendations useful  for the GAN community. 

In addition to the above modern applications, and besides their many classical applications in applied mathematics that include economics, equilibrium theory, game theory and optimal control \cite{facchinei2007finite}, variational inequalities remain popular in supervised learning (with non-separable loss \citep{Thorsten}; with non-separable regularizer~\citep{bach2011optimization}),  unsupervised learning (discriminative clustering \citep{NIPS2004_64036755};  matrix factorization \citep{bach2008convex}), image denoising \citep{esser2010general,chambolle2011first}, robust optimization \citep{BenTal2009:book}, and non-smooth optimization via smooth reformulations \citep{nesterov2005smooth,nemirovski2004prox}.

\subsection{Processing on the edge} With the proliferation of mobile phones, wearables, digital sensors, smart home appliances, and other devices capable of capturing, storing and processing data, there is an increased appetite to mine the richness  contained in these sources for the benefit of humanity. However, at the same time, the traditional centralized approach relying on moving the data into a single proprietary warehouse for  processing via suitable machine learning methods is problematic, and a new modus operandi is on the rise: processing the data at the source, on the edge, where it was first captured and where it is stored~\citep{FEDLEARN, FL2017-AISTATS, kairouz2019advances}, by the client's devices that own the data. There are many reasons for a gradual shift in this direction, including energy efficiency and data privacy. 

A key necessary characteristic for any viable algorithmic approach to work in such a massively decentralized regime is the ability to support decentralized processing reflecting the fact that the devices  are connected through a network of a potentially complicated topology, possibly varying in time. A central authority may be absent in such a system, and the methods need to rely on communication patters that correspond to the existing connection links.

\subsection{Decentralized algorithms for variational inequalities}

In this paper we study
\begingroup
\addtolength\leftmargini{-1cm}
 \vspace{-0.2cm}\begin{quote}{\em algorithms for solving variational inequalities over decentralized  communication networks}.\end{quote} \vspace{-0.2cm}
\endgroup
 In this regime, a  number of nodes (workers, devices, clients) are connected via a communication network, represented by a graph. Each node can perform computations using its local state and data, and is only allowed to communicate with its neighbors in the graph. 

Decentralized algorithms over fixed communication networks find their applications in  sensor networks
\cite{1307319}, network resource allocation \cite{beck20141}, cooperative control \cite{giselsson2013accelerated}, distributed spectrum sensing \cite{bazerque2009distributed}, power system control
\cite{gan2012optimal} and, of course, in machine learning \cite{scaman2017optimal}. 
Recently, decentralized methods over time-varying networks have gained particular popularity  due to their relevance to federated learning \cite{FEDLEARN, kairouz2019advances}, where communication failures between devices are a common problem.

Decentralized minimization methods are well studied \cite{9084356, koloskova2020unified}. In particular, lower bounds and optimal algorithms for such problems are known in the {\em fixed} \cite{scaman2017optimal,hendrikx2020optimal, kovalev2020optimal} and {\em time-varying} \cite{kovalev2021lower, li2021accelerated} network topology regimes.
\begingroup
\addtolength\leftmargini{-1cm}
\vspace{-0.2cm}\begin{quote}\em However, in significantly more general and hence potentially much more impactful formalism of variational inequalities, the question of optimal and efficient decentralized methods is still open. 
\end{quote}\vspace{-0.2cm}
\endgroup
Motivated by these considerations, our work is devoted to advancing the algorithmic and theoretical foundations of  decentralized variational inequalities, in both the fixed and time-varying network regimes.

\subsection{Our contributions and related work}

We now briefly summarize our main contributions.

\textbf{(a) Lower bounds} 

We present the {\bf first lower bounds for the communication and local computation complexities of decentralized variational inequalities in the stochastic (finite sum) case,   in both the  fixed and time-varying network topology regimes.} See Table \ref{tab:comparison0}.

Existing literature contains lower bounds for {\em non-distributed} finite-sum variational inequalities \citep{han2021lower}, which we recover as a special case. 
Existing literature also contains  lower bounds for {\em deterministic} decentralized variational inequalities in the fixed \citep{beznosikov2020distributed} and time-varying \citep{beznosikov2021} regimes. Our bounds covers these results too. See Table \ref{tab:comparison1} (Appendix \ref{sec:tables}).

\textbf{(b) Optimal decentralized algorithms}

We construct {\bf four new algorithms for stochastic (finite sum) decentralized variational inequalities: two for fixed networks, and two for time-varying networks. Two of these algorithms match our lower bounds, and are therefore  optimal in terms of communication and local iteration complexities.} These are the first algorithms for stochastic (finite sum) decentralized variational inequalities over fixed and time-varying networks. See Table \ref{tab:comparison0}.

Moreover, our results offer linear communication complexity for deterministic decentralized strongly monotone variational inequalities, which is an improvement upon the sublinear results of \cite{beznosikov2021distributed,beznosikov2020distributed,beznosikov2021decentralized,beznosikov2021}. Additionally, our algorithms have better guarantees on local computations than the methods developed by \cite{rogozin2021decentralized}. See Table \ref{tab:comparison1} (Appendix \ref{sec:tables}).

Let us also single out a number of works on decentralized saddle point problems or VIs which are not suitable for comparison with our results: \cite{NEURIPS2021_d994e372, BARAZANDEH2021108245} consider non-monotone(minty) problems, \cite{7403075} does not prove convergence, \cite{liu2019decentralized} assumes data homogeneity, and  \cite{6004889} considers a discrete problem.

\renewcommand{\arraystretch}{2}
\begin{table*}[h]
\vspace{-0.3cm}
    \centering
    \small
\captionof{table}{Summary of upper and lower bounds for communication and local computation complexities for finding an $\varepsilon$-solution for strongly monotone \textbf{stochastic (finite-sum)} \textbf{decentralized} variational inequality \eqref{eq:VI} over fixed and time-varying networks. Convergence is measured by the distance to the solution.}
\vspace{-0.2cm}
    \label{tab:comparison0}   
    \scriptsize
    \resizebox{\linewidth}{!}{
  \begin{threeparttable}
    \begin{tabular}{|c|c|c|c|c|c|}
    \cline{3-6}
    \multicolumn{2}{c|}{}
     & \textbf{\quad\quad\quad\quad\quad\quad Reference \quad\quad\quad\quad\quad\quad} & \textbf{Communication complexity} & \textbf{Local complexity}  & \textbf{Weaknesses} \\
    \hline
    \multirow{8}{*}{\rotatebox[origin=c]{90}{\textbf{Fixed \quad\quad}}} & \multirow{6}{*}{\rotatebox[origin=c]{90}{\textbf{Upper}\quad\quad}}
    & Mukherjee and Chakraborty \cite{Mukherjee2020:decentralizedminmax} \tnote{{\color{blue}(1,2)}}  & $\mathcal{O} \left( \myred{\chi^{\frac{4}{3}}} \frac{L^{\myred{\frac{4}{3}}}}{\mu^{\myred{\frac{4}{3}}}} \log \frac{1}{\varepsilon} \right)$ &  $\mathcal{O} \left( \myred{n\chi^{\frac{4}{3}}} \frac{L^{\myred{\frac{4}{3}}}}{\mu^{\myred{\frac{4}{3}}}} \log \frac{1}{\varepsilon} \right)$ & \makecell{{weak communication rates} \\ {weak local computation rates}}
    \\ \cline{3-6}
    && Beznosikov et al. \cite{beznosikov2021distributed}  \tnote{{\color{blue}(1,2)}}  & $\mathcal{O} \left( \sqrt{\chi} \frac{L}{\mu} \log^{\myred{2}} \frac{1}{\varepsilon} \right)$ &  $\mathcal{O} \left( \myred{n} \frac{L}{\mu} \myred{\log \frac{L+\mu}{\mu}} \log \frac{1}{\varepsilon}\right)$ & \makecell{{multiple gossip} \\ {no linear convergence}}
    \\ \cline{3-6}
    && Beznosikov et al. \cite{beznosikov2020distributed}  \tnote{{\color{blue}(1,3)}} & $\mathcal{O} \left( \sqrt{\chi} \frac{L}{\mu} \log^{\myred{2}} \frac{1}{\varepsilon} \right)$ &  $\mathcal{O} \left( \myred{n} \frac{L}{\mu}\log \frac{1}{\varepsilon}\right)$ & \makecell{{multiple gossip} \\ {no linear convergence}}
    \\ \cline{3-6}
    && Rogozin et al. \cite{rogozin2021decentralized}  \tnote{{\color{blue}(1,2,4)}} & $\mathcal{O} \left( \sqrt{\chi} \frac{L}{\mu} \log \frac{1}{\varepsilon} \right)$ &  $\mathcal{O} \left( \myred{n \sqrt{\chi}} \frac{L}{\mu}\log \frac{1}{\varepsilon}\right)$ & weak local computation rates
    \\ \cline{3-6}
    && \cellcolor{bgcolor2}{Alg. \ref{alg:vrvi} (this paper)}  & \cellcolor{bgcolor2}{$\mathcal{O} \left( \max[\sqrt{n}; \sqrt{\chi}] \frac{L}{\mu} \log \frac{1}{\varepsilon} \right)$ \tnote{{\color{blue}(6)}}}  &  \cellcolor{bgcolor2}{$\mathcal{O} \left( \max[\sqrt{n}; \sqrt{\chi}]\frac{L}{\mu} \log \frac{1}{\varepsilon} \right)$ \tnote{{\color{blue}(6)}} } & \cellcolor{bgcolor2}{}
    \\ \cline{3-6}
    && \cellcolor{bgcolor2}{Alg. \ref{alg:vrvi} + Alg. \ref{alg:chebyshev_gossip} (this paper)}  & \cellcolor{bgcolor2}{$\mathcal{O} \left( \sqrt{\chi} \frac{L}{\mu} \log \frac{1}{\varepsilon} \right)$ \tnote{{\color{blue}(6)}}}  &  \cellcolor{bgcolor2}{$\mathcal{O} \left( \sqrt{n}\frac{L}{\mu} \log \frac{1}{\varepsilon} \right)$ \tnote{{\color{blue}(6)}}} & \cellcolor{bgcolor2}{multiple gossip}
    \\ \cline{2-6} 
    & \multirow{2}{*}{\rotatebox[origin=c]{90}{\textbf{Lower}}} & Beznosikov et al. \cite{beznosikov2020distributed}  \tnote{{\color{blue}(3)}} & $\Omega \left( \sqrt{\chi} \frac{L}{\mu} \log \frac{1}{\varepsilon} \right)$  &  $\Omega \left( \frac{L}{\mu} \log \frac{1}{\varepsilon} \right)$ & 
    \\ \cline{3-6} 
    && \cellcolor{bgcolor2}{Thm. \ref{th:lower_fixed} + Cor. \ref{cor:lower_fixed} (this paper)} & \cellcolor{bgcolor2}{$\Omega \left( \sqrt{\chi} \frac{L}{\mu} \log \frac{1}{\varepsilon} \right)$}  &  \cellcolor{bgcolor2}{$\Omega \left( \myblue{\sqrt{n}}\frac{L}{\mu} \log \frac{1}{\varepsilon} \right)$} & \cellcolor{bgcolor2}{}
    \\\hline\hline 
    \multirow{6}{*}{\rotatebox[origin=c]{90}{\textbf{Time-varying}\quad}} & \multirow{4}{*}{\rotatebox[origin=c]{90}{\textbf{Upper}\quad \quad}} 
    & Beznosikov et al. \cite{beznosikov2021decentralized} \tnote{{\color{blue}(3)}} & $\mathcal{O} \left( \chi \frac{L}{\mu} \log \frac{1}{\varepsilon} + \myred{\chi \frac{L D }{\mu^2   \sqrt{\varepsilon}}} \right)$ \tnote{{\color{blue}(5)}} & $\mathcal{O} \left( \myred{n\chi} \frac{L}{\mu} \log \frac{1}{\varepsilon} + \myred{n\chi \frac{L D }{\mu^2   \sqrt{\varepsilon}} }\right)$ & \makecell{{$D$-homogeneity}  \\ {no linear convergence}}
    \\ \cline{3-6}
    && Beznosikov et al. \cite{beznosikov2021}  \tnote{{\color{blue}(1,2)}} & $\mathcal{O} \left( \chi \frac{L}{\mu} \log^{\myred{2}} \frac{1}{\varepsilon} \right)$  &  $\mathcal{O} \left( \myred{n}\frac{L}{\mu} \log \frac{1}{\varepsilon} \right)$ & \makecell{{multiple gossip} \\ {no linear convergence}}
    \\ \cline{3-6}
    && \cellcolor{bgcolor2}{ Alg. \ref{2dvi2:alg} (this paper)}  & \cellcolor{bgcolor2}{\hspace{-0.3cm}$\mathcal{O} \left( \max[\sqrt{n}; \chi] \frac{L}{\mu} \log \frac{1}{\varepsilon} \right)$ \tnote{{\color{blue}(5),(6)}}}  &  \cellcolor{bgcolor2}{$\mathcal{O} \left( \max[\sqrt{n}; \chi]\frac{L}{\mu} \log \frac{1}{\varepsilon} \right)$ \tnote{{\color{blue}(6)}} }  & \cellcolor{bgcolor2}{}
    \\ \cline{3-6}
    && \cellcolor{bgcolor2}{ Alg. \ref{2dvi2:alg} + Eq. \ref{eq:tv_gossip} (this paper)}  & \cellcolor{bgcolor2}{$\mathcal{O} \left( \chi \frac{L}{\mu} \log \frac{1}{\varepsilon} \right)$ \tnote{{\color{blue}(5), (6)}}}  &  \cellcolor{bgcolor2}{$\mathcal{O} \left( \sqrt{n}\frac{L}{\mu} \log \frac{1}{\varepsilon} \right)$ \tnote{{\color{blue}(6)}}} & \cellcolor{bgcolor2}{multiple gossip}
    \\ \cline{2-6}
    & \multirow{2}{*}{\rotatebox[origin=c]{90}{\textbf{Lower}}} 
    & Beznosikov et al. \cite{beznosikov2021} \tnote{{\color{blue}(2)}}  & $\Omega \left( \chi \frac{L}{\mu} \log \frac{1}{\varepsilon} \right)$  &  $\Omega \left( \frac{L}{\mu} \log \frac{1}{\varepsilon} \right)$ & 
    \\ \cline{3-6}
    && \cellcolor{bgcolor2}{Thm. \ref{th:lower_tv} + Cor. \ref{cor:lower_tv} (This paper)}  & \cellcolor{bgcolor2}{$\Omega \left( \chi \frac{L}{\mu} \log \frac{1}{\varepsilon} \right)$ \tnote{{\color{blue}(5)}}}   &  \cellcolor{bgcolor2}{$\Omega \left( \myblue{\sqrt{n}}\frac{L}{\mu} \log \frac{1}{\varepsilon} \right)$} & \cellcolor{bgcolor2}{}
    \\\hline 
    \end{tabular}   
    \begin{tablenotes}
    {\small   
    \item [] \tnote{{\color{blue}(1)}} for saddle point problems; \tnote{{\color{blue}(2)}} deterministic; \tnote{{\color{blue}(3)}} stochastic, but not finite sum; \tnote{{\color{blue}(4)}}  convex-concave (monotone) case (we re-analyzed for strongly monotone case);
    \tnote{{\color{blue}(5)}}$B$-connected graphs \cite{nedich2016geometrically} are also considered. For simplicity in comparison with other works, we put $B = 1$. To get estimates for $B \neq 1$, one need to change $\chi$ to $B \chi$; \tnote{{\color{blue}(6)}} can include additional factors such as $n \log \frac{1}{\varepsilon}$, $\chi \log \frac{1}{\varepsilon}$, for full complexities, see details in Section \ref{sec:opt_alg}.
 \item [] {\em Notation:} $\mu$ = constant of strong monotonicity of operator $F$, $L$ = Lipschitz constants of $L_{m,i}$, $\chi$ = characteristic number of the network (see Assumptions \ref{ass:fixed} and \ref{ass:tv}),  $n$ =  size of the local dataset.     
    }
\end{tablenotes}    
    \end{threeparttable}
    }
\vspace{-0.3cm}
\end{table*}

\textbf{(c) Optimal non-distributed/centralized algorithms} 

We believe it is notable that despite the generality of our setup and algorithms, {\bf our results, when specialized to handle this simpler case, improve upon the current state-of-the-art results in the non-distributed/centralized setting.} In particular, unlike existing methods, our algorithms support {\em batching}: while the complexity of the best available algorithms grows with the batch size, our algorithms are not sensitive to this. This property is of crucial importance when working in the large batch mode, which is used in the practice \cite{brock2018large,zhu2019freelb,you2019large}. See Table \ref{tab:comparison2} (Appendix \ref{sec:tables}).

\textbf{(d) Experiments}

Numerical experiments on bilinear problems and  robust regression problems confirm the practical efficiency of our methods, both in the non-distributed stochastic setup and in the decentralized deterministic one.

\section{Problem Setup and Assumptions}

We write $\la x,y \ra \eqdef \sum_{i=1}^nx_i y_i$ to denote the standard inner product of vectors $x,y\in\R^n$, where $x_i$ corresponds to the $i$-th component of $x$ in the standard basis in $\R^n$. This induces the standard $\ell_2$-norm in $\R^n$ in the following way: $\|x\| \eqdef \sqrt{\la x, x \ra}$. To denote the Kronecker product of two matrices $\mA\in\R^{m\times m}$ and $\mB\in\R^{n\times n}$, we use $A \otimes B \in \R^{nm \times nm}$. The identity matrix of size $n\times n$ is denoted by $\mI_n$. We write $[n]\eqdef \{1,2,\dots,n\}$. $\N$ is the set of positive integers.

\subsection{Variational inequality}

We study variational inequalities (VI) of the  form
\begin{equation}\begin{aligned}
    \label{eq:VI}
    \text{Find} \quad z^* &\in \R^d \quad \text{such that} \quad 
    \langle F(z^*), z - z^* \rangle + g(z) - g(z^*) \geq 0, \quad \forall z \in \R^d,
\end{aligned}\end{equation}
where $F: \R^d \to \R^d $ is an operator, and $g: \R^d \to \R \cup \{ + \infty\}$ is a proper lower semicontinuous convex function. We also assume that $g$ is proximal friendly, i.e. the computation of the operator $\text{prox}_{\rho g}(z) = \arg\min_{y \in \R^d} \{ \rho g(y) + \tfrac{1}{2}\|y-z \|^2\}$ (with $\rho > 0$) is done for free or costs very low.

To showcase the expressive power of the formalism \eqref{eq:VI}, we now give a  few examples of variational inequalities arising in machine learning.

\textbf{Example 1 [Convex minimization].} Consider the convex regularized minimization problem:
\begin{align}
\label{eq:min}
\min_{z \in \R^d} f(z) + g(z),
\end{align}
where $f$ is typically a smooth data-fidelity term, and $g$ a possibly nonsmooth regularizer. 
If we define $F(z) \eqdef \nabla f(z)$, then it can be proved that $z^* \in \dom g$ is a solution for \eqref{eq:VI} if and only if $z^* \in \dom g$ is a solution for \eqref{eq:min}. So, the regularized optimization problem \eqref{eq:min} can be cast as a VI \eqref{eq:VI}.

While minimization problems are widely studied in a separate literature, the next class of problems is much more strongly tied to variational inequalities.

\textbf{Example 2 [Convex-concave saddles].} Consider the convex-concave saddle point problem
\begin{align}
\label{eq:minmax}
\min_{x \in \R^{d_x}} \max_{y \in \R^{d_y}} f(x,y) + g_1 (x) - g_2(y),
\end{align}
where $g_1$ and $g_2$ can also be interpreted as regularizers. 
If we let $F(z) \eqdef F(x,y) = [\nabla_x f(x,y), -\nabla_y f(x,y)]$ and $g(z) = g(x,y) = g_1 (x) + g_2(y)$, then it can be proved that $z^* \in \dom g$ is a solution for \eqref{eq:VI} if and only if $z^* \in \dom g$ is a solution for \eqref{eq:minmax}. So, convex-concave saddle point problems  \eqref{eq:minmax} can be cast as a VI \eqref{eq:VI}.

Saddle point problems are strongly related to variational inequalities. In particular, lower bounds for the former are also valid for the latter. Moreover, upper bounds for variational inequalities are valid for saddle point problems. However, what is perhaps more important is  that these lower and upper bounds match. This is in contrast to minimization, where the lower bounds are weaker. 

\subsection{Decentralized variational inequalities}
We consider the decentralized case of  problem \eqref{eq:VI}, namely we assume that $F$ is distributed across $M$ workers,
\begin{equation}
    \label{eq:distr}
  \textstyle  
  F(z) \eqdef \sum\limits_{m=1}^M F_m(z),
\end{equation}
while each $F_m:\R^d \to \R^d$, $m \in [M]$, has the finite sum structure
\begin{equation}
    \label{eq:fs}
  \textstyle    
   F_m(z) \eqdef   \frac{1}{n}\sum\limits_{i=1}^n F_{m,i}(z).
\end{equation}
The data describing $F_{m}$ being stored on  worker $m$. For example, $F_{m,i}$ can correspond to  the value of the operator on the $i$th data point of the $m$-th dataset.

\subsection{Assumptions} \label{sec:as}

\begin{assumption}[Lipschitzness] \label{as:Lipsh}
Each operator $F_m$ is $L$-Lipschitz continuous, i.e. for all $u, v \in \R^d$ we have
$
\| F_m(u)-F_m(v) \|  \leq L\|u - v\|.
$

Further, the collection of operators is $\Lavg$-average Lipschitz continuous, i.e., for all $u, v \in \R^d$ it holds
$
  \textstyle    
\frac{1}{n}\sum_{i=1}^n \| F_{m,i}(u)-F_{m,i}(v)\|^2  \leq \Lavg^2 \|u - v\|^2.
$
\end{assumption}
In the context of \eqref{eq:min} and \eqref{eq:minmax},  $L$-Lipschitzness of the operator means that the functions $f(z)$ and $f(x,y)$ are $L$-smooth.

\begin{assumption}[Strong monotonicity]\label{as:strmon}
Each operator $F_m$ is $\mu$-strongly monotone, i.e., for all $u, v \in \R^d$ we have
$
\langle F_m(u) - F_m(v); u - v \rangle \geq \mu \| u-v\|^2.
$
\end{assumption}
 In the context of \eqref{eq:min} and \eqref{eq:minmax}, strong monotonicity of $F$ means strong convexity of $f(z)$ and strong convexity-strong concavity of $f(x,y)$.

\subsection{Communication and gossip} \label{sec:comm}

Typically, decentralized communication is realized via a {\em gossip protocol} \cite{XIAO200465,boyd2006randomized,nedic2009distributed}, which is merely  matrix-vector multiplication with a gossip matrix $\mW$, described below, which is different in the fixed and time-varying cases. Let 
$
	\textstyle \cL = \{\z = (z_1,\ldots,z_M)^\top  \in (\R^d)^M: z_1 = \ldots=z_M \}
$
be the {\em consensus space}. 

\begin{assumption}[Fixed network \cite{scaman2017optimal}] \label{ass:fixed}
For a fixed network, communication can be modeled via an undirected connected graph, $\mathcal{G} = (\mathcal{V}, \mathcal{E})$, where $\mathcal{V} = [n]$ are vertices (workers) and  $\mathcal{E} = \{(i,j) \, |\, i,j \in \mathcal{V} \}$ are edges. Note that  $(i,j) \in \mathcal{E}$ if and only if there exists a communication link between agents $i$ and $j$. 
The gossip matrix $\mW$ satisfies the following three assumptions: 1) $\mW$ is symmetric positive semi-definite; 2) $\text{ker}\mW \supset \cL$; 3) $\mW$ is supported on the vertices and edges of the  network only: $w_{i,j} \neq 0$ if and only if $i = j$ or $(i, j)\in \mathcal{E}$.

To characterize the matrix $\mW$, which captures the properties of the network, we denote $\lambda_{\max}(\mW) = 1$ as the maximum eigenvalue of $\mW$, $\lambda^+_{\min}(\mW)$ as the minimum positive eigenvalue of $\mW$, and the characteristic number $\chi = \lambda_{\max}(\mW)/\lambda^+_{\min}(\mW) = 1/\lambda^+_{\min}(\mW)$.
\end{assumption}

\begin{assumption}[Time-varying network \cite{nedich2016geometrically}]  \label{ass:tv}
For a time-varying network, at any moment $t$, communication network can be modeled as a directed $B$-connected graph, $\mathcal{G}(t) = (\mathcal{V}, \mathcal{E}(t))$, where $\mathcal{E}(t) = \{(i,j) \, |\, i,j \in \mathcal{V} \}$ are directed edges. $B$-connectedness means that for any time $t$, the graph $\mathcal{G}_B(t)$ with the set of edges $\bigcup^{t+B-1}_{\tau = t} \mathcal{E}(t)$ is connected. To describe the gossip protocol for time-varying case, we define the multi-consensus gossip matrix
\begin{equation}
    \label{eq:tv_gossip}
   \textstyle       \mW_T (t) = \mI_M - \prod\limits_{\tau = t}^{t+T-1} \mW(\tau).
\end{equation}
One can also observe that multiplication with the matrix $\mW_T$ requires to perform multiplication with $T$
gossip matrices $\mW(t), \ldots, \mW(t+T-1)$, i.e., it requires $T$ decentralized communications. We further assume that the gossip matrices $\mW(t)$ (for $\mathcal{G}(t)$) and $\mW_B(t)$ satisfy: 1) $\mW(t)$ is supported on the nodes and edges of the  network: $w_{i,j}(t) \neq 0$ if and only if $i = j$ or $(i, j)\in \mathcal{E}(t)$; 2) $\text{ker}\mW (t) \supset \cL$; 3) $\text{range}\mW (t) \subset \{\z \in (\R^d)^M: \sum_{m=1}^M z_m = 0\}$; 4) there exists a characteristic number $\chi \geq 1$ such that $\|\mW_B (t) z - z \|^2 \leq (1 - \chi^{-1})\|z\|^2$ for all $z \in \text{range}\mW_B(t) $.
\end{assumption}

\section{Lower Bounds}

Our lower bounds apply to a specific class of algorithms which are, loosely speaking, allowed to communicate with neighbors, and compute any local first-order information. We now give a formal definition.

\begin{definition}[Oracle] \label{def:proc}
    Each agent $m$ has its own local memory $\mathcal{M}_{m}$ with  initialization $\mathcal{M}_{m} = \{0\}$. $\mathcal{M}_{m}$ is updated as follows. At each iteration, the algorithm either performs local computations or communicates.\\
    $\bullet$ \textbf{Local computation:} At each local iteration, device $m$ can sample uniformly and independently batch $S_m$ of any size $b$ from $\{F_{m,i}\}$ and adds to its $\mathcal{M}_m$ a finite number of points $z$, satisfying 
    \begin{equation}\label{eq:oracle-opt-step}
       \textstyle       z  \in \text{span} \left\{z'~,~ \sum_{i_m \in S_m} F_{m, i_m}(z''), \text{prox}_{\rho g}\left(\text{span} \left\{z'''~,~ \sum_{i_m \in S_m} F_{m, i_m}(z'')\right\}\right)\right\}
\end{equation}
    for $z', z'', z''' \in \mathcal{M}_{m}$ and $\rho > 0$. Such a call needs $b$ local computations to collect the batch. Batch of  size $n$ represents  $F_m$;\\    
    $\bullet$ \textbf{Communication:} Upon communication rounds among the neighboring nodes, and at  communication time $t$, $\mathcal{M}_{m}$ is updated according to
    \begin{equation}\begin{aligned}\label{eq:oracle-comm}
     \textstyle         \mathcal{M}_{m} \eqdef \text{span}\left\{\bigcup _{(i,m) \in \mathcal{E}(t)} \mathcal{M}_{i} \right\}.
   \end{aligned}\end{equation}
    $\bullet$ \textbf{Output:} 
    The final global output is calculated as 
 $        \hat z \in \text{span}\left\{\bigcup _{m=1}^M \mathcal{M}_{m} \right\}.
$
\end{definition} 
The structure of the above definition is typical for distributed lower bounds \cite{scaman2017optimal} and for stochastic lower bounds \cite{hendrikx2020optimal}. In particular, Definition \ref{def:proc} includes all the approaches for working with stochastic  problems, such as \algname{SGD} or variance reduction techniques (\algname{SVRG}, \algname{SARAH}). Note that while our algorithm can invoke the  deterministic oracle (full $F_m$) in local computations, in the work on lower bounds in the non-distributed case \cite{han2021lower}, there is no such a possibility. This narrows the class of algorithms for which results of \cite{han2021lower} are valid. In particular, they can not do \algname{SVRG}-type updates.

\begin{theorem}[Lower bound - fixed network] \label{th:lower_fixed}
For any $\Lavg \geq \mu >0$ and $\chi \geq 1$, $n \in \N$ and
$K, N \in \N$, there exists a decentralized variational inequality  (satisfying Assumptions \ref{as:Lipsh} and \ref{as:strmon}) on ${\R}^{d}$ (where $d$ is sufficiently
large) with $z^* \neq 0$ over a fixed network (satisfying Assumption \ref{ass:fixed}) with a gossip matrix $\mW$ and characteristic number $\chi$, such that for any output $\hat z$ of any procedure (Definition \ref{def:proc}) with $K$ communication rounds and $N$ local computations,  it holds that $\EE[\|\hat z - z^*\|^2]$ is
\begin{equation*}
   \textstyle  \Omega\left(\exp\left(  -\frac{80}{1 +  \sqrt{\frac{2L^2}{\mu^2} + 1}} \cdot \frac{K}{\sqrt{\chi}}\right)  R_0^2\right)  ~~ \text{and} ~~  \Omega\left(\exp\left(-\frac{16}{n +  \sqrt{\frac{2n \Lavg^2}{\mu^2} + n^2}}\cdot N\right) R_0^2\right),
\end{equation*}
where $R_0^2 = \|z^0 - z^* \|^2$ and $L = \tfrac{\Lavg}{\sqrt{n}}$.
\end{theorem}
\begin{corollary} \label{cor:lower_fixed}
In the setting of Theorem~\ref{th:lower_fixed}, the number of communication rounds and local computations required to obtain an $\varepsilon$-solution (in expectation) is lower bounded by
\begin{align*}
    \textstyle    \Omega\left( \sqrt{\chi}\left(1 + \frac{L}{\mu}\right) \cdot  \log \left(\frac{R_0^2}{\varepsilon}\right)\right) ~~\text{and} ~~
     \Omega\left(\left(n + \sqrt{n}\cdot \frac{\Lavg}{\mu}\right) \cdot  \log \left(\frac{R_0^2}{\varepsilon}\right)\right), ~~\text{respectively.}
\end{align*}
\end{corollary}
\begin{theorem}[Lower bound - time varying network] \label{th:lower_tv}
For any $\Lavg \geq \mu >0$ and $\chi \geq 3$, $n \in \N$ and
$K, N \in \N$, there exist a decentralized variational inequality  (satisfying Assumptions \ref{as:Lipsh} and \ref{as:strmon}) on $\R^{d}$ (where $d$ is sufficiently
large) with $z^* \neq 0$ over a time-varying network (satisfying Assumption \ref{ass:tv}) with a sequence of gossip matrices $\mW(t)$ and characteristic number $\chi$, such that for any output $\hat z$ of any procedure (Definition \ref{def:proc}) with $K$ communication rounds and $N$ local computations, it holds that $\EE[\|\hat z - z^*\|^2]$ is
\begin{equation*}
   \textstyle \Omega\left(\exp\left(  -\frac{64}{\left(1 +  \sqrt{\frac{2L^2}{\mu^2} + 1}\right)} \cdot \frac{K}{B\chi}\right)  R_0^2\right) ~~ \text{and} ~~
   \Omega\left(\exp\left(-\frac{16}{n +  \sqrt{\frac{2n \Lavg^2}{\mu^2} + n^2}}\cdot N\right) R_0^2\right),
\end{equation*}
where $R_0^2 = \|z^0 - z^* \|^2$ and $L = \tfrac{\Lavg}{\sqrt{n}}$.
\end{theorem}
\begin{corollary} \label{cor:lower_tv}
In the setting of Theorem~\ref{th:lower_tv}, the number of communication rounds and local computations required to obtain an $\varepsilon$-solution (in expectation) is lower bounded by
\begin{align*}
    \textstyle    \Omega\left( B\chi\left(1 + \frac{L}{\mu}\right) \cdot  \log \left(\frac{R_0^2}{\varepsilon}\right)\right) ~~\text{and}~~    \Omega\left(\left(n + \sqrt{n}\cdot \frac{\Lavg}{\mu}\right) \cdot  \log \left(\frac{R_0^2}{\varepsilon}\right)\right), ~~\text{respectively.}
\end{align*}
\end{corollary}
See proofs for Theorems \ref{th:lower_fixed} and \ref{th:lower_tv} in Appendix \ref{sec:pr_lb}. The proof uses the idea (example of bad functions) of non-distributed deterministic lower bounds from \cite{zhang2019lower}. This idea is further extended to distributed stochastic VIs.

Note that in the time-varying case, the lower bounds for communication differ by the constant $B$ from the estimates that were previously encountered in the literature \cite{beznosikov2021}. This is due to the fact that we consider a more general setup with a $B$-connected graph (see Assumption \ref{ass:tv}), while the existing literature on lower bounds focuses on the simpler $B=1$ case.

\section{Optimal Algorithms} \label{sec:opt_alg}

Since in the decentralized gossip protocol each local worker $m$ stores its own $z_m$ vector, we consider the problem
\begin{align}
    \label{eq:VI_lift}
    \text{Find} \quad \z^* \in (\R^d)^M \quad & \text{such that} \quad 
    \langle \F(\z^*), \z - \z^* \rangle + \g(\z) - \g(\z^*) \geq 0, \forall \z \in (\R^d)^M,
\end{align}
where we use new notation: $\z = (z_1,\ldots,z_M)^\top$ and $\z^* = (z^*_1,\ldots,z^*_M)^\top$. Additionally, here we introduce the lifted operator $\F : (\R^d)^M \to (\R^d)^M $ given as
$
\textstyle \F(\z) = (F_1(z_1),\ldots,F_M(z_M))^\top,
$
and the lifted operator $\g : (\R^d)^M \to \R \cup \{ + \infty\}$ defined by
$
  \textstyle    \g(\z) = \frac{1}{M}\sum_{m=1}^M g(z_m).
$

One can note that \eqref{eq:VI_lift} is a set of $M$ unrelated variational inequalities with their own variables. 
But the original problem \eqref{eq:VI} + \eqref{eq:fs} is a sum of variational inequalities with the same variables:
$\sum_{m=1}^M \left[\langle F(z^*), z - z^* \rangle + \tfrac{1}{M}g(z) - \tfrac{1}{M}g(z^*)\right]$.
To eliminate this issue and move on to problem \eqref{eq:VI} + \eqref{eq:fs},  it is easy to get the following modification of \eqref{eq:VI_lift}
\begin{align}
    \label{eq:VI_new}
    \text{Find} \quad \z^* &\in \cL \quad \text{such that} \quad 
    \langle \F(\z^*), \z - \z^* \rangle + \g(\z) - \g(\z^*) \geq 0, ~~ \forall \z \in \cL,
\end{align}
where $\cL$ is the consensus space. Problem \eqref{eq:VI_new} is equivalent to  \eqref{eq:VI} + \eqref{eq:fs}.
Due to Assumptions \ref{as:Lipsh} and \ref{as:strmon}, $\F$ is $L$-Lipschitz continuous, 
$\Lavg$-average Lipschitz continuous and $\mu$-strongly monotone.

\subsection{Fixed networks}

We present Algorithm~\ref{alg:vrvi} for fixed networks. In Appendix \ref{sec:dis_int} we give a discussion and intuition. In particular, we give the deterministic variant as well as the non-distributed version of Algorithm~\ref{alg:vrvi}. 
The next result gives the iteration complexity of Algorithm~\ref{alg:vrvi}.
\begin{wrapfigure}[27]{r}{8.2cm}
\vspace{-0.7cm}
\begin{minipage}{0.6\textwidth}
\begin{algorithm}[H]
	\caption{}
	\label{alg:vrvi}
	\begin{algorithmic}[1]
		\STATE {\bf Parameters:}  Stepsizes $\eta, \theta>0$, momentums $\alpha, \beta, \gamma$, batchsize $b \in \{1,\ldots,n\}$, probability $p  \in (0,1)$
        \STATE {\bf Initialization:} Choose  $\z^0 = \w^0 \in (\dom g)^M$, $\y^0 \in \cL^\perp$. Put $\z^{-1} = \z^0, \w^{-1} = \w^0$, $\y^{-1} = \y^0$
		\FOR{$k=0,1,2\ldots$}
			\STATE Sample $j_{m,1}^k, \ldots,j_{m,b}^k$ independently from $[n]$
			\STATE $S^k = \{j_{m,1}^k, \ldots,j_{m,b}^k\}$
			\STATE Sample $j_{m,1}^{k+1/2}, \ldots,j_{m,b}^{k+1/2}$ independently from $[n]$
			\STATE $S^{k+1/2} = \{j_{m,1}^{k+1/2}, \ldots,j_{m,b}^{k+1/2}\}$
			\STATE $\delta^k = \frac{1}{b}\sum_{j\in S^k} \Big(\F_j(\z^k) - \F_j (\w^{k-1}) $ \\\hspace{1.8cm}  $+ \alpha[\F_j(\z^k) - \F_j(\z^{k-1})]\Big) + \F(\w^{k-1})$ \label{vrvi:line:Delta}
			\STATE $\Delta^k = \delta^k - (\y^k + \alpha(\y^k - \y^{k-1}))$ \label{vrvi:line:g}
			\STATE $\z^{k+1} = \prox_{\eta \g} (\z^k + \gamma (\w^k - \z^k)- \eta \Delta^k)$\label{vrvi:line:x}
			\STATE $\Delta^{k+1/2} = \frac{1}{b}\sum_{j\in S^{k+1/2}} \left(\F_j(\z^{k+1}) - \F_j (\w^{k})  \right)$\\\hspace{5.9cm} $+ \F(\w^{k})$  \label{vrvi:line:Delta1/2} 
			\STATE $\y^{k+1} = \y^k - \theta (\mW \otimes \mI_d )(\z^{k+1} - \beta(\Delta^{k+1/2} - \y^k))$\label{dvi:line:y}
			\STATE $\w^{k+1} = \begin{cases}
			\z^{k},& \text{with probability }p\\
			\w^{k},& \text{with probability }1-p
			\end{cases}$\label{vrvi:line:w}
		\ENDFOR
	\end{algorithmic}
\end{algorithm}
\vspace{-0.4cm}
$^*\F_j(\z) = (F_{1,j_{1, l}}(z_1), \ldots, F_{M,j_{M, l}} (z_M))^T$, $l \in \{1,\ldots, b\}$
\noindent\rule{8.4cm}{0.4pt}
\end{minipage}
\end{wrapfigure}

\begin{theorem}[Upper bound - fixed network] \label{th:ALg1_conv}
Consider the problem \eqref{eq:VI_new} (or \eqref{eq:VI} + \eqref{eq:fs}) under Assumptions~\ref{as:Lipsh} and \ref{as:strmon} over a fixed graph $\mathcal{G}$ (Assumption \ref{ass:fixed}) with a gossip matrix $\mW$. Let  $\{\z^k\}$ be the sequence generated by Algorithm~\ref{alg:vrvi} with tuning of $\eta, \theta, \alpha, \beta, \gamma$  as described in Appendix \ref{sec:pr_oa_fixed}. Then, given $\varepsilon>0$, the number of iterations for 
$\EE[\|\z^k - \z^*\|^2] \leq \varepsilon$ is 
\begin{equation*}
  \textstyle        \cO\left( \left[\frac{1}{p} +\chi + \frac{1}{\sqrt{pb}}\frac{\Lavg}{\mu}+  \sqrt{\chi} \frac{L}{\mu}\right] \log \frac{1}{\varepsilon} \right).
\end{equation*}
\end{theorem}
See the proof in Appendix \ref{sec:pr_oa_fixed}.
Let us discuss the results of Theorem. First of all, we are interested in how to obtain the complexity of communications and local computations from iterative complexity. At each iteration we require (in average) $\mathcal{O}(b + p n)$ local computations, because we need to store batch $b$ twice and with probability $p$ we update the point $\w^{k+1}$ by $\z^k$, this requires calculating the full $\F$ in the next iteration. Then, as the optimal $p$, one can choose $p \sim \nicefrac{b}{n}$. Then with such choice of $p$ we have the following local and  communication complexities
\begin{equation*}
    \textstyle      \cO\left( \left[n + b\chi + \sqrt{n}\frac{\Lavg}{\mu}+  b\sqrt{\chi} \frac{L}{\mu} \right] \log \frac{1}{\varepsilon}\right) \quad \text{and} \quad 
  \textstyle        \cO\left( \left[\frac{n}{b} + \chi + \frac{\sqrt{n}}{b}\frac{\Lavg}{\mu}+  \sqrt{\chi} \frac{L}{\mu}\right] \log \frac{1}{\varepsilon} \right),
\end{equation*}
respectively (since at each iteration Algorithm~\ref{alg:vrvi} performs $\mathcal{O}(1)$ communications).

Hence, with $b = 1$ we have the complexities the same as in Table \ref{tab:comparison0}. Depending on $\max\{\sqrt{n}; \sqrt{\chi}\}$, we have the optimality of either local communications or decentralized communications. One can note that it is enough to take $b \geq \nicefrac{\Lavg \sqrt{n} }{L}$ and guarantee the optimal communication complexity (see Corollary \ref{cor:lower_fixed}), but we have non-optimality in local iterations.

To make the algorithm optimal both in terms of communications and local computations, we need to slightly modify it. One can make it using Chebyshev acceleration (see Algorithm~\ref{alg:chebyshev_gossip} in Appendix \ref{sec:cheb}). Following  \cite{scaman2017optimal}, we can construct a polynomial $P$ such that 1) $P(\mW)$ is a gossip matrix, 2) multiplication by $P(\mW) \otimes \mI_d$ requires $\sqrt{\chi(\mW)}$ multiplications by $\mW$ (i.e. $\sqrt{\chi(\mW)}$ communication rounds)  3) $\chi(P(\mW)) \leq 4$. Then we can modify Algorithm~\ref{alg:vrvi} by replacing $\mW$ by $P(\mW)$ and get
\begin{theorem} [Upper bound - fixed network] \label{th:ALg1_cheb_conv}
Consider the problem \eqref{eq:VI_new} (or \eqref{eq:VI} + \eqref{eq:fs}) under Assumptions~\ref{as:Lipsh} and \ref{as:strmon} over a fixed connected graph $\mathcal{G}$ (Assumption \ref{ass:fixed}) with a gossip matrix $\mW$. Let  $\{\z^k\}$ be the sequence generated by Algorithm~\ref{alg:vrvi} with Chebyshev polynomial $P(\mW)$ as a gossip matrix and with tuning of $\eta, \theta, \alpha, \beta, \gamma$  as described in  Appendix \ref{sec:pr_oa_fixed}. Then, given $\varepsilon>0$, the number of iterations for 
$\EE[\|\z^k - \z^*\|^2] \leq \varepsilon$ is 
\begin{equation*}
     \textstyle     \cO\left( \left[\frac{1}{p} + \frac{1}{\sqrt{pb}}\frac{\Lavg}{\mu}+  \frac{L}{\mu} \right] \log \frac{1}{\varepsilon} \right).
\end{equation*}
\end{theorem}
In this case, the communication complexity of one iteration is $\chi$, and the local complexity (in average) is still $\mathcal{O}(b + p n)$. Then with $p = \nicefrac{b}{n}$ we get the following local and  communication complexities
\begin{equation*}
   \textstyle       \cO\left(\left[ n + \sqrt{n}\frac{\Lavg}{\mu}+  b\frac{L}{\mu} \right] \log \frac{1}{\varepsilon}\right) \quad \text{and} \quad 
  \textstyle        \cO\left(\left[\sqrt{\chi} \frac{n}{b}  + \sqrt{\chi}\frac{\sqrt{n}}{b}\frac{\Lavg}{\mu}+  \sqrt{\chi} \frac{L}{\mu} \right] \log \frac{1}{\varepsilon} \right),
\end{equation*}
respectively. To get optimal results from Table \ref{tab:comparison0} we just need to take $b = \nicefrac{\Lavg \sqrt{n} }{L}$.
\\
\begin{wrapfigure}[34]{r}{8.2cm}
\vspace{-1.1cm}
\begin{minipage}{0.6\textwidth}
\begin{algorithm}[H]
	\caption{}
	\label{2dvi2:alg}
	\begin{algorithmic}[1]
		\STATE {\bf Parameters:}  Stepsizes $\eta_z, \eta_y, \eta_x, \theta>0$, momentums $\alpha, \gamma, \omega, \tau$, parameters $\nu, \beta$, batchsize $b \in \{1,\ldots,n\}$, probability $p  \in (0,1)$
        \STATE {\bf Initialization:} Choose  $\z^0 = \w^0 \in (\dom g)^M$, $\y^0 \in (\R^d)^M$, $\x^0 \in \cL^\perp$. Put $\z^{-1} = \z^0, \w^{-1} = \w^0$, $\y_f = \y^{-1} = \y^0$, $\x_f = \x^{-1} = \x^0$, $m_0 = \textbf{0}^{dM}$
		\FOR{$k=0,1,2,\ldots$}
		\STATE Sample $j_{m,1}^k, \ldots,j_{m,b}^k$ independently from $[n]$
		\STATE $S^k = \{j_{m,1}^k, \ldots,j_{m,b}^k\}$
		\STATE Sample $j_{m,1}^{k+1/2}, \ldots,j_{m,b}^{k+1/2}$ independently from $[n]$
		\STATE $S^{k+1/2} = \{j_{m,1}^{k+1/2}, \ldots,j_{m,b}^{k+1/2}\}$
		\STATE $\delta^k = \frac{1}{b}\sum_{j\in S^k} \Big(\F_j(\z^k) - \F_j (\w^{k-1}) $ \\\hspace{1.8cm}  $+ \alpha[\F_j(\z^k) - \F_j(\z^{k-1})]\Big) + \F(\w^{k-1})$
		\STATE $\Delta_z^k = \delta^k - \nu \z^k - \y^k - \alpha(\y^k - \y^{k-1})$
		\STATE $\z^{k+1} = \prox_{\eta_z \g}(\z^k + \omega (\w^k - \z^k) - \eta_z \Delta_z^k)$
		\STATE $\y_c^k = \tau \y^k + (1-\tau)\y_f^k$
		\STATE $\x_c^k = \tau \x^k + (1-\tau)\x_f^k$
		\STATE $\Delta_y^k = \nu^{-1} (\y_c^k + \x_c^k) + \z^{k+1} + \gamma (\y^k + \x^k + \nu \z^k)$
		\STATE $\delta^{k+1/2} = \frac{1}{b}\sum_{j\in S^{k+1/2}} \left(\F_j(\z^{k+1}) - \F_j (\w^{k})  \right)$\\\hspace{5.9cm} $+ \F(\w^{k})$
		\STATE $\Delta_x^k = \nu^{-1} (\y_c^k + \x_c^k) + \beta(\x^k + \delta^{k+1/2})$
		\STATE $\y^{k+1} = \y^k - \eta_y \Delta_y^k$
		\STATE $\x^{k+1} = \x^k - (\mW_T(Tk) \otimes \mI_d) (\eta_x\Delta_x^k + m^k)$
		\STATE $m^{k+1} = \eta_x\Delta_x^k + m^k$\\\hspace{1.7cm} $- (\mW_T(Tk) \otimes \mI_d) (\eta_x\Delta_x^k + m^k)$
		\STATE $\y_f^{k+1} = \y_c^k + \tau(\y^{k+1} - \y^k)$
		\STATE $\x_f^{k+1} = \x_c^k - \theta(\mW_T(Tk) \otimes \mI_d)(\y_c^k + \x_c^k)$
		\STATE $\w^{k+1} = \begin{cases}
			\z^{k},& \text{with probability }p\\
			\w^{k},& \text{with probability }1-p
			\end{cases}$
		\ENDFOR
	\end{algorithmic}
\end{algorithm}
\end{minipage}
\end{wrapfigure}
In contrast to algorithms of \cite{beznosikov2021distributed, beznosikov2020distributed} (the closest papers in theoretical convergence), our Algorithm~\ref{alg:vrvi} needs  multi-consensus/Chebyshev acceleration for both optimal rates, but can work without these additional procedures. Algorithms \cite{beznosikov2021distributed, beznosikov2020distributed} requires $\mathcal{O}(\sqrt{\chi} \log \varepsilon^{-1})$ iterations for Chebyshev acceleration, which makes the algorithms less practical.

\subsection{Time-varying networks}
We present Algorithm \ref{2dvi2:alg} for time-varying networks. It needs to compute $\mW_T$ using \eqref{eq:tv_gossip}, it requires $T$ communications. In Appendix \ref{sec:dis_int} we give a discussion and intuition of this algorithm. 
The next result gives the iteration complexity of Algorithm~\ref{2dvi2:alg}.

\begin{theorem} [Upper bound - time varying network] \label{th:Alg2_conv}
Consider the problem \eqref{eq:VI_new} (or \eqref{eq:VI} + \eqref{eq:fs}) under Assumptions~\ref{as:Lipsh} and \ref{as:strmon} over a sequence of time-varying graphs $\mathcal{G}(k)$ (Assumption \ref{ass:tv}) with gossip matrices $\mW(k)$. Let  $\{\z^k\}$ be the sequence generated by Algorithm~\ref{2dvi2:alg} with $T \geq B$ and tuning of parameters  as described in Appendix \ref{sec:pr_oa_tv}. Let the choice of $T$ guarantees contraction property (Assumption \ref{ass:tv} point 4) with $\chi(T)$. Then , given $\varepsilon>0$, the number of iterations for 
$\EE[\|\z^k - \z^*\|^2] \leq \varepsilon$ is 
\begin{align*}
   \textstyle \mathcal{\tilde O}\left( \chi^2(T) + \frac{1}{p} + \chi(T)\frac{L}{\mu} + \frac{1}{\sqrt{bp}}\frac{\Lavg}{\mu}\right).
 \end{align*}
\end{theorem}
Note that an important detail of the method is that $T \geq B$. This limitation is due to the fact that network is $B$-connected. In particular, for $B >1$ it can happen that in some communications we use empty graphs. Therefore, the requirement $T \geq B$ is natural to guarantee the contraction property (Assumption \ref{ass:tv}, point 4).
This means that if $B > 1$ we have to use multi-consensus \eqref{eq:tv_gossip}. 

But with $B=T=1$, we can avoid multi-consensus, let us this case first. The same way as in fixed graph case we choose $p = \nicefrac{b}{n}$. Then we get the estimates on communications and local calls
\begin{align*}
  \textstyle     \cO\left( \left[\chi^2 + \frac{n}{b} + \chi\frac{L}{\mu} + \frac{\sqrt{n}}{b}\frac{\Lavg}{\mu}\right] \log \frac{1}{\varepsilon} \right).
\end{align*}
If we put $b=1$, we have the same estimates as in Table~\ref{tab:comparison0}.

Now we consider general case with any $B$. We use a multi-gossip step and take $T > 1$. In particular, let us choose $T = B \cdot \lceil \chi \ln 2\rceil$. Than using \eqref{eq:tv_gossip} and  point 4 of Assumption~\ref{ass:tv}, we can guarantee that
$   \|\mW_T (t) z - z \|^2 \leq \frac{1}{2}\|z\|^2.
$
Therefore, $\chi(T) = 2$, but we need $T$ communications per iteration. With $p=\nicefrac{b}{n}$ and $b = \nicefrac{\Lavg \sqrt{n} }{L}$, the iteration complexity from Theorem \ref{th:Alg2_conv} can be rewritten as follows
\begin{align*}
  \textstyle    \cO\left( \left[1 + \frac{\sqrt{n} L}{\Lavg} + \frac{L}{\mu}\right] \log \frac{1}{\varepsilon} \right). 
\end{align*}
Using that per iteration we make $\mathcal{O}(B \cdot \lceil \chi \ln 2\rceil)$ communications and $\cO\left(\nicefrac{\Lavg \sqrt{n} }{L}\right)$ local computations, we get 
\begin{align*}
  \textstyle    \cO\left( \left[B \chi + B\chi \frac{\sqrt{n} L}{\Lavg} + B\chi \frac{L}{\mu}\right] \log \frac{1}{\varepsilon} \right) ~\text{ commn-s and }~  \cO\left( \left[n + \frac{\Lavg \sqrt{n} }{L} + \sqrt{n}\frac{\Lavg}{\mu}\right] \log \frac{1}{\varepsilon} \right) ~~\text{local calls.}
\end{align*}
These results are reflected in Table \ref{tab:comparison0}.

\section{Experiments}

We now perform several experiments with the goal of corroborating our theoretical results. 
Note though that we are the first who consider the decentralized stochastic (finite-sum) setting for VIs, and hence there are no competing methods. Therefore, we compare the non-distributed finite sum setting and decentralized deterministic setting separately.

\subsection{Variance reduction} \label{sec:vr_exp}

In this section, we compare the main methods for solving strongly monotone non-distributed stochastic (finite-sum) variational inequalities with non-distributed version of our Algorithm \ref{alg:vrvi}.

\textbf{Problem.} We first consider bilinear problem:
\begin{align}
    \label{bilinear}
     \textstyle  \min\limits_{x\in \triangle^d} \max\limits_{y\in \triangle^d} \frac{1}{n} \sum \limits_{i=1}^n x^\top \mA_i y,
\end{align}
where $\triangle^d$ is the unit simplex in $\R^d$. We use the same experimental setup as in \cite{alacaoglu2021stochastic}, in particular we consider 
policeman and burglar matrix from \cite{nemirovski2013mini} and two test matrices from \cite{nemirovski2009robust}.

\begin{wrapfigure}[14]{r}{7cm}
\centering
\vspace{-0.5cm}
\captionof{figure}{Comparison epoch complexities of Algorithm~\ref{alg:vrvi}, \algname{EG-Alc-Alg1}, \algname{EG-Alc-Alg2} and \algname{EG-Car} on \eqref{bilinear} with matrix from \cite{nemirovski2013mini}. Dashed lines give convergence with theoretical parameters, solid lines --  with tuned parameters.}
\label{fig:comp1}
\vspace{-0.3cm}
\begin{minipage}[][][b]{0.5\textwidth}
\centering
\includegraphics[width=0.48\textwidth]{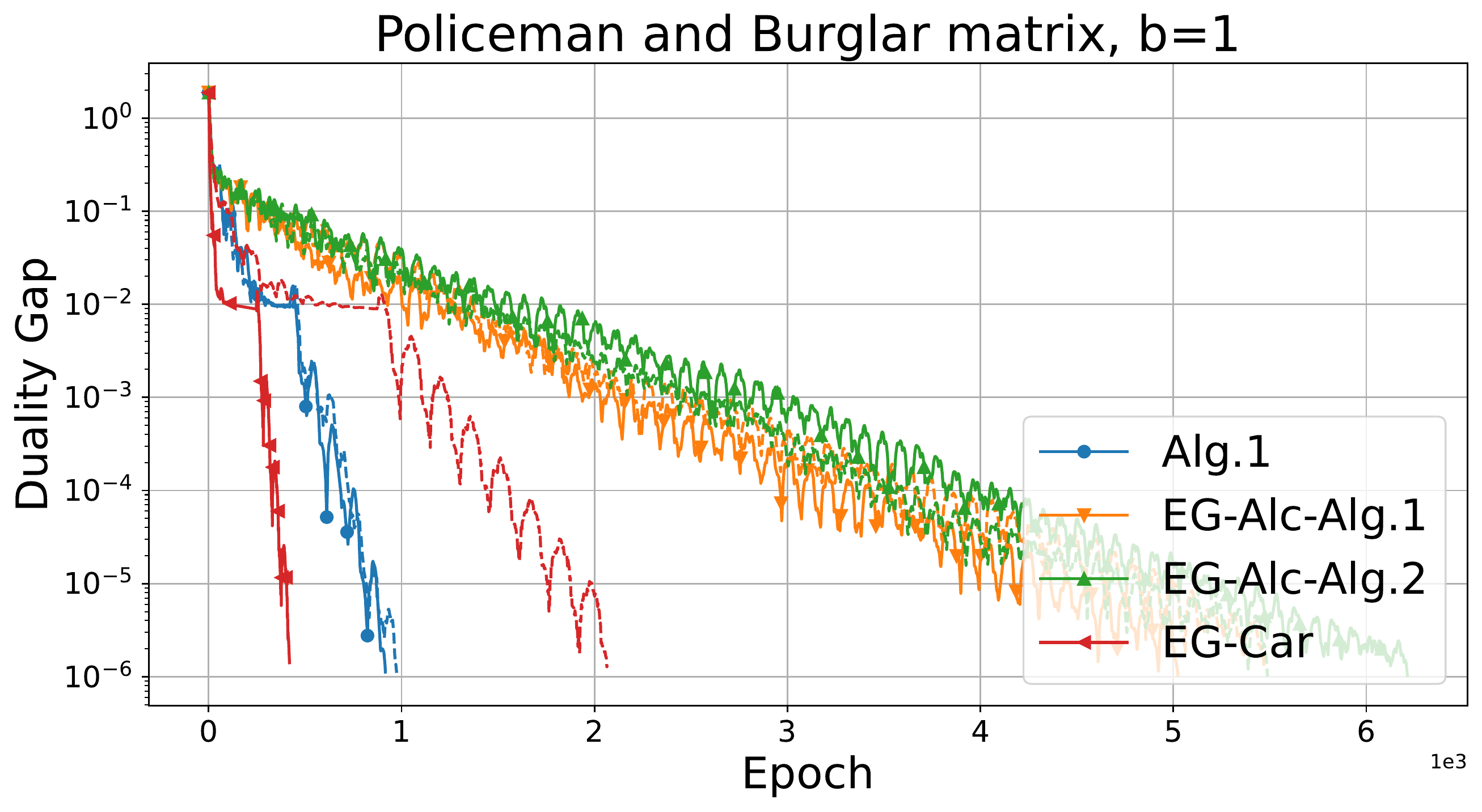}
\includegraphics[width=0.48\textwidth]{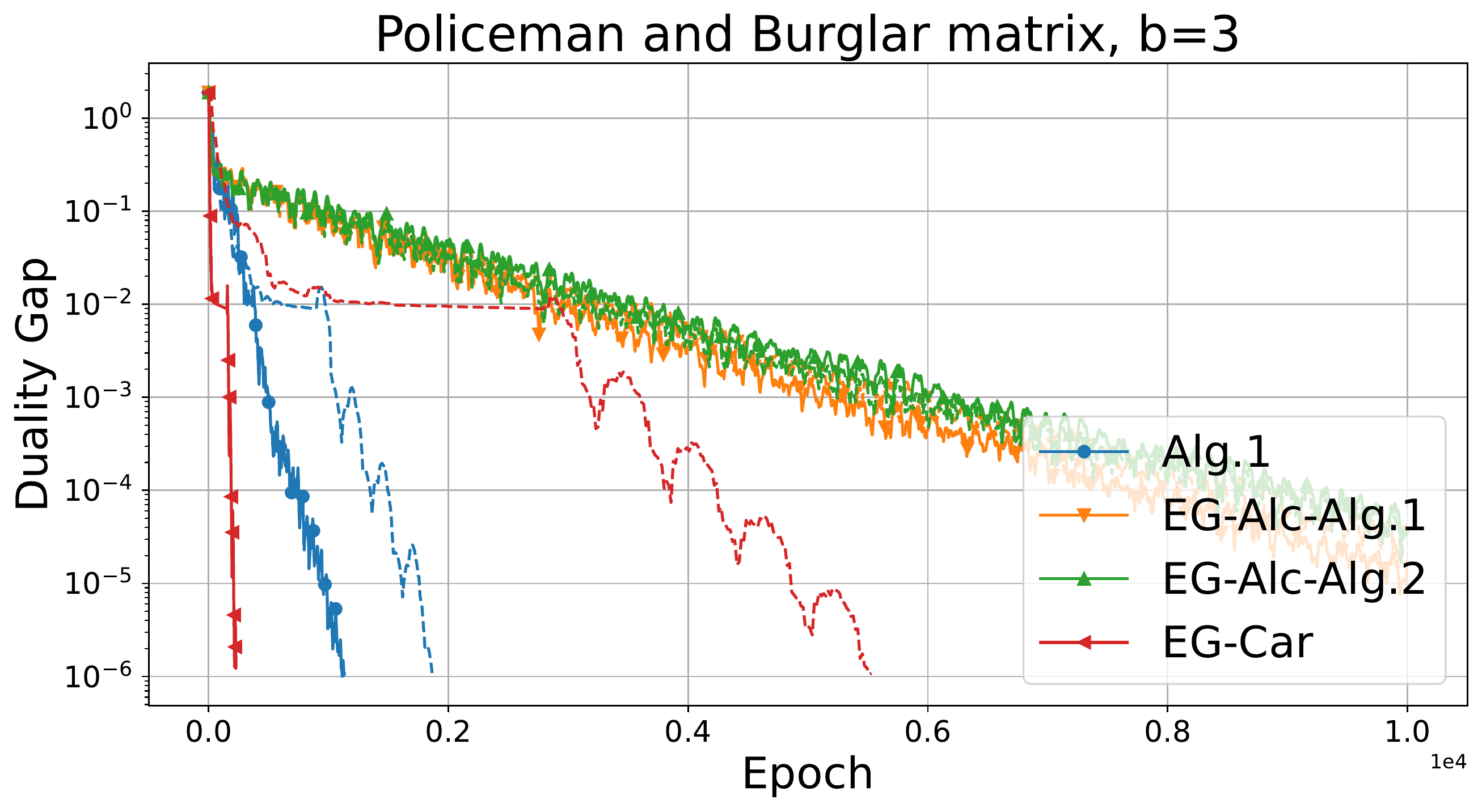}
\includegraphics[width=0.48\textwidth]{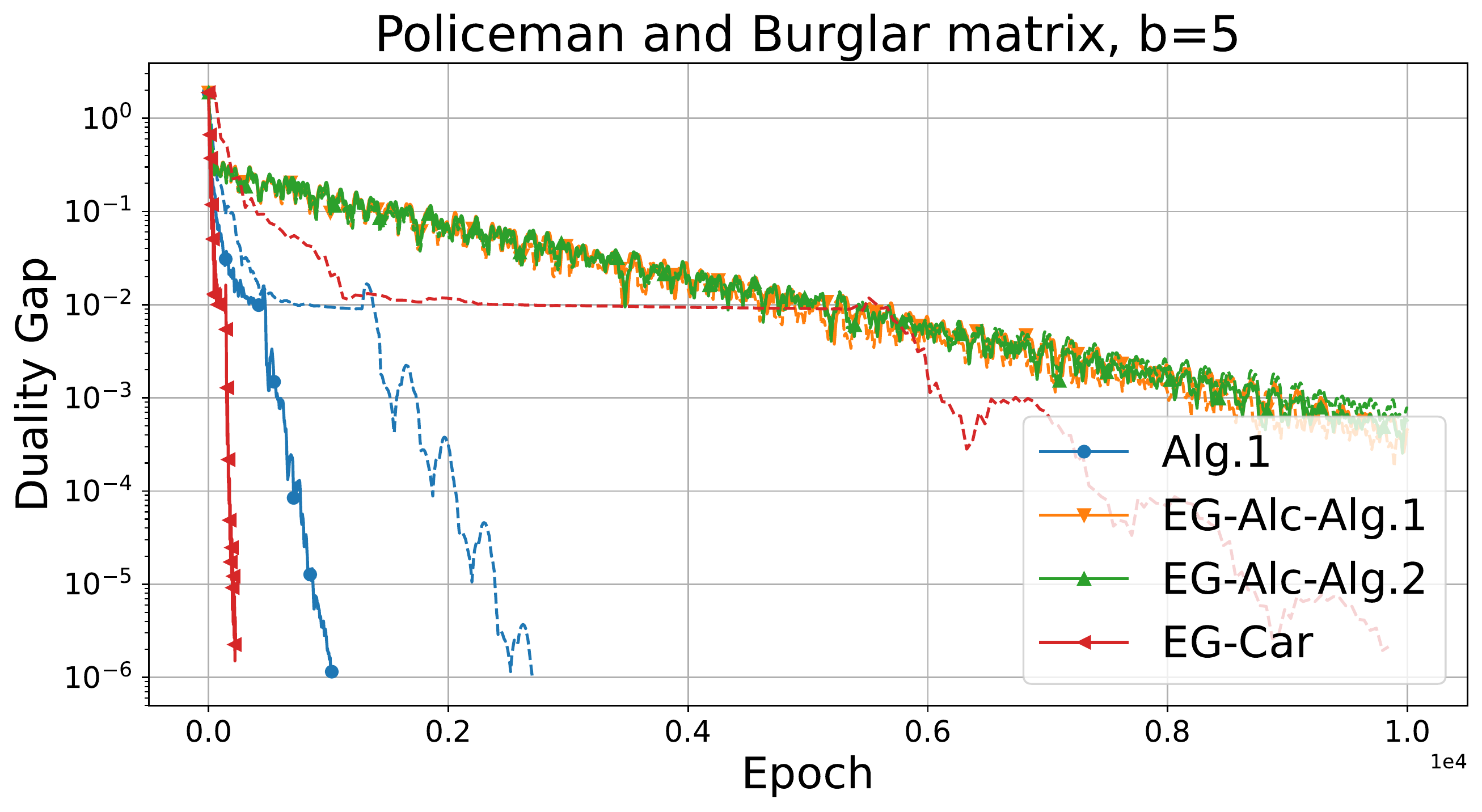}
\includegraphics[width=0.48\textwidth]{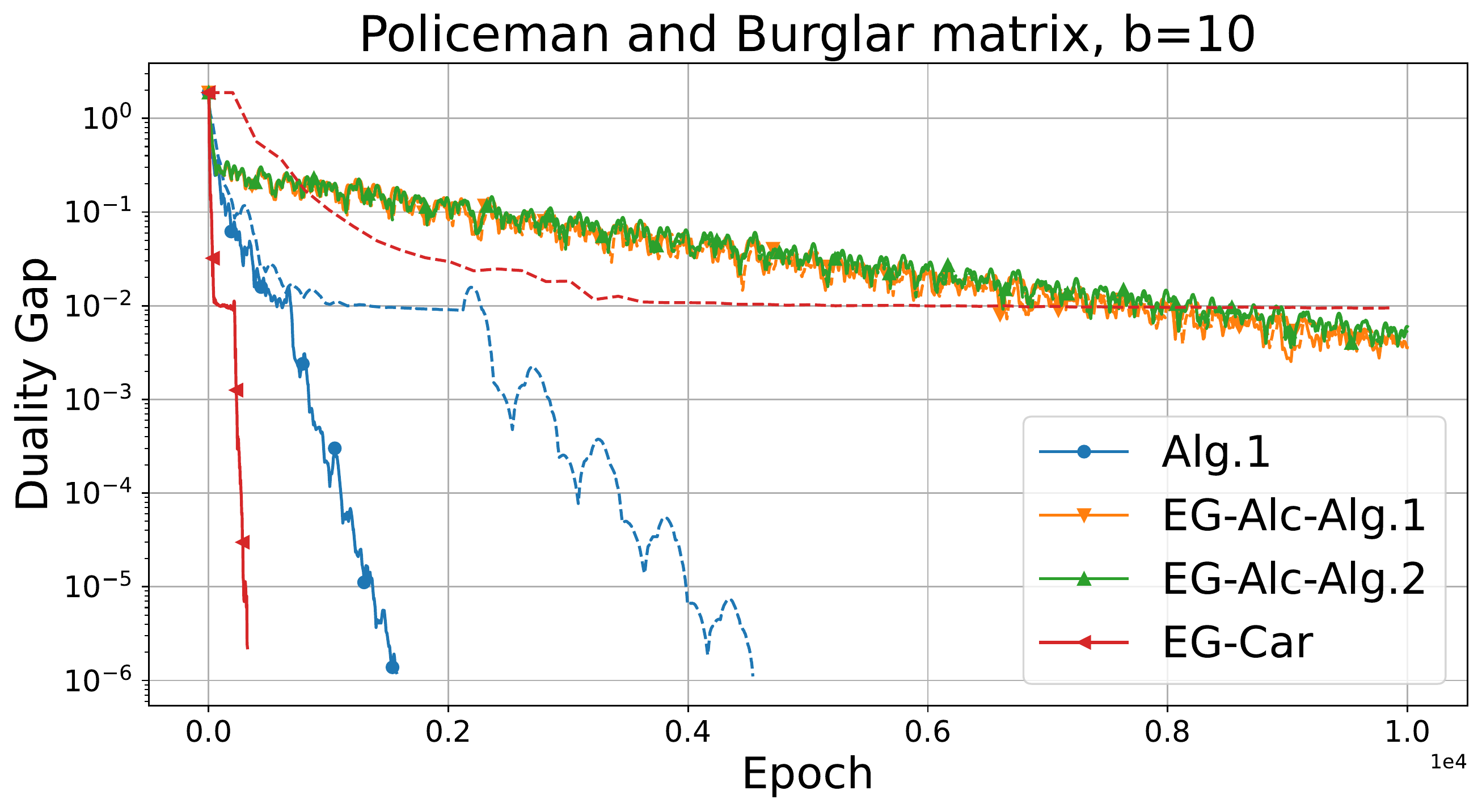}
\end{minipage}
\end{wrapfigure}

\textbf{Setting.} For comparison, we took methods from Table \ref{tab:comparison2}. In particular, we chose \algname{EG-Alc-Alg1} and \algname{EG-Alc-Alg2}  from \cite{alacaoglu2021stochastic}, \algname{EG-Car} from \cite{carmon2019variance}. 
The parameters of all methods are selected in two ways: 1) as described in the theory of the corresponding papers, and 2) tuned for the best convergence. We run all methods with different batch sizes. The comparison criterion is the number of epochs (one full gradient = epoch). 

\textbf{Results.} The plots from Figure \ref{fig:comp1} show that in the case of a theoretical choice of parameters, our Algorithm \ref{alg:vrvi} is ahead of other methods for any batch size (including $b=1$). In the case of tuning parameters, the specialized method from \cite{carmon2019variance} is better than our algorithm. See more experiments with other matrices in Appendix \ref{sec:app_vr_exp}.

\subsection{Decentralized methods}

In this section, we compare the state-of-the-art methods for solving strongly monotone decentralized variational inequalities over fixed and time-varying networks with our Algorithms~\ref{alg:vrvi} and \ref{2dvi2:alg}.

\textbf{Problem.} We now consider  robust linear regression:  
\begin{align}\label{eq:regression}
   \textstyle       \min\limits_{w} \max \limits_{\|r\|\leq e} \frac{1}{N} \sum\limits_{i=1}^N (w^\top (x_i + r_i) - y_i)^2   + \frac{\lambda}{2} \| w\|^2 - \frac{\beta}{2} \|r \|^{2},
\end{align}
where $w$ are model weights, $\{(x_i, y_i)\}_{i=1}^N$ are pairs of the training data,  $r_i$ are noise vectors, and   $\lambda$ and $\beta$ are  regularization parameters. The noises $r_i$ resist training the model, thereby inducing more robustness and stability. 

\textbf{Setting.} For comparison, we took methods from Table~\ref{tab:comparison0} for decentralized problems over fixed and time-varying networks. In particular, we choose \algname{EGD-GT} from \cite{Mukherjee2020:decentralizedminmax}, \algname{EGD-Con} from \cite{beznosikov2020distributed, beznosikov2021} and \algname{Sliding} from \cite{beznosikov2021distributed}. Note that only \algname{EGD-Con} has a theory for fixed and time-varying networks, despite this, we use all methods in both cases. 

For a fair comparison, we consider the deterministic setup, i.e., each worker can compute  full gradients. We take datasets from LiBSVM \cite{chang2011libsvm} and divided unevenly across $M=25$ workers. For communication networks we chose the star, the ring and the grid topologies. For time-varying networks, the topologies remain the same, but the locations of the vertices in them change randomly. 
All methods are tuned for the best convergence. The comparison criterion is the number of communication rounds. 

\subsubsection{Fixed networks} \label{sec:dec_f_exp}

\begin{wrapfigure}[8]{r}{7cm}
\centering
\vspace{-1.3cm}
\captionof{figure}{Comparison communication complexities of Algorithm \ref{alg:vrvi}, \algname{EGD-GT}, \algname{EGD-Con} and \algname{Sliding} on \eqref{eq:regression} over fixed networks.}
\label{fig:comp2}
\vspace{-0.3cm}
\begin{minipage}[][][b]{0.5\textwidth}
\centering
\includegraphics[width=0.48\textwidth]{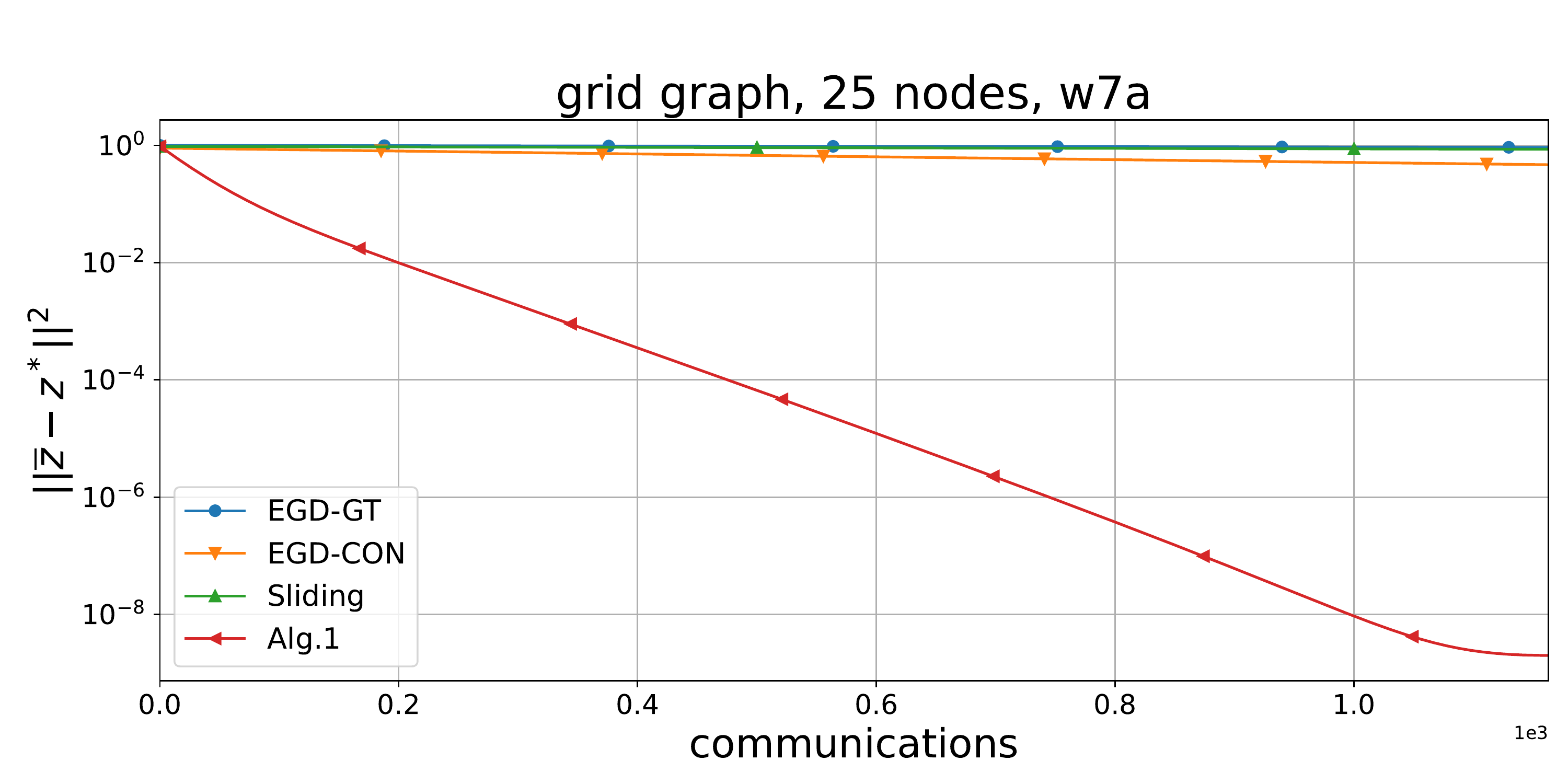}
\includegraphics[width=0.48\textwidth]{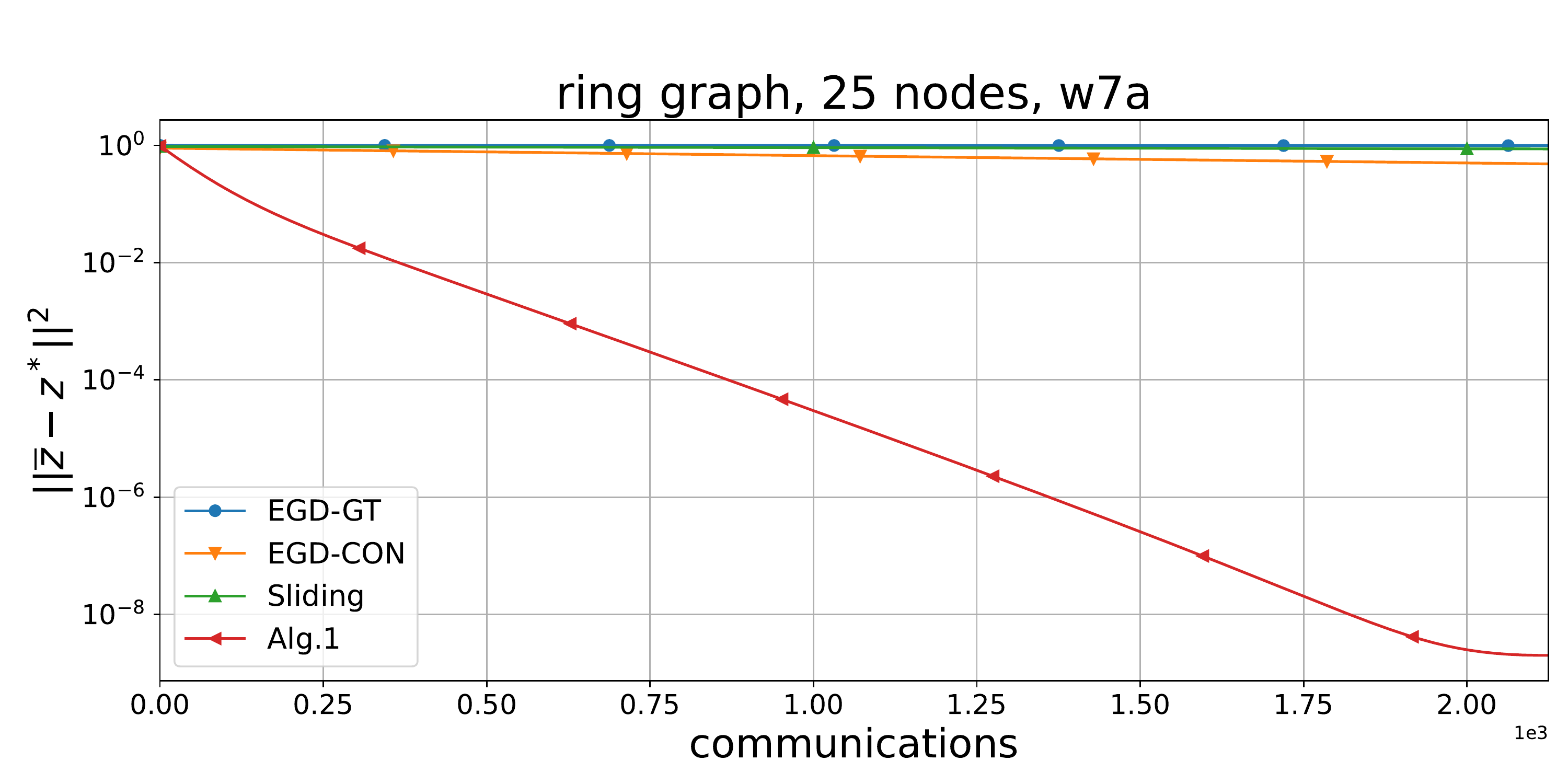}
\includegraphics[width=0.48\textwidth]{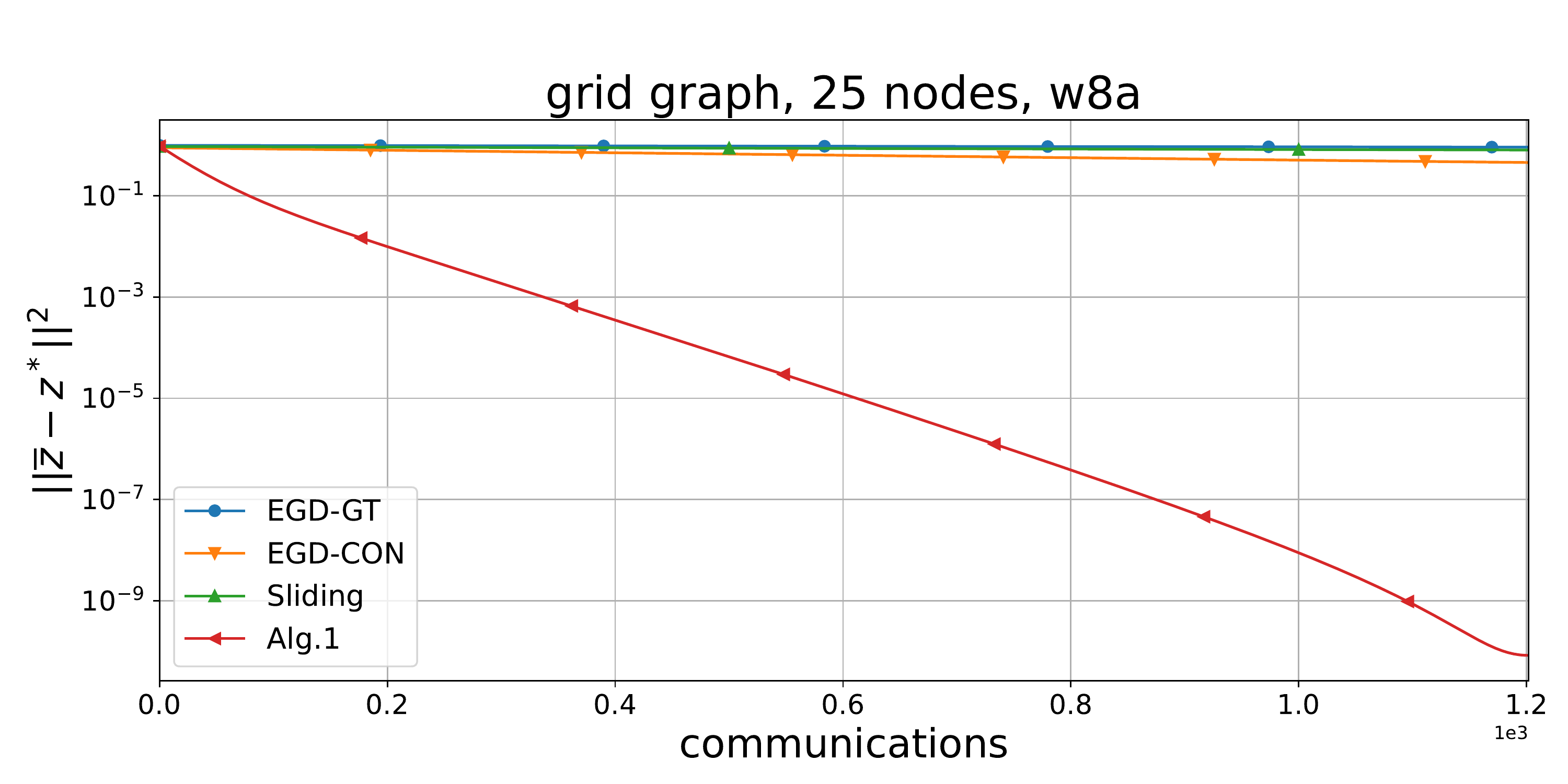}
\includegraphics[width=0.48\textwidth]{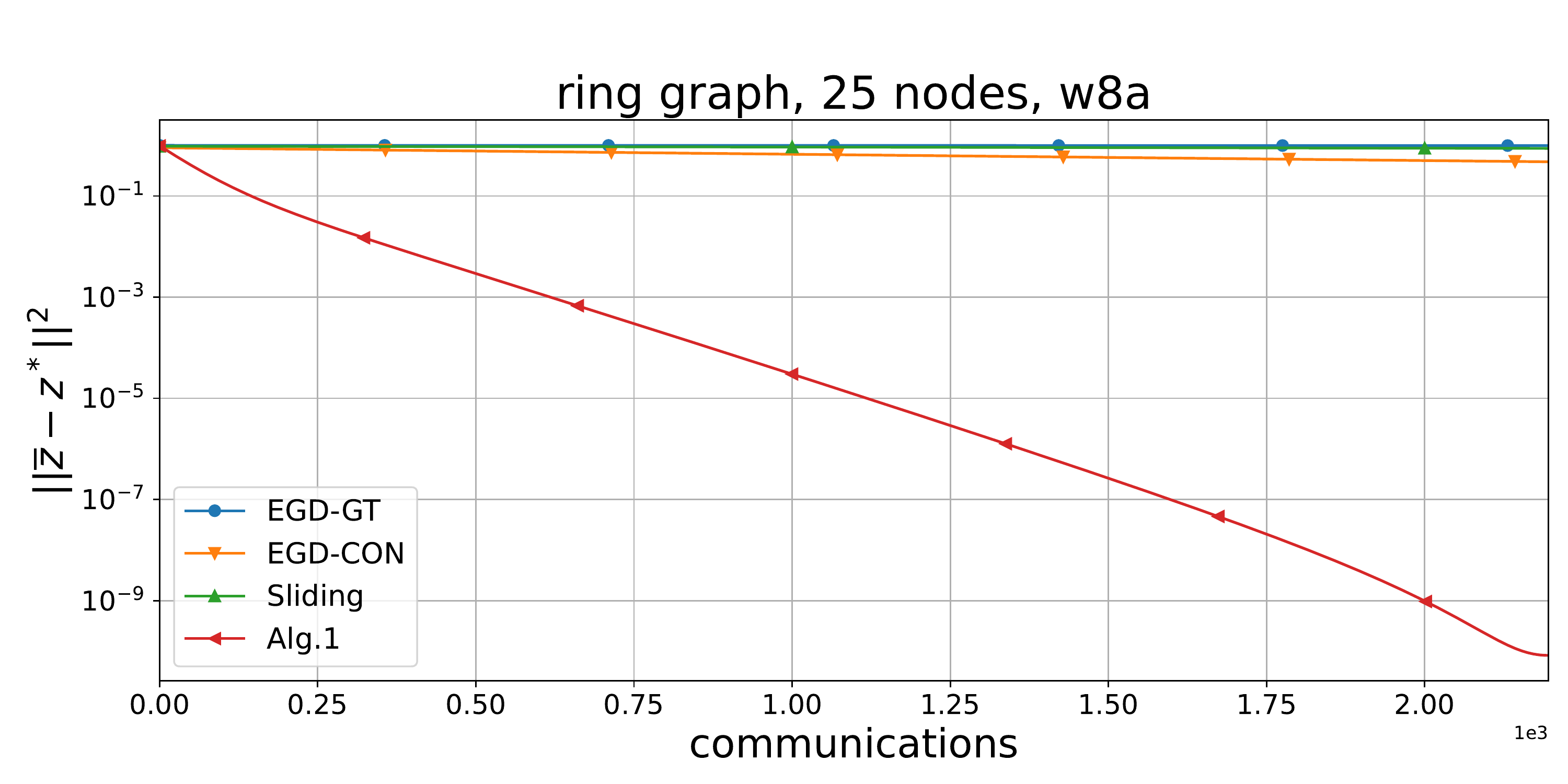}
\end{minipage}
\end{wrapfigure}

\textbf{Results.} The plots from Figure~\ref{fig:comp2} show that our Algorithm~\ref{alg:vrvi} is ahead of other methods. Among other things, it is ahead of \algname{Sliding} from \cite{beznosikov2021distributed}, which has a fast theoretical communication complexity. However, this happens when the dataset is relatively homogeneous and uniformly divided across the devices. In our setting, this is not the case. See more experiments with other datasets in Appendix \ref{sec:app_dec_f_exp}.

\subsubsection{Time-varying networks} \label{sec:dec_tv_exp}

\begin{wrapfigure}[9]{r}{7cm}
\centering
\vspace{-0.7cm}
\captionof{figure}{Comparison communication complexities of Algorithm \ref{2dvi2:alg}, \algname{EGD-GT}, \algname{EGD-Con} and \algname{Sliding} on \eqref{eq:regression} over time-varying networks.}
\label{fig:comp3}
\vspace{-0.3cm}
\begin{minipage}[][][b]{0.5\textwidth}
\centering
\includegraphics[width=0.48\textwidth]{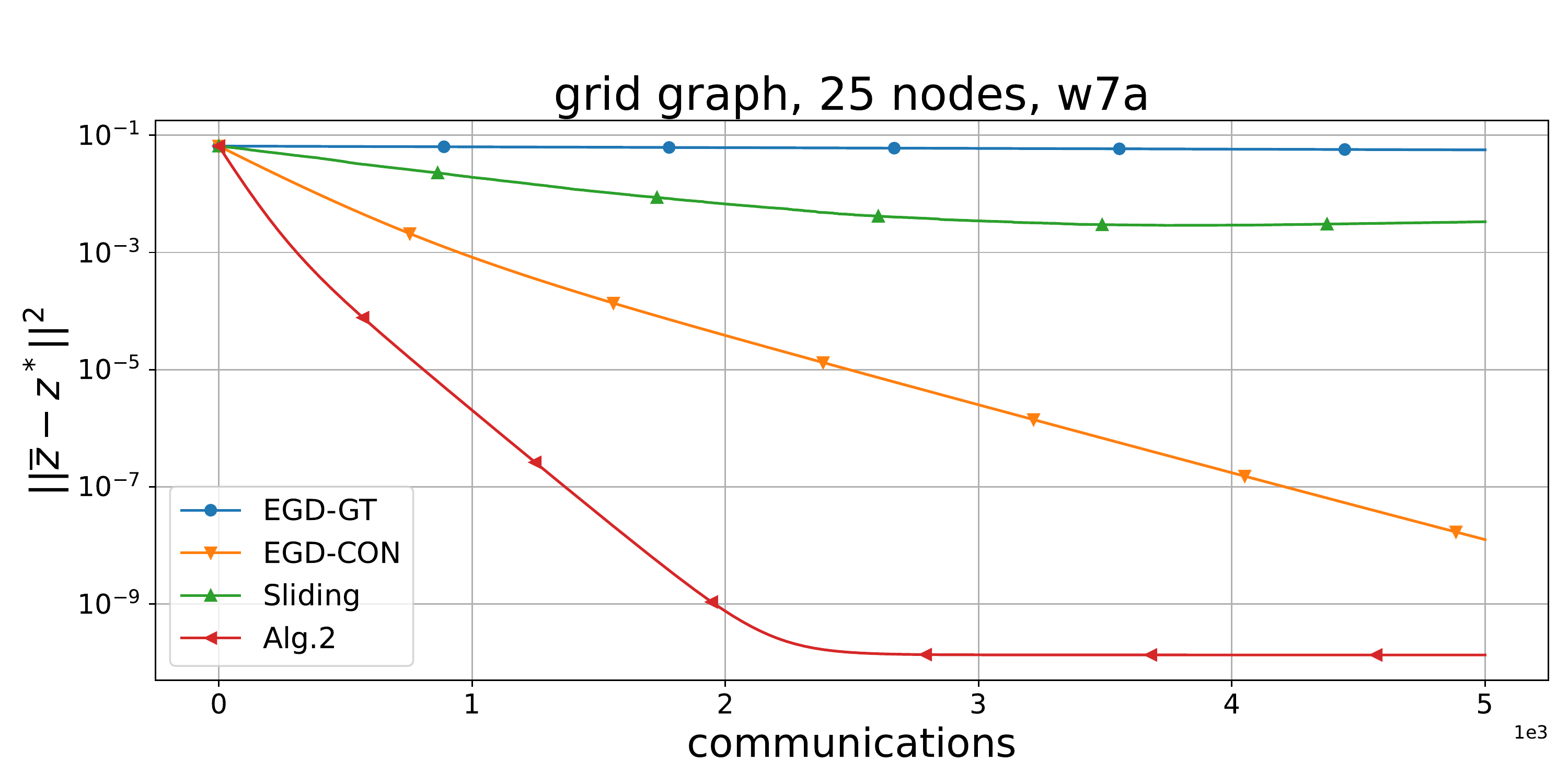}
\includegraphics[width=0.48\textwidth]{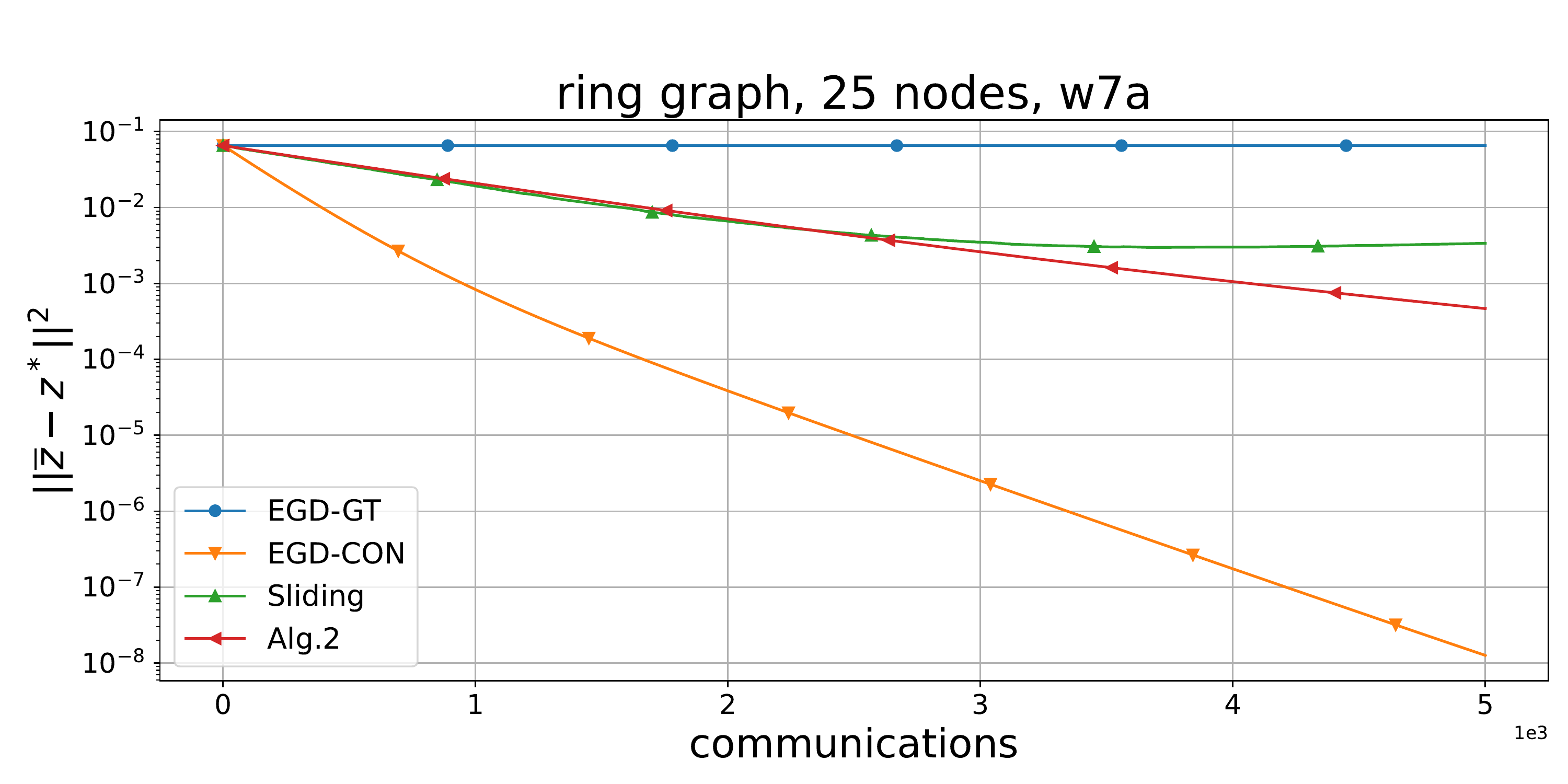}
\includegraphics[width=0.48\textwidth]{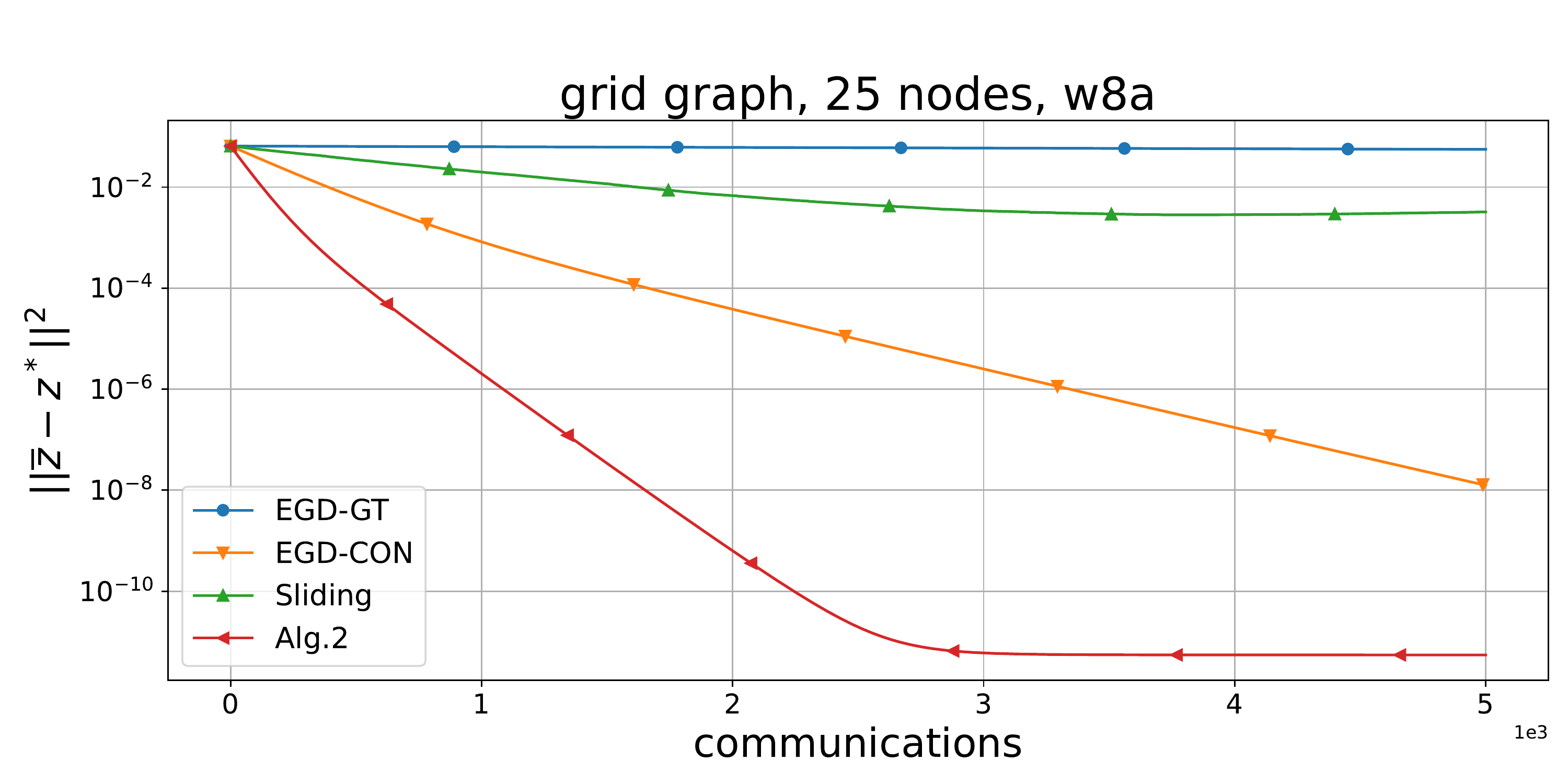}
\includegraphics[width=0.48\textwidth]{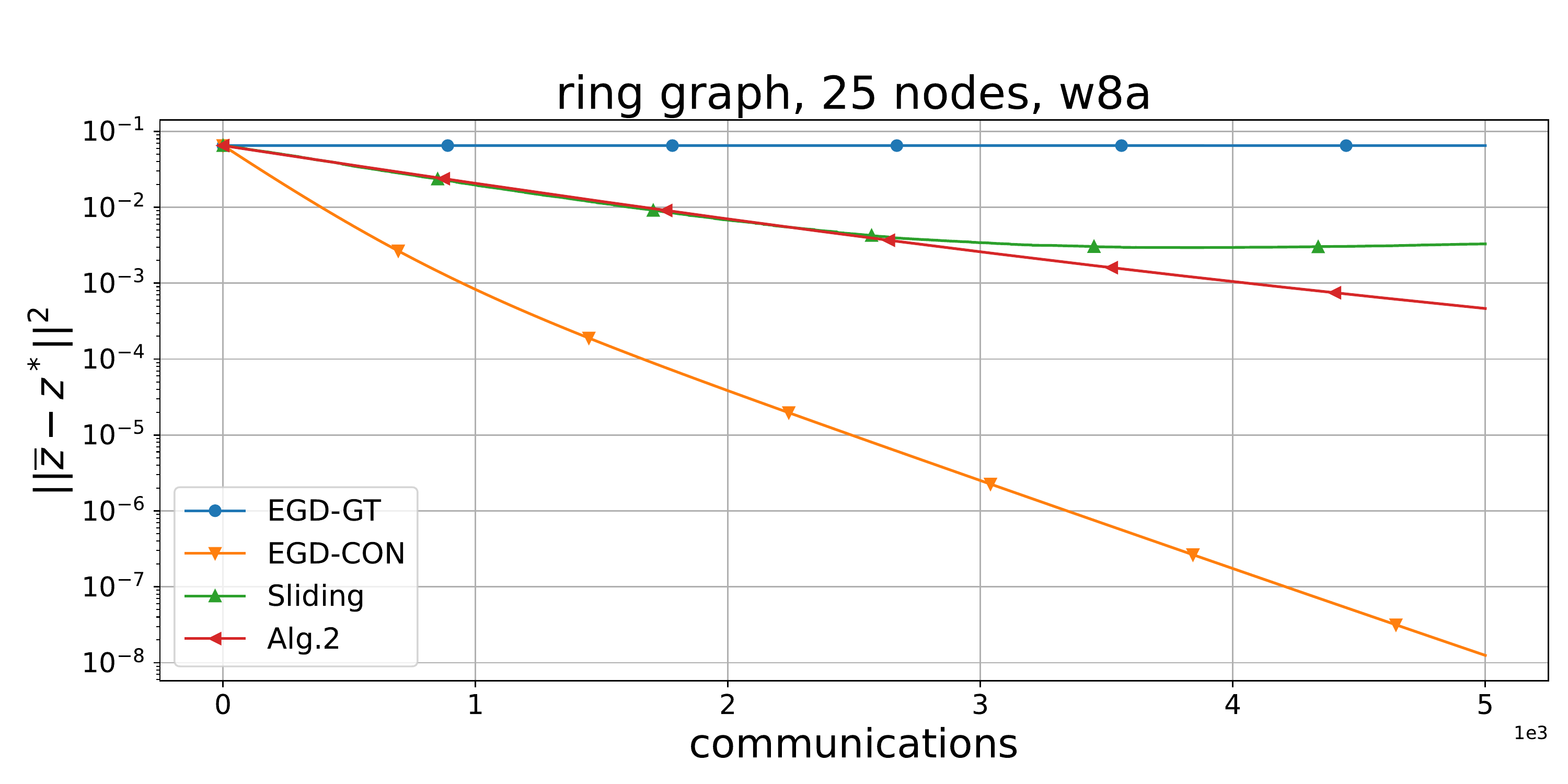}
\end{minipage}
\end{wrapfigure}

\textbf{Results.} The plots from Figure~\ref{fig:comp3} show that our Algorithm~\ref{2dvi2:alg} is ahead of other methods in the case of the grid network. In the case of the ring topology, \algname{EGD-Con} shows the best results. Indeed, such an algorithm in fact implements centralized communications via a decentralized protocol and this approach is not always the fastest, but reliable. See more experiments with other datasets in Appendix \ref{sec:app_dec_tv_exp}.

\section*{Acknowledgments}

This work was supported by a grant for research centers in the field of artificial intelligence, provided by the Analytical Center for the Government of the Russian Federation in accordance with the subsidy agreement (agreement identifier 000000D730321P5Q0002) and the agreement with the Moscow Institute of Physics and Technology dated November 1, 2021 No. 70-2021-00138.


\bibliography{ltr}
\bibliographystyle{plain}

\newpage
\appendix
\onecolumn
\small{
\tableofcontents
}

\newpage

\part*{Supplementary Materials}

\section{Chebyshev acceleration} \label{sec:cheb}

\begin{algorithm}[H]
    \caption{Chebyshev gossip subroutine \cite{scaman2017optimal}}
    \label{alg:chebyshev_gossip}
    \begin{algorithmic}[1]
        \STATE {\bf Input:} $\z, \mW$
        \STATE $c_2 = \frac{\chi + 1}{\chi - 1}$, $ a_0 = 1,~ a_1 = c_2$, $c_3 = \frac{2}{\lambda_{\max}(\mW) + \lambda_{\min}^+(\mW)}$.
        \STATE{$\z^0 = \z,~ \z^1 = c_2(\mI - c_3\mW)\z$}.
        \FOR{$k = 1, \ldots, K - 1$}
            \STATE{$a_{k+1} = 2c_2 a_k - a_{k-1}$}.
            \STATE{$\z^{k+1} = 2c_2(\mI - c_3\mW)\z^k - \z^{k-1}$}.
        \ENDFOR
        \STATE {\bf Output:}{$\z^0 - \frac{\z^K}{a_K}$}. 
    \end{algorithmic}
\end{algorithm}

\section{Discussion and intuition of Algorithms \ref{alg:vrvi} and \ref{2dvi2:alg}} \label{sec:dis_int}

The design of \Cref{alg:vrvi} and \Cref{2dvi2:alg} is based on a combination of two optimization ideas: variance reduction for variational inequalities (Algorithm \ref{alg_non_distr}) and optimal algorithms for decentralized distributed optimization on fixed (Algorithm \ref{alg_fixed}) and time-varying (Algorithm \ref{alg_tv}) networks.

{\bf Variance reduction for variational inequalities.} The variance reduction technique used in the design of \Cref{alg:vrvi,2dvi2:alg} is based on the FoRB with variance reduction \cite{alacaoglu2021stochastic,Yura2021}. However, there are crucial differences with the FoRB algorithm that we now describe. Consider the variance reduced ``gradient'' estimator of FoRB in the case of batch size being equal to 1 ($b=1$). It can be written as follows:
\begin{equation}
	\delta^k = F_j(z^k) - F_j(w^{k-1}) + F(w^k),
\end{equation}
where $j$ is sampled form $\{1,\ldots,n\}$ uniformly at random. This step in some sense combines the variance reduction technique and so-called ``optimistic'' step, which leads to a bad convergence rate $\cO\left(n + \sqrt{bn}\frac{L}{\mu}\log\frac{1}{\varepsilon}\right)$ in the single node setting (see \Cref{tab:comparison2}). Our Algorithm \ref{alg_non_distr} (basic for \Cref{alg:vrvi} and \Cref{2dvi2:alg}) is different. They use the following ``gradient'' estimator (see line~6 of \Cref{alg_non_distr}):
\begin{equation}\label{huj}
	\delta^k  = \underbrace{F_j(z^k) - F_j (w^{k-1}) +  F(w^{k-1})}_{\text{variance reduction}} + \alpha\underbrace{\left(F_j(z^k) - F_j(z^{k-1})\right)}_{\text{``optimistic'' step}}.
\end{equation}
One can observe that our gradient estimator in some sense separates the variance reduction step and the ``optimistic'' step. This leads to the better convergence rate $\cO\left(n + \sqrt{n}\frac{L}{\mu}\log\frac{1}{\varepsilon}\right)$ in the single node setting (see \Cref{tab:comparison2}) and allows to develop the optimal algorithms in the decentralized distributed setting.

{\bf Decentralized distributed optimization.} To provide the optimal algorithms for decentralized stochastic variational inequalities, we use the optimal algorithms for solving decentralized distributed minimization problems for fixed networks \cite{kovalev2020optimal} and for time-varying networks \cite{kovalev2021lower}. These algorithms use Nesterov acceleration \cite{nesterov1983method} which we replace with our ``gradient'' estimator \eqref{huj}. Another important difference is that we extend our results to the composite case, i.e., $g(z)\neq 0$ in the main problem~\eqref{eq:VI} with the help of the proximal operator $\prox_{g}(\cdot)$. This is one of the important contributions of our paper. To the best of our knowledge, optimal algorithms in the decentralized distributed setting exist neither for minimization problems nor for the variational inequalities.

\begin{algorithm}[H]
	\caption{}
	\label{alg_non_distr}
	\begin{algorithmic}[1]
	    \STATE {\bf Parameters:}  Stepsizes $\eta>0$, momentums $\alpha, \gamma$, batchsize $b \in \{1,\ldots,n\}$, probability $p  \in (0,1)$
        \STATE {\bf Initialization:} Choose  $z^0 = w^0 \in \dom g$. Put $z^{-1} = z^0, w^{-1} = w^0$
		\FOR{$k=0,1,2\ldots$}
			\STATE Sample $j_1^k, \ldots,j_b^k$ independently from $\{1,\ldots,m\}$ uniformly at random
			\STATE $S^k = \{j_1^k, \ldots,j_b^k\}$
			\STATE $\Delta^k = \frac{1}{b}\sum_{j\in S^k} \left(F_j(x^k) - F_j (w^{k-1}) + \alpha(F_j(x^k) - F_j(x^{k-1}))\right) + F(w^{k-1})$
			\STATE $x^{k+1} = \prox_{\eta g} (x^k + \gamma (w^k - x^k)- \eta \Delta^k)$
			\STATE $w^{k+1} = \begin{cases}
			x^{k+1},& \text{with probability }p\\
			w^{k},& \text{with probability }1-p
			\end{cases}$
		\ENDFOR
	\end{algorithmic}
\end{algorithm}

\begin{algorithm}[H]
	\caption{}
	\label{alg_fixed}
	\begin{algorithmic}
		\STATE {\bf Parameters:}  Stepsizes $\eta, \theta>0$, momentums $\alpha, \beta, \gamma$
        \STATE {\bf Initialization:} Choose  $\z^0 = \w^0 \in (\dom g)^M$, $\y^0 \in \cL^\perp$. Put $\z^{-1} = \z^0, \w^{-1} = \w^0$, $\y^{-1} = \y^0$
		\FOR{$k=0,1,2,\ldots$}
			\STATE $\Delta^k = \F(x^k) + \alpha(\F(x^k) - \F(x^{k-1})) - (y^k + \alpha(y^k - y^{k-1}))$
			\STATE $x^{k+1} = \prox_{\eta\g}(x^k -\eta\Delta^k)$
			\STATE $y^{k+1} = y^k - \theta (\mW \otimes \mI_d )(x^{k+1} - \beta(\F(x^{k+1}) - y^k))$
		\ENDFOR
	\end{algorithmic}
\end{algorithm}

\begin{algorithm}[H]
	\caption{}
	\label{alg_tv}
	\begin{algorithmic}[1]
		\STATE {\bf Parameters:}  Stepsizes $\eta_z, \eta_y, \eta_x, \theta>0$, momentums $\alpha, \gamma, \omega, \tau$, parameters $\nu, \beta$
        \STATE {\bf Initialization:} Choose  $\z^0 = \w^0 \in (\dom g)^M$, $\y^0 \in (\R^d)^M$, $\x^0 \in \cL^\perp$. Put $\z^{-1} = \z^0, \w^{-1} = \w^0$, $\y_f = \y^{-1} = \y^0$, $\x_f = \x^{-1} = \x^0$, $m_0 = \textbf{0}^{dM}$
		\FOR{$k=0,1,2,\ldots$}
		\STATE $\Delta_z^k = \F(\z^k) + \alpha(\F(\z^k) - \F(\z^{k-1})) - \nu \z^k - \y^k - \alpha(\y^k - \y^{k-1})$
		\STATE $\z^{k+1} = \prox_{\eta_z \g}(\z^k - \eta_z \Delta_z^k)$
		\STATE $\y_c^k = \tau \y^k + (1-\tau)\y_f^k$
		\STATE $\x_c^k = \tau \x^k + (1-\tau)\x_f^k$
		\STATE $\Delta_y^k = \nu^{-1} (\y_c^k + \x_c^k) + \z^{k+1} + \gamma (\y^k + \x^k + \nu \z^k)$
		\STATE $\Delta_x^k = \nu^{-1} (\y_c^k + \x_c^k) + \beta(\x^k + \F(\z^k))$
		\STATE $\y^{k+1} = \y^k - \eta_y \Delta_y^k$
		\STATE $\x^{k+1} = \x^k - (\mW_T(Tk) \otimes \mI_d) (\eta_x\Delta_x^k + m^k)$
		\STATE $m^{k+1} = \eta_x\Delta_x^k + m^k - (\mW_T(Tk) \otimes \mI_d) (\eta_x\Delta_x^k + m^k)$
		\STATE $\y_f^{k+1} = \y_c^k + \tau(\y^{k+1} - \y^k)$
		\STATE $\x_f^{k+1} = \x_c^k - \theta(\mW_T(Tk) \otimes \mI_d)(\y_c^k + \x_c^k)$
		\ENDFOR
	\end{algorithmic}
\end{algorithm}

\newpage

\section{Proof of Theorems \ref{th:lower_fixed} and \ref{th:lower_tv}} \label{sec:pr_lb}

The idea of obtaining lower bounds is to find some "bad" example in the class of problems. In particular, we consider saddle point problems (as a special case of VIs) and look for the "bad" function among them. The paper \cite{zhang2019lower} gives such an example for non-distributed deterministic problems. To get a distributed bounds, this problem needs to be divided between computing devices, which are connected in some kind of "bad" computing network \cite{scaman2017optimal,kovalev2021lower}. Then on each of the devices the function must be further divided in order to obtain a stochastic sum type problem \cite{hendrikx2020optimal,han2021lower}.

As stated above, to obtain lower bounds, we consider a particular case of the variational inequality \eqref{eq:VI}+\eqref{eq:fs}, the saddle point problem:
\begin{equation}
    \label{app:lb_problem}
    \min_{x \in \R^d}\max_{y \in \R^d} f(x,y) \eqdef \sum\limits_{m=1}^M f_m (x,y) \eqdef \sum\limits_{m=1}^M \frac{1}{n}\sum_{i=1}^n f_{m, i} (x, y).
\end{equation}
In this case $F(z) = F(x,y) = [\nabla_x f(x,y), -\nabla_y f(x,y)]$ and $g \equiv 0$. Next we rewrite Assumptions \ref{as:Lipsh} and \ref{as:strmon} for \eqref{app:lb_problem}.

\begin{assumption} Suppose that 
\label{app:lb_as1} \\
$\bullet$ each $f_m$ is $\Lavg$-average smooth, on $\R^d \times \R^d$, i.e. for all $x_1, x_2 \in \R^d$, $y_1, y_2 \in \R^d$ it holds that
\begin{equation}
\label{app:avr_smooth}
    \begin{split}
        \frac{1}{n}\sum_{i=1}^n (\|\nabla_x f_{m, i} (x_1, y_1) - \nabla_x f_{m, i} (x_2, y_2) \|^2 + \|\nabla_y & f_{m, i} (x_1, y_1) - \nabla_y f_{m i} (x_2, y_2) \|^2 ) \\
     &\leq \Lavg^2 \left( \|x_1 - x_2 \|^2 + \|y_1 - y_2 \|^2\right);
    \end{split}
\end{equation}
$\bullet$ each $f_m$ is $L$-smooth, on $\R^d \times \R^d$, i.e. for all $x_1, x_2 \in \R^d$, $y_1, y_2 \in \R^d$ it holds that
\begin{equation}
\label{app:smooth}
    \begin{split}
   \|\nabla_x f_m (x_1, y_1) - \nabla_x f_m (x_2, y_2) \|^2 + \|\nabla_y f_m (x_1, & y_1) - \nabla_y f_m (x_2, y_2) \|^2 \\
     &\leq L^2 \left( \|x_1 - x_2 \|^2 + \|y_1 - y_2 \|^2\right);
     \end{split}
\end{equation}
$\bullet$ $f$ is $\mu$ - strongly-convex-strongly-concave on $\R^d \times \R^d$, i.e. for all $x_1, x_2 \in \R^d$, $y_1, y_2 \in \R^d$ it holds that
\begin{equation}
\label{app:strcc}
    \begin{split}
    \langle \nabla_x f(x_1, y_1) - \nabla_x f (x_2, y_2) ; x_1 - x_2\rangle - \langle \nabla_y f & (x_1, y_1) - \nabla_y f (x_2, y_2) ; y_1 - y_2 \rangle \\
    &\geq \mu \left( \|x_1 - x_2 \|^2 + \|y_1 - y_2 \|^2\right).
    \end{split}
\end{equation}
\end{assumption}

Next we rewrite Definition \ref{def:proc} for \eqref{app:lb_problem}. Since $g \equiv 0$, then $\text{prox}_{\rho g}(z) = z$.
\begin{definition}[Oracle] \label{app:proc}
    Each agent $m$ has its own local memories $\mathcal{M}^x_{m}$ and $\mathcal{M}^y_{m}$ for the $x$- and $y$-variables, respectively--with  initialization $\mathcal{M}_{m}^x = \mathcal{M}_{m}^y= \{0\}$.    $\mathcal{M}_{m}^x$ and $\mathcal{M}_{m}^x$ are updated as follows.\\
   
    $\bullet$ \textbf{Local computation:} At each local iteration device $m$ can sample uniformly and independently batch $S_m$ of any size $b$ from $\{f_{m,i}\}$ and adds to its $\mathcal{M}^x_{m}$ and $\mathcal{M}^y_{m}$ a finite number of points $x,y$, satisfying 
    \begin{equation}\begin{aligned}\label{app:oracle-opt-step}
        x \in \text{span} \big\{&x'~,~\sum_{i_m \in S_m} \nabla_x f_{m, i_m}(x'',y'') \big\},\\
        y \in \text{span} \big\{&y'~,~\sum_{i_m \in S_m}\nabla_y f_{m, i_m}(x'',y'')\big\}, 
    \end{aligned}\end{equation}
    for given $x', x'' \in \mathcal{M}^x_{m}$ and  $y', y'' \in \mathcal{M}^y_{m}$. Such call needs $b$ local computations to collect the batch. Batch of the size $n$ is equal to the full $f_m$;

    $\bullet$ \textbf{Communication:} Based upon communication rounds  among neighbouring nodes, at the communication with the number $t$,  $\mathcal{M}^x_{m}$ and $\mathcal{M}^y_{m}$ are updated according to
    \begin{equation}\label{app:oracle-comm}
        \mathcal{M}^x_{m} := \text{span}\left\{\bigcup_{(i,m) \in \mathcal{E}(t)} \mathcal{M}^x_{i} \right\}, \quad 
        \mathcal{M}^{y}_{m} := \text{span}\left\{\bigcup_{(i,m) \in \mathcal{E}(t)} \mathcal{M}^y_{i} \right\}.
    \end{equation}

    $\bullet$ \textbf{Output:} 
    The final global output is calculated as: 
    \begin{align*}
        \hat x \in \text{span}\left\{\bigcup_{m=1}^M \mathcal{M}^x_{m} \right\},~~\hat y \in \text{span}\left\{\bigcup_{m=1}^M \mathcal{M}^y_{m} \right\}.
    \end{align*}
\end{definition}

We construct the following bilinearly functions with $\Lavg, \mu$. Let us consider a communication graph $G$ with vertexes $\{1, \ldots M\}$. Define $B = \{1\}$  and  $\bar B = \{M\}$, with  $|B| = |\bar B| = 1$. We then construct the following   bilinear functions on the graph:
\begin{eqnarray}
\label{t2}
f_m (x,y) = 
\begin{cases}
f_1 (x,y) =  \frac{\Lavg}{4\sqrt{n}}  x^T A_1 y + \frac{\mu}{6}\|x\|^2 - \frac{\mu}{6}\|y\|^2 + \frac{\Lavg^2}{2n\mu}e_1^T y, & m \in B;\\
f_2 (x,y) =  \frac{\Lavg}{4\sqrt{n}} x^T A_2 y + \frac{\mu}{6}\|x\|^2 - \frac{\mu}{6}\|y\|^2, & m \in \bar B;\\
f_3 (x,y) = \frac{1}{M-2}\left(\frac{\mu}{6}\|x\|^2 - \frac{\mu}{6}\|y\|^2\right), & \text{otherwise};
\end{cases}
\end{eqnarray}
where $e_1 = (1,0 \ldots, 0)$ and
\begin{eqnarray*}
A_1 = \left(
\begin{array}{cccccccc}
1&0 & & & & & &  \\
&1 &-2 & & & & &  \\
& &1 &0 & & & & \\
& & &1 &-2 & & & \\
& & & &\ldots &\ldots & & \\
& & & & &1  &-2   & \\
& & &   & & &1 &0 \\
& & &  & & & &1 \\
\end{array}
\right), \quad
A_2 = \left(
\begin{array}{cccccccc}
1&-2 & & & & & &  \\
&1 &0 & & & & &  \\
& &1 &-2 & & & & \\
& & &1 &0 & & & \\
& & & &\ldots &\ldots & & \\
& & & & &1  &0   & \\
& & &   & & &1 &-2 \\
& & &  & & & &1 \\
\end{array}
\right).
\end{eqnarray*}
Then we construct functions $f_{m,i}$. Let define $a_{1, q}$ $q$th row of the matrix $A_1$. Then we split function $f_1$ to finite sum $\frac{1}{n}\sum_{i=1}^n f_{1,i}$ in following way:
\begin{align}
    \label{f_m}
    f_{1,i} =& \sqrt{n} \cdot \frac{\Lavg}{4} x^T \left[ \sum\limits_{j \equiv (i-1) ~\text{mod}~ n}\left(e_{2j+1} a^T_{1, 2j+1} + e_{2j+2} a^T_{1, 2j+2}\right) \right]y \notag\\
    &+ \frac{\mu}{6}\|x\|^2 - \frac{\mu}{6}\|y\|^2 + \frac{\Lavg^2}{2n\mu}e_1^T y
\end{align}

Let define $A_{1, i} = \sum_{j \equiv (i-1) ~\text{mod}~ n}\left(e_{2j+1} a^T_{1, 2j+1} + e_{2j+2} a^T_{1, 2j+2}\right)$. The same way we can construct $f_{2,i}$ with $A_{2,i}$ and $f_{3,i}$ with $A_{3,i} = 0$. Consider the global objective function:
\begin{align}
    \label{t144}
    f(x,y) &= \sum\limits_{m=1}^M f_m(x,y) =  |B| \cdot f_1(x,y) + |\bar B| \cdot f_2(x,y) + (M - |B| - |\bar B|) \cdot f_3(x,y) \nonumber\\
    &= \frac{\Lavg}{2\sqrt{n}} x^T A y + \frac{\mu}{2}\|x\|^2 - \frac{\mu}{2}\|y\|^2 +  \frac{\Lavg^2}{2n\mu} e_1^T y,
\end{align}
with $A = \frac{1}{2}(A_1 + A_2)$. 

\begin{lemma}
Problem \eqref{t2} with $f_m$ from \eqref{f_m} satisfies Assumption \ref{app:lb_as1} with $\Lavg, L = \tfrac{\Lavg}{\sqrt{n}}, \mu$.
\end{lemma}
\begin{proof}
It is easy to verify that $f$ is $\mu$ - strongly-convex-strongly-concave. Also with $\|A_1 \|, \|A_2 \| \leq 3$, we get that $f_m(x,y)$ is $L = \frac{\Lavg}{\sqrt{n}}$ - smooth. Then we need to check that $f_m$ is $\Lavg$ - average smooth:
\begin{align*}
    \frac{1}{n} \sum\limits_{i=1}^n& \left[\|\nabla_x f_{1,i} (x_1 , y_1) -  \nabla_x f_{1,i} (x_2, y_2)\|^2 + \|\nabla_y f_{1,i} (x_1 , y_1) -  \nabla_y f_{1,i} (x_2, y_2)\|^2 \right]\nonumber \\
    &= \frac{1}{n} \sum\limits_{i=1}^n \left[\left\| \sqrt{n} \cdot \frac{\Lavg}{4} A_{1,i}(y_1 - y_2) + \frac{\mu}{3} (x_1 - x_2)\right\|^2 + \left\| \sqrt{n} \cdot \frac{\Lavg}{4} A^T_{1,i}(x_2 - x_1) + \frac{\mu}{3} (y_1 - y_2)\right\|^2 \right]\nonumber \\
    &\leq \frac{1}{n} \sum\limits_{i=1}^n \left[\frac{nL^2}{8} \left\|  A_{1,i}(y_1 - y_2)\right\|^2 + \frac{2\mu^2}{9} \left\| x_1 - x_2\right\|^2 + \frac{nL^2}{8} \left\|  A^T_{1,i}(x_1 - x_2)\right\|^2 + \frac{2\mu^2}{9} \left\| y_1 - y_2\right\|^2 \right] \nonumber \\
    &= \frac{\Lavg^2}{8} \sum\limits_{i=1}^n \left[(y_1 - y_2)^T A_{1,i}^T A_{1,i}(y_1 - y_2)  +  (x_1 - x_2)^T A_{1,i}A^T_{1,i}(x_1 - x_2)  \right] \nonumber \\
    &\hspace{0.4cm}+ \frac{2\mu^2}{9} \left\| x_1 - x_2\right\|^2 +\frac{2\mu^2}{9} \left\| y_1 - y_2\right\|^2 \nonumber \\
    &= \frac{\Lavg^2}{8} \sum\limits_{i=1}^n (y_1 - y_2)^T \left[\sum_{j \equiv (i-1) ~\text{mod}~ n}\left( a_{1, 2j+1} e_{2j+1}^T +  a_{1, 2j+2} e_{2j+2}^T\right)\right] \\
    &\hspace{3.4cm} \left[\sum_{j \equiv (i-1) ~\text{mod}~ n}\left(e_{2j+1} a^T_{1, 2j+1} + e_{2j+2} a^T_{1, 2j+2}\right)\right] (y_1 - y_2)  
    \nonumber \\
    &\hspace{0.4cm} +  \frac{\Lavg^2}{8} \sum\limits_{i=1}^n (x_1 - x_2)^T \left[\sum_{j \equiv (i-1) ~\text{mod}~ n}\left(e_{2j+1} a^T_{1, 2j+1} + e_{2j+2} a^T_{1, 2j+2}\right)\right] \\
    &\hspace{3.8cm} \left[\sum_{j \equiv (i-1) ~\text{mod}~ n}\left( a_{1, 2j+1} e^T_{2j+1}+  a_{1, 2j+2} e^T_{2j+2}\right)\right] (x_1 - x_2)  
    \nonumber \\
    &\hspace{0.4cm}+ \frac{2\mu^2}{9} \left\| x_1 - x_2\right\|^2 +\frac{2\mu^2}{9} \left\| y_1 - y_2\right\|^2 
    \nonumber \\
    &= \frac{\Lavg^2}{8} \sum\limits_{i=1}^n (y_1 - y_2)^T \left[\sum_{j \equiv (i-1) ~\text{mod}~ n}\left( a_{1, 2j+1} a^T_{1, 2j+1} +  a_{1, 2j+2} a^T_{1, 2j+2}\right)\right]  (y_1 - y_2)  
    \nonumber \\
    &\hspace{0.4cm} +  \frac{\Lavg^2}{8} \sum\limits_{i=1}^n (x_1 - x_2)^T \Bigg[\sum_{j \equiv (i-1) ~\text{mod}~ n}\bigg(e_{2j+1}  e^T_{2j+1} + 5e_{2j+2} e^T_{2j+2} \bigg)\Bigg]^T (x_1 - x_2)  
    \nonumber \\
    &\hspace{0.4cm}+ \frac{2\mu^2}{9} \left\| x_1 - x_2\right\|^2 +\frac{2\mu^2}{9} \left\| y_1 - y_2\right\|^2 \nonumber \\
    &\leq \frac{\Lavg^2}{8} \sum\limits_{j=1}^d (y_1 - y_2)^T \left( a_{1, j} a^T_{1, j}\right)  (y_1 - y_2) +  \frac{5\Lavg^2}{8} \sum\limits_{j=1}^d (x_1 - x_2)^T \left(e_{j}  e^T_{j}\right) (x_1 - x_2)  
    \nonumber \\
    &\hspace{0.4cm}+ \frac{2\mu^2}{9} \left\| x_1 - x_2\right\|^2 +\frac{2\mu^2}{9} \left\| y_1 - y_2\right\|^2 \nonumber \\
    &\leq \frac{\Lavg^2}{8} \sum\limits_{j=1}^d (y_1 - y_2)^T \left( a_{1, j} a^T_{1, j}\right)  (y_1 - y_2) +  \frac{5\Lavg^2}{8} \left\| x_1 - x_2\right\|^2 \\
    &\hspace{0.4cm} + \frac{2\mu^2}{9} \left\| x_1 - x_2\right\|^2 +\frac{2\mu^2}{9} \left\| y_1 - y_2\right\|^2 \nonumber \\
    &\leq \frac{3\Lavg^2}{4} \left\| y_1 - y_2\right\|^2 +  \frac{5\Lavg^2}{8} \left\| x_1 - x_2\right\|^2 + \frac{2\mu^2}{9} \left\| x_1 - x_2\right\|^2 +\frac{2\mu^2}{9} \left\| y_1 - y_2\right\|^2.
\end{align*}
The last inequality follows from $\lambda_{\max} \left(\sum\limits_{j=1}^d \left( a_{1, j} a^T_{1, j}\right) \right) \leq 6$. Finally, with $\mu \leq \Lavg$ we get
\begin{align*}
    \frac{1}{n} \sum\limits_{i=1}^n& \left[\|\nabla_x f_{1,i} (x_1 , y_1) -  \nabla_x f_{1,i} (x_2, y_2)\|^2 + \|\nabla_y f_{1,i} (x_1 , y_1) -  \nabla_y f_{1,i} (x_2, y_2)\|^2 \right]\nonumber \\
    &\leq \Lavg^2 \left(\left\| x_1 - x_2\right\|^2 + \left\| y_1 - y_2\right\|^2\right).
\end{align*}
\end{proof}

The next two lemmas give an idea of how quickly we approximate the solution of \eqref{t144} depending on the number of communications and local iterations. For simplicity, we divide a situation in two parts: one part is devoted to communications (taking into account the fact that we are not limited in the number of local iterations), the second - on the contrary (we concentrate on local computations and assume that communications cost nothing).

\begin{lemma} \label{l2}
Let Problem~\eqref{t2} be solved by any method  that satisfies Definition \ref{app:proc}. Then after $K$ communication rounds,  only the first $\left\lfloor \frac{K}{l} \right\rfloor$ coordinates of the   global output can be non-zero while  the rest of the $d-\left\lfloor \frac{K}{l} \right\rfloor$ coordinates are strictly equal to zero. Here $l$ is "distance" between $B$ and $\bar B$ (how quickly can we transfer information from $B$ to $\bar B$).
\end{lemma}

\begin{proof}
We begin introducing some notation for our proof. Let
\begin{align*}
    E_{0} := \{ 0\}, \quad E_{K} := \text{span} \{ e_1, \ldots, e_K\}.
\end{align*}
Note that, the initialization gives $\mathcal{M}^x_{m} = E_0$, $\mathcal{M}^y_{m} = E_0$.

Suppose that, for some $m$,  $\mathcal{M}^x_{m} = E_K$ and $\mathcal{M}^y_{m} = E_K$, at some given time. Let us analyze how $\mathcal{M}^x_{m}, \mathcal{M}^y_{m}$ can change by performing only local computations. 

Firstly, we consider the case when 
$K$ odd. After one local update, we have the following: 

$\bullet$ For   machines $m$  which own $f_1$, it holds
\begin{equation}\begin{aligned}
\label{update_lower1}
        x \in \text{span} \big\{&e_1~,~ x'~,~A_{1} y'\big\} = E_K,\\
        y \in \text{span} \big\{&e_1~,~ y'~,~A_1^T x'\big\} = E_K, 
\end{aligned}\end{equation}
for given $x' \in \mathcal{M}^x_{m}$ and  $y' \in \mathcal{M}^y_{m}$. In details, each local iteration uses matrices $A_{1,i}$ (in the stochastic) or $A_1$ (in the deterministic). But here we talks only about communications and do not pay attention to the number of local iterations. Therefore, without loss of generality, we can immediately assume that all local calculations change our output according to \eqref{update_lower1}.
Since $A_1$ has a block diagonal structure,   after local computations, we have  $\mathcal{M}^x_{m} = E_K$ and $\mathcal{M}^y_{m} = E_K$. The situation does not change, no matter how many local computations one  does.

$\bullet$ For machines $m$ which own $f_2$, it holds
\begin{equation*}\begin{aligned}
        x \in \text{span} \big\{&x'~,~A_2 y'\big\} = E_{K+1},\\
        y  \in \text{span} \big\{&y'~,~A_2^T x'\big\} = E_{K+1}, 
\end{aligned}\end{equation*}
for given $x' \in \mathcal{M}^x_{m}$ and  $y' \in \mathcal{M}^y_{m}$. It means that, after local computations, one has  $\mathcal{M}^x_{m} = E_{K+1}$ and $\mathcal{M}^y_{m} = E_{K+1}$. Therefore,  machines with function $f_2$ can progress by one new non-zero coordinate.

This means that we constantly have to transfer progress from the group of machines with $f_1$ to the group of machines with $f_2$ and back. Initially, all devices have zero coordinates. Further, machines with $f_1$ can receive the first nonzero coordinate (but only the first, the second is not), and the rest of the devices are left with all zeros. Next, we pass the first non-zero coordinate to machines with $f_2$. To do this,    $l$ communication rounds are needed. By doing so,  they can make the second coordinate non-zero, and then transfer this progress to the machines with $f_1$. Then the process continues in the same way. This  completes the proof.
\end{proof}

In the next lemma, we will give an understanding of how local progress towards a solution occurs. For this, we will assume that communications cost nothing.

\begin{lemma} \label{l3}
Let Problem~\eqref{t2} be solved by any method  that satisfies Definition \ref{app:proc}. Then after $N$ local calls (for each node), in expectation only the first $\left\lfloor \frac{N}{n} \right\rfloor$ coordinates of the   global output can be non-zero while  the rest of the $d-\left\lfloor \frac{N}{n} \right\rfloor$ coordinates are strictly equal to zero. 
\end{lemma}

\begin{proof} As is clear from the previous lemma, communications make sense if an update ($E_K \to E_{K+1}$) is reached on one of the nodes. Depending on $K$, this happens on the nodes with $f_1$ or $f_2$ (but not simultaneously). The question is how many local calls should be made to get this update. One can understand it by looking at the structure of matrices $A_{1,i}$ and $A_{2,i}$ from \eqref{f_m}. Only one of $n$ matrices is suitable for us. For example, in case of $K = 2k$, we need $f_{1,j}$ with $j \equiv k ~\text{mod}~ n$. 

Suppose that $s_1, s_2\ldots$ times call stochastic oracle with batchsize $1, 2 \ldots$. Then $\sum_{j=1}^n j s_j + N$. Due to the fact that the choice of batches is random and uniform, the random variable responsible for the total number of updates during the operation of the algorithm has the sum of binomial distribution with pairs of parameters $\left(s_j ; \frac{j}{n}\right)$. It remains only to take the mathematical expectation and the lemma is proved.
\end{proof}

The next lemma is   devoted to provide an approximate solution of   problem \eqref{t144}, and shows that this approximation is close to a real solution. 
The proof of the lemma follows closely that of  Lemma 3.3 from \cite{zhang2019lower}, and is reported for the sake of completeness.

\begin{lemma}[Lemma 3.3 from \cite{zhang2019lower}]\label{lemma2}
Let $\alpha = \frac{2n\mu^2}{\Lavg^2}$ and $q = \frac{1}{2}\left(2 + \alpha - \sqrt{\alpha^2 + 4\alpha} \right) \in (0;1)$--the smallest root of $q^2 - (2 + \alpha) q + 1 = 0$; and let define
\begin{equation*}
    \bar y^*_i = \frac{q^i}{1-q},\quad i\in [d].
\end{equation*}
The following bound holds when   $\bar y^*:=[y^*_1,\ldots y^*_d]^\top$ is used to  approximate   the solution $ y^*$: 
\begin{equation*}
    \|\bar y^* - y^*\| \leq \frac{q^{d+1}}{\alpha(1-q)}.
\end{equation*}
\end{lemma}
\begin{proof} Let us write the dual function for \eqref{t144}:
\begin{equation*}
    h(y) = -\frac{1}{2}y^T \left(\frac{\Lavg^2}{2n\mu}A^T A + \mu I \right)y + \frac{\Lavg^2}{2n\mu} e_1^T y,
\end{equation*}
where it is not difficult to check that 
\begin{eqnarray*}
A A^T = \left(
\begin{array}{cccccccc}
1&-1 & & & & & &  \\
-1&2 &-1 & & & & &  \\
&-1 &2 & -1 & & & & \\
& & -1&2 &-1 & & & \\
& & &-1 &2 &-1 & & \\
& & & & &\ldots & & \\
& & & & &-1 &2 &-1 \\
& & & & & &-1 &2 \\
\end{array}
\right).
\end{eqnarray*}
The optimality of dual problem $\nabla h(y^*) = 0$ gives
\begin{equation*}
    \left(\frac{\Lavg^2}{2n\mu}A^T A + \mu I \right)y^* = \frac{\Lavg^2}{2n\mu} e_1,
\end{equation*}
or
\begin{equation*}
    \left(A^T A + \alpha I \right)y^* = e_1.
\end{equation*}
Equivalently, we can write 
\begin{eqnarray*}
\left\{
\begin{array}{l}
(1+\alpha)y_1^* - y_2^* = 1, \\
-y_1^* + (2 + \alpha) y^*_2 - y^*_3 = 0,\\
\ldots \\
-y_{d-2}^* + (2 + \alpha) y^*_{d-1} - y^*_d = 0,\\
-y^*_{d-1} + (2+\alpha)y^*_d = 0.
\end{array}
\right .
\end{eqnarray*}
On the other hand, the    approximation $\bar y^*$ satisfies the following set of equations:
\begin{eqnarray*}
\left\{
\begin{array}{l}
(1+\alpha)\bar y_1^* - \bar y_2^* = 1, \\
-\bar y_1^* + (2 + \alpha) \bar y^*_2 - \bar y^*_3 = 0,\\
\ldots \\
-\bar y_{d-2}^* + (2 + \alpha) \bar y^*_{d-1} - \bar y^*_d = 0,\\
-\bar y^*_{d-1} + (2+\alpha) \bar y^*_d = \frac{q^{d+1}}{1-q},
\end{array}
\right .
\end{eqnarray*}
or equivalently
\begin{equation*}
    \left(A^T A + \alpha I \right)\bar y^* = e_1 + \frac{q^{d+1}}{1-q}e_d.
\end{equation*}
Therefore,  the difference between  $\bar y^*$ and   $y^*$   reads
\begin{equation*}
    \bar y^* - y^* = \left(A^T A + \alpha I \right)^{-1}\frac{q^{d+1}}{1-q}e_d.
\end{equation*}
The statement of the lemma follow from the above equality and $\alpha^{-1} I \succeq \left(A^T A + \alpha I \right)^{-1} \succ 0$.  
\end{proof}

The next lemma 
provides a lower bound for the solution of \eqref{t144} in the distributed case \eqref{t2}. The proof follows closely that of Lemma 3.4 from \cite{zhang2019lower} and is reported for the sake of completeness. 

\begin{lemma} \label{l234}
Consider  a distributed saddle-point problem with objective function given by \eqref{t144}. For any $K, N$, choose any  problem size  $d \geq \max \left\{ 2 \log_q \left( \frac{\alpha}{4\sqrt{2}}\right), 2K, 2N\right\}$, where $\alpha = \frac{2n\mu^2}{\Lavg^2}$ and $q = \frac{1}{2}\left(2 + \alpha - \sqrt{\alpha^2 + 4\alpha} \right) \in (0;1)$. Then, any output $\hat x, \hat y$ produced by  any method satisfying Definition \ref{app:proc} after $K$ communications and $N$ local calls, is such that 
\begin{equation*}
    \EE\left[\|\hat x - x^*\|^2 + \|\hat y - y^*\|^2 \right]\geq \left(q^{\frac{2K}{l}} + q^{\frac{2N}{n}}\right) \frac{\| y_0 - y^*\|^2}{16}.
\end{equation*}
\end{lemma}
\begin{proof} Let us assume that in output we have $k$ non-zero coordinates. By definition of $\bar y^*$, with $q < 1$ and $k \leq \frac{d}{2}$, we have
\begin{eqnarray*}
    \|\hat y - \bar y^*\|^2 &\geq& \sqrt{\sum\limits_{j=k+1}^d  (\bar y^*_j)^2} = \frac{q^k}{1-q} \sqrt{q^2 + q^4 + \ldots + q^{2(d-k)}} \\
    &\geq& \frac{q^k}{\sqrt{2}(1-q)} \sqrt{q^2 + q^4 + \ldots + q^{2d}} = \frac{q^k}{\sqrt{2}} \| \bar y^*\|^2 = \frac{q^k}{\sqrt{2}} \| y_0 - \bar y^*\|^2.
\end{eqnarray*}
Using Lemma \ref{lemma2} for $d \geq 2 \log_q \left(\frac{\alpha}{4\sqrt{2}} \right)$ we can guarantee that $\bar y^* \approx y^*$ (for more detailed proof see \cite{zhang2019lower}) and
\begin{equation*}
    \|\hat x - x^*\|^2 + \|\hat y - y^*\|^2 \geq \|\hat y - y^*\|^2 \geq \frac{q^{2k}}{16} \| y_0 - y^*\|^2.
\end{equation*}
It remains only to note that $k$ depends on the number of nonzero coordinates from communications $k_c$ and local computations $k_l$. For this we use Lemmas \ref{l2} and \ref{l3}:
\begin{equation*}
    \|\hat x - x^*\|^2 + \|\hat y - y^*\|^2 \geq \frac{q^{2k}}{16} \| y_0 - y^*\|^2 \geq  \frac{q^{2 \min (k_c, k_l)}}{16} \| y_0 - y^*\|^2 \geq \frac{q^{k_c + k_l}}{16} \| y_0 - y^*\|^2.
\end{equation*}
By Lemma \ref{l2} we have $k_c \leq \left\lfloor \frac{K}{l} \right\rfloor$, where $l$ is "distance" between $B$ and $\bar B$. By Lemma \ref{l3} we get that $k_l$ has binomial distribution with parameters $N$ and $\frac{1}{n}$.
\begin{align*}
    \EE\left[\|\hat x - x^*\|^2 + \|\hat y - y^*\|^2\right] &\geq \frac{q^{2 \left\lfloor \frac{K}{l} \right\rfloor}}{32} \| y_0 - y^*\|^2 + \EE\left[\frac{q^{2 k_l}}{32} \| y_0 - y^*\|^2\right] \\
    &\geq \frac{q^{\frac{2K}{l}}}{32} \| y_0 - y^*\|^2 + \frac{\EE\left[q^{2 k_l}\right]}{32} \| y_0 - y^*\|^2 \\
    &\geq \left(q^{\frac{2K}{l}} + q^{2 \EE[k_l]}\right)\cdot \frac{1}{32} \| y_0 - y^*\|^2 \\
    &\geq \left(q^{\frac{2K}{l}} + q^{\frac{2N}{n}}\right)\cdot \frac{1}{32} \| y_0 - y^*\|^2.
\end{align*}
Here we use Jensen's inequality.  
\end{proof}

It remains to get an estimate on $l$ ("distance" between $B$ and $\bar B$). For the fixed network it is real distance, for time-varying -- rate how fast information transmits in the network from  $B$ to $\bar B$.

\subsection{Fixed network} \label{app:low_fixed}
\begin{theorem}[Theorem \ref{th:lower_fixed}]
Let $\Lavg > \mu > 0$, $n \in \N$ (with $\Lavg/\mu \geq \sqrt{n}$), $\chi \geq 1$ and $K, N \in \N$. There exists a distributed saddle-point problem over fixed network (Assumption \ref{ass:fixed}). For which the following statements are true:

$\bullet$ a gossip matrix $\mW$ has $\chi(\mW) = \chi$,

$\bullet$ $f = \sum\limits_{m=1}^M \frac{1}{n}\sum\limits_{i=1}^n f_{m.i} : \R^d \times \R^d \to \R$ is $\mu$ -- strongly-convex-strongly-concave,

$\bullet$ $f_m$ are $\Lavg$-average smooth and $L = \tfrac{\Lavg}{\sqrt{n}}$ - smooth,

$\bullet$ size $d \geq \max \left\{ 2 \log_q \left( \frac{\alpha}{4\sqrt{2}}\right), 2K, 2N\right\}$, where  $\alpha = \frac{2n\mu^2}{L^2}$ and $q = \frac{1}{2}\left(2 + \alpha - \sqrt{\alpha^2 + 4\alpha} \right) \in (0;1)$,

$\bullet$ the solution of the problem is non-zero: $x^* \neq 0$, $y^* \neq 0$.

Then for any output $\hat z$ of any procedure (Definition \ref{app:proc}) with $K$ communication rounds and $N$ local computations, one can obtain the following estimate:
\begin{equation*}
    \|\hat z - z^*\|^2 = \Omega\left(\exp\left(  -\frac{80}{1 +  \sqrt{\frac{2 L^2}{\mu^2} + 1}} \cdot \frac{K}{\sqrt{\chi}}\right)  \| y_0 - y^*\|^2\right);
\end{equation*}
\begin{equation*}
    \|\hat z - z^*\|^2 = \Omega\left(\exp\left(-\frac{16}{n +  \sqrt{\frac{2n \Lavg^2}{\mu^2} + n^2}}\cdot N\right) \| y_0 - y^*\|^2\right).
\end{equation*}

\end{theorem}

\begin{proof}
Applying Lemma \ref{l234}, we have
\begin{equation*}
    \left(\frac{1}{q}\right)^{\frac{2K}{l}} \geq  \frac{\| y_0 - y^*\|^2}{32\EE\left[\|\hat x - x^*\|^2 + \|\hat y - y^*\|^2\right]} ~~\text{and}~~\left(\frac{1}{q}\right)^{\frac{2N}{n}} \geq  \frac{\| y_0 - y^*\|^2}{32\EE\left[\|\hat x - x^*\|^2 + \|\hat y - y^*\|^2\right]}.
\end{equation*}
Taking the logarithm on both sides, we get
\begin{equation*}
    \frac{2K}{l}  \geq  \ln\left(\frac{\| y_0 - y^*\|^2}{32\EE\left[\|\hat x - x^*\|^2 + \|\hat y - y^*\|^2\right]}\right) \frac{1}{\ln(q^{-1})}.
\end{equation*}
Next, we work with
\begin{eqnarray*}
    \frac{1}{\ln (q^{-1})} &=& \frac{1}{\ln (1+ (1-q)/q))} \geq \frac{q}{1-q} = \frac{1 + \frac{n\mu^2}{\Lavg^2} - \sqrt{\frac{2n\mu^2}{\Lavg^2} + \left(\frac{n\mu^2}{\Lavg^2}\right)^2}}{\sqrt{\frac{2n\mu^2}{\Lavg^2} + \left(\frac{n\mu^2}{\Lavg^2}\right)^2} - \frac{n\mu^2}{\Lavg^2}} \nonumber\\ 
    &\geq& \frac{\sqrt{\frac{2n\mu^2}{\Lavg^2} + \left(\frac{n\mu^2}{\Lavg^2}\right)^2} + \frac{n\mu^2}{\Lavg^2}}{\frac{8n\mu^2}{\Lavg^2}} = \frac{1}{8}\left( 1 +  \sqrt{\frac{2\Lavg^2}{n\mu^2} + 1}\right).
\end{eqnarray*}
One can then obtain 
\begin{equation*}
    \frac{2K}{l}  \geq  \ln\left(\frac{\| y_0 - y^*\|^2}{32\EE\left[\|\hat x - x^*\|^2 + \|\hat y - y^*\|^2\right]}\right) \cdot \frac{1}{8}\left(  1 +  \sqrt{\frac{2\Lavg^2}{n\mu^2} + 1}\right) ,
\end{equation*}
and 
\begin{equation*}
    \exp\left(\frac{1}{1 +  \sqrt{\frac{2\Lavg^2}{n\mu^2} + 1}}\frac{16 K}{l}\right)  \geq  \frac{\| y_0 - y^*\|^2}{32\EE\left[\|\hat x - x^*\|^2 + \|\hat y - y^*\|^2\right]}.
\end{equation*}

The next proof follow similar steps as in the  proof of Theorem 2 from \cite{scaman2017optimal}. 
Let $\gamma_M = \frac{1 - \cos \frac{\pi}{M}}{1 + \cos \frac{\pi}{M}}$ be a  decreasing sequence of positive numbers. Since $\gamma_2 = 1$ and $\lim_m \gamma_M = 0$, there exists $M \geq 2$ such that $\gamma_M \geq \chi^{-1} > \gamma_{M+1}$.

$\bullet$ If $M \geq 3$, let us consider linear graph of size $M$ with vertexes $v_1, \ldots v_M$, and weighted with $w_{1,2} = 1 - a$ and  $w_{i,i+1} = 1$ for $i \geq 2$. Then we applied Lemmas 1 and 3 and get:
\begin{equation*}
    \|\hat x - x^*\|^2 + \|\hat y - y^*\|^2 \geq q^{\frac{2K}{l} } \frac{\| y_0 - y^*\|^2}{32}.
\end{equation*}
If $\mW_a$ is the normalized Laplacian of the weighted graph $\mathcal G$, one can note that with $a = 0$, $\chi^{-1}(W_a) = \gamma_M$, with $a = 1$ -- $\chi^{-1}(\mW_a) = 0$. Hence, there exists $a \in (0;1]$ such that $\chi^{-1}(\mW_a) = \chi^{-1}$. Then $\chi^{-1} \geq \gamma_{M+1} \geq \frac{2}{(M+1)^2}$, and $M \geq \sqrt{2\chi} - 1 \geq \frac{\sqrt{\chi}}{4 }$. Finally, $l = M - 1 \geq \frac{15M}{16} - 1 \geq \frac{15}{16} \left(\sqrt{2\chi} - 1\right) - 1 \geq \frac{\sqrt{\chi}}{5}$ since $\chi^{-1} \leq \gamma_3 = \frac{1}{3}$. Hence,
\begin{equation}
    \label{r509}
    \exp\left(80\frac{K}{\sqrt{\chi}} \cdot \frac{1}{ 1 +  \sqrt{\frac{2\Lavg^2}{n\mu^2} + 1}} \right)  \geq  \frac{\| y_0 - y^*\|^2}{32(\|\hat x - x^*\|^2 + \|\hat y - y^*\|^2)}.
\end{equation}
$\bullet$ If $M = 2$, we construct a totally connected network with 3 nodes with weight $w_{1,3} = a \in [0;1]$. Let $W_a$ is the normalized Laplacian. If $a = 0$, then the network is a linear graph and $\chi^{-1}(\mW_a) = \gamma_3 =\frac{1}{3}$. Hence, there exists $a \in [0;1]$ such that $\chi^{-1}(\mW_a) = \chi^{-1}$. Finally, $B = \{v_1\}$, $\bar B = \{v_3\}$ and $l \geq 1 \geq \frac{1}{2\sqrt{\chi^{-1}}}$. Whence it follows that in this case \eqref{r509} is also valid.

The same way we can work with (but without considering graph):
\begin{equation*}
    \left(\frac{1}{q}\right)^{\frac{2N}{n}} \geq  \frac{\| y_0 - y^*\|^2}{32\EE\left[\|\hat x - x^*\|^2 + \|\hat y - y^*\|^2\right]}.
\end{equation*}
\end{proof}

\subsection{Time-varying network} \label{app:low_tv}

\begin{theorem}[Theorem \ref{th:lower_tv}]
Let $\Lavg > \mu > 0$, $n \in \N$ (with $\Lavg/\mu \geq \sqrt{n}$), $\hat \chi \geq 3$ and $K, N \in \N$. There exists a distributed saddle-point problem over time-varying network (Assumption \ref{ass:tv}). For which the following statements are true:

$\bullet$ Assumption \ref{ass:tv} holds with $\chi = \hat \chi$,

$\bullet$ $f = \sum\limits_{m=1}^M \frac{1}{n}\sum\limits_{i=1}^n f_{m.i} : \R^d \times \R^d \to \R$ is $\mu$ -- strongly-convex-strongly-concave,

$\bullet$ $f_m$ are $\Lavg$-average smooth and $L = \tfrac{\Lavg}{\sqrt{n}}$ - smooth,

$\bullet$ size $d \geq \max \left\{ 2 \log_q \left( \frac{\alpha}{4\sqrt{2}}\right), 2K, 2N\right\}$, where  $\alpha = \frac{2n\mu^2}{L^2}$ and $q = \frac{1}{2}\left(2 + \alpha - \sqrt{\alpha^2 + 4\alpha} \right) \in (0;1)$,

$\bullet$ the solution of the problem is non-zero: $x^* \neq 0$, $y^* \neq 0$.

Then for any output $\hat z$ of any procedure (Definition \ref{app:proc}) with $K$ communication rounds and $N$ local computations, one can obtain the following estimate:
\begin{equation*}
    \|\hat z - z^*\|^2 = \Omega\left(\exp\left(  -\frac{64}{1 +  \sqrt{\frac{2 L^2}{\mu^2} + 1}} \cdot \frac{K}{B\hat \chi}\right)  \| y_0 - y^*\|^2\right);
\end{equation*}
\begin{equation*}
    \|\hat z - z^*\|^2 = \Omega\left(\exp\left(-\frac{16}{n +  \sqrt{\frac{2n \Lavg^2}{\mu^2} + n^2}}\cdot N\right) \| y_0 - y^*\|^2\right).
\end{equation*}
\end{theorem}
\begin{proof}
The same way as in the previous Theorem we can obtain 
\begin{equation*}
    \exp\left(\frac{1}{1 +  \sqrt{\frac{2\Lavg^2}{n\mu^2} + 1}}\frac{16 K}{l}\right)  \geq  \frac{\| y_0 - y^*\|^2}{32\EE\left[\|\hat x - x^*\|^2 + \|\hat y - y^*\|^2\right]}.
\end{equation*}
Following \cite{kovalev2021lower}, we can consider the next sequences of graphs. Let us choose $M = \lfloor\hat \chi\rfloor$. Each communication $t$ such that $t\neq 8\tau$ (for any $\tau \in \N$) we consider empty network. Each communication $t$ such that $t=8\tau$ (for some $\tau \in \N$) we construct star graph with vertex $[(\tau-1) \mod (M-2)] + 2$ in the center of this star. It means that in the 8th communication vertex $2$ is in the center; in the 16th communication vertex $3$ is in the center etc. One can note that only vertexes with $f_3$ are in the center and they change sequentially. As matrices $\mW(t)$ we consider the normalized Laplacians. Then it holds $\chi = M = \lfloor\hat \chi\rfloor$. It holds that Assumption \ref{ass:tv} is valid with $\hat \chi$.

It is left to estimate $l$. Suppose we need to transfer information from $B$ to $\bar B$. At best, the following will happen:

$\bullet$ vertex $j$ in the center, it means that we can transfer information to $j$;

$\bullet$ $B-1$ empty graphs -- "empty" communications;

$\bullet$ vertex $j+1$ in the center, it means that we can transfer information to $j+1$;

\ldots

$\bullet$ vertex $j-1$ in the center, it means that we can transfer information to $j-1$ (still there is no information in $\bar B$);

$\bullet$ $B-1$ empty graphs -- "empty" communications;

$\bullet$ vertex $j$ in the center, now we can transfer from $j$ to $\bar B$.

In this (the best variant) we spend $M-1 + (B-1)(M-2)$ communication rounds. It means that $l \geq M-1 + (B-1)(M-2) \geq B(M-2) = B(\lfloor\hat \chi\rfloor - 2) \geq \tfrac{B\hat \chi}{4}$ (for $\hat \chi \geq 3$). Then we get
\begin{equation*}
    \exp\left(\frac{1}{1 +  \sqrt{\frac{2\Lavg^2}{n\mu^2} + 1}}\frac{64 K}{B\hat \chi}\right)  \geq  \frac{\| y_0 - y^*\|^2}{32\EE\left[\|\hat x - x^*\|^2 + \|\hat y - y^*\|^2\right]}.
\end{equation*}

\end{proof}

\section{Proof of Theorem \ref{th:ALg1_conv}} \label{sec:pr_oa_fixed}

We start the proof from the following lemma on $\delta^k$ and $\Delta^{k+1/2}$ from Algorithm \ref{alg:vrvi}.

\begin{lemma} \label{var_lem_fix}
	The following inequality holds:
	\begin{equation}\label{vrvi:eq:1}
		\Ek{\sqn{\delta^k - \Ek{\delta^k}}} \leq \frac{2\Lavg^2}{b}\E{\sqn{\z^{k} - \w^{k-1}} + \alpha^2\sqn{\z^k- \z^{k-1}}}.
	\end{equation}
	\begin{equation}\label{vrvi:eq:2}
		\Ek{\sqn{\Delta^{k+1/2} - \F(\z^*)}} \leq \frac{2\Lavg^2}{b} \E{\sqn{\z^{k+1} - \w^{k}}} + 2L^2\E{\sqn{\z^{k+1} - \z^*}}.
	\end{equation}
	where $\Ek{\delta^k}$ is equal to
	\begin{equation}\label{vrvi:eq:3}
		\Ek{\delta^k} = F(\z^k) + \alpha(F(\z^k) - F(\z^{k-1})).
	\end{equation}
\end{lemma}
\begin{proof}
    Due to the fact that the batch $S^k$ is generated uniformly and independently for all workers, we can make sure that \eqref{vrvi:eq:3} is correct. Then using definition of $\Delta^k$ from Algorithm~\ref{alg:vrvi}, we get
	\begin{align*}
		\Ek{\sqn{\delta^k - \Ek{\delta^k}}}&\leq
		2\Ek{\sqN{\frac{1}{b}\sum_{j\in S^k} \left(\F_j(\z^k) - \F_j (\w^{k-1}) \right) - \left(\F(\z^k) - \F(\w^{k-1})\right)}}
		\\&\quad+
		2\Ek{\sqN{\frac{\alpha}{b}\sum_{j\in S^k} \left(\F_j(\z^k) - \F_j (\z^{k-1}) \right) - \left(\F(\z^k) - \F(\z^{k-1})\right)}}
	\end{align*}
By randomness and independence of indexes in $S^k$, we obtain	
	\begin{align*}
		\Ek{\sqn{\delta^k - \Ek{\delta^k}}} &=
		\frac{2}{b^2}\Ek{\sum_{j\in S^k}\sqn{\left(\F_j(\z^k) - \F_j (\w^{k-1}) \right) - \left(\F(\z^k) - \F(\w^{k-1})\right)}}
		\\&\quad+
		\frac{2\alpha^2}{b^2}\Ek{\sum_{j\in S^k}\sqn{\left(\F_j(\z^k) - \F_j (\z^{k-1}) \right) - \left(\F(\z^k) - \F(\z^{k-1})\right)}}
	\end{align*}
The property of the second moment $\EE\|\xi - \EE \xi \|^2 = \EE\|\xi\|^2 - \|\EE \xi \|^2$ gives
	\begin{align*}	
		\Ek{\sqn{\delta^k - \Ek{\delta^k}}} \leq&
		\frac{2}{b^2}\Ek{\sum_{j\in S^k}\sqn{\F_j(\z^k) - \F_j (\w^{k-1})}} \\
		&+ \frac{2\alpha^2}{b^2}\Ek{\sum_{j\in S^k}\sqn{\F_j(\z^k) - F_j (\z^{k-1})}}
    \end{align*}
Again from the fact that $S^k$ is generated uniformly and independently for all samples and workers, we can obtain that for each worker $m$ indexes $j_{m,1}^k, \ldots j_{m,b}^k$ have the same uniform distribution, that means
    \begin{align*}
    \Ek{\sqn{\delta^k - \Ek{\delta^k}}} \leq&
		\frac{2}{b^2}\Ek{\sum_{j\in S^k}\sqn{\F_j(\z^k) - \F_j (\w^{k-1})}} \\
		&+ \frac{2\alpha^2}{b^2}\Ek{\sum_{j\in S^k}\sqn{\F_j(\z^k) - F_j (\z^{k-1})}}
		\\=&
		\frac{2}{mb}\sum_{j=1}^m \left(\sqn{\F_j(\z^k) - \F_j(\w^{k-1})} + \alpha^2\sqn{\F_j(\z^k) - \F_j(\z^{k-1})}\right).
	\end{align*}
Using Assumption \ref{as:Lipsh}, we get
    \begin{align*}
	\Ek{\sqn{\delta^k - \Ek{\delta^k}}}&\leq
	\frac{2\Lavg^2}{b}\left(\sqn{\z^{k} - \w^{k-1}} + \alpha^2\sqn{\z^k- \z^{k-1}}\right).
	\end{align*}
This concludes the proof of \eqref{vrvi:eq:1}.	

The proof chain for \eqref{vrvi:eq:2} is very similar.	
	\begin{align*}
		\E{\sqn{\Delta^{k+1/2} - \F(\z^*)}}&\leq
		2\E{\sqN{\Delta^{k+1/2} - \F(\z^{k+1})}} + 2\E{\sqN{\F(\z^{k+1}) - \F(\z^*)}}
		\\&=
		\E{\frac{2}{b^2}\mathbb{E}_{k+1/2}{\sum_{j\in S^{k+1/2}}\sqn{\left(\F_j(\z^{k+1}) - \F_j (\w^{k}) \right) - \left(\F(\z^{k+1}) - \F(\w^{k})\right)}}}
		\\&\quad+
		2L^2\E{\sqN{\z^{k+1} - \z^*}}
		\\&\leq
		\E{\frac{2}{b^2}\mathbb{E}_{k+1/2}{\sum_{j\in S^{k+1/2}}\sqn{\F_j(\z^{k+1}) - \F_j (\w^{k})}}} + 2L^2\E{\sqN{\z^{k+1} - \z^*}}
		\\&\leq 
		\frac{2\Lavg^2}{b} \E{\sqn{\z^{k+1} - \w^{k}}} + 2L^2\E{\sqN{\z^{k+1} - \z^*}}.
	\end{align*}
Here, we used Assumption \ref{as:Lipsh}.
\end{proof}
Before proving the main lemma of this section, let us introduce an auxiliary notation. Throughout the proof, we denote $\mW^{\dagger}: \text{range} \mW \to \text{range} \mW$ the inverse of the map $\mW: \text{range} \mW \to \text{range} \mW$. And we denote $\sqnw{\y} = \langle (\mW^{\dagger} \otimes \mI_d)y ; y \rangle$.

We define the following Lyapunov function:
	\begin{align}
	\label{eq:Lf_fixed}
		\Psi^k&= \bigg(\frac{1}{\eta} + \frac{3\mu}{2}\bigg) \sqn{\z^{k+1} - \z^*} + \frac{1}{\theta}\sqnw{\y^{k+1} - \y^*} 
	    \notag\\&\quad + 2\<\F(\z^k) - \F(\z^{k+1}) - (\y^k - \y^{k+1}), \z^{k+1} - \z^*>
	    +\frac{1}{2\theta}\sqn{\y^{k+1} - \y^k}  \\&\quad  +\frac{1}{8\eta}\sqn{\z^{k+1} - \z^k}
	    + \frac{\gamma + \frac{1}{2}\eta\mu}{p\eta}\sqn{\w^{k+1} - \z^*} + \frac{\gamma}{2\eta}\sqn{\w^k - \z^{k+1}} \notag.
	\end{align}
Here we also use
    \begin{equation}
    \label{eq:ystar_fixed}
		\y^* = \mP\F(\z^*)
    \end{equation}
with $\mP \in \R^{nd\times nd}$, an orthogonal projection matrix onto  the subspace $\cL^\perp$, given as
\begin{equation*}
	\mP = (\mI_M - \tfrac{1}{M}\ones_M\ones_M^\top) \otimes \mI_d,
\end{equation*}
where $\ones_M = (1,\ldots,1)^\top \in \R^M $.

\begin{lemma}\label{lem:fixed_key}
	Consider the problem \eqref{eq:VI_new} (or \eqref{eq:VI} + \eqref{eq:fs}) under Assumptions~\ref{as:Lipsh} and \ref{as:strmon} over a fixed connected graph $\mathcal{G}$ with a gossip matrix $\mW$. Let  $\{\z^k\}$ be the sequence generated by Algorithm~\ref{alg:vrvi} with parameters 
	\begin{align*}
		&\gamma \leq \frac{1}{8}, \quad \eta \leq \min\left\{\frac{\sqrt{\alpha\gamma b}}{\sqrt{8} \cdot \Lavg}, \frac{1}{16L}\right\}, \quad
		\beta \leq \min\left\{\frac{\mu}{4L^2};\frac{b\gamma}{4 \eta \Lavg^2}\right\}, \quad \theta \leq \min\left\{\frac{1}{2\beta};\frac{1}{16 \eta}\right\},
		\\
		&\quad \quad \quad \quad \quad \alpha = \max \left[\bigg(1 - \frac{\mu\eta}{4}\bigg);  \left(1 - \beta\theta\chi^{-1}\right) ; \left( 1 - \frac{p\eta\mu }{2\gamma + \eta\mu} \right)\right].
	\end{align*}
	Then, after $k$ iterations we get
\begin{align*}
	\E{\frac{1}{2\eta} \|\z^k - \z^* \|^2} \leq \max \left[\bigg(1 - \frac{\mu\eta}{4}\bigg);  \left(1 - \beta\theta\chi^{-1}\right) ; \left( 1 - \frac{p\eta\mu }{2\gamma + \eta\mu} \right)\right]^k \cdot \Psi^0.
\end{align*}
\end{lemma}
\begin{proof} 
\textbf{Part 1.} We start the proof from considering update of $\y^{k+1}$ in Algorithm~\ref{alg:vrvi}.
	\begin{align*}
	\frac{1}{\theta}\sqnw{\y^{k+1} - \y^*}
	&=
	\frac{1}{\theta}\sqnw{\y^{k} - \y^*}
	-
	\frac{1}{\theta}\sqnw{\y^{k+1} - \y^k}
	\\&-
	2\<(\mW \otimes \mI_d )(\z^{k+1} - \beta(\Delta^{k+1/2} - \y^k)),(\mWp\otimes \mI_d )(\y^{k+1} - \y^*)>.
	\end{align*}
	Next, we use the fact that $(\mW \otimes \mI_d )(\mWp \otimes \mI_d )=\mP$ and obtain
	\begin{align*}
	\frac{1}{\theta}\sqnw{\y^{k+1} - \y^*}
	&=
	\frac{1}{\theta}\sqnw{\y^{k} - \y^*}
	-
	\frac{1}{\theta}\sqnw{\y^{k+1} - \y^k}
	\\& \quad
	-
	2\<\mP(\z^{k+1} - \beta(\Delta^{k+1/2} - \y^k)),\y^{k+1} - \y^*>.
	\end{align*}
One can observe, that $\z^* \in \cL$ and then $\mP\z^* = 0$. Additionally, using update $\y^{k+1}$, we can note that $\y^k \in \cL^\perp$ for all $k = 0,1,2,\ldots$. Hence,
\begin{align*}
	\frac{1}{\theta}\sqnw{\y^{k+1} - \y^*}
	&=
	\frac{1}{\theta}\sqnw{\y^{k} - \y^*}
	-
	\frac{1}{\theta}\sqnw{\y^{k+1} - \y^k}
	\\&\quad
	- 2\<\mP(\z^{k+1} - \z^*),\y^{k+1} - \y^*>
	+ 2\beta\<\mP(\Delta^{k+1/2} - \F(\z^*)),\y^{k+1} - \y^*>
	\\&\quad
	- 2\beta\<\mP\y^k - \y^*,\y^{k+1} - \y^*> \\
	&=
	\frac{1}{\theta}\sqnw{\y^{k} - \y^*}
	-
	\frac{1}{\theta}\sqnw{\y^{k+1} - \y^k}
	\\&\quad
	-
	2\<\mP(\z^{k+1} - \z^*),\y^{k+1} - \y^*>
	+
	2\beta\<\mP(\Delta^{k+1/2} - \F(\z^*)),\y^{k+1} - \y^*>
	\\&\quad
	-
	2\beta\<\mP(\y^k - \y^*),\y^{k+1} - \y^*>
	\\
	&=
	\frac{1}{\theta}\sqnw{\y^{k} - \y^*}
	-
	\frac{1}{\theta}\sqnw{\y^{k+1} - \y^k}
	\\&\quad
	-
	2\<\z^{k+1} - \z^*,\y^{k+1} - \y^*>
	+
	2\beta\<\Delta^{k+1/2} - \F(\z^*),\y^{k+1} - \y^*>
	\\&\quad
	-
	2\beta\<\y^k - \y^*,\y^{k+1} - \y^*>
\end{align*}	
By the simple fact $\| a + b\|^2 = \|a\|^2 + 2\langle a; b \rangle + \|b\|^2$, we get	
\begin{align*}	
	\frac{1}{\theta}\sqnw{\y^{k+1} - \y^*}&\leq
	\frac{1}{\theta}\sqnw{\y^{k} - \y^*}
	-
	\frac{1}{\theta}\sqnw{\y^{k+1} - \y^k}
	\\&\quad
	-
	2\<\z^{k+1} - \z^*,\y^{k+1} - \y^*>
	+
	\beta\sqn{\Delta^{k+1/2} - \F(\z^*)}
	+
	\beta\sqn{\y^{k+1}-\y^*}
	\\&\quad -
	\beta\sqn{\y^k - \y^*}
	-
	\beta\sqn{\y^{k+1} - \y^*}
	+
	\beta\sqn{\y^{k+1} - \y^k}
	\\&=
	\frac{1}{\theta}\sqnw{\y^{k} - \y^*}
	-
	\frac{1}{\theta}\sqnw{\y^{k+1} - \y^k}
	\\&\quad
	-
	2\<\z^{k+1} - \z^*,\y^{k+1} - \y^*>
	+
	\beta\sqn{\Delta^{k+1/2} - \F(\z^*)}
	-
	\beta\sqn{\y^k - \y^*}
	\\&\quad
	+
	\beta\sqn{\y^{k+1} - \y^k}.
\end{align*}
Using the fact, that $\lambda^+_{\min}(\mW) = \chi^{-1} $ and $\lambda_{\max}(\mW) = 1$, we can note 
	$-\lambda^+_{\min}(\mWp) \leq -1$ and $\chi^{-1}\lambda_{\max}(\mWp) \leq 1$ and then get
	\begin{align*}
	\frac{1}{\theta}\sqnw{\y^{k+1} - \y^*}
	&\leq
	\frac{1}{\theta}\sqnw{\y^{k} - \y^*}
	-
	\frac{1}{\theta}\sqn{\y^{k+1} - \y^k}
	-
	2\<\z^{k+1} - \z^*,\y^{k+1} - \y^*>
	\\&\quad+
	\beta\sqn{\Delta^{k+1/2} - \F(\z^*)}
	-
	\beta\chi^{-1}\sqnw{\y^k - \y^*}
	+
	\beta\sqn{\y^{k+1} - \y^k}
	\\&=
	\left(\frac{1}{\theta} - \beta\chi^{-1}\right)\sqnw{\y^{k} - \y^*}
	-
	\left(\frac{1}{\theta}-\beta\right)\sqn{\y^{k+1} - \y^k}
	\\&\quad-
	2\<\z^{k+1} - \z^*, \y^{k+1} - \y^*>
	+
	\beta\sqn{\Delta^{k+1/2} - \F(\z^*)}.
	\end{align*}
Then we take a full expectation and use \eqref{vrvi:eq:2}:
	\begin{equation}
	\begin{split}\label{dvi:eq:1}
	\E{\frac{1}{\theta}\sqnw{\y^{k+1} - \y^*}}
	&\leq
	\left(\frac{1}{\theta} - \beta\chi^{-1}\right)\E{\sqnw{\y^{k} - \y^*}}
	-
	\left(\frac{1}{\theta}-\beta\right)\E{\sqn{\y^{k+1} - \y^k}}
	\\&\quad-
	2\E{\<\z^{k+1} - \z^*,\y^{k+1} - \y^*>}
	+
	2\beta L^2\E{\sqn{\z^{k+1} - \z^*}} 
	\\&\quad+ 
	\frac{2 \beta \Lavg^2}{b}\E{\sqn{\z^{k+1} - \w^{k}}}.
	\end{split}
	\end{equation}
Here we stop and put aside \eqref{dvi:eq:1}. We will not return to it later. 

\textbf{Part 2.} Now we work with
\begin{align*}
		\frac{1}{\eta}\sqn{\z^{k+1} - \z^*}
		=&
		\frac{1}{\eta}\sqn{\z^k - \z^*} + \frac{2}{\eta}\<\z^{k+1} - \z^k,\z^{k+1} - \z^*> - \frac{1}{\eta}\sqn{\z^{k+1} - \z^k}
		\\=&
		\frac{1}{\eta}\sqn{\z^k - \z^*} + \frac{2\gamma}{\eta} \<\w^k - \z^k, \z^{k+1} - \z^*> - 2\<\Delta^k - (\F(\z^*) - \y^*),\z^{k+1} - \z^*> 
		\\&- \frac{1}{\eta}\sqn{\z^{k+1} - \z^k} 
		\\& -
		2\<\tfrac{1}{\eta}(\z^k + \gamma (\w^k - \z^k) - \eta \Delta^k - \z^{k+1}) + (\F(\z^*) - \y^*),\z^{k+1} - \z^*>.
\end{align*}
Optimality condition for \eqref{eq:VI}+\eqref{eq:fs} it follows, that
	\begin{equation*}
		-F(z^*) \in \partial g(z^*).
	\end{equation*}
Let us define $\Delta^* \in (\R^d)^M$ as
	\begin{equation*}
	\Delta^*= \tfrac{1}{M}[F(z^*), \ldots F(z^*)]^T.
	\end{equation*}
It is to note that $-\Delta^* \in M \partial\g(\z^*)$. On the other hand 
\begin{equation*}
	\Delta^*= (\ones_M\ones_M^\top \otimes \mI_d) \F(\z^*) = M (\F(\z^*) - \mP \F(\z^*)) = M (\F(\z^*) - \y^*))
\end{equation*}
This means that $-(\F(\z^*) - \y^*) \in \partial\g(\z^*)$. From update for $\z^{k+1}$ of Algorithm~\ref{alg:vrvi} it follows, that
	\begin{equation*}
		\z^k + \gamma(\w^k - \z^k) - \eta \Delta^k - \z^{k+1} \in \partial (\eta \g)(\z^{k+1}).
	\end{equation*}
Hence, from  monotonicity of $\partial \g(\cdot)$ we get
	\begin{align*}
	\frac{1}{\eta}\E{\sqn{\z^{k+1} - \z^*}}
	&\leq
	\frac{1}{\eta} \E{\sqn{\z^k - \z^*}} + \frac{2\gamma}{\eta} \E{\<\w^k - \z^k, \z^{k+1} - \z^*>}
	\\&\quad -2\E{\<\Delta^k - (\F(\z^*) - \y^*), \z^{k+1} - \z^*>}  - \frac{1}{\eta}\E{\sqn{\z^{k+1} - \z^k}}
	\\&= 
	\frac{1}{\eta} \E{\sqn{\z^k - \z^*}} + \frac{2\gamma}{\eta} \E{\<\w^k - \z^*, \z^{k+1} - \z^*>}
	\\&\quad - \frac{2\gamma}{\eta}\E{\<\z^k - \z^*, \z^{k+1} - \z^*>}
	-2\E{\<\Delta^k - (\F(\z^*) - \y^*), \z^{k+1} - \z^*>} 
	\\&\quad - \frac{1}{\eta}\E{\sqn{\z^{k+1} - \z^k}}
	\\&=
	\frac{1}{\eta} \E{\sqn{\z^k - \z^*}} + \frac{\gamma}{\eta} \E{\sqn{\w^k - \z^*} + \sqn{\z^{k+1} - \z^*} - \sqn{\z^{k+1} - \w^k}} 
	\\&\quad- \frac{\gamma}{\eta}\E{\sqn{\z^{k+1} - \z^*} + \sqn{\z^k - \z^*} - \sqn{\z^{k+1} - \z^k}}
	\\&\quad-2\E{\<\Delta^k - (\F(\z^*) - \y^*), \z^{k+1} - \z^*>}  - \frac{1}{\eta}\E{\sqn{\z^{k+1} - \z^k}}
	\\&=
	\frac{1}{\eta} \E{\sqn{\z^k - \z^*}} + \frac{\gamma}{\eta}\E{\sqn{\w^k - \z^*}} - \frac{\gamma}{\eta}\E{\sqn{\z^k - \z^*}} 
	\\&\quad
	- \frac{\gamma}{\eta}\E{\sqn{\w^k - \z^{k+1}}}
	-
	2\E{\<\Delta^k - (\F(\z^*) - \y^*), \z^{k+1} - \z^*>}
	\\&\quad
	 - \frac{1-\gamma}{\eta}\E{\sqn{\z^{k+1} - \z^k}}.
	\end{align*}
In previous we also use the simple fact $\| a + b\|^2 = \|a\|^2 + 2\langle a; b \rangle + \|b\|^2$ twice. Small rearrangement gives
	\begin{align*}
	\frac{1}{\eta} \E{\sqn{\z^{k+1} - \z^*}}
	&\leq
	\frac{1}{\eta}\E{\sqn{\z^k - \z^*}} + \frac{\gamma}{\eta}\E{\sqn{\w^k - \z^*}} - \frac{\gamma}{\eta}\E{\sqn{\z^k - \z^*}} 
	\\&\quad
	- \frac{\gamma}{\eta}\E{\sqn{\w^k - \z^{k+1}}}  - \frac{1-\gamma}{\eta}\E{\sqn{\z^{k+1} - \z^k}} 
	\\&\quad
	- 2\E{\<\Ek{\delta^k} - (\y^k + \alpha(\y^k - \y^{k-1})) - (\F(\z^*) - \y^*), \z^{k+1} - \z^*>}
	\\&\quad
	- 2\E{\<\delta^k - \Ek{\delta^k}, \z^{k+1} - \z^k>} - 2\E{\<\delta^k - \Ek{\delta^k}, \z^{k} - \z^*>}.
	\end{align*}
Using the tower property of expectation we can obtain the following:
	\begin{align*}
		\E{\<\Ek{\delta^k} - \delta^k, \z^k - \z^*>} &= \E{\Ek{\<\Ek{\delta^k} - \delta^k, \z^k - \z^*>}}
		\\&=
		\E{{\<\Ek{\Ek{\delta^k} - \delta^k}, \z^k - \z^*>}}
		\\&=
		\E{{\<\Ek{\delta^k} - \Ek{\delta^k}, \z^k - \z^*>}} = 0.
	\end{align*}
Hence, with \eqref{vrvi:eq:3}
	\begin{align*}
	\frac{1}{\eta} \E{\sqn{\z^{k+1} - \z^*}}
	&\leq
	\frac{1}{\eta}\E{\sqn{\z^k - \z^*}} + \frac{\gamma}{\eta}\E{\sqn{\z^k - \z*}} - \frac{\gamma}{\eta}\E{\sqn{\z^k - \z^*}}  
	\\&\quad
	- \frac{\gamma}{\eta}\E{\sqn{\w^k - \z^{k+1}}} - \frac{1-\gamma}{\eta}\E{\sqn{\z^{k+1} - \z^k}} 
	\\&\quad - 2\E{\<\F(\z^k) + \alpha(\F(\z^k) - \F(\z^{k-1})) - \F(\z^*) - (\y^k + \alpha(\y^k - \y^{k-1}) - \y^*), \z^{k+1} - \z^*>}
	\\&\quad+ 2\E{\<\Ek{\delta^k} - \delta^k , \z^{k+1} - \z^k>}.
	\end{align*}
Using the Young's inequality we get
	\begin{align*}
	\frac{1}{\eta} \E{\sqn{\z^{k+1} - \z^*}}
	&\leq
	\frac{1}{\eta}\E{\sqn{\z^k - \z^*}} + \frac{\gamma}{\eta}\E{\sqn{\w^k - \z^*}} - \frac{\gamma}{\eta}\E{\sqn{\z^k - \z^*}} 
	\\& \quad
	- \frac{\gamma}{\eta}\E{\sqn{\w^k - \z^{k+1}}} - \frac{1-\gamma}{\eta}\E{\sqn{\z^{k+1} - \z^k}} 
	\\& \quad - 2\E{\<\F(\z^k) + \alpha(\F(\z^k) - \F(\z^{k-1})) - \F(\z^*) - (\y^k + \alpha(\y^k - \y^{k-1}) - \y^*), \z^{k+1} - \z^*>}
	\\& \quad+ 2\eta\E{\sqn{\Ek{\delta^k} - \delta^k}}  + \frac{1}{2\eta}\E{\sqn{\z^{k+1} - \z^k}}
    \\&= \frac{1}{\eta}\E{\sqn{\z^k - \z^*}} + \frac{\gamma}{\eta}\E{\sqn{\w^k - \z^*}} - \frac{\gamma}{\eta}\E{\sqn{\z^k - \z^*}}  
	\\& \quad
	- \frac{\gamma}{\eta}\E{\sqn{\w^k - \z^{k+1}}} - \frac{1/2-\gamma}{\eta}\E{\sqn{\z^{k+1} - \z^k}}  + 2\eta\E{\sqn{\Ek{\delta^k} - \delta^k}}
	\\&\quad - 2\E{\<\F(\z^k) + \alpha(\F(\z^k) - \F(\z^{k-1})) - \F(\z^*) - (\y^k + \alpha(\y^k - \y^{k-1}) - \y^*), \z^{k+1} - \z^*>}.
	\end{align*}
By \eqref{vrvi:eq:1} we get
	\begin{align*}
	\frac{1}{\eta} \E{\sqn{\z^{k+1} - \z^*}}
	&\leq
	\frac{1}{\eta}\E{\sqn{\z^k - \z^*}} + \frac{\gamma}{\eta}\E{\sqn{\w^k - \z^*}} - \frac{\gamma}{\eta}\E{\sqn{\z^k - \z^*}}  
	\\&\quad
	- \frac{\gamma}{\eta}\E{\sqn{\w^k - \z^{k+1}}}
	- \frac{1/2-\gamma}{\eta}\E{\sqn{\z^{k+1} - \z^k}}  
	\\&\quad + \frac{4\Lavg^2\eta}{b}\E{\sqn{\z^{k} - \w^{k-1}}} + \frac{4\alpha^2\Lavg^2\eta}{b} \E{\sqn{\z^k- \z^{k-1}}}
	\\&\quad - 2\E{\<\F(\z^k) + \alpha(\F(\z^k) - \F(\z^{k-1})) - \F(\z^*) - (\y^k + \alpha(\y^k - \y^{k-1}) - \y^*), \z^{k+1} - \z^*>}
	\\&=\frac{1}{\eta}\E{\sqn{\z^k - \z^*}} - 2\E{\<\F(\z^{k+1}) - \F(\z^*), \z^{k+1} - \z^*>} + \frac{\gamma}{\eta}\E{\sqn{\w^k - \z^*}} 
	\\& \quad
	- \frac{\gamma}{\eta}\E{\sqn{\z^k - \z^*}} - \frac{\gamma}{\eta}\E{\sqn{\w^k - \z^{k+1}}}  - \frac{1/2-\gamma}{\eta}\E{\sqn{\z^{k+1} - \z^k}}  
	\\& \quad  + \frac{4\Lavg^2\eta}{b}\E{\sqn{\z^{k} - \w^{k-1}}} + \frac{4\alpha^2\Lavg^2\eta}{b} \E{\sqn{\z^k- \z^{k-1}}} 
	\\& \quad - 2\E{\<\F(\z^k) - \F(\z^{k+1}) + \alpha(\F(\z^k) - \F(\z^{k-1})) - (\y^k + \alpha(\y^k - \y^{k-1}) - \y^*), \z^{k+1} - \z^*>}.
	\end{align*}
Assumption \ref{as:strmon} about $\mu$-strong monotonicity of $F$ gives
	\begin{align*}
	\frac{1}{\eta} &\E{\sqn{\z^{k+1} - \z^*}} 
	\\&\leq
	\frac{1}{\eta}\E{\sqn{\z^k - \z^*}} - 2\mu\E{\sqn{\z^{k+1} - \z^*}} + \frac{\gamma}{\eta}\E{\sqn{\w^k - \z^*}} - \frac{\gamma}{\eta}\E{\sqn{\z^k - \z^*}} 
	\\&\quad
	- \frac{\gamma}{\eta}\E{\sqn{\w^k - \z^{k+1}}}  - \frac{1/2-\gamma}{\eta}\E{\sqn{\z^{k+1} - \z^k}}
	\\&\quad
	+ \frac{4\Lavg^2\eta}{b}\E{\sqn{\z^{k} - \w^{k-1}}} + \frac{4\alpha^2\Lavg^2\eta}{b} \E{\sqn{\z^k- \z^{k-1}}}
	\\&\quad - 2\E{\<\F(\z^k) - \F(\z^{k+1}) + \alpha(\F(\z^k) - \F(\z^{k-1})) - (\y^k + \alpha(\y^k - \y^{k-1}) - \y^*), \z^{k+1} - \z^*>}
	\\&=\frac{1}{\eta}\E{\sqn{\z^k - \z^*}} - 2\mu\E{\sqn{\z^{k+1} - \z^*}} + \frac{\gamma}{\eta}\E{\sqn{\w^k - \z^*}} - \frac{\gamma}{\eta}\E{\sqn{\z^k - \z^*}} 
	\\& \quad
	- \frac{\gamma}{\eta}\E{\sqn{\w^k - \z^{k+1}}}  - \frac{1/2-\gamma}{\eta}\E{\sqn{\z^{k+1} - \z^k}} 
	\\&\quad
	+ \frac{4\Lavg^2\eta}{b}\E{\sqn{\z^{k} - \w^{k-1}}} + \frac{4\alpha^2\Lavg^2\eta}{b} \E{\sqn{\z^k- \z^{k-1}}}
	\\&\quad -2\E{\<\F(\z^k) - \F(\z^{k+1}) - (\y^k - \y^{k+1}), \z^{k+1} - \z^*>} + 2\E{\<\y^{k+1} - \y^*, \z^{k+1} - \z^*>}
	\\&\quad -2\alpha\E{\<\F(\z^k) - \F(\z^{k-1}) - (\y^{k} - \y^{k-1}), \z^{k} - \z^*>}
	\\&\quad -2\alpha\E{\<\F(\z^k) - \F(\z^{k-1}) - (\y^{k} - \y^{k-1}), \z^{k+1} - \z^k>}.
	\end{align*}
\textbf{Part 3.} Now we are ready to consider \eqref{dvi:eq:1}. We sum up previous expression and \eqref{dvi:eq:1}. After small rearrangement, we have
	\begin{align*}
	\frac{1}{\eta} &\E{\sqn{\z^{k+1} - \z^*}} + \frac{1}{\theta}\sqnw{\y^{k+1} - \y^*} + 2\E{\<\F(\z^k) - \F(\z^{k+1}) - (\y^k - \y^{k+1}), \z^{k+1} - \z^*>} 
	\\&\leq
	\frac{1}{\eta}\E{\sqn{\z^k - \z^*}} - (2\mu - 2\beta L^2)\E{\sqn{\z^{k+1} - \z^*}} +\left(1 - \beta\theta\chi^{-1}\right)\cdot \frac{1}{\theta}\sqnw{\y^{k} - \y^*} 
	\\&\quad-2\alpha\E{\<\F(\z^k) - \F(\z^{k-1}) - (\y^{k} - \y^{k-1}), \z^{k} - \z^*>} + \frac{\gamma}{\eta}\E{\sqn{\w^k - \z^*}} - \frac{\gamma}{\eta}\E{\sqn{\z^k - \z^*}} 
	\\&\quad
	- \left(\frac{\gamma}{\eta} - \frac{2 \beta \Lavg^2}{b}\right)\E{\sqn{\w^k - \z^{k+1}}}  - \frac{1/2-\gamma}{\eta}\E{\sqn{\z^{k+1} - \z^k}}  + \frac{4\Lavg^2\eta}{b}\E{\sqn{\z^{k} - \w^{k-1}}} 
	\\&\quad + \frac{4\alpha^2\Lavg^2\eta}{b} \E{\sqn{\z^k- \z^{k-1}}} -2\alpha\E{\<\F(\z^k) - \F(\z^{k-1}), \z^{k+1} - \z^k>}
	\\&\quad
	+2\alpha\E{\<\y^{k} - \y^{k-1}, \z^{k+1} - \z^k>} - \left(\frac{1}{\theta}-\beta\right)\E{\sqn{\y^{k+1} - \y^k}}.
	\end{align*}
With $\theta \leq \frac{1}{2\beta}$ (or $\beta \leq \tfrac{1}{2\theta}$) and Young's inequality we get
	\begin{align*}
	\frac{1}{\eta} &\E{\sqn{\z^{k+1} - \z^*}} + \frac{1}{\theta}\sqnw{\y^{k+1} - \y^*} + 2\E{\<\F(\z^k) - \F(\z^{k+1}) - (\y^k - \y^{k+1}), \z^{k+1} - \z^*>} 
	\\&\leq
	\frac{1}{\eta}\E{\sqn{\z^k - \z^*}} - (2\mu - 2\beta L^2)\E{\sqn{\z^{k+1} - \z^*}} +\left(1 - \beta\theta\chi^{-1}\right)\cdot \frac{1}{\theta}\sqnw{\y^{k} - \y^*} 
	\\&\quad-2\alpha\E{\<\F(\z^k) - \F(\z^{k-1}) - (\y^{k} - \y^{k-1}), \z^{k} - \z^*>} + \frac{\gamma}{\eta}\E{\sqn{\w^k - \z^*}} - \frac{\gamma}{\eta}\E{\sqn{\z^k - \z^*}} 
	\\&\quad
	- \left(\frac{\gamma}{\eta} - \frac{2 \beta \Lavg^2}{b}\right)\E{\sqn{\w^k - \z^{k+1}}}  - \frac{1/2-\gamma}{\eta}\E{\sqn{\z^{k+1} - \z^k}}  + \frac{4\Lavg^2\eta}{b}\E{\sqn{\z^{k} - \w^{k-1}}} 
	\\&\quad + \frac{4\alpha^2\Lavg^2\eta}{b} \E{\sqn{\z^k- \z^{k-1}}} +\frac{1}{8\eta}\E{\sqn{\z^{k+1} - \z^k}} + 8\eta\alpha^2\E{\sqn{\F(\z^k) - \F(\z^{k-1})}} 
	\\&\quad+ 2\theta\alpha\E{\sqn{\z^{k+1} - \z^k}} +\frac{\alpha}{2\theta}\E{\sqn{\y^{k} - \y^{k-1}}} -\frac{1}{2\theta}\E{\sqn{\y^{k+1} - \y^k}}.
	\end{align*}
By choosing $\theta \leq \frac{1}{16\eta}$ and $\alpha \leq 1$
    \begin{align*}
	\frac{1}{\eta} &\E{\sqn{\z^{k+1} - \z^*}} + \frac{1}{\theta}\sqnw{\y^{k+1} - \y^*} + 2\E{\<\F(\z^k) - \F(\z^{k+1}) - (\y^k - \y^{k+1}), \z^{k+1} - \z^*>} 
	\\&\leq
	\frac{1}{\eta}\E{\sqn{\z^k - \z^*}} - (2\mu - 2\beta L^2)\E{\sqn{\z^{k+1} - \z^*}} +\left(1 - \beta\theta\chi^{-1}\right)\cdot \frac{1}{\theta}\sqnw{\y^{k} - \y^*} 
	\\&\quad-2\alpha\E{\<\F(\z^k) - \F(\z^{k-1}) - (\y^{k} - \y^{k-1}), \z^{k} - \z^*>} + \frac{\gamma}{\eta}\E{\sqn{\w^k - \z^*}} - \frac{\gamma}{\eta}\E{\sqn{\z^k - \z^*}} 
	\\&\quad
	- \left(\frac{\gamma}{\eta} - \frac{2 \beta \Lavg^2}{b}\right)\E{\sqn{\w^k - \z^{k+1}}}  - \frac{1/4-\gamma}{\eta}\E{\sqn{\z^{k+1} - \z^k}}  + \frac{4\Lavg^2\eta}{b}\E{\sqn{\z^{k} - \w^{k-1}}} 
	\\&\quad + \frac{4\alpha^2\Lavg^2\eta}{b} \E{\sqn{\z^k- \z^{k-1}}} + 8\eta\alpha^2\E{\sqn{\F(\z^k) - \F(\z^{k-1})}}+\frac{\alpha}{2\theta}\E{\sqn{\y^{k} - \y^{k-1}}} 
	\\&\quad -\frac{1}{2\theta}\E{\sqn{\y^{k+1} - \y^k}}.
	\end{align*}
With $L$-Lipshitzness of $\F$ (Assumption \ref{as:Lipsh}) and $\gamma \leq \frac{1}{8}$, we have
    \begin{align*}
	\frac{1}{\eta} &\E{\sqn{\z^{k+1} - \z^*}} + \frac{1}{\theta}\sqnw{\y^{k+1} - \y^*} + 2\E{\<\F(\z^k) - \F(\z^{k+1}) - (\y^k - \y^{k+1}), \z^{k+1} - \z^*>} 
	\\&\leq
	\frac{1}{\eta}\E{\sqn{\z^k - \z^*}} - (2\mu - 2\beta L^2)\E{\sqn{\z^{k+1} - \z^*}} +\left(1 - \beta\theta\chi^{-1}\right)\cdot \frac{1}{\theta}\sqnw{\y^{k} - \y^*} 
	\\&\quad-2\alpha\E{\<\F(\z^k) - \F(\z^{k-1}) - (\y^{k} - \y^{k-1}), \z^{k} - \z^*>} + \frac{\gamma}{\eta}\E{\sqn{\w^k - \z^*}} - \frac{\gamma}{\eta}\E{\sqn{\z^k - \z^*}} 
	\\&\quad
	- \left(\frac{\gamma}{\eta} - \frac{2 \beta \Lavg^2}{b}\right)\E{\sqn{\w^k - \z^{k+1}}}  - \frac{1}{8\eta}\E{\sqn{\z^{k+1} - \z^k}}  + \frac{4\Lavg^2\eta}{b}\E{\sqn{\z^{k} - \w^{k-1}}} 
	\\&\quad + \left(\frac{32\alpha^2\Lavg^2\eta^2}{b} 
	 + 64\eta^2\alpha^2L^2\right)  \cdot \frac{1}{8\eta} \E{\sqn{\z^k- \z^{k-1}}} +\frac{\alpha}{2\theta}\E{\sqn{\y^{k} - \y^{k-1}}}
	 \\&\quad-\frac{1}{2\theta}\E{\sqn{\y^{k+1} - \y^k}}.
	\end{align*}
Small rearrangement gives
	\begin{align*}
	\frac{1}{\eta} &\E{\sqn{\z^{k+1} - \z^*}} + \frac{1}{\theta}\sqnw{\y^{k+1} - \y^*} + 2\E{\<\F(\z^k) - \F(\z^{k+1}) - (\y^k - \y^{k+1}), \z^{k+1} - \z^*>} 
	\\&\quad + \frac{1}{2\theta}\E{\sqn{\y^{k+1} - \y^k}} +\frac{1}{8\eta}\E{\sqn{\z^{k+1} - \z^k}}
	\\&\leq
	\frac{1}{\eta}\E{\sqn{\z^k - \z^*}} - (2\mu - 2\beta L^2)\E{\sqn{\z^{k+1} - \z^*}} +\left(1 - \beta\theta\chi^{-1}\right)\cdot \frac{1}{\theta}\sqnw{\y^{k} - \y^*} 
	\\&\quad-2\alpha\E{\<\F(\z^k) - \F(\z^{k-1}) - (\y^{k} - \y^{k-1}), \z^{k} - \z^*>} + \frac{\gamma}{\eta}\E{\sqn{\w^k - \z^*}} - \frac{\gamma}{\eta}\E{\sqn{\z^k - \z^*}} 
	\\&\quad
	- \left(\frac{\gamma}{\eta} - \frac{2 \beta \Lavg^2}{b}\right)\E{\sqn{\w^k - \z^{k+1}}}  + \frac{4\Lavg^2\eta}{b}\E{\sqn{\z^{k} - \w^{k-1}}} 
	\\&\quad + \left(\frac{32\alpha^2\Lavg^2\eta^2}{b} 
	 + 64\eta^2\alpha^2L^2\right)  \cdot \frac{1}{8\eta} \E{\sqn{\z^k- \z^{k-1}}} +\frac{\alpha}{2\theta}\E{\sqn{\y^{k} - \y^{k-1}}}.
	\end{align*}
With our choice of $\eta \leq \min\left\{\frac{\sqrt{\alpha\gamma b}}{\sqrt{8} \cdot \Lavg}, \frac{1}{16L}\right\}$ and $\alpha < 1$, we get
	\begin{align*}
	\frac{1}{\eta} &\E{\sqn{\z^{k+1} - \z^*}} + \frac{1}{\theta}\sqnw{\y^{k+1} - \y^*} + 2\E{\<\F(\z^k) - \F(\z^{k+1}) - (\y^k - \y^{k+1}), \z^{k+1} - \z^*>}
	\\&\quad +\frac{1}{2\theta}\sqn{\y^{k+1} - \y^k} +\frac{1}{8\eta}\E{\sqn{\z^{k+1} - \z^k}}
	\\&\leq
	\frac{1}{\eta}\E{\sqn{\z^k - \z^*}} - (2\mu - 2\beta L^2)\E{\sqn{\z^{k+1} - \z^*}} +\left(1 - \beta\theta\chi^{-1}\right)\cdot \frac{1}{\theta}\sqnw{\y^{k} - \y^*} 
	\\&\quad-\alpha \cdot 2\E{\<\F(\z^k) - \F(\z^{k-1}) - (\y^{k} - \y^{k-1}), \z^{k} - \z^*>} +\alpha \cdot\frac{1}{2\theta}\E{\sqn{\y^{k} - \y^{k-1}}}
	\\&\quad
	+ \alpha \cdot \frac{1}{8\eta}\E{\sqn{\z^k - \z^{k-1}}}
	\\&\quad
	+ \frac{\gamma}{\eta}\E{\sqn{\w^k - \z^*}} - \frac{\gamma}{\eta}\E{\sqn{\z^k - \z^*}} 
	\\&\quad
	- \left(\frac{\gamma}{\eta} - \frac{2 \beta \Lavg^2}{b}\right)\E{\sqn{\w^k - \z^{k+1}}}   + \frac{4\Lavg^2\eta}{b}\E{\sqn{\z^{k} - \w^{k-1}}} .
	\end{align*}
Now, we add $\tfrac{\mu}{2} \E{\sqn{\z^{k+1} - \z^*}} + \tfrac{\gamma + \tfrac{1}{2}\eta\mu}{p\eta}\E{\sqn{\w^{k+1} - \z^*}}$ to both sides and use update for $w^{k+1}$  of Algorithm~\ref{alg:vrvi} 
	\begin{align*}
	\frac{\gamma + \tfrac{1}{2}\eta\mu}{p\eta}\E{\mathbb{E}_{w^{k+1}}{\sqn{\w^{k+1} - \z^*}}} = \frac{\gamma + \tfrac{1}{2}\eta\mu}{\eta}\E{\sqn{\z^{k} - \z^*}} + \frac{(\gamma + \tfrac{1}{2}\eta\mu)(1-p)}{\eta p}\E{\sqn{\w^{k} - \z^*}},
	\end{align*}
and get 
	\begin{align*}
	\bigg(\frac{1}{\eta} &+ \frac{\mu}{2}\bigg) \E{\sqn{\z^{k+1} - \z^*}} + \frac{1}{\theta}\sqnw{\y^{k+1} - \y^*}
	\\&\quad
	+ 2\E{\<\F(\z^k) - \F(\z^{k+1}) - (\y^k - \y^{k+1}), \z^{k+1} - \z^*>} + \frac{\gamma + \tfrac{1}{2}\eta\mu}{p\eta}\E{\sqn{\w^{k+1} - \z^*}}
	\\&\quad  +\frac{1}{2\theta}\sqn{\y^{k+1} - \y^k} +\frac{1}{8\eta}\E{\sqn{\z^{k+1} - \z^k}} + \left(\frac{\gamma}{\eta} - \frac{2 \beta \Lavg^2}{b}\right)\E{\sqn{\w^k - \z^{k+1}}}
	\\&\leq
	\left(\frac{1}{\eta} + \frac{\mu}{2}\right)\E{\sqn{\z^k - \z^*}} - \left(\frac{3}{2}\mu - 2\beta L^2\right)\E{\sqn{\z^{k+1} - \z^*}} 
	\\&\quad+\left(1 - \beta\theta\chi^{-1}\right)\cdot \frac{1}{\theta}\sqnw{\y^{k} - \y^*} -\alpha \cdot 2\E{\<\F(\z^k) - \F(\z^{k-1}) - (\y^{k} - \y^{k-1}), \z^{k} - \z^*>} 
	\\&\quad
	+\alpha \cdot\frac{1}{2\theta}\E{\sqn{\y^{k} - \y^{k-1}}}
	+ \alpha \cdot \frac{1}{8\eta}\E{\sqn{\z^k - \z^{k-1}}}
	\\&\quad
	+ \left(1-p +  \frac{p\gamma}{\gamma + \tfrac{1}{2}\eta\mu}\right) \frac{(\gamma + \tfrac{1}{2}\eta\mu)}{\eta p} \E{\sqn{\w^k - \z^*}}  + \frac{4\Lavg^2\eta}{b}\E{\sqn{\z^{k} - \w^{k-1}}} .
	\end{align*}
Note that $\eta \leq \min\left\{\frac{\sqrt{\alpha\gamma b}}{\sqrt{8} \cdot \Lavg}, \frac{1}{16L}\right\}$, $\beta \leq \min\left\{\frac{\mu}{4L^2};\frac{b\gamma}{4 \eta \Lavg^2}\right\}$, then we get that
$$
\left(\frac{\gamma}{\eta} - \frac{2 \beta \Lavg^2}{b}\right) \geq \frac{\gamma}{2\eta}; \quad -\left(\frac{3}{2}\mu - 2\beta L^2\right) \leq -\mu;\quad \frac{4\Lavg^2\eta}{b} \leq \alpha \cdot \frac{\gamma}{2\eta}.
$$
Hence, it holds
\begin{align*}
	\bigg(\frac{1}{\eta} &+ \frac{3\mu}{2}\bigg) \E{\sqn{\z^{k+1} - \z^*}} + \frac{1}{\theta}\sqnw{\y^{k+1} - \y^*} 
	\\&\quad + 2\E{\<\F(\z^k) - \F(\z^{k+1}) - (\y^k - \y^{k+1}), \z^{k+1} - \z^*>} + \frac{\gamma + \tfrac{1}{2}\eta\mu}{p\eta}\E{\sqn{\w^{k+1} - \z^*}} 
	\\&\quad
	+\frac{1}{2\theta}\sqn{\y^{k+1} - \y^k} +\frac{1}{8\eta}\E{\sqn{\z^{k+1} - \z^k}} + \frac{\gamma}{2\eta}\E{\sqn{\w^k - \z^{k+1}}}
	\\&\leq
	\left(\frac{1}{\eta} + \frac{\mu}{2}\right)\E{\sqn{\z^k - \z^*}} +\left(1 - \beta\theta\chi^{-1}\right)\cdot \frac{1}{\theta}\sqnw{\y^{k} - \y^*} 
	\\&\quad-\alpha \cdot 2\E{\<\F(\z^k) - \F(\z^{k-1}) - (\y^{k} - \y^{k-1}), \z^{k} - \z^*>} 
	\\&\quad
	+\alpha \cdot\frac{1}{2\theta}\E{\sqn{\y^{k} - \y^{k-1}}}
	+ \alpha \cdot \frac{1}{8\eta}\E{\sqn{\z^k - \z^{k-1}}}
	\\&\quad
	+ \left( 1 - \frac{p\eta\mu }{2\gamma + \eta\mu} \right) \cdot \frac{\gamma + \tfrac{1}{2}\eta\mu}{p\eta} \E{\sqn{\w^k - \z^*}}  + \alpha \cdot \frac{\gamma}{2\eta}\E{\sqn{\x^{k} - \w^{k-1}}} .
	\end{align*}
With our $\eta$ we have $\eta \mu \leq 1$ and then
$$
\bigg(\frac{1}{\eta} + \frac{\mu}{2}\bigg) \leq \bigg(\frac{1}{\eta} + \frac{3\mu}{2}\bigg) \bigg(1 - \frac{\mu\eta}{4}\bigg). 
$$
It gives
\begin{align*}
	\bigg(\frac{1}{\eta} &+ \frac{3\mu}{2}\bigg) \E{\sqn{\z^{k+1} - \z^*}} + \frac{1}{\theta}\sqnw{\y^{k+1} - \y^*} 
	\\&\quad + 2\E{\<\F(\z^k) - \F(\z^{k+1}) - (\y^k - \y^{k+1}), \z^{k+1} - \z^*>} + \frac{\gamma + \tfrac{1}{2}\eta\mu}{p\eta}\E{\sqn{\w^{k+1} - \z^*}}
	\\&\quad
	+\frac{1}{2\theta}\sqn{\y^{k+1} - \y^k} +\frac{1}{8\eta}\E{\sqn{\z^{k+1} - \z^k}} + \frac{\gamma}{2\eta}\E{\sqn{\w^k - \z^{k+1}}}
	\\&\leq
	\bigg(1 - \frac{\mu\eta}{4}\bigg) \cdot \bigg(\frac{1}{\eta} + \frac{3\mu}{2}\bigg)\E{\sqn{\z^k - \z^*}} +\left(1 - \beta\theta\chi^{-1}\right)\cdot \frac{1}{\theta}\sqnw{\y^{k} - \y^*} 
	\\&\quad-\alpha \cdot 2\E{\<\F(\z^k) - \F(\z^{k-1}) - (\y^{k} - \y^{k-1}), \z^{k} - \z^*>} 
	\\&\quad
	+\alpha \cdot\frac{1}{2\theta}\E{\sqn{\y^{k} - \y^{k-1}}}
	+ \alpha \cdot \frac{1}{8\eta}\E{\sqn{\z^k - \z^{k-1}}}
	\\&\quad
	+ \left( 1 - \frac{p\eta\mu }{2\gamma + \eta\mu} \right) \cdot \frac{\gamma + \tfrac{1}{2}\eta\mu}{p\eta} \E{\sqn{\w^k - \z^*}}  + \alpha \cdot \frac{\gamma}{2\eta}\E{\sqn{x^{k} - w^{k-1}}} .
	\end{align*}
Definition \eqref{eq:Lf_fixed} of the Lyapunov function and the choice $\alpha = \max \left[\left(1 - \tfrac{\mu\eta}{4}\right);  \left(1 - \beta\theta\chi^{-1}\right) ; \left( 1 - \tfrac{p\eta\mu }{2\gamma + \eta\mu} \right)\right] $ move us to
	\begin{align*}
	\EE{\Psi^{k+1}} \leq \max \left[\bigg(1 - \frac{\mu\eta}{4}\bigg);  \left(1 - \beta\theta\chi^{-1}\right) ; \left( 1 - \frac{p\eta\mu }{2\gamma + \eta\mu} \right)\right]\EE{\Psi^k}.
	\end{align*}
	It remains to show, that $$\Psi^k \geq \frac{1}{2\eta}\sqn{\z^{k} - \z^*}.$$
	\begin{align*}
		\Psi^k&\geq \frac{1}{\eta}\sqn{\z^{k} - \z^*}
		+
		\frac{1}{8\eta}\sqn{\z^{k} - \z^{k-1}}
		+
		\frac{1}{2\theta}\sqn{\y^{k} - \y^{k-1}}
		\\&\quad
		+
		2\<\F(\z^{k}) -\F(\z^{k-1}) - (\y^{k} - \y^{k-1}),\z^*-\z^k>
		\\&\geq 
		\frac{1}{\eta}\sqn{\z^{k} - \z^*}
		+
		\frac{1}{8\eta}\sqn{\z^{k} - \z^{k-1}}
		+
		\frac{1}{2\theta}\sqn{\y^{k} - \y^{k-1}}
		\\&\quad 
		-
		\frac{1}{2\theta}\sqn{\y^k - \y^{k-1}}
		-
		2\theta\sqn{\z^k - \z^*}
		-
		\frac{1}{8\eta L^2}\sqn{\F(\z^k) - \F(\z^{k-1})}
		-
		8\eta L^2\sqn{\z^k - \z^*}.
	\end{align*}
	Using $L$-Lipschitzness of $\F$ we get 
	\begin{align*}
	\Psi^k&\geq \frac{1}{2\eta} \left( 2 - 8\eta^2 L^2 - 4 \theta \eta\right)\sqn{\z^{k} - \z^*}.
	\end{align*}
	With $\eta \leq \frac{1}{16 L}$ and $\theta \leq \frac{1}{16 \eta}$ we get $$\Psi^k \geq \frac{1}{2\eta}\sqn{\z^{k} - \z^*}.$$
\end{proof}

\begin{theorem}[Theorem \ref{th:ALg1_conv}]\label{th:app_fixed}
	Consider the problem \eqref{eq:VI_new} (or \eqref{eq:VI} + \eqref{eq:fs}) under Assumptions~\ref{as:Lipsh} and \ref{as:strmon} over a fixed connected graph $\mathcal{G}$ with a gossip matrix $\mW$. Let  $\{\z^k\}$ be the sequence generated by Algorithm~\ref{alg:vrvi} with parameters 
	\begin{align*}
		&\gamma = p \leq \frac{1}{8}, \quad \eta = \min\left\{\frac{\sqrt{\gamma b}}{4 \Lavg}, \frac{1}{16L \sqrt{\chi}}\right\}, \quad \beta = \min\left\{\frac{\mu}{4L^2};\frac{b\gamma}{4 \eta \Lavg^2}\right\}, \quad \theta = \min\left\{\frac{1}{2\beta};\frac{1}{16 \eta}\right\}, 
		\\
		&\hspace{2cm} \alpha = \max \left[\bigg(1 - \frac{\mu\eta}{4}\bigg);  \left(1 - \beta\theta\chi^{-1}\right) ; \left( 1 - \frac{p\eta\mu }{2\gamma + \eta\mu} \right)\right].
	\end{align*}
	Then, given $\varepsilon>0$, the number of iterations for 
$\EE[\|\z^k - \z^*\|^2] \leq \varepsilon$ is 
\begin{equation*}
    O\left( \left[\frac{1}{p} +\chi + \frac{1}{\sqrt{pb}}\frac{\Lavg}{\mu}+  \sqrt{\chi} \frac{L}{\mu}\right] \log \frac{1}{\varepsilon} \right).
\end{equation*}
\end{theorem}
\begin{proof}
It is easy to check that $\alpha \geq \tfrac{1}{2}$, then also one can verify that the choice of $\gamma$, $\eta$, $\beta$, $\theta$, $\alpha$  satisfies the conditions of Lemma \ref{lem:fixed_key}. We can get that the iteration complexity of Algorithm~\ref{alg:vrvi}:
\begin{align*}
O\left( \left[1 + \frac{1}{\eta \mu}+ \frac{1}{\beta\theta\chi^{-1}}+  \frac{\gamma + \eta\mu}{p\eta\mu }\right] \log \frac{1}{\varepsilon} \right) &= O\left( \left[\frac{1}{p} + \frac{1}{\eta \mu}+ \frac{1}{\beta\theta\chi^{-1}}\right] \log \frac{1}{\varepsilon} \right) \\
 &= O\left( \left[\frac{1}{p} + \chi + \frac{1}{\eta \mu}+ \frac{\chi\eta}{\beta}\right] \log \frac{1}{\varepsilon} \right)\\
 &= O\left( \left[\frac{1}{p} + \chi + \frac{1}{\eta \mu}+ \frac{\chi \eta^2 \Lavg^2}{b p} + \frac{\chi \eta L^2}{\mu}\right] \log \frac{1}{\varepsilon} \right) \\
 &= O\left( \left[\frac{1}{p} +\chi + \frac{1}{\sqrt{pb}}\frac{\Lavg}{\mu}+  \sqrt{\chi} \frac{L}{\mu}\right] \log \frac{1}{\varepsilon} \right).
\end{align*}
\end{proof}

\newpage

\section{Proof of Theorem \ref{th:Alg2_conv}} \label{sec:pr_oa_tv}

We start the proof from the following lemma on $\delta^k$ and $\delta^{k+1/2}$ from Algorithm~\ref{2dvi2:alg}.

\begin{lemma}
	The following inequality holds:
	\begin{equation}\label{2vrvi:eq:1}
		\Ek{\sqn{\delta^k - \Ek{\delta^k}}} \leq \frac{2\Lavg^2}{b}\E{\sqn{\z^{k} - \w^{k-1}} + \alpha^2\sqn{\z^k- \z^{k-1}}}.
	\end{equation}
	\begin{equation}\label{2vrvi:eq:2}
		\Ek{\sqn{\delta^{k+1/2} - \F(\z^*)}} \leq \frac{2\Lavg^2}{b} \E{\sqn{\z^{k+1} - \w^{k}}} + 2L^2\E{\sqn{\z^{k+1} - \z^*}}.
	\end{equation}
	where $\Ek{\delta^k}$ is equal to
	\begin{equation}\label{2vrvi:eq:3}
		\Ek{\delta^k} = F(\z^k) + \alpha(F(\z^k) - F(\z^{k-1})).
	\end{equation}
\end{lemma}
\begin{proof}
    The proof is the same as the proof for Lemma \ref{var_lem_fix}.
\end{proof}
Before proving the main lemma of this section, let us introduce an auxiliary notation. Let $\hat z^k$ be defined for all $k= 0,1,2,\ldots$ as follows:
	\begin{equation}
	\label{eq:hatz}
		 \hat \x^k = \x^k - \mP m^k,
	\end{equation}
and the Lyapunov function $\Psi^k$ be denoted by
	\begin{equation}
	\label{eq:lyap_tv}
	\begin{split}
		\Psi^{k} 
		&= 
		\frac{1}{\eta_y}\sqn{\y^{k} - \y^*} +\frac{1}{\eta_x}\sqn{\hat \x^{k} - \x^*} + \frac{2}{\eta_x}\sqn{m^{k}}_\mP + \frac{1}{2\eta_y}\sqn{\y^{k} - \y^{k-1}}
		\\&\quad
		+\frac{\nu^{-1}}{\tau}\sqn{\y_f^{k} + \x_f^{k} - \y^* - \x^*} - 2\<\F(\z^{k}) - \F(\z^{k-1}) - (\y^{k} - \y^{k-1}), \z^{k} - \z^*>
		\\&\quad + \frac{1}{4\eta_z}\sqn{\z^{k} - \z^{k-1}} + \left(1 + \frac{3\mu \eta_z}{2}\right) \cdot\frac{1}{\eta_z}\sqn{\z^{k} - \z^*} + \frac{\omega}{2\eta_z}\E{\sqn{\z^{k} - \w^{k-1}}}.
	\end{split}
	\end{equation}
Here we also use the next notation
	\begin{equation}
	\label{eq:tv_ystar}
		\y^* = \mP \F(\z^*) - \nu \z^*,
	\end{equation}
and:
	\begin{equation}
	\label{eq:tv_xstar}
		\x^* = -\mP\F(\z^*).
	\end{equation}

\begin{lemma} \label{lem:key_tv}
	Consider the problem \eqref{eq:VI_new} (or \eqref{eq:VI} + \eqref{eq:fs}) under Assumptions~\ref{as:Lipsh} and \ref{as:strmon} over a sequence of time-varying graphs $\mathcal{G}(k)$ with gossip matrices $\mW(k)$. Let  $\{\z^k\}$ be the sequence generated by Algorithm~\ref{2dvi2:alg} with  
$$
	T \geq B, \quad \omega = p \leq \frac{1}{16}; \quad \theta = \frac{1}{2}; \quad \beta = 5\gamma; \quad \nu \leq \frac{\mu}{4}; \quad \tau \in (0;1);
$$
$$
    \eta_z \leq \min\left\{\frac{1}{8L}, \frac{1}{32\eta_y}, \frac{\sqrt{\alpha b \omega}}{8 \Lavg}\right\}; ~~ \eta_y \leq \min\left\{\frac{1}{4\gamma}, \frac{\nu}{8\tau}\right\}; ~~ \eta_x \leq \min\left\{ \frac{1}{900\chi(T)\gamma}, \frac{\nu}{36\tau\chi^2(T)}\right\}; 
$$
$$
    \gamma \leq \min\left\{ \frac{\mu}{16L^2}; \frac{b \omega}{24 \Lavg^2 \eta_z}\right\}, \quad 
$$
$$
    \alpha = \max\left[
	\left(1 - \frac{ p\eta_z \nu }{p + \eta_z \nu}\right);
	1 - \eta_y\gamma;
	1-\eta_x\gamma;
	1 - \frac{\mu \eta_z}{8};
	1-\tau;
	1-\frac{1}{4\chi(T)}
	\right].
$$
Let the choice of $T$ guarantees contraction property (Assumption \ref{ass:tv} point 4) with $\chi(T)$. Then, after $k$ iterations we get
\begin{align*}
	&\E{\frac{1}{4\eta_z}\sqn{\z^{k} - \z^*}}
	    \\
	    &\quad \leq \max\left[
		\left(1 - \frac{ p\eta_z \nu }{\omega + \eta_z \nu}\right);
		1 - \eta_y\gamma;
		1-\eta_x\gamma, \left(1 - \frac{\mu \eta_z}{8}\right);
		1-\tau;
		\left(1-\frac{1}{4\chi(T)}\right)
		\right]^k \cdot \Psi^0.
\end{align*}
\end{lemma}

\begin{proof}
	
{\bf Part 1.} 
We start from
	\begin{align*}
		\frac{1}{\eta_z}\sqn{\z^{k+1} - \z^*} &= 
		\frac{1}{\eta_z}\sqn{\z^k - \z^*} + \frac{2}{\eta_z}\<\z^{k+1} - \z^k, \z^{k+1} - \z^*> - \frac{1}{\eta_z}\sqn{\z^{k+1} - \z^k}
		\\&= \frac{1}{\eta_z}\sqn{\z^k - \z^*} + \frac{2 \omega}{\eta_z}\<\w^{k} - \z^k, \z^{k+1} - \z^*>  
		\\&\quad
		- 2\<\Delta_z^k - (\F(\z^*) - \y^* - \nu \z^*), \z^{k+1} - \z^*> 
		-
		\frac{1}{\eta_z}\sqn{\z^{k+1} - \z^k} 
		\\&\quad- 
		2\<\tfrac{1}{\eta_z}(\z^k + \omega (\w^k - \z^k) - \eta_z \Delta^k_z - \z^{k+1})  - (- (\F(\z^*) - \y^* - \nu \z^*)) ,\z^{k+1} - \z^*>.
	\end{align*}
Optimality condition for \eqref{eq:VI}+\eqref{eq:fs} it follows, that
	\begin{equation*}
		-F(z^*) \in \partial g(z^*).
	\end{equation*}
Let us define $\Delta^* \in (\R^d)^M$ as
	\begin{equation*}
	\Delta^*= \tfrac{1}{M}[F(z^*), \ldots F(z^*)]^T.
	\end{equation*}
It is to note that $-\Delta^* \in M \partial\g(\z^*)$. On the other hand with notation \eqref{eq:tv_ystar}, we get
\begin{equation*}
	\Delta^*= (\ones_M\ones_M^\top \otimes \mI_d) \F(\z^*) = M (\F(\z^*) - \mP \F(\z^*)) = M (\F(\z^*) - \y^* - \nu \z^*)
\end{equation*}
It means that $-(\F(\z^*) - \y^* - \nu \z^*) \in \partial\g(\z^*)$.  
From update for $\z^{k+1}$ of Algorithm~\ref{2dvi2:alg} it follows, that	$ \tfrac{1}{\eta_z}(\z^k + \omega (\w^k - \z^k)  -  \z^{k+1}- \eta_z \Delta_z^k) \in \partial \g(\z^{k+1})$. This together with $-(\F(\z^*) - \y^* - \nu \z^*) \in \partial\g(\z^*)$ and monotonicity of $\partial \g(\cdot)$ implies
	\begin{align*}
		\frac{1}{\eta_z}\E{\sqn{\z^{k+1} - \z^*}} &\leq 
		\frac{1}{\eta_z}\E{\sqn{\z^k - \z^*}} + \frac{2 \omega}{\eta_z}\E{\<\w^{k} - \z^k, \z^{k+1} - \z^*>}
		\\&\quad- 
		2\E{\<\Delta_z^k - (\F(\z^*) - \y^* - \nu \z^*), \z^{k+1} - \z^*>} - \frac{1}{\eta_z}\E{\sqn{\z^{k+1} - \z^k}}
		\\&=
		\frac{1}{\eta_z}\E{\sqn{\z^k - \z^*}}  + \frac{2 \omega}{\eta_z}\E{\<\w^{k} - \z^*, \z^{k+1} - \z^*>}  
		\\&\quad- \frac{2 \omega}{\eta_z}\E{\<\z^{k} - \z^*, \z^{k+1} - \z^*>} - 2\E{\<\Delta_z^k - (\F(\z^*) - \y^* - \nu \z^*), \z^{k+1} - \z^*>} 
		\\&\quad
		-\frac{1}{\eta_z}\E{\sqn{\z^{k+1} - \z^k}}
		\\&=
		\frac{1}{\eta_z}\E{\sqn{\z^k - \z^*}}  + \frac{\omega}{\eta_z} \E{\sqn{\w^k - \z^*} + \sqn{\z^{k+1} - \z^*} - \sqn{\z^{k+1} - \w^k}}
		\\&\quad
		- \frac{\omega}{\eta_z}\E{\sqn{\z^{k+1} - \z^*} + \sqn{\z^k - \z^*} - \sqn{\z^{k+1} - \z^k}} 
		\\&\quad
		- 2\E{\<\Delta_z^k - (\F(\z^*) - \y^* - \nu \z^*), \z^{k+1} - \z^*>} -\frac{1}{\eta_z}\E{\sqn{\z^{k+1} - \z^k}}
		\\&=
		\frac{1}{\eta_z}\E{\sqn{\z^k - \z^*}}  + \frac{\omega}{\eta_z} \E{\sqn{\w^k - \z^*}} - \frac{\omega}{\eta_z} \E{\sqn{\z^{k+1} - \w^k}}
		\\&\quad
		- \frac{\omega}{\eta_z}\E{\sqn{\z^k - \z^*}} - 2\E{\<\Delta_z^k - (\F(\z^*) - \y^* - \nu \z^*), \z^{k+1} - \z^*>} 
		\\&\quad
		-\frac{1-\omega}{\eta_z}\E{\sqn{\z^{k+1} - \z^k}}.
	\end{align*}
In previous we also use the simple fact $\| a + b\|^2 = \|a\|^2 + 2\langle a; b \rangle + \|b\|^2$ twice. Definition $\Delta_z^k$ and small rearrangement gives	
	\begin{align*}
		\frac{1}{\eta_z}\E{\sqn{\z^{k+1} - \z^*}} &\leq 
		\frac{1}{\eta_z}\E{\sqn{\z^k - \z^*}}  + \frac{\omega}{\eta_z} \E{\sqn{\w^k - \z^*}} - \frac{\omega}{\eta_z} \E{\sqn{\z^{k+1} - \w^k}}
		\\&\quad
		- \frac{\omega}{\eta_z}\E{\sqn{\z^k - \z^*}} -\frac{1-\omega}{\eta_z}\E{\sqn{\z^{k+1} - \z^k}}
		\\&\quad
		- 2\E{\<\E{\delta^k} - \nu \z^k - y^k - \alpha(\y^k - \y^{k-1}) - (\F(\z^*) - \y^* - \nu \z^*), \z^{k+1} - \z^*>}
		\\&\quad
		+ 2\E{\<\E{\delta^k} - \delta^k, \z^{k+1} - \z^k>} + 2\E{\<\E{\delta^k} - \delta^k, \z^{k} - \z^*>}.
	\end{align*}
Using the tower property of expectation we can obtain the following:
	\begin{align*}
		\E{\<\Ek{\delta^k} - \delta^k, \z^k - \z^*>} &= \E{\Ek{\<\Ek{\delta^k} - \delta^k, \z^k - \z^*>}}
		\\&=
		\E{{\<\Ek{\Ek{\delta^k} - \delta^k}, \z^k - \z^*>}}
		\\&=
		\E{{\<\Ek{\delta^k} - \Ek{\delta^k}, \z^k - \z^*>}} = 0.
	\end{align*}	
Then we have the following	
    \begin{align*}
		\frac{1}{\eta_z}\E{\sqn{\z^{k+1} - \z^*}} &\leq 
		\frac{1}{\eta_z}\E{\sqn{\z^k - \z^*}}  + \frac{\omega}{\eta_z} \E{\sqn{\w^k - \z^*}} - \frac{\omega}{\eta_z} \E{\sqn{\z^{k+1} - \w^k}}
		\\&\quad
		- \frac{\omega}{\eta_z}\E{\sqn{\z^k - \z^*}} -\frac{1-\omega}{\eta_z}\E{\sqn{\z^{k+1} - \z^k}}
		\\&\quad
		- 2\E{\<\E{\delta^k} - \nu \z^k - \y^k - \alpha(\y^k - \y^{k-1}) - (\F(\z^*) - \y^* - \nu \z^*), \z^{k+1} - \z^*>}
		\\&\quad
		+ 2\E{\<\E{\delta^k} - \delta^k, \z^{k+1} - \z^k>}.
	\end{align*}	
Using the Young's inequality we get
    \begin{align*}
		\frac{1}{\eta_z}\E{\sqn{\z^{k+1} - \z^*}} &\leq 
		\frac{1}{\eta_z}\E{\sqn{\z^k - \z^*}}  + \frac{\omega}{\eta_z} \E{\sqn{w^k - \z^*}} - \frac{\omega}{\eta_z} \E{\sqn{\z^{k+1} - \w^k}}
		\\&\quad
		- \frac{\omega}{\eta_z}\E{\sqn{\z^k - \z^*}} -\frac{1-\omega}{\eta_z}\E{\sqn{\z^{k+1} - \z^k}}
		\\&\quad
		- 2\E{\<\E{\delta^k} - \nu \z^k - \y^k - \alpha(\y^k - \y^{k-1}) - (\F(\z^*) - \y^* - \nu \z^*), \z^{k+1} - \z^*>}
		\\&\quad
		+ 2\eta_z\E{\sqn{\E{\delta^k} - \delta^k}} + \frac{1}{2\eta_z} \E{\sqn{\z^{k+1} - \z^k}}
		\\& = 
		\frac{1}{\eta_z}\E{\sqn{\z^k - \z^*}}  + \frac{\omega}{\eta_z} \E{\sqn{\w^k - \z^*}} - \frac{\omega}{\eta_z} \E{\sqn{\z^{k+1} - \w^k}}
		\\&\quad
		- \frac{\omega}{\eta_z}\E{\sqn{\z^k - \z^*}} -\frac{1/2-\omega}{\eta_z}\E{\sqn{\z^{k+1} - \z^k}}
		\\&\quad
		- 2\E{\<\E{\delta^k} - \nu \z^k - \y^k - \alpha(\y^k - \y^{k-1}) - (\F(\z^*) - \y^* - \nu \z^*), \z^{k+1} - \z^*>}
		\\&\quad
		+ 2\eta_z\E{\sqn{\E{\delta^k} - \delta^k}}.
	\end{align*}
With \eqref{2vrvi:eq:1} and \eqref{2vrvi:eq:3} we have
    \begin{align*}
		\frac{1}{\eta_z}\E{\sqn{\z^{k+1} - \z^*}} &\leq 
		\frac{1}{\eta_z}\E{\sqn{\z^k - \z^*}}  + \frac{\omega}{\eta_z} \E{\sqn{w^k - \z^*}} - \frac{\omega}{\eta_z} \E{\sqn{\z^{k+1} - \w^k}}
		\\&\quad
		- \frac{\omega}{\eta_z}\E{\sqn{\z^k - \z^*}} -\frac{1/2-\omega}{\eta_z}\E{\sqn{\z^{k+1} - \z^k}}
		\\&\quad
		- 2\E{\<\F(\z^k) + \alpha(\F(\z^k) - \F(\z^{k-1})) - \nu \z^k - \y^k, \z^{k+1} - \z^*>}
		\\&\quad
		- 2\E{\<- \alpha(\y^k - \y^{k-1}) - (\F(\z^*) - \y^* - \nu \z^*), \z^{k+1} - \z^*>}
		\\&\quad
		+ \frac{4\eta_z\Lavg^2}{b}\E{\sqn{\z^{k} - \w^{k-1}}} + \frac{4\eta_z\alpha^2\Lavg^2}{b} \E{\sqn{\z^k- \z^{k-1}}}.
	\end{align*}
Small rearrangement gives
    \begin{align*}
		\frac{1}{\eta_z}\E{\sqn{\z^{k+1} - \z^*}} &\leq 
		\frac{1}{\eta_z}\E{\sqn{\z^k - \z^*}}  + \frac{\omega}{\eta_z} \E{\sqn{\w^k - \z^*}} - \frac{\omega}{\eta_z} \E{\sqn{\z^{k+1} - \w^k}}
		\\&\quad
		- \frac{\omega}{\eta_z}\E{\sqn{\z^k - \z^*}} -\frac{1/2-\omega}{\eta_z}\E{\sqn{\z^{k+1} - \z^k}}
		\\&\quad-
		2\E{\<\F(\z^{k+1}) - \F(\z^*), \z^{k+1} - \z^*>}
		+
		2 \nu\E{\<\z^k - \z^*, \z^{k+1} - \z^*>}
		\\&\quad+
		2\E{\<\y^{k+1} - \y^*,\z^{k+1} - \z^*>}
		\\&\quad+
		2\E{\<\F(\z^{k+1}) - \F(\z^k) - (\y^{k+1} - \y^k), \z^{k+1} - \z^*>}
		\\&\quad-
		2\alpha\E{\<\F(\z^k) - \F(\z^{k-1}) - (\y^k - \y^{k-1}), \z^k - \z^*>}
		\\&\quad-
		2\alpha\E{\<\F(\z^k) - \F(\z^{k-1}) - (\y^k - \y^{k-1}), \z^{k+1} - \z^k>}
		\\&\quad
		+ \frac{4\eta_z\Lavg^2}{b}\E{\sqn{\z^{k} - \w^{k-1}}} + \frac{4\eta_z\alpha^2\Lavg^2}{b} \E{\sqn{\z^k- \z^{k-1}}}.
	\end{align*}
Using Assumption \ref{as:strmon} of $\mu$-strong monotonicity and Young's inequality we get
	\begin{align*}
		\frac{1}{\eta_z}\E{\sqn{\z^{k+1} - \z^*}} &\leq 
		\frac{1}{\eta_z}\E{\sqn{\z^k - \z^*}}  + \frac{\omega}{\eta_z} \E{\sqn{\w^k - \z^*}} - \frac{\omega}{\eta_z} \E{\sqn{\z^{k+1} - \w^k}}
		\\&\quad
		- \frac{\omega}{\eta_z}\E{\sqn{\z^k - \z^*}} -\frac{1/2-\omega}{\eta_z}\E{\sqn{\z^{k+1} - \z^k}}
		\\&\quad-
		2\mu\E{\sqn{\z^{k+1} - \z^*}}
		+
		2 \nu \E{\<\z^k - \z^*, \z^{k+1} - \z^*>}
		\\&\quad+
		2\E{\<\y^{k+1} - \y^*,\z^{k+1} - \z^*>}
		\\&\quad+
		2\E{\<\F(\z^{k+1}) - \F(\z^k) - (\y^{k+1} - \y^k), \z^{k+1} - \z^*>}
		\\&	\quad-
		2\alpha\E{\<\F(\z^k) - \F(\z^{k-1}) - (\y^k - \y^{k-1}), \z^k - \z^*>}
		\\&\quad-
		2\alpha\E{\<\F(\z^k) - \F(\z^{k-1}) - (\y^k - \y^{k-1}), \z^{k+1} - \z^k>}
		\\&\quad
		+ \frac{4\eta_z\Lavg^2}{b}\E{\sqn{\z^{k} - \w^{k-1}}} + \frac{4\eta_z\alpha^2\Lavg^2}{b} \E{\sqn{\z^k- \z^{k-1}}}
		\\&\leq 
		\frac{1}{\eta_z}\E{\sqn{\z^k - \z^*}}  + \frac{\omega}{\eta_z} \E{\sqn{w^k - \z^*}} - \frac{\omega}{\eta_z} \E{\sqn{\z^{k+1} - \w^k}}
		\\&\quad
		- \frac{\omega}{\eta_z}\E{\sqn{\z^k - \z^*}} -\frac{1/2-\omega}{\eta_z}\E{\sqn{\z^{k+1} - \z^k}}
		\\&\quad-
		2\mu\E{\sqn{\z^{k+1} - \z^*}}
		+
		\nu \E{\sqn{\z^k - \z^*}} + \nu \E{\sqn{\z^{k+1} - \z^*}}
		\\&\quad+
		2\E{\<\y^{k+1} - \y^*,\z^{k+1} - \z^*>}
		\\&\quad+
		2\E{\<\F(\z^{k+1}) - \F(\z^k) - (\y^{k+1} - \y^k), \z^{k+1} - \z^*>}
	    \\&\quad-
		2\alpha\E{\<\F(\z^k) - \F(\z^{k-1}) - (\y^k - \y^{k-1}), \z^k - \z^*>}
		\\&\quad+
		2\alpha\E{\|\F(\z^k) - \F(\z^{k-1}) \| \| \z^{k+1} - \z^k \|} + 2\alpha \E{\|\y^k - \y^{k-1}\| \| \z^{k+1} - \z^k\|}
		\\&\quad
		+ \frac{4\eta_z\Lavg^2}{b}\E{\sqn{\z^{k} - \w^{k-1}}} + \frac{4\eta_z\alpha^2\Lavg^2}{b} \E{\sqn{\z^k- \z^{k-1}}}.
	\end{align*}
With $L$-Lipschitzness of $\F(\cdot)$ (Assumption \ref{as:Lipsh}) we have
    \begin{align*}
		\frac{1}{\eta_z}\E{\sqn{\z^{k+1} - \z^*}} &\leq 
		\frac{1}{\eta_z}\E{\sqn{\z^k - \z^*}}  + \frac{\omega}{\eta_z} \E{\sqn{\w^k - \z^*}} - \frac{\omega}{\eta_z} \E{\sqn{\z^{k+1} - \w^k}}
		\\&\quad
		- \left(\frac{\omega}{\eta_z} - \nu \right)\E{\sqn{\z^k - \z^*}} -\frac{1/2-\omega}{\eta_z}\E{\sqn{\z^{k+1} - \z^k}}
		\\&\quad-
		(2\mu - \nu)\E{\sqn{\z^{k+1} - \z^*}}
		+2\E{\<y^{k+1} - y^*,\z^{k+1} - \z^*>}
		\\&\quad+
		2\E{\<\F(\z^{k+1}) - \F(\z^k) - (\y^{k+1} - \y^k), \z^{k+1} - \z^*>}
		\\&\quad-
		2\alpha\E{\<\F(\z^k) - \F(\z^{k-1}) - (\y^k - \y^{k-1}), \z^k - \z^*>}
		\\&\quad+
		2\alpha L \E{\|\z^k - \z^{k-1}\| \| \z^{k+1} - \z^k \|} + 2\alpha \E{\|\y^k - \y^{k-1}\| \| \z^{k+1} - \z^k\|}
		\\&\quad
		+ \frac{4\eta_z\Lavg^2}{b}\E{\sqn{\z^{k} - \w^{k-1}}} + \frac{4\eta_z\alpha^2\Lavg^2}{b} \E{\sqn{\z^k- \z^{k-1}}}
	\end{align*}	
Young's inequality we get		
	\begin{align*}	
		\frac{1}{\eta_z}\E{\sqn{\z^{k+1} - \z^*}} &\leq
		\frac{1}{\eta_z}\E{\sqn{\z^k - \z^*}}  + \frac{\omega}{\eta_z} \E{\sqn{\w^k - \z^*}} - \frac{\omega}{\eta_z} \E{\sqn{\z^{k+1} - \w^k}}
		\\&\quad
		- \left(\frac{\omega}{\eta_z} - \nu \right)\E{\sqn{\z^k - \z^*}} -\frac{1/2-\omega}{\eta_z}\E{\sqn{\z^{k+1} - \z^k}}
		\\&\quad-
		(2\mu - \nu)\E{\sqn{\z^{k+1} - \z^*}}
		+2\E{\<y^{k+1} - y^*,\z^{k+1} - \z^*>}
		\\&\quad+
		2\E{\<\F(\z^{k+1}) - \F(\z^k) - (y^{k+1} - y^k), \z^{k+1} - \z^*>}
		\\&\quad-
		2\alpha\E{\<\F(\z^k) - \F(\z^{k-1}) - (y^k - y^{k-1}), \z^k - \z^*>}
		\\&\quad+
		\alpha L \E{\sqn{\z^k - \z^{k-1}}} + \alpha L \E{\sqn{ \z^{k+1} - \z^k}} 
		\\&\quad
		+ \frac{\alpha}{2\eta_y} \E{\sqn{\y^k - \y^{k-1}}} +  2\alpha \eta_y \E{\sqn{ \z^{k+1} - \z^k}}
		\\&\quad
		+ \frac{4\eta_z\Lavg^2}{b}\E{\sqn{\z^{k} -\w^{k-1}}} + \frac{4\eta_z\alpha^2\Lavg^2}{b} \E{\sqn{\z^k- \z^{k-1}}}.
	\end{align*}
Using the assumption on $\eta_z$ we get that $L \leq \tfrac{1}{8 \eta_z}$ and $\eta_y \leq \tfrac{1}{32 \eta_z}$
    \begin{align*}
		\frac{1}{\eta_z}\E{\sqn{\z^{k+1} - \z^*}} &\leq
		\frac{1}{\eta_z}\E{\sqn{\z^k - \z^*}}  + \frac{\omega}{\eta_z} \E{\sqn{\w^k - \z^*} } - \frac{\omega}{\eta_z} \E{\sqn{\z^{k+1} - \w^k}}
		\\&\quad
		- \left(\frac{\omega}{\eta_z} - \nu \right)\E{\sqn{\z^k - \z^*}} -\frac{1/2-\omega}{\eta_z}\E{\sqn{\z^{k+1} - \z^k}}
		\\&\quad-
		(2\mu - \nu)\E{\sqn{\z^{k+1} - \z^*}}
		+2\E{\<\y^{k+1} - \y^*,\z^{k+1} - \z^*>}
		\\&\quad+
		2\E{\<\F(\z^{k+1}) - \F(\z^k) - (\y^{k+1} - \y^k), \z^{k+1} - \z^*>}
		\\&\quad-
		2\alpha\E{\<\F(\z^k) - \F(\z^{k-1}) - (\y^k - \y^{k-1}), \z^k - \z^*>}
		\\&\quad+
		\frac{\alpha}{8\eta_z} \E{\sqn{\z^k - \z^{k-1}}} + \frac{\alpha}{8\eta_z} \E{\sqn{ \z^{k+1} - \z^k}} 
		\\&\quad+
		\frac{\alpha}{2\eta_y} \E{\sqn{\y^k - \y^{k-1}}} +  \frac{\alpha}{16\eta_z} \E{\sqn{ \z^{k+1} - \z^k}}
		\\&\quad
		+ \frac{4\eta_z\Lavg^2}{b}\E{\sqn{\z^{k} - \w^{k-1}}} + \frac{4\eta_z\alpha^2\Lavg^2}{b} \E{\sqn{\z^k- \z^{k-1}}}.
	\end{align*}
With $\alpha \leq 1$ and $\eta_z \leq \tfrac{\sqrt{\alpha b}}{8\Lavg}$ we obtain
    \begin{align*}
		\frac{1}{\eta_z}\E{\sqn{\z^{k+1} - \z^*}} &\leq
		\frac{1}{\eta_z}\E{\sqn{\z^k - \z^*}}  + \frac{\omega}{\eta_z} \E{\sqn{\w^k - \z^*}} - \frac{\omega}{\eta_z} \E{\sqn{\z^{k+1} - \w^k}}
		\\&\quad
		- \left(\frac{\omega}{\eta_z} - \nu \right)\E{\sqn{\z^k - \z^*}} -\frac{1/2-\omega}{\eta_z}\E{\sqn{\z^{k+1} - \z^k}}
		\\&\quad-
		(2\mu - \nu)\E{\sqn{\z^{k+1} - \z^*}}
		+2\E{\<\y^{k+1} - \y^*,\z^{k+1} - \z^*>}
		\\&\quad+
		2\E{\<\F(\z^{k+1}) - \F(\z^k) - (\y^{k+1} - \y^k), \z^{k+1} - \z^*>}
		\\&\quad-
		2\alpha\E{\<\F(\z^k) - \F(\z^{k-1}) - (\y^k - \y^{k-1}), \z^k - \z^*>}
		\\&\quad+
		\frac{\alpha}{8\eta_z} \E{\sqn{\z^k - \z^{k-1}}} + \frac{1}{8\eta_z} \E{\sqn{ \z^{k+1} - \z^k}}
		\\&\quad + \frac{\alpha}{2\eta_y} \E{\sqn{\y^k - \y^{k-1}}} +  \frac{1}{16\eta_z} \E{\sqn{ \z^{k+1} - \z^k}}
		\\&\quad
		+ \frac{4\eta_z\Lavg^2}{b}\E{\sqn{\z^{k} - \w^{k-1}}} + \frac{\alpha}{16\eta_z} \E{\sqn{\z^k- \z^{k-1}}}
		\\&\leq
		\frac{1}{\eta_z}\E{\sqn{\z^k - \z^*}}  + \frac{\omega}{\eta_z} \E{\sqn{\w^k - \z^*}} -  \frac{\omega}{\eta_z}\E{\sqn{\z^{k+1} - \w^k}}
		\\&\quad
		- \left(\frac{\omega}{\eta_z} - \nu \right)\E{\sqn{\z^k - \z^*}} -\frac{5/16-\omega}{\eta_z}\E{\sqn{\z^{k+1} - \z^k}} 
		\\&\quad
		+ \frac{\alpha}{4\eta_z} \E{\sqn{\z^k - \z^{k-1}}}
		-
		(2\mu - \nu)\E{\sqn{\z^{k+1} - \z^*}}
		\\&\quad
		+2\E{\<\y^{k+1} - \y^*,\z^{k+1} - \z^*>}
		\\&\quad+
		2\E{\<\F(\z^{k+1}) - \F(\z^k) - (\y^{k+1} - \y^k), \z^{k+1} - \z^*>}
		\\&\quad-
		2\alpha\E{\<\F(\z^k) - \F(\z^{k-1}) - (\y^k - \y^{k-1}), \z^k - \z^*>}
		\\&\quad + \frac{\alpha}{2\eta_y} \E{\sqn{\y^k - \y^{k-1}}}
		+ \frac{4\eta_z\Lavg^2}{b}\E{\sqn{\z^{k} - \w^{k-1}}}.
	\end{align*}
Choice of $\omega \leq \tfrac{1}{16}$ and small rearrangement give
    \begin{align*}
		\frac{1}{\eta_z}&\E{\sqn{\z^{k+1} - \z^*}} + (2\mu - \nu)\E{\sqn{\z^{k+1} - \z^*}} + \frac{1}{4\eta_z}\E{\sqn{\z^{k+1} - \z^k}} 
		\\&\quad- 2\E{\<\F(\z^{k+1}) - \F(\z^k) - (\y^{k+1} - \y^k), \z^{k+1} - \z^*>}\\
		&\leq
		\frac{1}{\eta_z}\E{\sqn{\z^k - \z^*}}  - \left(\frac{\omega}{\eta_z} - \nu \right)\E{\sqn{\z^k - \z^*}}
		\\&\quad
		+ \alpha \cdot \frac{1}{4\eta_z} \E{\sqn{\z^k - \z^{k-1}}} - \alpha  \cdot 2\E{\<\F(\z^k) - \F(\z^{k-1}) - (\y^k - \y^{k-1}), \z^k - \z^*>}
		\\&\quad
		+ \frac{\omega}{\eta_z} \E{\sqn{\w^k - \z^*}} -  \frac{\omega}{\eta_z}\E{\sqn{\z^{k+1} - \w^k}}
		+2\E{\<\y^{k+1} - \y^*,\z^{k+1} - \z^*>}
		\\&\quad + \frac{\alpha}{2\eta_y} \E{\sqn{\y^k - \y^{k-1}}}
		+ \frac{4\eta_z\Lavg^2}{b}\E{\sqn{\z^{k} - \w^{k-1}}}.
	\end{align*}
{\bf Part 2.} Update for $m^{k+1}$ and Assumption \ref{ass:tv} on time-varying graph give
	\begin{align*}
		\sqn{m^{k+1}}_\mP
		&\leq
		(1-\chi^{-1}(T))\sqn{m^k + \eta_x \Delta_x^k}_\mP
		\\&\leq
		(1-\chi^{-1}(T))\left(
		\left(1 + (2\chi(T))^{-1}\right)\sqn{m^k}_\mP
		+\left(1 + 2\chi(T)\right)\sqn{\eta_x\Delta_x^k}_\mP
		\right)
		\\&\leq
		(1-(2\chi(T))^{-1})\sqn{m^k}_\mP
		+2\eta_x^2\chi(T)\sqn{\Delta_x^k}_\mP.
	\end{align*}
After rearranging we get
	\begin{equation}\label{dvi2:eq:m}
		\sqn{m^k}_\mP
		\leq
		(1-(4\chi(T))^{-1})4\chi(T)\sqn{m^k}_\mP
		-4\chi(T)\sqn{m^{k+1}}_\mP
		+8\eta_x^2\chi^2(T)\sqn{\Delta_x^k}_\mP.
	\end{equation}
	
{\bf Part 3.} Updates for $\x^{k+1}$ and $m^{k+1}$  of Algorithm~\ref{2dvi2:alg} imply $\hat \x^{k+1} = \hat \x^k-\eta_x \mP \Delta_x^k$.
By this together with update for $\y^{k+1}$ of Algorithm~\ref{2dvi2:alg} we obtain
	\begin{align*}
		\frac{1}{\eta_y}&\sqn{\y^{k+1} - \y^*}
		+\frac{1}{\eta_x}\sqn{\hat \x^{k+1} - \x^*} 
		\\&=
		\frac{1}{\eta_y}\sqn{\y^{k} - \y^*}
		+\frac{1}{\eta_x}\sqn{\hat \x^{k} - \x^*} 
		-\frac{1}{\eta_y}\sqn{\y^{k+1} - \y^k}
		+\frac{1}{\eta_x}\sqn{\hat \x^{k+1} - \hat \x^k} 
		\\&\quad
		+\frac{2}{\eta_y}\<\y^{k+1} - \y^{k}, \y^{k+1} - \y^*>
		+\frac{2}{\eta_x}\<\hat \x^{k+1} - \hat \x^k, \hat \x^{k} - \x^*>
		\\&\leq
		\frac{1}{\eta_y}\sqn{\y^{k} - \y^*}
		+\frac{1}{\eta_x}\sqn{\hat \x^{k} - \x^*} 
		-\frac{1}{\eta_y}\sqn{\y^{k+1} - \y^k}
		\\&\quad
		-2\<\Delta_y^k, \y^{k+1} - \y^*>
		-2\<\mP\Delta_x^k, \x^{k} - \x^*>
		+2\norm{\Delta_x^k}_\mP\norm{m^k}_\mP
		+\eta_x\sqn{\Delta_x^k}_\mP.
	\end{align*}
Using the definitions of $\Delta_y^k$ and $\Delta_x^k$  we get
	\begin{align*}
		\frac{1}{\eta_y}&\sqn{\y^{k+1} - \y^*}
		+\frac{1}{\eta_x}\sqn{\hat \x^{k+1} - \x^*} 
		\\&\leq
		\frac{1}{\eta_y}\sqn{\y^{k} - \y^*}
		+\frac{1}{\eta_x}\sqn{\hat \x^{k} - \x^*} 
		-\frac{1}{\eta_y}\sqn{\y^{k+1} - \y^k}
		+2\norm{\Delta_x^k}_\mP\norm{m^k}_\mP
		+\eta_x\sqn{\Delta_x^k}_\mP
		\\&\quad
		-2\<\nu^{-1} (\y_c^k + \x_c^k) + \z^{k+1} + \gamma (\y^k + \x^k + \nu \z^k), \y^{k+1} - \y^*>
		\\&\quad
		-2\<\nu^{-1} \mP(\y_c^k + \x_c^k) + \beta \mP (\x^k + \Delta^{k+1/2}), \x^{k} - \x^*>
		\\&=
		\frac{1}{\eta_y}\sqn{\y^{k} - \y^*}
		+\frac{1}{\eta_x}\sqn{\hat \x^{k} - \x^*} 
		-\frac{1}{\eta_y}\sqn{\y^{k+1} - \y^k}
		+2\norm{\Delta_x^k}_\mP\norm{m^k}_\mP
		+\eta_x\sqn{\Delta_x^k}_\mP
		\\&\quad
		-2\<\nu^{-1} (\y_c^k + \x_c^k) + \z^{k+1} + \gamma (\y^k + \x^k + \nu \z^k), \y^{k+1} - \y^*>
		\\&\quad
		-2\<\nu^{-1} \mP(\y_c^k + \x_c^k), \x^{k} - \x^*>
		-2\<\beta \mP (\x^k + \Delta^{k+1/2}), \x^{k} - \x^*>.
	\end{align*}
Definitions \eqref{eq:tv_ystar}, \eqref{eq:tv_xstar} gives $\mP\F(\z^*) + \x^* = 0$ $\mP(\y^* + \x^*) = 0$, $\nu^{-1}(\y^* + \x^*) + \z^* = 0$. Additionally, using $\mP \x^k = \x^k$, we get
    \begin{align*}
		\frac{1}{\eta_y}&\sqn{\y^{k+1} - \y^*}
		+\frac{1}{\eta_x}\sqn{\hat \x^{k+1} - \x^*} 
		\\&\leq
		\frac{1}{\eta_y}\sqn{\y^{k} - \y^*}
		+\frac{1}{\eta_x}\sqn{\hat \x^{k} - \x^*} 
		-\frac{1}{\eta_y}\sqn{\y^{k+1} - \y^k}
		+2\norm{\Delta_x^k}_\mP\norm{m^k}_\mP
		+\eta_x\sqn{\Delta_x^k}_\mP
		\\&\quad
		-2\<\nu^{-1} (\y_c^k + \x_c^k) + \z^{k+1} + \gamma (\y^k + \x^k + \nu \z^k), \y^{k+1} - \y^*>
		\\&\quad
		-2\<\nu^{-1} (\y_c^k + \x_c^k), \x^{k} - \x^*>
		-2\<\beta  (\x^k + \mP\Delta^{k+1/2}), \x^{k} - \x^*>
		\\&=
		\frac{1}{\eta_y}\sqn{\y^{k} - \y^*}
		+\frac{1}{\eta_x}\sqn{\hat \x^{k} - \x^*} 
		-\frac{1}{\eta_y}\sqn{\y^{k+1} - \y^k}
		+2\norm{\Delta_x^k}_\mP\norm{m^k}_\mP
		+\eta_x\sqn{\Delta_x^k}_\mP
		\\&\quad
		-2\<\nu^{-1} (\y_c^k + \x_c^k) + \z^{k+1} - \nu^{-1}(\y^* + \x^*) - \z^*, \y^{k+1} - \y^*> 
		\\&\quad
		-2\gamma \<\y^k + \x^k + \nu \z^k - (\y^* + \x^* + \nu\z^*), \y^{k+1} - \y^*>
		\\&\quad
		-2\<\nu^{-1} (\y_c^k + \x_c^k - \y^* - \x^*), \x^{k} - \x^*>
		\\&\quad
		-2\beta\<  \x^k + \mP\Delta^{k+1/2} -\mP\F(\z^*) - \x^*, \x^{k} - \x^*>.
	\end{align*}
And then
    \begin{align*}
		\frac{1}{\eta_y}&\sqn{\y^{k+1} - \y^*}
		+\frac{1}{\eta_x}\sqn{\hat \x^{k+1} - \x^*} 
		\\&\leq
		\frac{1}{\eta_y}\sqn{\y^{k} - \y^*}
		+\frac{1}{\eta_x}\sqn{\hat \x^{k} - \x^*} 
		-\frac{1}{\eta_y}\sqn{\y^{k+1} - \y^k}
		+2\norm{\Delta_x^k}_\mP\norm{m^k}_\mP
		+\eta_x\sqn{\Delta_x^k}_\mP
		\\&\quad
		-2\<\z^{k+1} - \z^*, \y^{k+1} - \y^*> -2\nu^{-1}\< \y_c^k + \x_c^k - \y^* - \x^*, \y^{k+1}+\x^k - \y^*-\x^*>
		\\&\quad
		-2\gamma\<  \y^k + \x^k + \nu \z^k - \y^* - \x^* - \nu\z^*, \y^{k+1} - \y^*>
		\\&\quad
		-2\beta\<\x^k + \mP\Delta^{k+1/2} - \x^* - \mP\F(\z^*), \x^{k} - \x^*>.
	\end{align*}
Using parallelogram rule we get
	\begin{align*}
		\frac{1}{\eta_y}&\sqn{\y^{k+1} - \y^*}
		+\frac{1}{\eta_x}\sqn{\hat \x^{k+1} - \x^*} 
		\\&\leq
		\frac{1}{\eta_y}\sqn{\y^{k} - \y^*}
		+\frac{1}{\eta_x}\sqn{\hat \x^{k} - \x^*} 
		-\frac{1}{\eta_y}\sqn{\y^{k+1} - \y^k} +2\norm{\Delta_x^k}_\mP\norm{m^k}_\mP
		+\eta_x\sqn{\Delta_x^k}_\mP
		\\&\quad
		-2\<\z^{k+1} - \z^*, \y^{k+1} - \y^*>
		-2\nu^{-1}\< \y_c^k + \x_c^k - \y^* - \x^*, \y^{k+1}+\x^k - \y^*-\x^*>
		\\&\quad
		-\gamma\left(
		\sqn{\y^k - \y^*}
		+\sqn{\y^{k+1} - \y^*} - \sqn{\y^{k+1} - \y^k}
		\right)
		\\&\quad
		+\gamma\left(
		\sqn{\x^k - \x^* + \nu(\z^k - \z^*)}
		+\sqn{\y^{k+1} - \y^*}
		\right)
		\\& \quad
		-2\beta\sqn{\x^k - \x^*} + \beta\sqn{\x^k - \x^*}
		+\beta\sqn{\Delta^{k+1/2} - \F(\z^*)}
		\\& \leq
		\frac{1}{\eta_y}\sqn{\y^{k} - \y^*}
		+\frac{1}{\eta_x}\sqn{\hat \x^{k} - \x^*} 
		-\left(\frac{1}{\eta_y} - \gamma\right)\sqn{\y^{k+1} - \y^k}
		\\&\quad
		+2\norm{\Delta_x^k}_\mP\norm{m^k}_\mP
		+\eta_x\sqn{\Delta_x^k}_\mP
		\\&\quad -2\<\z^{k+1} - \z^*, \y^{k+1} - \y^*>
		-2\nu^{-1}\< \y_c^k + \x_c^k - \y^* - \x^*, \y^{k+1}+\x^k - \y^*-\x^*>
		\\& \quad
		-\gamma\sqn{\y^k - \y^*}
		-(\beta - 2\gamma)\sqn{\x^k - \x^*}
		+2\gamma\nu^2\sqn{\z^k - \z^*}
		+\beta\sqn{\Delta^{k+1/2} - \F(\z^*)}.
	\end{align*}
The definition of $\beta = 5 \gamma$ gives
	\begin{align*}
		\frac{1}{\eta_y}&\sqn{\y^{k+1} - \y^*}
		+\frac{1}{\eta_x}\sqn{\hat \x^{k+1} - \x^*} 
		\\& \leq
		\frac{1}{\eta_y}\sqn{\y^{k} - \y^*}
		+\frac{1}{\eta_x}\sqn{\hat \x^{k} - \x^*} 
		-\left(\frac{1}{\eta_y} - \gamma\right)\sqn{\y^{k+1} - \y^k}
		\\&\quad
		+2\norm{\Delta_x^k}_\mP\norm{m^k}_\mP
		+\eta_x\sqn{\Delta_x^k}_\mP
		\\& \quad -2\<\z^{k+1} - \z^*, \y^{k+1} - \y^*>
		-2\nu^{-1}\< \y_c^k + \x_c^k - \y^* - \x^*, \y^{k+1}+\x^k - \y^*-\x^*>
		\\& \quad
		-\gamma\sqn{\y^k - \y^*}
		-3\gamma\sqn{\x^k - \x^*}
		+2\gamma\nu^2\sqn{\z^k - \z^*}
		+5\gamma\sqn{\Delta^{k+1/2} - \F(\z^*)}.
	\end{align*}
Using the definition \eqref{eq:hatz} of $\hat \x^k$, we get
	\begin{align*}
		\frac{1}{\eta_y}&\sqn{\y^{k+1} - \y^*}
		+\frac{1}{\eta_x}\sqn{\hat \x^{k+1} - \x^*} 
		\\& \leq
		\frac{1}{\eta_y}\sqn{\y^{k} - \y^*}
		+\frac{1}{\eta_x}\sqn{\hat \x^{k} - \x^*} 
		-\left(\frac{1}{\eta_y} - \gamma\right)\sqn{\y^{k+1} - \y^k}
		\\&\quad
		+2\norm{\Delta_x^k}_\mP\norm{m^k}_\mP
		+\eta_x\sqn{\Delta_x^k}_\mP
		\\& \quad -2\<\z^{k+1} - \z^*, \y^{k+1} - \y^*>
		-2\nu^{-1}\< \y_c^k + \x_c^k - \y^* - \x^*, \y^{k+1}+\x^k - \y^*-\x^*>
		\\& \quad
		-\gamma\sqn{\y^k - \y^*}
		-\gamma\sqn{\hat \x^k - \x^*}
		-\gamma\sqn{\x^k - \x^*}
		+2\gamma\sqn{m^k}_\mP
		+2\gamma\nu^2\sqn{\z^k - \z^*}
		\\&\quad
		+5\gamma\sqn{\Delta^{k+1/2} - \F(\z^*)}
		\\& =
		\left(\frac{1}{\eta_y} - \gamma\right)\sqn{\y^{k} - \y^*}
		+\left(\frac{1}{\eta_x} - \gamma\right)\sqn{\hat \x^{k} - \x^*} 
		-\left(\frac{1}{\eta_y} - \gamma\right)\sqn{\y^{k+1} - \y^k}
		\\& \quad
		+2\gamma\sqn{m^k}_\mP
		+2\norm{\Delta_x^k}_\mP\norm{m^k}_\mP
		+\eta_x\sqn{\Delta_x^k}_\mP
		-2\<\z^{k+1} - \z^*, \y^{k+1} - \y^*>
		\\& \quad
		-2\nu^{-1}\< \y_c^k + \x_c^k - \y^* - \x^*, \y^{k+1}+\x^k - \y^*-\x^*>
		\\& \quad
		-\gamma\sqn{\x^k - \x^*}
		+2\gamma\nu^2\sqn{\z^k - \z^*}
		+5\gamma\sqn{\Delta^{k+1/2} - \F(\z^*)}.
	\end{align*}
With updates of $\x_f^{k+1}$ and $\y_f^{k+1}$ from Algorithm~\ref{2dvi2:alg} we have
	\begin{align*}
		\frac{1}{\eta_y}&\sqn{\y^{k+1} - \y^*}
		+\frac{1}{\eta_x}\sqn{\hat \x^{k+1} - \x^*} 
		\\& \leq
		\left(\frac{1}{\eta_y} - \gamma\right)\sqn{\y^{k} - \y^*}
		+\left(\frac{1}{\eta_x} - \gamma\right)\sqn{\hat \x^{k} - \x^*} 
		-\left(\frac{1}{\eta_y} - \gamma\right)\sqn{\y^{k+1} - \y^k}
		\\& \quad
		+2\gamma\sqn{m^k}_\mP
		+2\norm{\Delta_x^k}_\mP\norm{m^k}_\mP
		+\eta_x\sqn{\Delta_x^k}_\mP
		-2\<\z^{k+1} - \z^*, \y^{k+1} - \y^*>
		\\& \quad
		-\gamma\sqn{\x^k - \x^*}
		+2\gamma\nu^2\sqn{\z^k - \z^*}
		+5\gamma\sqn{\Delta^{k+1/2} - \F(\z^*)}
		\\&\quad
		-\frac{2\nu^{-1}}{\tau}\< \y_c^k + \x_c^k - \y^* - \x^*, \y_f^{k+1} + \x_f^{k+1} - \y_c^k  - \x_c^k + \theta(\mW(k) \otimes \mI_d)(\y_c^k + \x_c^k)>
		\\& \quad
		-2\nu^{-1}\< \y_c^k + \x_c^k - \y^* - \x^*, \y^k+\x^k - \y^*-\x^*> 
		\\&=
		\left(\frac{1}{\eta_y} - \gamma\right)\sqn{\y^{k} - \y^*}
		+\left(\frac{1}{\eta_x} - \gamma\right)\sqn{\hat \x^{k} - \x^*} 
		-\left(\frac{1}{\eta_y} - \gamma\right)\sqn{\y^{k+1} - \y^k}
		\\& \quad
		+2\gamma\sqn{m^k}_\mP
		+2\norm{\Delta_x^k}_\mP\norm{m^k}_\mP
		+\eta_x\sqn{\Delta_x^k}_\mP
		-2\<\z^{k+1} - \z^*, \y^{k+1} - \y^*>
		\\& \quad
		-\gamma\sqn{\x^k - \x^*}
		+2\gamma\nu^2\sqn{\z^k - \z^*}
		+5\gamma\sqn{\Delta^{k+1/2} - \F(\z^*)}
		\\&\quad
		-\frac{2\theta\nu^{-1}}{\tau}\< \y_c^k + \x_c^k - \y^* - \x^*, (\mW(k) \otimes \mI_d)(\y_c^k + \x_c^k)>
		\\&\quad
		-\frac{2\nu^{-1}}{\tau}\< \y_c^k + \x_c^k - \y^* - \x^*, \y_f^{k+1} + \x_f^{k+1} - \y_c^k  - \x_c^k>
		\\& \quad
		-2\nu^{-1}\< \y_c^k + \x_c^k - \y^* - \x^*, \y^k+\x^k - \y^*-\x^*> 
	\end{align*}		
Using definitions \eqref{eq:tv_ystar}, \eqref{eq:tv_xstar} 	we get $\y^* + \x^* \in \cL$ and then	
	\begin{align*}
		\frac{1}{\eta_y}&\sqn{\y^{k+1} - \y^*}
		+\frac{1}{\eta_x}\sqn{\hat \x^{k+1} - \x^*} 
		\\& \leq
		\left(\frac{1}{\eta_y} - \gamma\right)\sqn{\y^{k} - \y^*}
		+\left(\frac{1}{\eta_x} - \gamma\right)\sqn{\hat \x^{k} - \x^*} 
		-\left(\frac{1}{\eta_y} - \gamma\right)\sqn{\y^{k+1} - \y^k}
		\\& \quad
		+2\gamma\sqn{m^k}_\mP
		+2\norm{\Delta_x^k}_\mP\norm{m^k}_\mP
		+\eta_x\sqn{\Delta_x^k}_\mP
		-2\<\z^{k+1} - \z^*, \y^{k+1} - \y^*>
		\\& \quad
		-\gamma\sqn{\x^k - \x^*}
		+2\gamma\nu^2\sqn{\z^k - \z^*}
		+5\gamma\sqn{\Delta^{k+1/2} - \F(\z^*)}
		\\&\quad
		-\frac{2\theta\nu^{-1}}{\tau}\< \y_c^k + \x_c^k, (\mW(k) \otimes \mI_d)(\y_c^k + \x_c^k)>
		\\&\quad
		-\frac{2\nu^{-1}}{\tau}\< \y_c^k + \x_c^k - \y^* - \x^*, \y_f^{k+1} + \x_f^{k+1} - \y_c^k  - \x_c^k>
		\\& \quad
		-2\nu^{-1}\< \y_c^k + \x_c^k - \y^* - \x^*, \y^k+\x^k - \y^*-\x^*> 
	    \\&=
		\left(\frac{1}{\eta_y} - \gamma\right)\sqn{\y^{k} - \y^*}
		+\left(\frac{1}{\eta_x} - \gamma\right)\sqn{\hat \x^{k} - \x^*} 
		-\left(\frac{1}{\eta_y} - \gamma\right)\sqn{\y^{k+1} - \y^k}
		\\& \quad
		+2\gamma\sqn{m^k}_\mP
		+2\norm{\Delta_x^k}_\mP\norm{m^k}_\mP
		+\eta_x\sqn{\Delta_x^k}_\mP
		-2\<\z^{k+1} - \z^*, \y^{k+1} - \y^*>
		\\& \quad
		-\gamma\sqn{\x^k - \x^*}
		+2\gamma\nu^2\sqn{\z^k - \z^*}
		+5\gamma\sqn{\Delta^{k+1/2} - \F(\z^*)}
		\\&\quad
		-\frac{2\theta\nu^{-1}}{\tau}\sqn{\y_c^k+\x_c^k}_{(\mW(k) \otimes \mI_d)}
		\\&\quad
		-\frac{2\nu^{-1}}{\tau}\< \y_c^k + \x_c^k - \y^* - \x^*, \y_f^{k+1} + \x_f^{k+1} - \y_c^k  - \x_c^k>
		\\& \quad
		-2\nu^{-1}\< \y_c^k + \x_c^k - \y^* - \x^*, \y^k+\x^k - \y^*-\x^*> 
	\end{align*}
By parallelogram rule we obtain	
	\begin{align*}		
		\frac{1}{\eta_y}&\sqn{\y^{k+1} - \y^*}
		+\frac{1}{\eta_x}\sqn{\hat \x^{k+1} - \x^*} 
		\\& \leq
		\left(\frac{1}{\eta_y} - \gamma\right)\sqn{\y^{k} - \y^*}
		+\left(\frac{1}{\eta_x} - \gamma\right)\sqn{\hat \x^{k} - \x^*} 
		-\left(\frac{1}{\eta_y} - \gamma\right)\sqn{\y^{k+1} - \y^k}
		\\& \quad
		+2\gamma\sqn{m^k}_\mP
		+2\norm{\Delta_x^k}_\mP\norm{m^k}_\mP
		+\eta_x\sqn{\Delta_x^k}_\mP
		-2\<\z^{k+1} - \z^*, \y^{k+1} - \y^*>
		\\& \quad
		-\gamma\sqn{\x^k - \x^*}
		+2\gamma\nu^2\sqn{\z^k - \z^*}
		+5\gamma\sqn{\Delta^{k+1/2} - \F(\z^*)}
		\\& \quad
		-\frac{\nu^{-1}}{\tau}\sqn{\y_f^{k+1} + \x_f^{k+1} - \y^* - \x^*}
		+\frac{\nu^{-1}}{\tau}\sqn{\y_c^{k} + \x_c^{k} - \y^* - \x^*}
		\\&\quad
		+\frac{\nu^{-1}}{\tau}\sqn{\y_f^{k+1} + \x_f^{k+1} - \y_c^k - \x_c^k}
		\\& \quad
		-\frac{2\theta\nu^{-1}}{\tau}\sqn{\y_c^k+\x_c^k}_{(\mW(k) \otimes \mI_d)}
		-2\nu^{-1}\< \y_c^k + \x_c^k - \y^* - \x^*, \y^k+\x^k - \y^*-\x^*>
		\\& \leq
		\left(\frac{1}{\eta_y} - \gamma\right)\sqn{\y^{k} - \y^*}
		+\left(\frac{1}{\eta_x} - \gamma\right)\sqn{\hat \x^{k} - \x^*} 
		-\left(\frac{1}{\eta_y} - \gamma\right)\sqn{\y^{k+1} - \y^k}
		\\& \quad
		+2\gamma\sqn{m^k}_\mP
		+2\norm{\Delta_x^k}_\mP\norm{m^k}_\mP
		+\eta_x\sqn{\Delta_x^k}_\mP
		-2\<\z^{k+1} - \z^*, \y^{k+1} - \y^*>
		\\& \quad
		-\gamma\sqn{\x^k - \x^*}
		+2\gamma\nu^2\sqn{\z^k - \z^*}
		+5\gamma\sqn{\Delta^{k+1/2} - \F(\z^*)}
		\\& \quad
		-\frac{\nu^{-1}}{\tau}\sqn{\y_f^{k+1} + \x_f^{k+1} - \y^* - \x^*}
		+\frac{\nu^{-1}}{\tau}\sqn{\y_c^{k} + \x_c^{k} - \y^* - \x^*}
		\\& \quad
		+\frac{2\nu^{-1}}{\tau}\sqn{\y_f^{k+1} - \y_c^k}
		+\frac{2\nu^{-1}}{\tau}\sqn{\x_f^{k+1} - \x_c^k}
		-\frac{2\theta\nu^{-1}}{\tau}\sqn{\y_c^k+\x_c^k}_{(\mW(k) \otimes \mI_d)}
		\\& \quad
		-2\nu^{-1}\< \y_c^k + \x_c^k - \y^* - \x^*, \y^k+\x^k - \y^*-\x^*>
	\end{align*}		
Again by expressions for $\x_f^{k+1}$ and $\y_f^{k+1}$
	\begin{align*}
	\frac{1}{\eta_y}&\sqn{\y^{k+1} - \y^*}
		+\frac{1}{\eta_x}\sqn{\hat \x^{k+1} - \x^*}
		\\& \leq
		\left(\frac{1}{\eta_y} - \gamma\right)\sqn{\y^{k} - \y^*}
		+\left(\frac{1}{\eta_x} - \gamma\right)\sqn{\hat \x^{k} - \x^*} 
		-\left(\frac{1}{\eta_y} - \gamma\right)\sqn{\y^{k+1} - \y^k}
		\\& \quad
		+2\gamma\sqn{m^k}_\mP
		+2\norm{\Delta_x^k}_\mP\norm{m^k}_\mP
		+\eta_x\sqn{\Delta_x^k}_\mP
		-2\<x^{k+1} - \x^*, \y^{k+1} - \y^*>
		\\& \quad
		-\gamma\sqn{\x^k - \x^*}
		+2\gamma\nu^2\sqn{\z^k - \z^*}
		+5\gamma\sqn{\Delta^{k+1/2} - \F(\x^*)}
		\\& \quad
		-\frac{\nu^{-1}}{\tau}\sqn{\y_f^{k+1} + \x_f^{k+1} - \y^* - \x^*}
		+\frac{\nu^{-1}}{\tau}\sqn{\y_c^{k} + \x_c^{k} - \y^* - \x^*}
		\\& \quad
		+ 2\nu^{-1}\tau\sqn{\y^{k+1} - \y^k}
		+\frac{2\theta^2\nu^{-1}}{\tau}\sqn{\y_c^k + \x_c^k}_{(\mW(k) \otimes \mI_d)^2}
		-\frac{2\theta\nu^{-1}}{\tau}\sqn{\y_c^k+\x_c^k}_{(\mW(k) \otimes \mI_d)}
		\\& \quad
		-2\nu^{-1}\< \y_c^k + \x_c^k - \y^* - \x^*, \y^k+\x^k - \y^*-\x^*>.
	\end{align*}
Using the contraction property of the gossip matrix and the definition of $\theta = \tfrac{1}{2}$ we get
	\begin{align*}
		\frac{1}{\eta_y}&\sqn{\y^{k+1} - \y^*}
		+\frac{1}{\eta_x}\sqn{\hat \x^{k+1} - \x^*} 
		\\& \leq
		\left(\frac{1}{\eta_y} - \gamma\right)\sqn{\y^{k} - \y^*}
		+\left(\frac{1}{\eta_x} - \gamma\right)\sqn{\hat \x^{k} - \x^*} 
		-\left(\frac{1}{\eta_y} - \gamma - 2\nu^{-1}\tau\right)\sqn{\y^{k+1} - \y^k}
		\\& \quad
		+2\gamma\sqn{m^k}_\mP
		+2\norm{\Delta_x^k}_\mP\norm{m^k}_\mP
		+\eta_x\sqn{\Delta_x^k}_\mP
		-\frac{\nu^{-1}}{2\tau\chi(T)}\sqn{\y_c^k+\x_c^k}_\mP
		\\& \quad
		-2\<\z^{k+1} - \z^*, \y^{k+1} - \y^*>
		-\gamma\sqn{\x^k - \x^*}
		+2\gamma\nu^2\sqn{\z^k - \z^*}
		+5\gamma\sqn{\Delta^{k+1/2} - \F(\z^*)}
		\\& \quad
		-\frac{\nu^{-1}}{\tau}\sqn{\y_f^{k+1} + \x_f^{k+1} - \y^* - \x^*}
		+\frac{\nu^{-1}}{\tau}\sqn{\y_c^{k} + \x_c^{k} - \y^* - \x^*}
		\\& \quad
		-2\nu^{-1}\< \y_c^k + \x_c^k - \y^* - \x^*, \y^k+\x^k - \y^*-\x^*>.
	\end{align*}
By updates for $\x^k_c$ and $\y^k_c$ we get
    \begin{align*}
		\frac{1}{\eta_y}&\sqn{\y^{k+1} - \y^*}
		+\frac{1}{\eta_x}\sqn{\hat \x^{k+1} - \x^*} 
		\\& \leq
		\left(\frac{1}{\eta_y} - \gamma\right)\sqn{\y^{k} - \y^*}
		+\left(\frac{1}{\eta_x} - \gamma\right)\sqn{\hat \x^{k} - \x^*} 
		-\left(\frac{1}{\eta_y} - \gamma - 2\nu^{-1}\tau\right)\sqn{\y^{k+1} - \y^k}
		\\& \quad
		+2\gamma\sqn{m^k}_\mP
		+2\norm{\Delta_x^k}_\mP\norm{m^k}_\mP
		+\eta_x\sqn{\Delta_x^k}_\mP
		-\frac{\nu^{-1}}{2\tau\chi(T)}\sqn{\y_c^k+\x_c^k}_\mP
		\\& \quad
		-2\<\z^{k+1} - \z^*, \y^{k+1} - \y^*>
		-\gamma\sqn{\x^k - \x^*}
		+2\gamma\nu^2\sqn{\z^k - \z^*}
		+5\gamma\sqn{\Delta^{k+1/2} - \F(\z^*)}
		\\& \quad
		-\frac{\nu^{-1}}{\tau}\sqn{\y_f^{k+1} + \x_f^{k+1} - \y^* - \x^*} +\frac{\nu^{-1}}{\tau}\sqn{\y_c^{k} + \x_c^{k} - \y^* - \x^*}
		\\& \quad
		-2\nu^{-1}\< \y_c^k + \x_c^k - \y^* - \x^*, \y^k_c+\x^k_c - \y^*-\x^*> 
		\\& \quad
		-2\nu^{-1}\< \y_c^k + \x_c^k - \y^* - \x^*, \y^k+\x^k - \y_c^k-\x_c^k>
		\\& =
		\left(\frac{1}{\eta_y} - \gamma\right)\sqn{\y^{k} - \y^*}
		+\left(\frac{1}{\eta_x} - \gamma\right)\sqn{\hat \x^{k} - \x^*} 
		-\left(\frac{1}{\eta_y} - \gamma - 2\nu^{-1}\tau\right)\sqn{\y^{k+1} - \y^k}
		\\& \quad
		+2\gamma\sqn{m^k}_\mP
		+2\norm{\Delta_x^k}_\mP\norm{m^k}_\mP
		+\eta_x\sqn{\Delta_x^k}_\mP
		-\frac{\nu^{-1}}{2\tau\chi(T)}\sqn{\y_c^k+\x_c^k}_\mP
		\\& \quad
		-2\<\z^{k+1} - \z^*, \y^{k+1} - \y^*>
		-\gamma\sqn{\x^k - \x^*}
		+2\gamma\nu^2\sqn{\z^k - \z^*}
		+5\gamma\sqn{\Delta^{k+1/2} - \F(\z^*)}
		\\& \quad
		-\frac{\nu^{-1}}{\tau}\sqn{\y_f^{k+1} + \x_f^{k+1} - \y^* - \x^*}
		+\frac{\nu^{-1}}{\tau}\sqn{\y_c^{k} + \x_c^{k} - \y^* - \x^*}
		\\& \quad
		-2\nu^{-1}\sqn{\y_c^{k} + \x_c^{k} - \y^* - \x^*}
		\\& \quad
		+\frac{2\nu^{-1}(1-\tau)}{\tau}\< \y_c^k + \x_c^k - \y^* - \x^*, \y^k_f+\x^k_f - \y_c^k-\x_c^k>.
	\end{align*}
Using parallelogram rule we obtain	
     \begin{align*}
		\frac{1}{\eta_y}&\sqn{\y^{k+1} - \y^*}
		+\frac{1}{\eta_x}\sqn{\hat \x^{k+1} - \x^*} 
		\\& \leq
		\left(\frac{1}{\eta_y} - \gamma\right)\sqn{\y^{k} - \y^*}
		+\left(\frac{1}{\eta_x} - \gamma\right)\sqn{\hat \x^{k} - \x^*} 
		-\left(\frac{1}{\eta_y} - \gamma - 2\nu^{-1}\tau\right)\sqn{\y^{k+1} - \y^k}
		\\& \quad
		+2\gamma\sqn{m^k}_\mP
		+2\norm{\Delta_x^k}_\mP\norm{m^k}_\mP
		+\eta_x\sqn{\Delta_x^k}_\mP
		-\frac{\nu^{-1}}{2\tau\chi(T)}\sqn{\y_c^k+\x_c^k}_\mP
		\\& \quad
		-2\<\z^{k+1} - \z^*, \y^{k+1} - \y^*>
		-\gamma\sqn{\x^k - \x^*}
		+2\gamma\nu^2\sqn{\z^k - \z^*}
		+5\gamma\sqn{\Delta^{k+1/2} - \F(\z^*)}
		\\& \quad
		-\frac{\nu^{-1}}{\tau}\sqn{\y_f^{k+1} + \x_f^{k+1} - \y^* - \x^*}
		+\frac{\nu^{-1}}{\tau}\sqn{\y_c^{k} + \x_c^{k} - \y^* - \x^*}
		\\& \quad
		-2\nu^{-1}\sqn{\y_c^{k} + \x_c^{k} - \y^* - \x^*}
		\\& \quad
		+\frac{\nu^{-1}(1-\tau)}{\tau}\sqn{\y^k_f+\x^k_f - \y^* - \x^*} - \frac{\nu^{-1}(1-\tau)}{\tau}\sqn{\y_c^{k} + \x_c^{k} - \y^* - \x^*}
		\\& \leq
		\left(\frac{1}{\eta_y} - \gamma\right)\sqn{\y^{k} - \y^*}
		+\left(\frac{1}{\eta_x} - \gamma\right)\sqn{\hat \x^{k} - \x^*} 
		-\left(\frac{1}{\eta_y} - \gamma - 2\nu^{-1}\tau\right)\sqn{\y^{k+1} - \y^k}
		\\& \quad
		+2\gamma\sqn{m^k}_\mP
		+2\norm{\Delta_x^k}_\mP\norm{m^k}_\mP
		+\eta_x\sqn{\Delta_x^k}_\mP
		-\frac{\nu^{-1}}{2\tau\chi(T)}\sqn{\y_c^k+\x_c^k}_\mP
		\\& \quad
		-2\<\z^{k+1} - \z^*, \y^{k+1} - \y^*>
		-\gamma\sqn{\x^k - \x^*}
		+2\gamma\nu^2\sqn{\z^k - \z^*}
		+5\gamma\sqn{\Delta^{k+1/2} - \F(\z^*)}
		\\& \quad
		-\frac{\nu^{-1}}{\tau}\sqn{\y_f^{k+1} + \x_f^{k+1} - \y^* - \x^*}
		+\frac{\nu^{-1}(1-\tau)}{\tau}\sqn{\y^k_f+\x^k_f - \y^* - \x^*} .
	\end{align*}
Using Young's inequality we get
    \begin{align*}
		\frac{1}{\eta_y}&\sqn{\y^{k+1} - \y^*}
		+\frac{1}{\eta_x}\sqn{\hat \x^{k+1} - \x^*} 
		\\& \leq
		\left(\frac{1}{\eta_y} - \gamma\right)\sqn{\y^{k} - \y^*}
		+\left(\frac{1}{\eta_x} - \gamma\right)\sqn{\hat \x^{k} - \x^*} 
		-\left(\frac{1}{\eta_y} - \gamma - 2\nu^{-1}\tau\right)\sqn{\y^{k+1} - \y^k}
		\\& \quad
		+2\gamma\sqn{m^k}_\mP + 4\eta_x\chi(T)\sqn{\Delta_x^k}_\mP
		+(4\eta_x\chi(T))^{-1}\sqn{m^k}_\mP
		+\eta_x\sqn{\Delta_x^k}_\mP
		\\& \quad
		-\frac{\nu^{-1}}{2\tau\chi(T)}\sqn{\y_c^k+\x_c^k}_\mP
		-2\<\z^{k+1} - \z^*, \y^{k+1} - \y^*>
		-\gamma\sqn{\x^k - \x^*}
		\\& \quad
		+2\gamma\nu^2\sqn{\z^k - \z^*}
		+5\gamma\sqn{\Delta^{k+1/2} - \F(\z^*)}
		-\frac{\nu^{-1}}{\tau}\sqn{\y_f^{k+1} + \x_f^{k+1} - \y^* - \x^*}
		\\& \quad
		+\frac{\nu^{-1}(1-\tau)}{\tau}\sqn{\y^k_f+\x^k_f - \y^* - \x^*} 
		\\& =
		\left(\frac{1}{\eta_y} - \gamma\right)\sqn{\y^{k} - \y^*}
		+\left(\frac{1}{\eta_x} - \gamma\right)\sqn{\hat \x^{k} - \x^*} 
		-\left(\frac{1}{\eta_y} - \gamma - 2\nu^{-1}\tau\right)\sqn{\y^{k+1} - \y^k}
		\\& \quad
		+2\gamma\sqn{m^k}_\mP + 4\eta_x\chi(T)\sqn{\Delta_x^k}_\mP
		+(4\eta_x\chi(T))^{-1}\sqn{m^k}_\mP
		+\eta_x\sqn{\Delta_x^k}_\mP
		\\& \quad
		-\frac{\nu^{-1}}{2\tau\chi(T)}\sqn{\y_c^k+\x_c^k}_\mP
		-2\<\z^{k+1} - \z^*, \y^{k+1} - \y^*>
		-\gamma\sqn{\x^k - \x^*}
		\\& \quad
		+2\gamma\nu^2\sqn{\z^k - \z^*}
		+5\gamma\sqn{\Delta^{k+1/2} - \F(\z^*)}
		-\frac{\nu^{-1}}{\tau}\sqn{\y_f^{k+1} + \x_f^{k+1} - \y^* - \x^*}
		\\& \quad
		+\frac{\nu^{-1}(1-\tau)}{\tau}\sqn{\y^k_f+\x^k_f - \y^* - \x^*} .
	\end{align*}
By the assumption on $\eta_x \leq \tfrac{1}{8\chi(T)\gamma}$ we get that $2\gamma \leq \tfrac{1}{4 \eta_x}$ and
    \begin{align*}
		\frac{1}{\eta_y}&\sqn{\y^{k+1} - \y^*}
		+\frac{1}{\eta_x}\sqn{\hat \x^{k+1} - \x^*} 
		\\& \leq
		\left(\frac{1}{\eta_y} - \gamma\right)\sqn{\y^{k} - \y^*}
		+\left(\frac{1}{\eta_x} - \gamma\right)\sqn{\hat \x^{k} - \x^*} 
		-\left(\frac{1}{\eta_y} - \gamma - 2\nu^{-1}\tau\right)\sqn{\y^{k+1} - \y^k}
		\\& \quad + 
		4\eta_x\chi(T)\sqn{\Delta_x^k}_\mP
		+(2\eta_x\chi(T))^{-1}\sqn{m^k}_\mP
		+\eta_x\sqn{\Delta_x^k}_\mP
		-\frac{\nu^{-1}}{2\tau\chi(T)}\sqn{\y_c^k+\x_c^k}_\mP
		\\& \quad
		-2\<\z^{k+1} - \z^*, \y^{k+1} - \y^*>
		-\gamma\sqn{\x^k - \x^*}
		+2\gamma\nu^2\sqn{\z^k - \z^*}
		\\& \quad
		+5\gamma\sqn{\Delta^{k+1/2} - \F(\z^*)}
		-\frac{\nu^{-1}}{\tau}\sqn{\y_f^{k+1} + \x_f^{k+1} - \y^* - \x^*}
		\\& \quad
		+\frac{\nu^{-1}(1-\tau)}{\tau}\sqn{\y^k_f+\x^k_f - \y^* - \x^*} .
	\end{align*}
Using \eqref{dvi2:eq:m} we get
	\begin{align*}
		\frac{1}{\eta_y}&\sqn{\y^{k+1} - \y^*}
		+\frac{1}{\eta_x}\sqn{\hat \x^{k+1} - \x^*} 
		\\& \leq
		\left(\frac{1}{\eta_y} - \gamma\right)\sqn{\y^{k} - \y^*}
		+\left(\frac{1}{\eta_x} - \gamma\right)\sqn{\hat \x^{k} - \x^*} 
		-\left(\frac{1}{\eta_y} - \gamma - 2\nu^{-1}\tau\right)\sqn{\y^{k+1} - \y^k}
		\\& \quad + 
		8\eta_x\chi(T)\sqn{\Delta_x^k}_\mP
		+(1-(4\chi(T))^{-1})\frac{2}{\eta_x}\sqn{m^k}_\mP -\frac{2}{\eta_x}\sqn{m^{k+1}}_\mP
		+\eta_x\sqn{\Delta_x^k}_\mP
		\\& \quad
		-\frac{\nu^{-1}}{2\tau\chi(T)}\sqn{\y_c^k+\x_c^k}_\mP -2\<\z^{k+1} - \z^*, \y^{k+1} - \y^*>
		-\gamma\sqn{\x^k - \x^*}
		+2\gamma\nu^2\sqn{\z^k - \z^*}
		\\& \quad
	    +5\gamma\sqn{\Delta^{k+1/2} - \F(\z^*)}-\frac{\nu^{-1}}{\tau}\sqn{\y_f^{k+1} + \x_f^{k+1} - \y^* - \x^*}
	    \\& \quad
	    +\frac{\nu^{-1}(1-\tau)}{\tau}\sqn{\y^k_f+\x^k_f - \y^* - \x^*} .
	\end{align*}
Using the definition of $\Delta_x^k$ with definition of $\beta = 5 \gamma$ we get
    \begin{align*}
		\frac{1}{\eta_y}&\sqn{\y^{k+1} - \y^*}
		+\frac{1}{\eta_x}\sqn{\hat \x^{k+1} - \x^*} 
		\\& \leq
		\left(\frac{1}{\eta_y} - \gamma\right)\sqn{\y^{k} - \y^*}
		+\left(\frac{1}{\eta_x} - \gamma\right)\sqn{\hat \x^{k} - \x^*} 
		-\left(\frac{1}{\eta_y} - \gamma - 2\nu^{-1}\tau\right)\sqn{\y^{k+1} - \y^k}
		\\& \quad + 
		9\eta_x\chi(T)\sqn{\nu^{-1} (\y_c^k + \x_c^k) + 5\gamma(\x^k + \Delta^{k+1/2})}_\mP
		+(1-(4\chi(T))^{-1})\frac{2}{\eta_x}\sqn{m^k}_\mP
		\\& \quad
		-\frac{2}{\eta_x}\sqn{m^{k+1}}_\mP -\frac{\nu^{-1}}{2\tau\chi(T)}\sqn{\y_c^k+\x_c^k}_\mP -2\<\z^{k+1} - \z^*, \y^{k+1} - \y^*>
		-\gamma\sqn{\x^k - \x^*}
		\\& \quad
		+2\gamma\nu^2\sqn{\z^k - \z^*}
		+5\gamma\sqn{\Delta^{k+1/2} - \F(\z^*)}
		-\frac{\nu^{-1}}{\tau}\sqn{\y_f^{k+1} + \x_f^{k+1} - \y^* - \x^*}
		\\& \quad
		+\frac{\nu^{-1}(1-\tau)}{\tau}\sqn{\y^k_f+\x^k_f - \y^* - \x^*} .
	\end{align*}
By definition $\mP\F(\z^*) + \x^* = 0$ we obtain
    \begin{align*}
		\frac{1}{\eta_y}&\sqn{\y^{k+1} - \y^*}
		+\frac{1}{\eta_x}\sqn{\hat \x^{k+1} - \x^*} 
		\\& \leq
		\left(\frac{1}{\eta_y} - \gamma\right)\sqn{\y^{k} - \y^*}
		+\left(\frac{1}{\eta_x} - \gamma\right)\sqn{\hat \x^{k} - \x^*} 
		-\left(\frac{1}{\eta_y} - \gamma - 2\nu^{-1}\tau\right)\sqn{\y^{k+1} - \y^k}
		\\& \quad + 
		18\eta_x \nu^{-2}\chi(T)\sqn{ \y_c^k + \x_c^k}_\mP +
		450\eta_x\chi(T) \gamma^2 \sqn{\x^k - \x^* - \mP\F(\z^*) + \Delta^{k+1/2}}_\mP
		\\& \quad
		+(1-(4\chi(T))^{-1})\frac{2}{\eta_x}\sqn{m^k}_\mP
		-\frac{2}{\eta_x}\sqn{m^{k+1}}_\mP
		-\frac{\nu^{-1}}{2\tau\chi(T)}\sqn{\y_c^k+\x_c^k}_\mP
		\\& \quad
		-2\<\z^{k+1} - \z^*, \y^{k+1} - \y^*>
		-\gamma\sqn{\x^k - \x^*}
		+2\gamma\nu^2\sqn{\z^k - \z^*}
		+5\gamma\sqn{\Delta^{k+1/2}  - \F(\z^*)}
		\\& \quad
		-\frac{\nu^{-1}}{\tau}\sqn{\y_f^{k+1} + \x_f^{k+1} - \y^* - \x^*} +\frac{\nu^{-1}(1-\tau)}{\tau}\sqn{\y^k_f+\x^k_f - \y^* - \x^*} 
		\\& \leq
		\left(\frac{1}{\eta_y} - \gamma\right)\sqn{\y^{k} - \y^*}
		+\left(\frac{1}{\eta_x} - \gamma\right)\sqn{\hat \x^{k} - \x^*} 
		-\left(\frac{1}{\eta_y} - \gamma - 2\nu^{-1}\tau\right)\sqn{\y^{k+1} - \y^k}
		\\& \quad + 
		18\eta_x \nu^{-2}\chi(T)\sqn{ \y_c^k + \x_c^k}_\mP +
		900\eta_x\chi(T) \gamma^2 \sqn{\x^k - \x^*}_\mP
		\\& \quad
		+
		900\eta_x\chi(T) \gamma^2 \sqn{  \Delta^{k+1/2} - \mP\F(\z^*)}_\mP +(1-(4\chi(T))^{-1})\frac{2}{\eta_x}\sqn{m^k}_\mP 
		\\& \quad
		-\frac{2}{\eta_x}\sqn{m^{k+1}}_\mP
		-\frac{\nu^{-1}}{2\tau\chi(T)}\sqn{\y_c^k+\x_c^k}_\mP 
		-2\<\z^{k+1} - \z^*, \y^{k+1} - \y^*> -\gamma\sqn{\x^k - \x^*}
		\\& \quad
		+2\gamma\nu^2\sqn{\z^k - \z^*} +5\gamma\sqn{\Delta^{k+1/2}  - \F(\z^*)} 
		\\& \quad
		-\frac{\nu^{-1}}{\tau}\sqn{\y_f^{k+1} + \x_f^{k+1} - \y^* - \x^*} +\frac{\nu^{-1}(1-\tau)}{\tau}\sqn{\y^k_f+\x^k_f - \y^* - \x^*} .
	\end{align*}
With the assumption on $\eta_x \leq \frac{1}{900\chi(T)\gamma}; \frac{\nu}{36\tau\chi^2(T)}$ we get
    \begin{align*}
		\frac{1}{\eta_y}&\sqn{\y^{k+1} - \y^*}
		+\frac{1}{\eta_x}\sqn{\hat \x^{k+1} - \x^*} 
		\\& \leq
		\left(\frac{1}{\eta_y} - \gamma\right)\sqn{\y^{k} - \y^*}
		+\left(\frac{1}{\eta_x} - \gamma\right)\sqn{\hat \x^{k} - \x^*} 
		-\left(\frac{1}{\eta_y} - \gamma - 2\nu^{-1}\tau\right)\sqn{\y^{k+1} - \y^k}
		\\& \quad + 
		\frac{\nu^{-1}}{2\tau\chi(T)}\sqn{ \y_c^k + \x_c^k}_\mP +
		\gamma \sqn{\x^k - \x^*}_\mP
		+
		\gamma\sqn{  \Delta^{k+1/2} - \mP\F(\z^*)}_\mP
		\\& \quad
		+(1-(4\chi(T))^{-1})\frac{2}{\eta_x}\sqn{m^k}_\mP -\frac{2}{\eta_x}\sqn{m^{k+1}}_\mP
		-\frac{\nu^{-1}}{2\tau\chi(T)}\sqn{\y_c^k+\x_c^k}_\mP 
		\\& \quad
		-2\<\z^{k+1} - \z^*, \y^{k+1} - \y^*> -\gamma\sqn{\x^k - \x^*}
		+2\gamma\nu^2\sqn{\z^k - \z^*} +5\gamma\sqn{\Delta^{k+1/2}  - \F(\z^*)} 
		\\& \quad
		-\frac{\nu^{-1}}{\tau}\sqn{\y_f^{k+1} + \x_f^{k+1} - \y^* - \x^*} +\frac{\nu^{-1}(1-\tau)}{\tau}\sqn{\y^k_f+\x^k_f - \y^* - \x^*} .
	\end{align*}
With property of the projector: $\mP \mP = \mP$, we get
    \begin{align*}
		\frac{1}{\eta_y}&\sqn{\y^{k+1} - \y^*}
		+\frac{1}{\eta_x}\sqn{\hat \x^{k+1} - \x^*} 
		\\& \leq
		\left(\frac{1}{\eta_y} - \gamma\right)\sqn{\y^{k} - \y^*}
		+\left(\frac{1}{\eta_x} - \gamma\right)\sqn{\hat \x^{k} - \x^*} 
		-\left(\frac{1}{\eta_y} - \gamma - 2\nu^{-1}\tau\right)\sqn{\y^{k+1} - \y^k}
		\\& \quad
		+(1-(4\chi(T))^{-1})\frac{2}{\eta_x}\sqn{m^k}_\mP -\frac{2}{\eta_x}\sqn{m^{k+1}}_\mP
		-2\<\z^{k+1} - \z^*, \y^{k+1} - \y^*>
		\\& \quad
		+2\gamma\nu^2\sqn{\z^k - \z^*} +6\gamma\sqn{\Delta^{k+1/2}  - \F(\z^*)} 
		\\& \quad
		-\frac{\nu^{-1}}{\tau}\sqn{\y_f^{k+1} + \x_f^{k+1} - \y^* - \x^*} +\frac{\nu^{-1}(1-\tau)}{\tau}\sqn{\y^k_f+\x^k_f - \y^* - \x^*} .
	\end{align*}
Taking full expectation and using \eqref{2vrvi:eq:2} we have
    \begin{align*}
		\frac{1}{\eta_y}&\E{\sqn{\y^{k+1} - \y^*}}
		+\frac{1}{\eta_x}\E{\sqn{\hat \x^{k+1} - \x^*}}
		\\& \leq
		\left(\frac{1}{\eta_y} - \gamma\right)\E{\sqn{\y^{k} - \y^*}}
		+\left(\frac{1}{\eta_x} - \gamma\right)\E{\sqn{\hat \x^{k} - \x^*}}
		\\& \quad
		-\left(\frac{1}{\eta_y} - \gamma - 2\nu^{-1}\tau\right)\E{\sqn{\y^{k+1} - \y^k}} +(1-(4\chi(T))^{-1})\frac{2}{\eta_x}\E{\sqn{m^k}_\mP}
		\\& \quad
		-\frac{2}{\eta_x}\E{\sqn{m^{k+1}}_\mP} -2\E{\<\z^{k+1} - \z^*, \y^{k+1} - \y^*>} +2\gamma\nu^2\E{\sqn{\z^k - \z^*}} 
		\\& \quad
		+ 12 \gamma L^2\E{\sqn{\z^{k+1} - \z^*}} + \frac{12\gamma \Lavg^2}{b} \E{\sqn{\z^{k+1} - \w^{k}}} 
		\\& \quad
		-\frac{\nu^{-1}}{\tau}\E{\sqn{\y_f^{k+1} + \x_f^{k+1} - \y^* - \x^*}} +\frac{\nu^{-1}(1-\tau)}{\tau}\E{\sqn{\y^k_f+\x^k_f - \y^* - \x^*}} .
	\end{align*}
Using the definition of $\eta_y$ we get that $\gamma \leq \frac{1}{4 \eta_y}$ and $2\nu^{-1}\tau \leq \frac{1}{4\eta_y}$
    \begin{align*}
		\frac{1}{\eta_y}&\E{\sqn{\y^{k+1} - \y^*}}
		+\frac{1}{\eta_x}\E{\sqn{\hat \x^{k+1} - \x^*}}
		\\& \leq
		\left(\frac{1}{\eta_y} - \gamma\right)\E{\sqn{\y^{k} - \y^*}}
		+\left(\frac{1}{\eta_x} - \gamma\right)\E{\sqn{\hat \x^{k} - \x^*}}
		-\frac{1}{2\eta_y}\E{\sqn{\y^{k+1} - \y^k}}
		\\& \quad
		+(1-(4\chi(T))^{-1})\frac{2}{\eta_x}\E{\sqn{m^k}_\mP}-\frac{2}{\eta_x}\E{\sqn{m^{k+1}}_\mP}
		-2\E{\<\z^{k+1} - \z^*, \y^{k+1} - \y^*>}
		\\& \quad
		+2\gamma\nu^2\E{\sqn{\z^k - \z^*}} + 12 \gamma L^2\E{\sqn{\z^{k+1} - \z^*}}+ \frac{12\gamma \Lavg^2}{b} \E{\sqn{\z^{k+1} - \w^{k}}}
		\\& \quad
		-\frac{\nu^{-1}}{\tau}\E{\sqn{\y_f^{k+1} + \x_f^{k+1} - \y^* - \x^*}} +\frac{\nu^{-1}(1-\tau)}{\tau}\E{\sqn{\y^k_f+\x^k_f - \y^* - \x^*}} .
	\end{align*}
{\bf Part 4.}
After combining parts~1 and~3 of this proof we get
    \begin{align*}
		\frac{1}{\eta_y}&\E{\sqn{\y^{k+1} - \y^*}}
		+\frac{1}{\eta_x}\E{\sqn{\hat \x^{k+1} - \x^*}} + \frac{1}{\eta_z}\E{\sqn{\z^{k+1} - \z^*}} + (2\mu - \nu)\E{\sqn{\z^{k+1} - \z^*}} 
		\\&\quad+ \frac{1}{4\eta_z}\E{\sqn{\z^{k+1} - \z^k}}  - 2\E{\<\F(\z^{k+1}) - \F(\z^k) - (\y^{k+1} - \y^k), \z^{k+1} - \z^*>}
		\\& \leq
		\left(\frac{1}{\eta_y} - \gamma\right)\E{\sqn{\y^{k} - \y^*}}
		+\left(\frac{1}{\eta_x} - \gamma\right)\E{\sqn{\hat \x^{k} - \x^*}}
		-\frac{1}{2\eta_y}\E{\sqn{\y^{k+1} - \y^k}}
		\\& \quad
		+(1-(4\chi(T))^{-1})\frac{2}{\eta_x}\E{\sqn{m^k}_\mP}-\frac{2}{\eta_x}\E{\sqn{m^{k+1}}_\mP}
		-2\E{\<\z^{k+1} - \z^*, \y^{k+1} - \y^*>}
		\\& \quad
		+2\gamma\nu^2\E{\sqn{\z^k - \z^*}} + 12 \gamma L^2\E{\sqn{\z^{k+1} - \z^*}}+ \frac{12\gamma \Lavg^2}{b} \E{\sqn{\z^{k+1} - \w^{k}}}
		\\& \quad
		-\frac{\nu^{-1}}{\tau}\E{\sqn{\y_f^{k+1} + \x_f^{k+1} - \y^* - \x^*}} +\frac{\nu^{-1}(1-\tau)}{\tau}\E{\sqn{\y^k_f+\x^k_f - \y^* - \x^*}} .
        \\& \quad
		+\frac{1}{\eta_z}\E{\sqn{\z^k - \z^*}}  - \left(\frac{\omega}{\eta_z} - \nu \right)\E{\sqn{\z^k - \z^*}}
		\\&\quad
		+ \alpha \cdot \frac{1}{4\eta_z} \E{\sqn{\z^k - \z^{k-1}}} - \alpha  \cdot 2\E{\<\F(\z^k) - \F(\z^{k-1}) - (\y^k - \y^{k-1}), \z^k - \z^*>}
		\\&\quad
		+ \frac{\omega}{\eta_z} \E{\sqn{\w^k - \z^*}} -  \frac{\omega}{\eta_z}\E{\sqn{\z^{k+1} - \w^k}}
		+2\E{\<\y^{k+1} - \y^*,\z^{k+1} - \z^*>}
		\\&\quad + \frac{\alpha}{2\eta_y} \E{\sqn{\y^k - \y^{k-1}}}
		+ \frac{4\eta_z\Lavg^2}{b}\E{\sqn{\z^{k} - \w^{k-1}}}.
	\end{align*}
Small rearrangement gives
	\begin{align*}
		\frac{1}{\eta_y}&\E{\sqn{\y^{k+1} - \y^*}}
		+\frac{1}{\eta_x}\E{\sqn{\hat \x^{k+1} - \x^*}} +\frac{1}{2\eta_y}\E{\sqn{\y^{k+1} - \y^k}} 
		\\&\quad  + \frac{1}{4\eta_z}\E{\sqn{\z^{k+1} - \z^k}} - 2\E{\<\F(\z^{k+1}) - \F(\z^k) - (\y^{k+1} - \y^k), \z^{k+1} - \z^*>}
		\\&\quad +\frac{2}{\eta_x}\E{\sqn{m^{k+1}}_\mP} + \frac{\nu^{-1}}{\tau}\E{\sqn{\y_f^{k+1} + \x_f^{k+1} - \y^* - \x^*}} 
		\\&\quad + \left(1- \frac{12\gamma \Lavg^2 \eta_z}{b\omega} \right)\frac{\omega}{\eta_z}\E{\sqn{\z^{k+1} - \w^k}}  + \left(\frac{1}{\eta_z} + 2\mu - \nu - 12 \gamma L^2\right)\E{\sqn{\z^{k+1} - \z^*}}
		\\& \leq (1 - \eta_y \gamma)\cdot \frac{1}{\eta_y}\E{\sqn{\y^{k} - \y^*}} +\left(1 - \eta_x \gamma\right)\cdot \frac{1}{\eta_x}\E{\sqn{\hat \x^{k} - \x^*}} + \alpha \cdot \frac{1}{2\eta_y} \E{\sqn{\y^k - \y^{k-1}}}
		\\&\quad + \alpha \cdot \frac{1}{4\eta_z} \E{\sqn{\z^k - \z^{k-1}}} - \alpha  \cdot 2\E{\<\F(\z^k) - \F(\z^{k-1}) - (\y^k - \y^{k-1}), \z^k - \z^*>}
		\\&\quad +(1-(4\chi(T))^{-1}) \cdot \frac{2}{\eta_x}\E{\sqn{m^k}_\mP} +(1-\tau) \cdot \frac{\nu^{-1}}{\tau}\E{\sqn{\y^k_f+\x^k_f - \y^* - \x^*}}
		\\&\quad + \frac{4\eta_z\Lavg^2}{b}\E{\sqn{\z^{k} - \w^{k-1}}} + 
		\left(\frac{1}{\eta_z} - \frac{\omega}{\eta_z} + \nu + 2\gamma\nu^2\right)\E{\sqn{\z^k - \z^*}}
		\\&\quad
		+ \frac{\omega}{\eta_z} \E{\sqn{\w^k - \z^*}}.
	\end{align*}
Now, we add $\tfrac{\gamma + \eta\nu}{p\eta_z}\E{\sqn{\w^{k+1} - \z^*}}$ to both sides and use update for $\w^{k+1}$
	\begin{align*}
	\frac{\omega + \eta_z\nu}{p\eta_z}\E{\mathbb{E}_{\w^{k+1}}{\sqn{\w^{k+1} - \z^*}}} = \frac{\omega + \eta_z\nu}{\eta_z}\E{\sqn{\z^{k} - \z^*}} + \frac{(\omega + \eta_z\nu)(1-p)}{\eta_z p}\E{\sqn{\w^{k} - \z^*}},
	\end{align*}
and then 
	\begin{align*}
		\frac{1}{\eta_y}&\E{\sqn{\y^{k+1} - \y^*}}
		+\frac{1}{\eta_x}\E{\sqn{\hat \x^{k+1} - \x^*}} +\frac{1}{2\eta_y}\E{\sqn{\y^{k+1} - \y^k}} 
		\\&\quad  + \frac{1}{4\eta_z}\E{\sqn{\z^{k+1} - \z^k}} - 2\E{\<\F(\z^{k+1}) - \F(\z^k) - (\y^{k+1} - \y^k), \z^{k+1} - \z^*>}
		\\&\quad +\frac{2}{\eta_x}\E{\sqn{m^{k+1}}_\mP} + \frac{\nu^{-1}}{\tau}\E{\sqn{\y_f^{k+1} + \x_f^{k+1} - \y^* - \x^*}} 
		\\&\quad + \left(1- \frac{12\gamma \Lavg^2 \eta_z}{b\omega} \right)\frac{\omega}{\eta_z}\E{\sqn{\z^{k+1} - \w^k}}  + \left(\frac{1}{\eta_z} + 2\mu - \nu - 12 \gamma L^2\right)\E{\sqn{\z^{k+1} - \z^*}}
		\\&\quad + 	\frac{\omega + \eta_z\nu}{p\eta_z}\E{\sqn{\w^{k+1} - \z^*}}
		\\& \leq (1 - \eta_y \gamma)\cdot \frac{1}{\eta_y}\E{\sqn{\y^{k} - \y^*}} +\left(1 - \eta_x \gamma\right)\cdot \frac{1}{\eta_x}\E{\sqn{\hat \x^{k} - \x^*}} + \alpha \cdot \frac{1}{2\eta_y} \E{\sqn{\y^k - \y^{k-1}}}
		\\&\quad + \alpha \cdot \frac{1}{4\eta_z} \E{\sqn{\z^k - \z^{k-1}}} - \alpha  \cdot 2\E{\<\F(\z^k) - \F(\z^{k-1}) - (\y^k - \y^{k-1}), \z^k - \z^*>}
		\\&\quad +(1-(4\chi(T))^{-1}) \cdot \frac{2}{\eta_x}\E{\sqn{m^k}_\mP} +(1-\tau) \cdot \frac{\nu^{-1}}{\tau}\E{\sqn{\y^k_f+\x^k_f - \y^* - \x^*}}
		\\&\quad + \frac{4\eta_z\Lavg^2}{b}\E{\sqn{\z^{k} - \w^{k-1}}} + 
		\left(\frac{1}{\eta_z}+ 2\nu + 2\gamma\nu^2\right)\E{\sqn{\z^k - \z^*}}
		\\&\quad
		+ \left(1 - \frac{ p\eta_z \nu }{\omega + \eta_z \nu}\right) \cdot \frac{\omega + \eta_z\nu}{p\eta_z} \E{\sqn{\w^k - \z^*}}.
	\end{align*}
With $\gamma \leq \tfrac{\mu}{48L^2}$ and $\nu \leq \tfrac{\mu}{4}$ we gat that 
$$
\left( 1+2\mu\eta_x - \nu \eta_x - 12 \gamma L^2 \eta_x\right) \geq 1 + \frac{3\mu \eta_x}{2},
$$
$$
\left(\frac{1}{\eta_x} + 2\nu + 2\gamma \nu^2 \right) \leq \left( 1 + \frac{\mu \eta_x}{2} + \frac{2 \mu^3 \eta_x}{16 \cdot 48 L^2} \right) \frac{1}{\eta_x} \leq \left( 1 + \mu \eta_x \right) \frac{1}{\eta_x}.
$$
Additionally, by 
$$
\left( 1 + \mu \eta_x \right) \leq \left(1 + \frac{3\mu \eta_x}{2}\right) \left(1 - \frac{\mu \eta_x}{8}\right)
$$
we obtain
	\begin{align*}
		\frac{1}{\eta_y}&\E{\sqn{\y^{k+1} - \y^*}}
		+\frac{1}{\eta_x}\E{\sqn{\hat \x^{k+1} - \x^*}} +\frac{1}{2\eta_y}\E{\sqn{\y^{k+1} - \y^k}} 
		\\&\quad  + \frac{1}{4\eta_z}\E{\sqn{\z^{k+1} - \z^k}} - 2\E{\<\F(\z^{k+1}) - \F(\z^k) - (\y^{k+1} - \y^k), \z^{k+1} - \z^*>}
		\\&\quad +\frac{2}{\eta_x}\E{\sqn{m^{k+1}}_\mP} + \frac{\nu^{-1}}{\tau}\E{\sqn{\y_f^{k+1} + \x_f^{k+1} - \y^* - \x^*}} 
		\\&\quad + \left(1- \frac{12\gamma \Lavg^2 \eta_z}{b\omega} \right)\frac{\omega}{\eta_z}\E{\sqn{\z^{k+1} - \w^k}}  + \left(1 + \frac{3\mu \eta_x}{2}\right)\E{\sqn{\z^{k+1} - \z^*}}
		\\&\quad + 	\frac{\omega + \eta_z\nu}{p\eta_z}\E{\sqn{\w^{k+1} - \z^*}}
		\\& \leq (1 - \eta_y \gamma)\cdot \frac{1}{\eta_y}\E{\sqn{\y^{k} - \y^*}} +\left(1 - \eta_x \gamma\right)\cdot \frac{1}{\eta_x}\E{\sqn{\hat \x^{k} - \x^*}} + \alpha \cdot \frac{1}{2\eta_y} \E{\sqn{\y^k - \y^{k-1}}}
		\\&\quad + \alpha \cdot \frac{1}{4\eta_z} \E{\sqn{\z^k - \z^{k-1}}} - \alpha  \cdot 2\E{\<\F(\z^k) - \F(\z^{k-1}) - (\y^k - \y^{k-1}), \z^k - \z^*>}
		\\&\quad +(1-(4\chi(T))^{-1}) \cdot \frac{2}{\eta_x}\E{\sqn{m^k}_\mP} +(1-\tau) \cdot \frac{\nu^{-1}}{\tau}\E{\sqn{\y^k_f+\x^k_f - \y^* - \x^*}}
		\\&\quad + \frac{4\eta_z\Lavg^2}{b}\E{\sqn{\z^{k} - \w^{k-1}}} + 
		\left(1 - \frac{\mu \eta_x}{8}\right) \cdot \left(1 + \frac{3\mu \eta_x}{2}\right) \E{\sqn{\z^k - \z^*}}
		\\&\quad
		+ \left(1 - \frac{ p\eta_z \nu }{\omega + \eta_z \nu}\right) \cdot \frac{\omega + \eta_z\nu}{p\eta_z} \E{\sqn{\w^k - \z^*}}.
	\end{align*}
The choice $\eta_z \leq \tfrac{\sqrt{\alpha \omega b}}{4\Lavg}$ and $\gamma \leq \frac{b \omega}{24 \Lavg^2 \eta_z}$ gives
	\begin{align*}
		\frac{1}{\eta_y}&\E{\sqn{\y^{k+1} - \y^*}}
		+\frac{1}{\eta_x}\E{\sqn{\hat \x^{k+1} - \x^*}} +\frac{1}{2\eta_y}\E{\sqn{\y^{k+1} - \y^k}} 
		\\&\quad  + \frac{1}{4\eta_z}\E{\sqn{\z^{k+1} - \z^k}} - 2\E{\<\F(\z^{k+1}) - \F(\z^k) - (\y^{k+1} - \y^k), \z^{k+1} - \z^*>}
		\\&\quad +\frac{2}{\eta_x}\E{\sqn{m^{k+1}}_\mP} + \frac{\nu^{-1}}{\tau}\E{\sqn{\y_f^{k+1} + \x_f^{k+1} - \y^* - \x^*}} 
		\\&\quad +\frac{\omega}{2\eta_z}\E{\sqn{\z^{k+1} - \w^k}}  + \left(1 + \frac{3\mu \eta_x}{2}\right)\E{\sqn{\z^{k+1} - \z^*}}
		\\&\quad + 	\frac{\omega + \eta_z\nu}{p\eta_z}\E{\sqn{\w^{k+1} - \z^*}}
		\\& \leq (1 - \eta_y \gamma)\cdot \frac{1}{\eta_y}\E{\sqn{\y^{k} - \y^*}} +\left(1 - \eta_x \gamma\right)\cdot \frac{1}{\eta_x}\E{\sqn{\hat \x^{k} - \x^*}} + \alpha \cdot \frac{1}{2\eta_y} \E{\sqn{\y^k - \y^{k-1}}}
		\\&\quad + \alpha \cdot \frac{1}{4\eta_z} \E{\sqn{\z^k - \z^{k-1}}} - \alpha  \cdot 2\E{\<\F(\z^k) - \F(\z^{k-1}) - (\y^k - \y^{k-1}), \z^k - \z^*>}
		\\&\quad +(1-(4\chi(T))^{-1}) \cdot \frac{2}{\eta_x}\E{\sqn{m^k}_\mP} +(1-\tau) \cdot \frac{\nu^{-1}}{\tau}\E{\sqn{\y^k_f+\x^k_f - \y^* - \x^*}}
		\\&\quad + \alpha \cdot \frac{\omega}{2\eta_z}\E{\sqn{\z^{k} - \w^{k-1}}} + 
		\left(1 - \frac{\mu \eta_x}{8}\right) \cdot \left(1 + \frac{3\mu \eta_x}{2}\right) \E{\sqn{\z^k - \z^*}}
		\\&\quad
		+ \left(1 - \frac{ p\eta_z \nu }{\omega + \eta_z \nu}\right) \cdot \frac{\omega + \eta_z\nu}{p\eta_z} \E{\sqn{\w^k - \z^*}}.
	\end{align*}
Putting $p = \omega$ and using definition of the Lyapunov function, we get
	\begin{align*}
		\EE[\Psi^{k+1}] &\leq \max\left[
		\left(1 - \frac{ p\eta_z \nu }{p + \eta_z \nu}\right);
		1 - \eta_y\gamma;
		1-\eta_x\gamma;
		1 - \frac{\mu \eta_z}{8};
		1-\tau;
		1-\frac{1}{4\chi(T)}
		\right]\cdot\EE[\Psi^k].
	\end{align*}
Finally,
	\begin{align*}
		\Psi^k
		&\geq
		\frac{1}{\eta_x}\sqn{x^{k} - \x^*}
		+\frac{1}{4\eta_x}\sqn{x^{k-1} - x^k}
		+\frac{1}{2\eta_y}\sqn{y^{k-1} - y^k} 
		\\&\quad -2\<\F(x^{k}) - \F(x^{k-1}) - (y^{k} - y^{k-1}), x^{k} - \x^*>
		\\&\geq
		\frac{1}{\eta_x}\sqn{x^{k} - \x^*}
		+\frac{1}{4\eta_x}\sqn{x^{k-1} - x^k}
		+\frac{1}{2\eta_y}\sqn{y^{k-1} - y^k}
		\\&\quad
		-\frac{1}{2\eta_y}\sqn{y^{k} - y^{k-1}} - 2\eta_y\sqn{x^k - \x^*}
		-\frac{1}{4\eta_x}\sqn{x^{k-1} - x^k}
		-4\eta_xL^2\sqn{x^k - \x^*}
		\\&=
		\left(1 - 2\eta_x\eta_y - 4\eta_x^2 L^2\right) \frac{1}{\eta_x}\sqn{x^{k} - \x^*}
		\\&\geq
		\frac{1}{4\eta_x}\sqn{x^{k} - \x^*}.
	\end{align*}
\end{proof}	

\begin{theorem}[Theorem \ref{th:Alg2_conv}]\label{th:app_tv}
	Consider the problem \eqref{eq:VI_new} (or \eqref{eq:VI} + \eqref{eq:fs}) under Assumptions~\ref{as:Lipsh} and \ref{as:strmon} over a sequence of time-varying graphs $\mathcal{G}(k)$ with gossip matrices $\mW(k)$. Let  $\{\z^k\}$ be the sequence generated by Algorithm~\ref{2dvi2:alg} with parameters  
$$
	T \geq B; \quad \omega = p \leq \frac{1}{16}; \quad \theta = \frac{1}{2}; \quad \beta = 5\gamma; \quad \nu = \frac{\mu}{4}; \quad \tau = \min\left\{\frac{\mu}{32L \chi(T)}; \frac{\mu \sqrt{bp}}{32 \Lavg} \right\};
$$
$$
    \eta_z = \min\left\{\frac{1}{8L\chi(T)}, \frac{1}{32\eta_y}, \frac{\sqrt{\alpha b \omega}}{8 \Lavg}\right\}; ~~ \eta_y = \min\left\{\frac{1}{4\gamma}, \frac{\nu}{8\tau}\right\}; ~~ \eta_x =\min\left\{ \frac{1}{900\chi(T)\gamma}, \frac{\nu}{36\tau\chi^2(T)}\right\};
$$ 
$$
\gamma = \min\left\{ \frac{\mu}{16L^2}; \frac{b \omega}{24 \Lavg^2 \eta_z}\right\},
$$
$$
\alpha = \max\left[
	\left(1 - \frac{ p\eta_z \nu }{p + \eta_z \nu}\right);
	1 - \eta_y\gamma;
	1-\eta_x\gamma;
	1 - \frac{\mu \eta_z}{8};
	1-\tau;
	1-\frac{1}{4\chi(T)}
	\right].
$$
Let the choice of $T$ guarantees contraction property (Assumption \ref{ass:tv} point 4) with $\chi(T)$.	Then, given $\varepsilon>0$, the number of iterations for 
$\EE[\|\z^k - \z^*\|^2] \leq \varepsilon$ is 
\begin{equation*}
    O\left( \left[\chi^2(T) + \frac{1}{p} + \chi(T)\frac{L}{\mu} + \frac{1}{\sqrt{bp}}\frac{\Lavg}{\mu}\right] \log \frac{1}{\varepsilon} \right).
\end{equation*}
\end{theorem}
\begin{proof}
It is easy to check that $\alpha \geq \tfrac{1}{2}$, then also one can verify that the choice of $\omega$, $\theta$, $\beta$, $\nu$, $\tau$, $\eta_z$, $\eta_y$, $\eta_x$, $\gamma$, $\alpha$  satisfies the conditions of Lemma \ref{lem:key_tv}. We can get that the
iteration complexity of Algorithm~\ref{2dvi2:alg}:
\begin{align*}
O\left( \left[\chi(T) + \frac{1}{p} + \frac{1}{\eta_y \gamma} + \frac{1}{\eta_x \gamma} + \frac{1}{\eta_z \mu} + \frac{1}{\eta_z \nu} + \frac{1}{\tau} + \frac{1}{1-\alpha}\right] \log \frac{1}{\varepsilon} \right).  
\end{align*}
Substituting $\alpha, \nu, \omega$, we get
\begin{align*}
O\left( \left[\chi(T) + \frac{1}{p} + \frac{1}{\eta_y \gamma} + \frac{1}{\eta_x \gamma} + \frac{1}{\eta_z \mu} + \frac{1}{\tau}\right] \log \frac{1}{\varepsilon} \right). 
\end{align*}
Putting $\eta_z$
\begin{align*}
O\left( \left[\chi(T) + \frac{1}{p} + \chi(T)\frac{L}{\mu} + \frac{1}{\sqrt{bp}}\frac{\Lavg}{\mu} + \frac{1}{\eta_y \gamma} + \frac{1}{\eta_x \gamma} + \frac{\eta_y}{ \mu} + \frac{1}{\tau}\right] \log \frac{1}{\varepsilon} \right). 
\end{align*}
Substituting $\eta_y$ and $\eta_x$ 
\begin{align*}
O\left( \left[\chi(T) + \frac{1}{p} + \chi(T)\frac{L}{\mu} + \frac{1}{\sqrt{bp}}\frac{\Lavg}{\mu} + \frac{\tau \chi^2(T)}{\mu\gamma} + \frac{\eta_y}{ \mu} + \frac{1}{\tau}\right] \log \frac{1}{\varepsilon} \right). 
\end{align*}
Using $\eta_y \leq \tfrac{\nu}{8 \tau}$
\begin{align*}
O\left( \left[\chi(T) + \frac{1}{p} + \chi(T)\frac{L}{\mu} + \frac{1}{\sqrt{bp}}\frac{\Lavg}{\mu} + \frac{\tau \chi^2(T)}{\mu\gamma} + \frac{1}{\tau}\right] \log \frac{1}{\varepsilon} \right). 
\end{align*}
Substituting $\tau$ we get
\begin{align*}
O\left( \left[\chi(T) + \frac{1}{p} + \chi(T)\frac{L}{\mu} + \frac{1}{\sqrt{bp}}\frac{\Lavg}{\mu} + \frac{\tau \chi^2(T)}{\mu\gamma}\right] \log \frac{1}{\varepsilon} \right). 
\end{align*}
Putting $\gamma$ we obtain
\begin{align*}
O\left( \left[\chi(T) + \frac{1}{p} + \chi(T)\frac{L}{\mu} + \frac{1}{\sqrt{bp}}\frac{\Lavg}{\mu} + \frac{\tau \Lavg^2 \eta_z \chi^2(T)}{\mu bp} + \frac{\tau L^2\chi^2(T)}{\mu^2}\right] \log \frac{1}{\varepsilon} \right). 
\end{align*}
With $\eta_z \leq \tfrac{\sqrt{abp}}{8 \Lavg}$ and $\tau \leq \tfrac{\mu \sqrt{bp}}{\Lavg}$ we have
\begin{align*}
O\left( \left[\chi^2(T) + \frac{1}{p} + \chi(T)\frac{L}{\mu} + \frac{1}{\sqrt{bp}}\frac{\Lavg}{\mu} + \frac{\tau L^2\chi^2(T)}{\mu^2}\right] \log \frac{1}{\varepsilon} \right). 
\end{align*}
Finally, $\tau \leq \tfrac{1}{8 L \chi(T)}$
\begin{align*}
O\left( \left[\chi^2(T) + \frac{1}{p} + \chi(T)\frac{L}{\mu} + \frac{1}{\sqrt{bp}}\frac{\Lavg}{\mu}\right] \log \frac{1}{\varepsilon} \right). 
\end{align*}
\end{proof}

\newpage

\section{Additional experiments} \label{sec:add_exp}

\subsection{Variance reduction} \label{sec:app_vr_exp}

Here we give additional experiments for Section \ref{sec:vr_exp} on other matrices. 
\begin{figure}[h]
\centering
\captionof{figure}{Comparison epoch complexities of Algorithm~\ref{alg:vrvi}, \algname{EG-Alc-Alg1}, \algname{EG-Alc-Alg2} and \algname{EG-Car} on \eqref{bilinear} with matrices from \cite{nemirovski2009robust} (two upper lines correspond to test matrix 1, the two bottom lines -- to test matrix 2). Dashed lines give convergence with theoretical parameters, solid lines --  with tuned parameters.}
\vspace{-0.3cm}
\begin{minipage}[][][b]{\textwidth}
\centering
\includegraphics[width=0.35\textwidth]{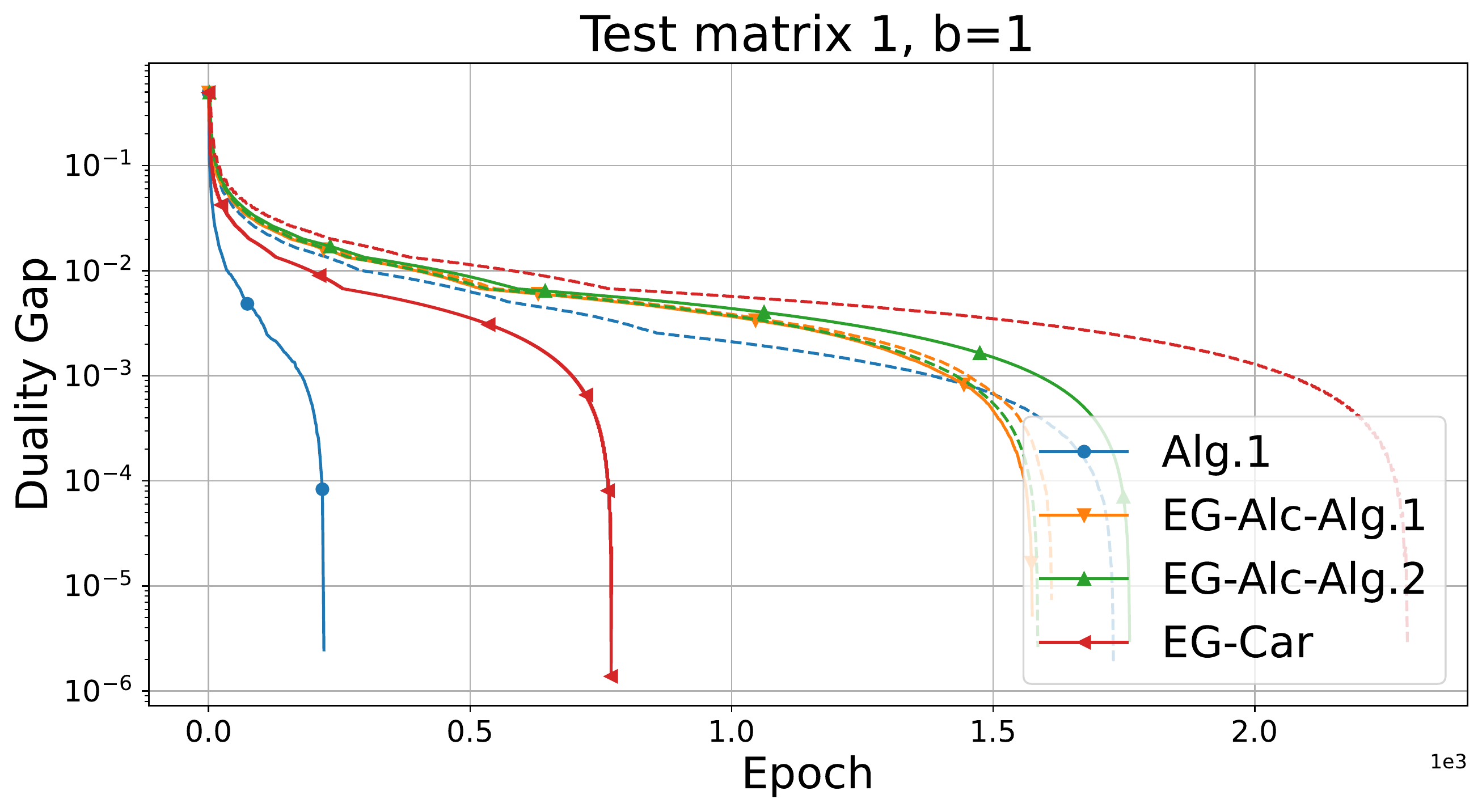}
\includegraphics[width=0.35\textwidth]{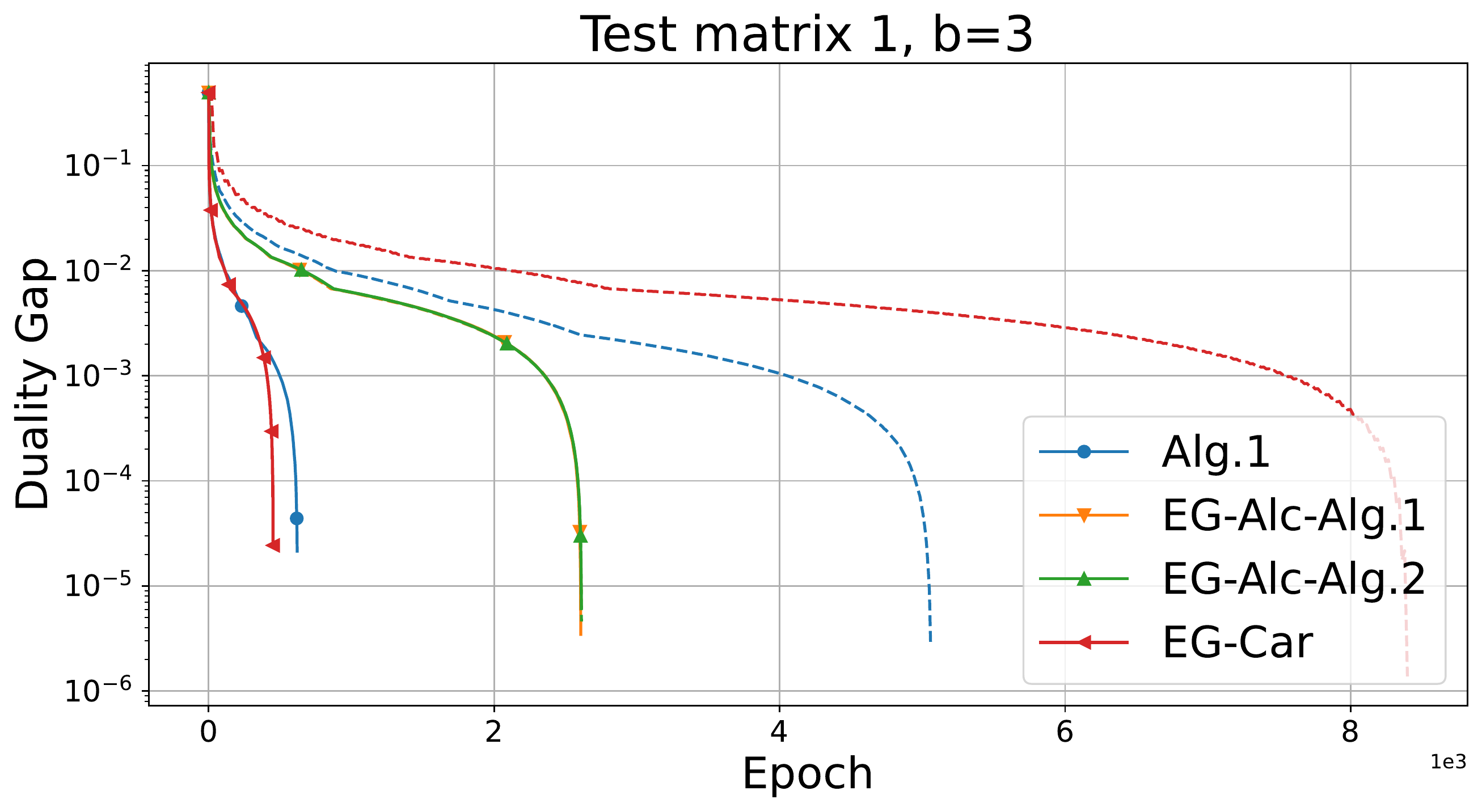}
\\
\includegraphics[width=0.35\textwidth]{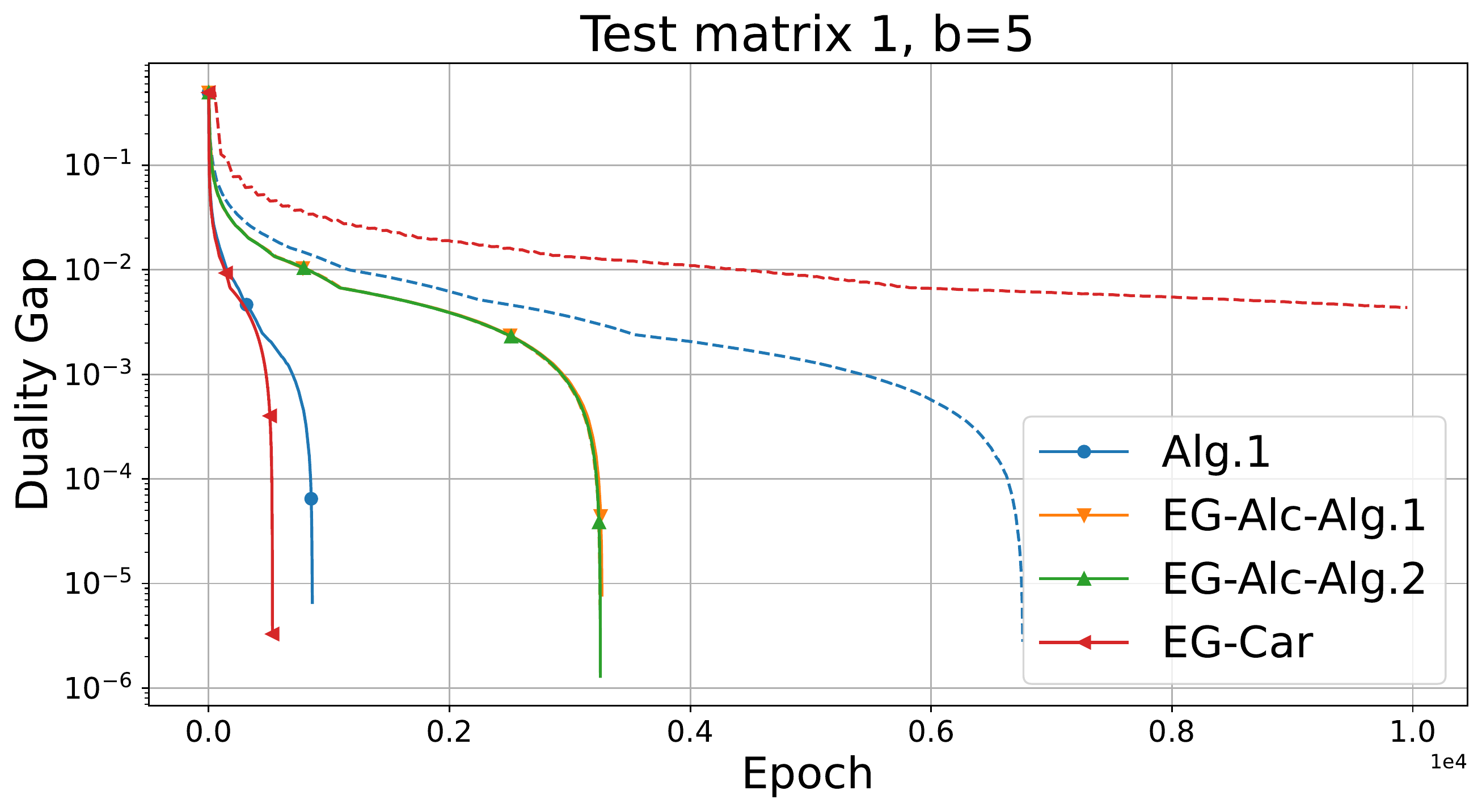}
\includegraphics[width=0.35\textwidth]{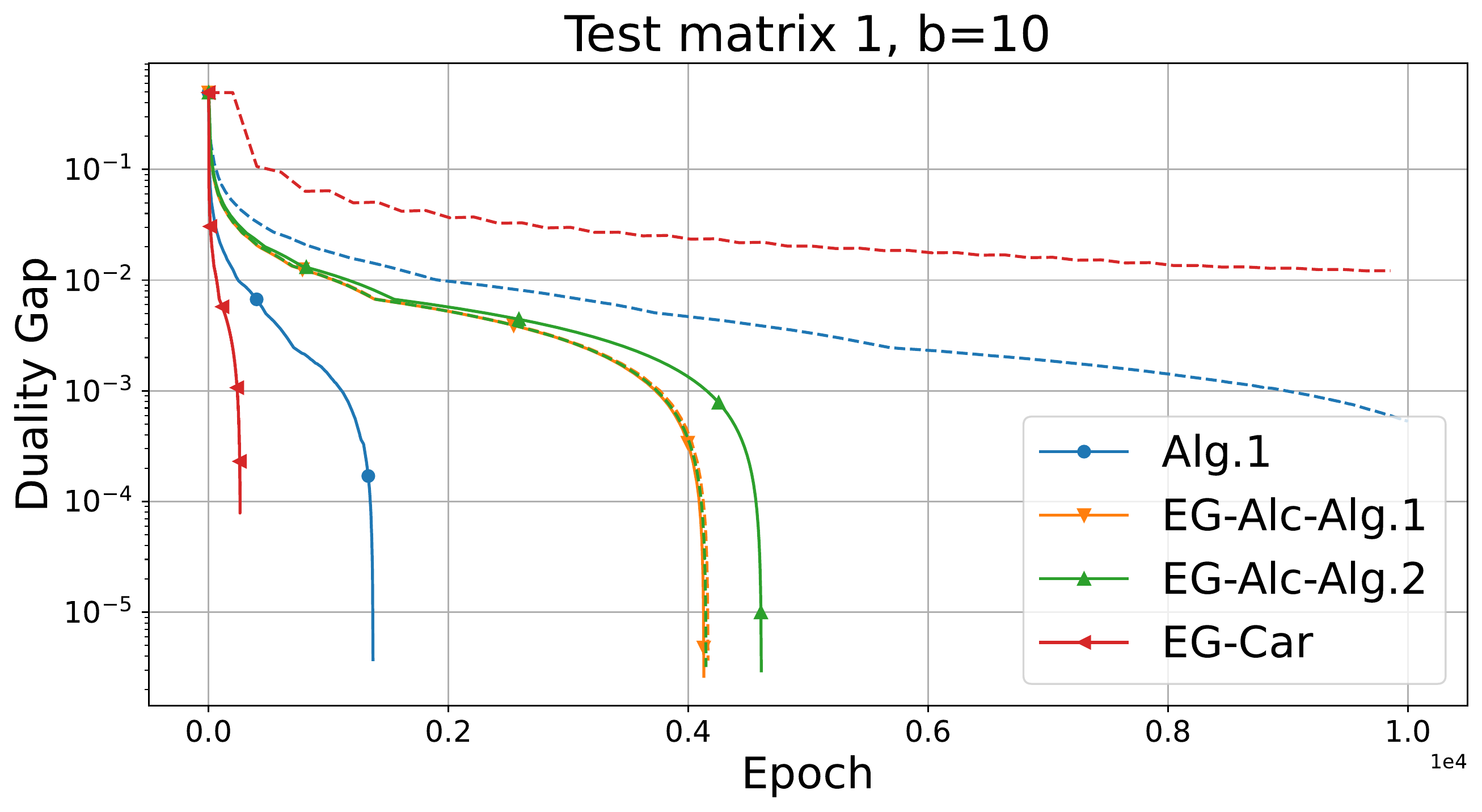}
\\
\includegraphics[width=0.35\textwidth]{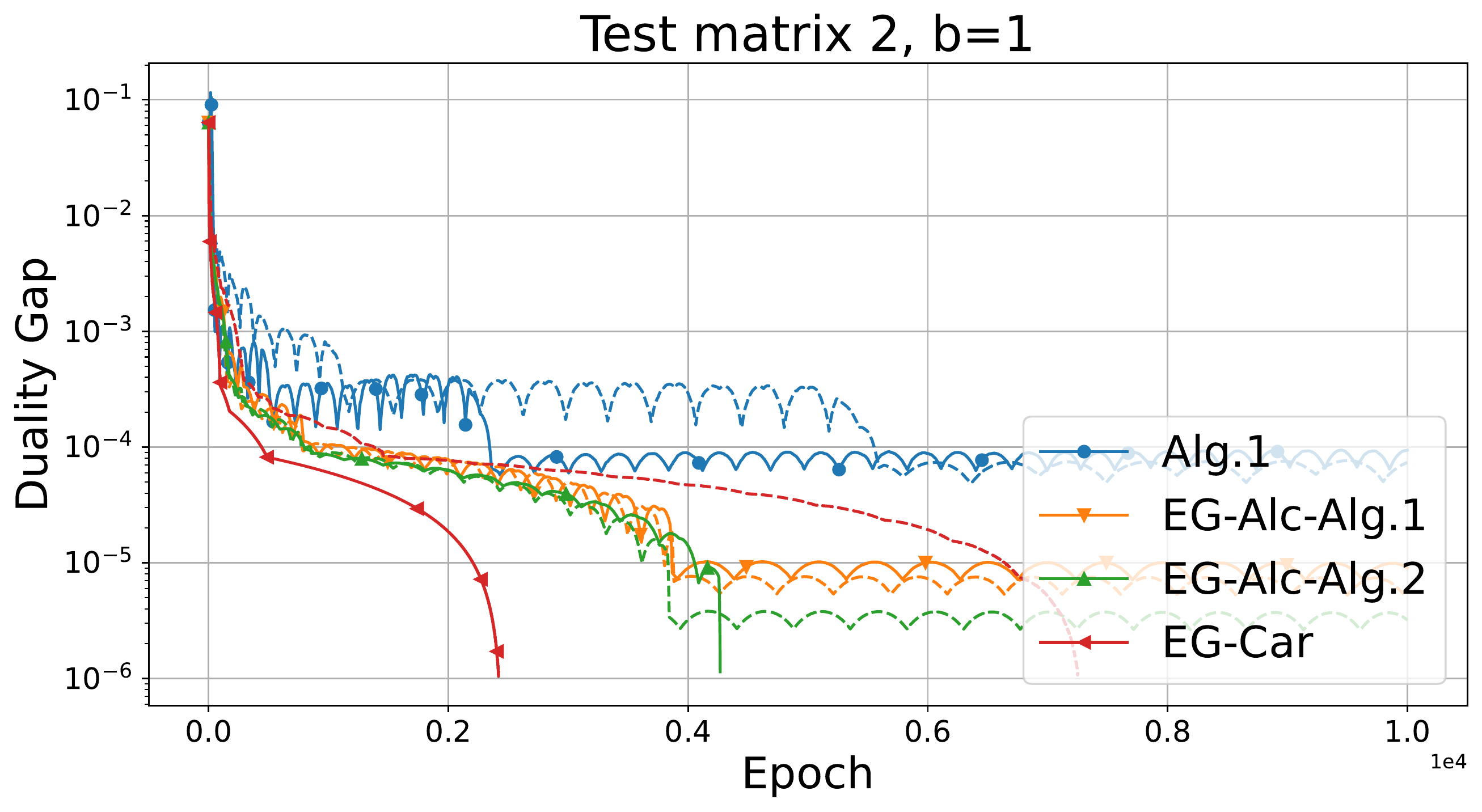}
\includegraphics[width=0.35\textwidth]{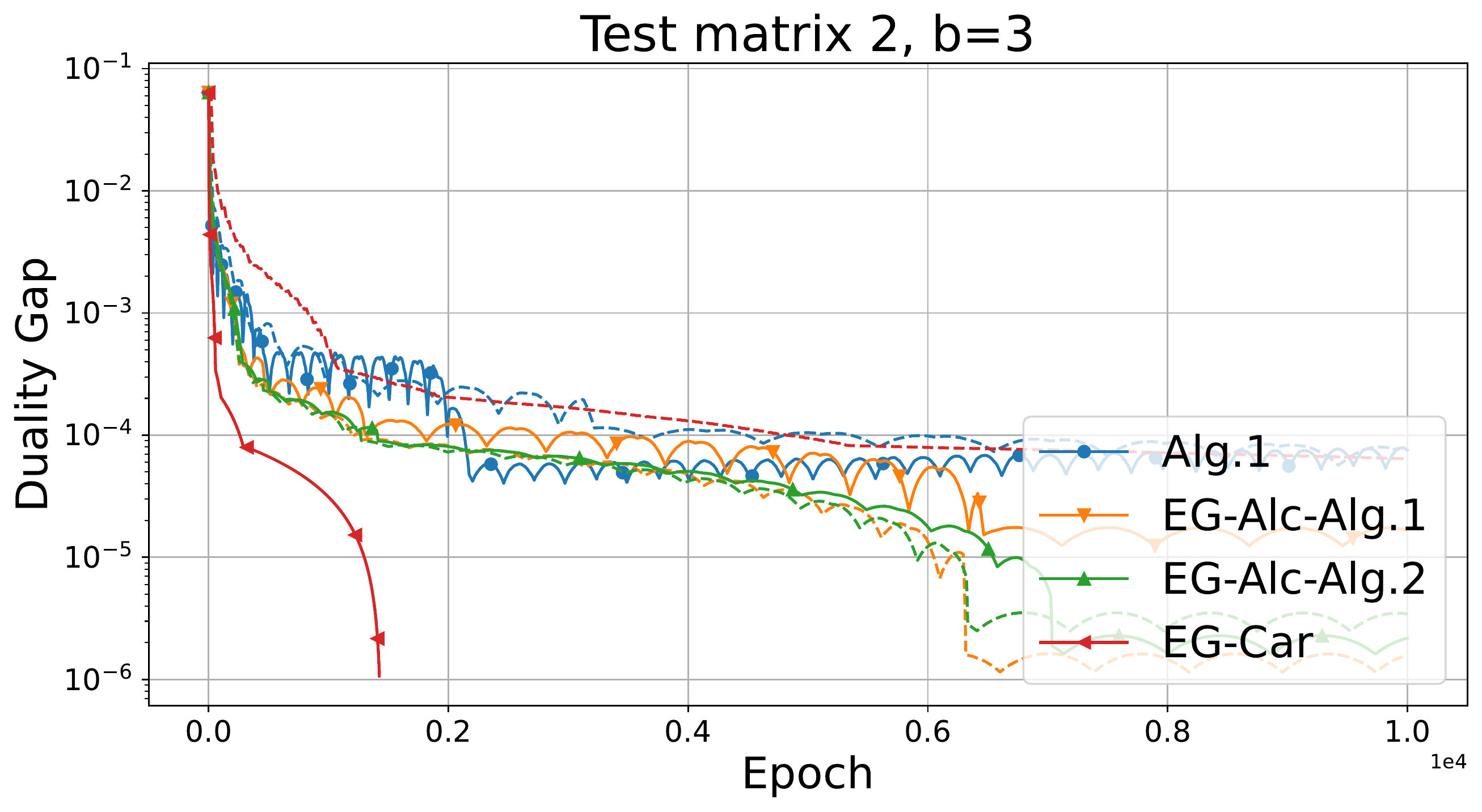}
\\
\includegraphics[width=0.35\textwidth]{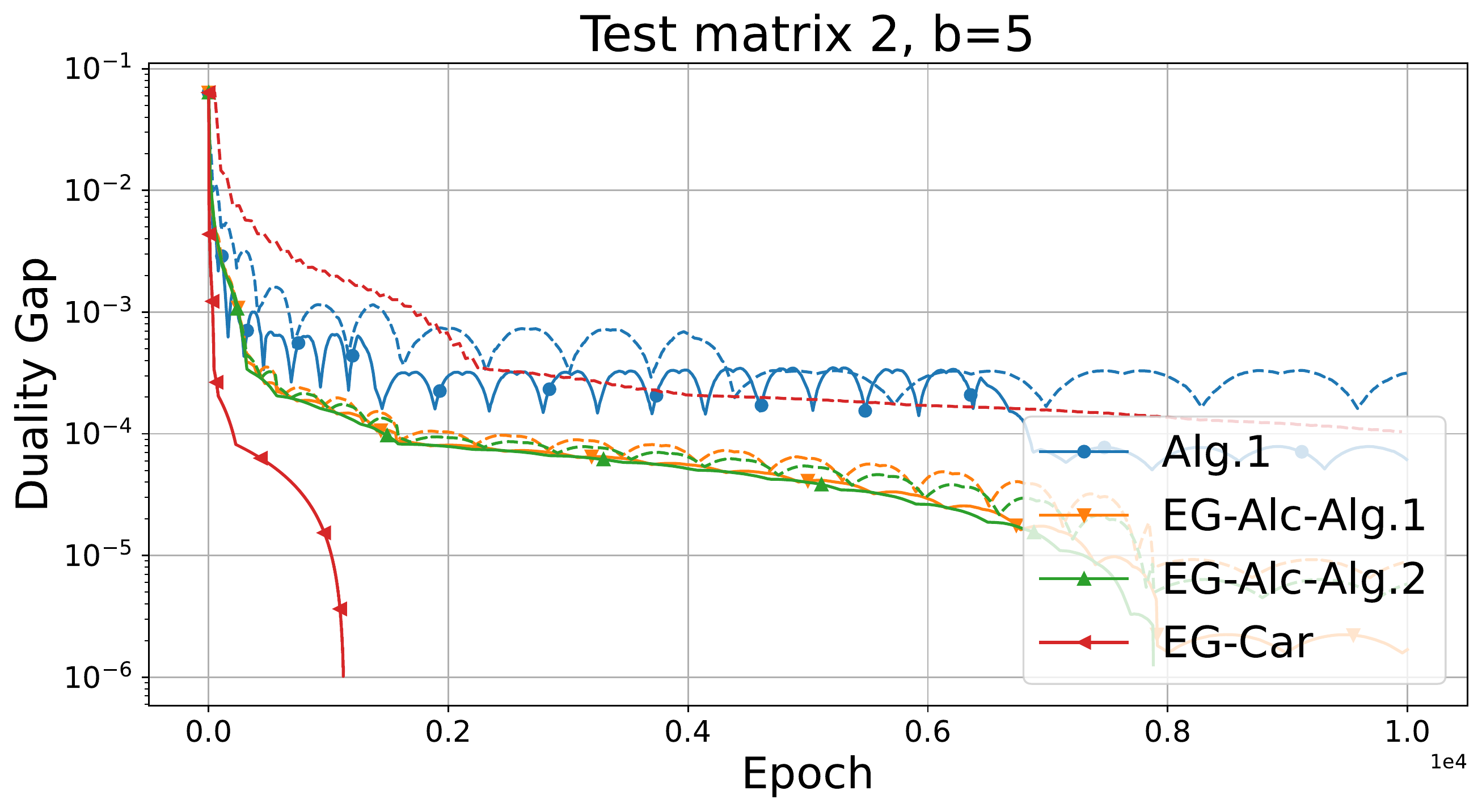}
\includegraphics[width=0.35\textwidth]{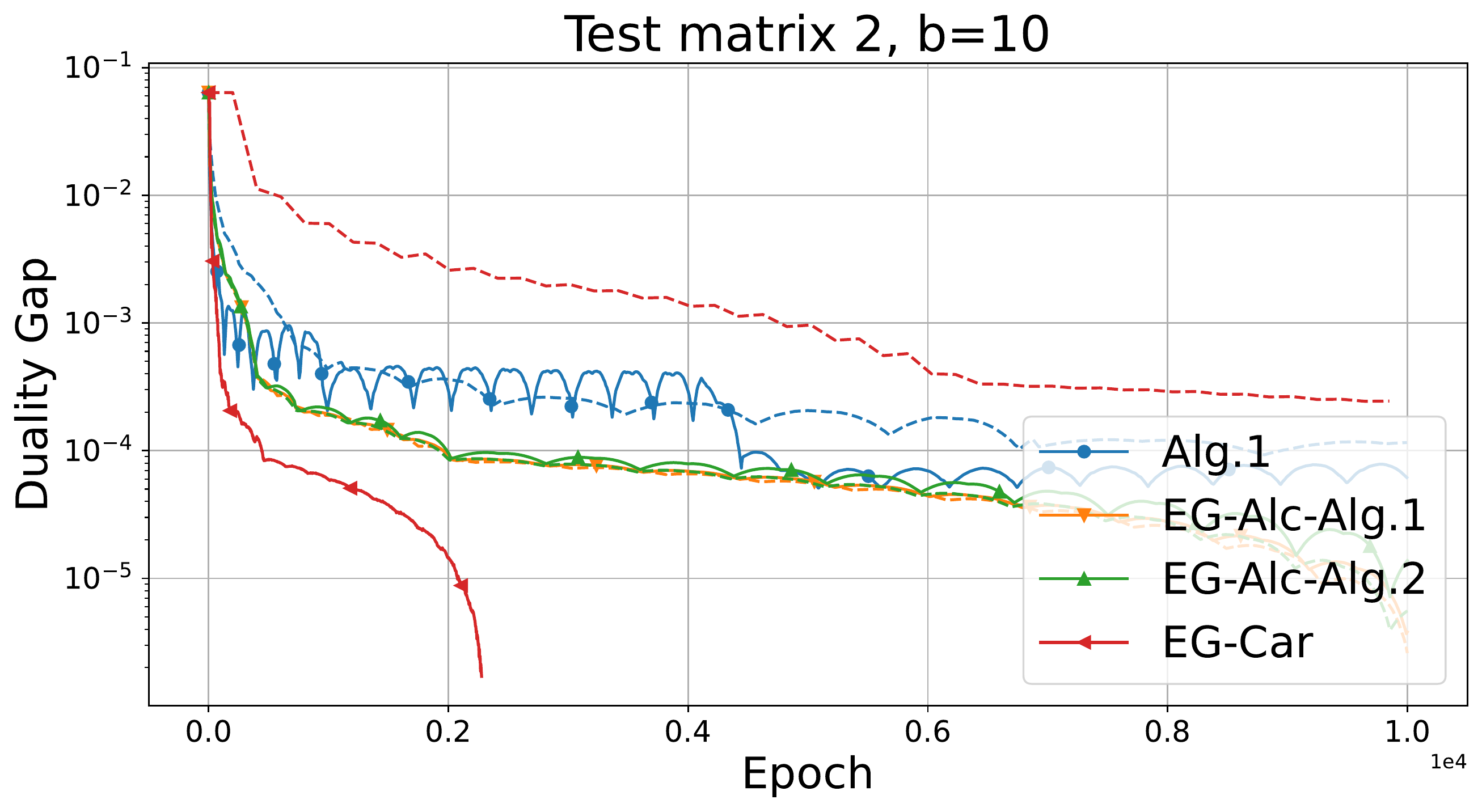}
\end{minipage}
\end{figure}

\subsection{Decentralized methods}

\subsubsection{Fixed networks} \label{sec:app_dec_f_exp}

Here we give additional experiments for Section \ref{sec:dec_f_exp}. 

\begin{figure}[h]
\centering
\captionof{figure}{Comparison communication complexities of Algorithm \ref{alg:vrvi}, \algname{EGD-GT}, \algname{EGD-Con} and \algname{Sliding} on \eqref{eq:regression} with over fixed grid networks.}
\vspace{-0.3cm}
\begin{minipage}[][][b]{\textwidth}
\centering
\includegraphics[width=0.35\textwidth]{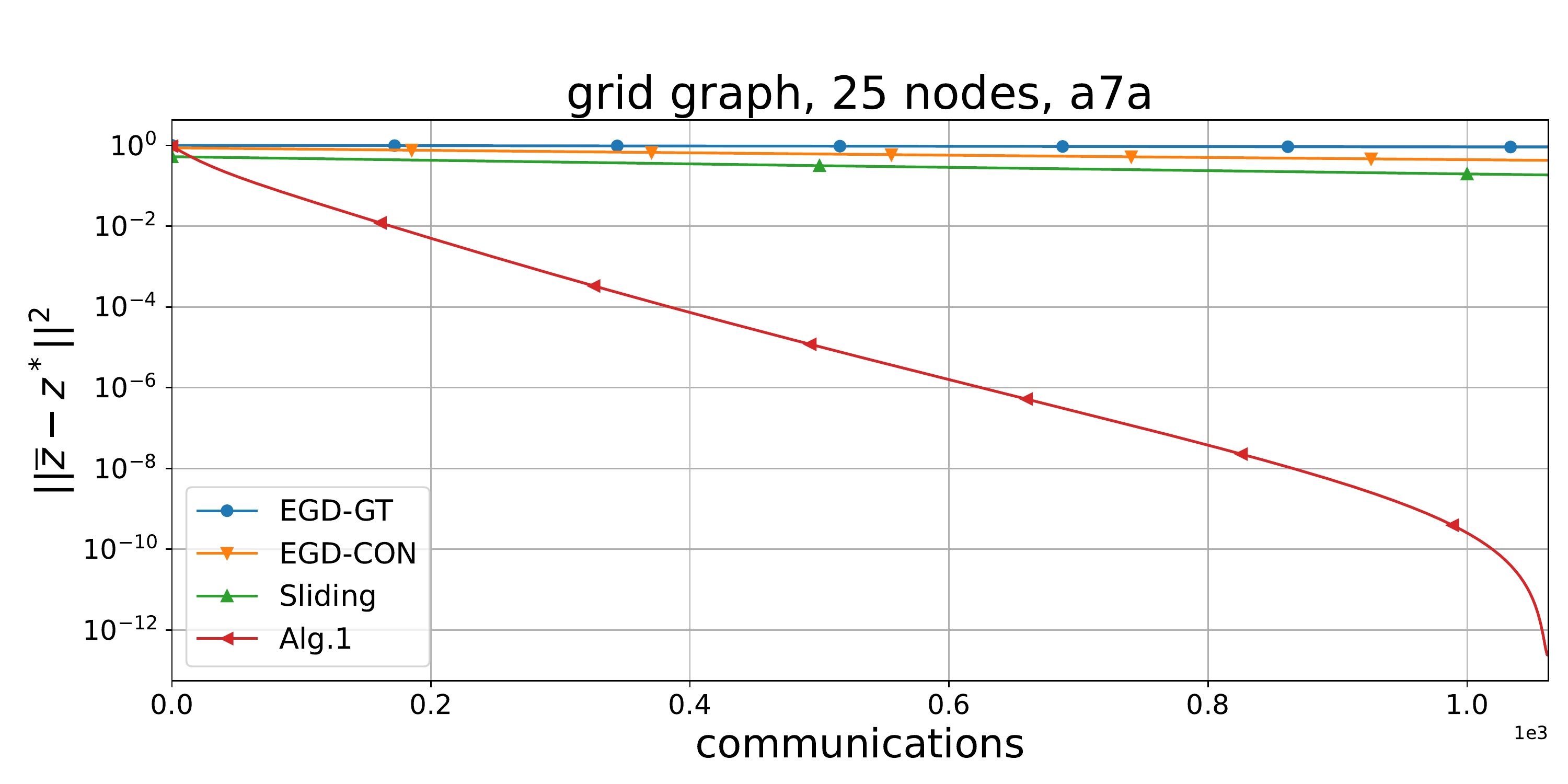}
\includegraphics[width=0.35\textwidth]{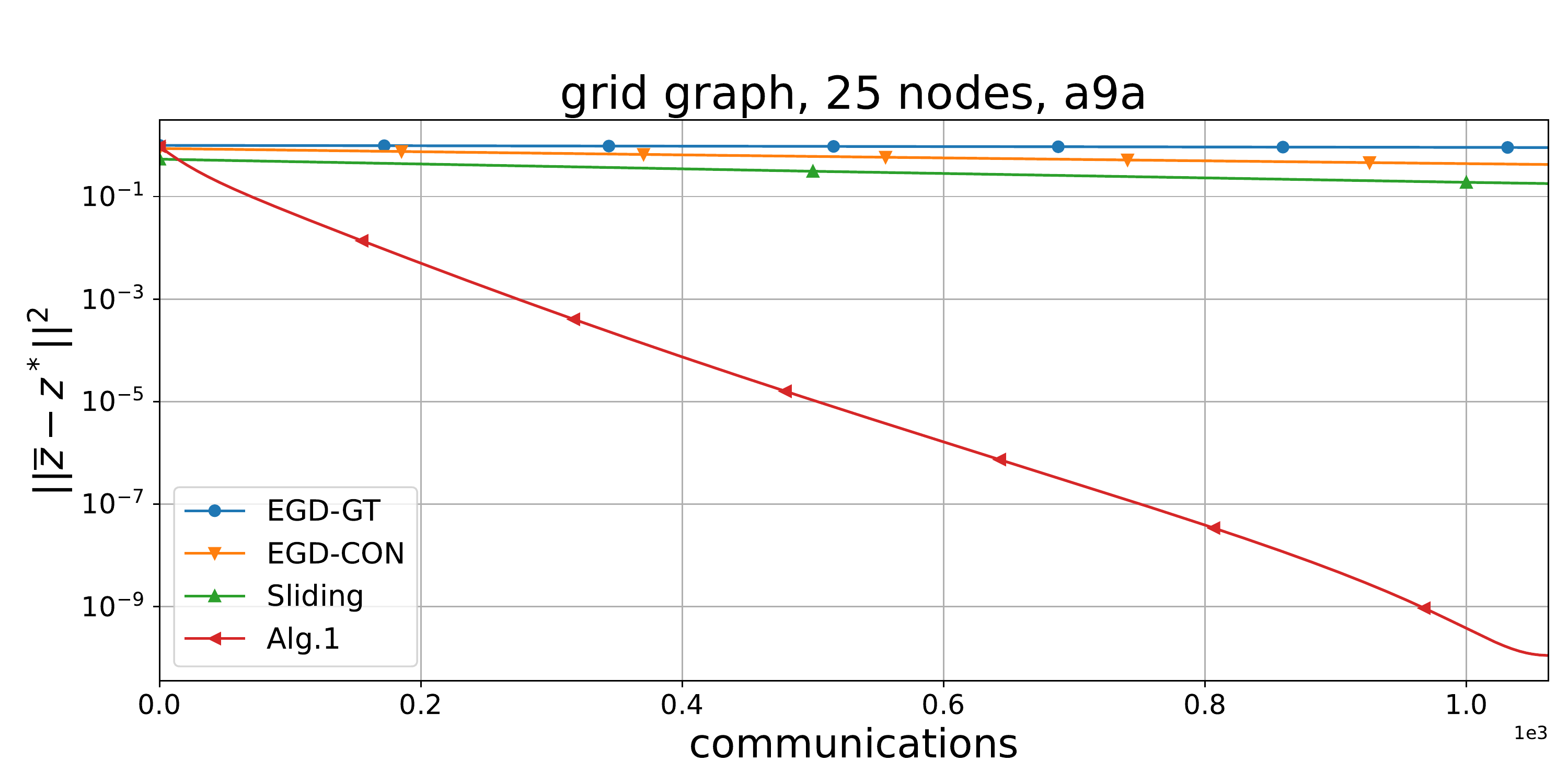}
\\
\includegraphics[width=0.35\textwidth]{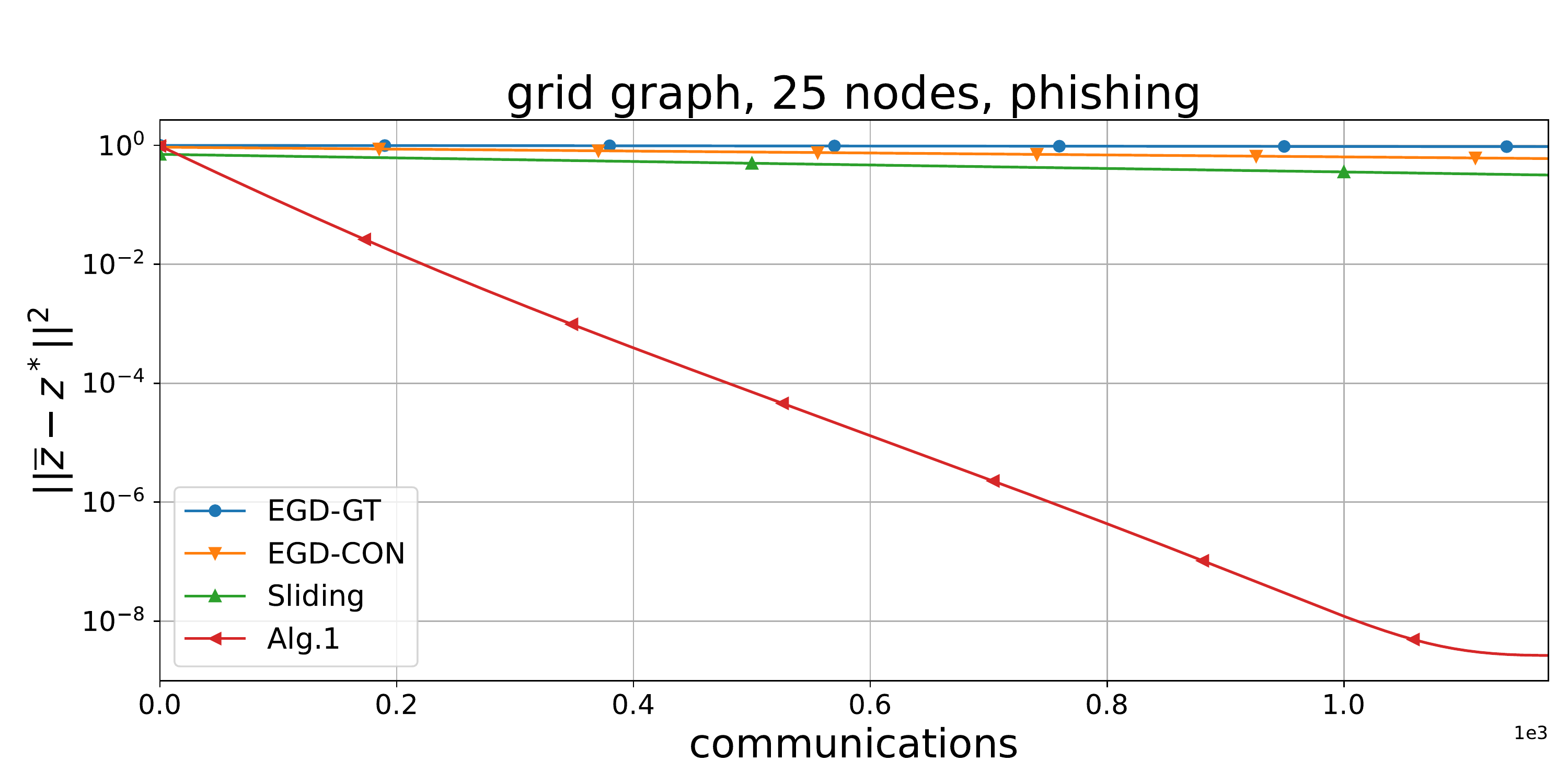}
\end{minipage}
\end{figure}
\newpage
\begin{figure}[h]
\centering
\captionof{figure}{Comparison communication complexities of Algorithm \ref{alg:vrvi}, \algname{EGD-GT}, \algname{EGD-Con} and \algname{Sliding} on \eqref{eq:regression} with over fixed ring networks.}
\vspace{-0.3cm}
\begin{minipage}[][][b]{\textwidth}
\centering
\includegraphics[width=0.35\textwidth]{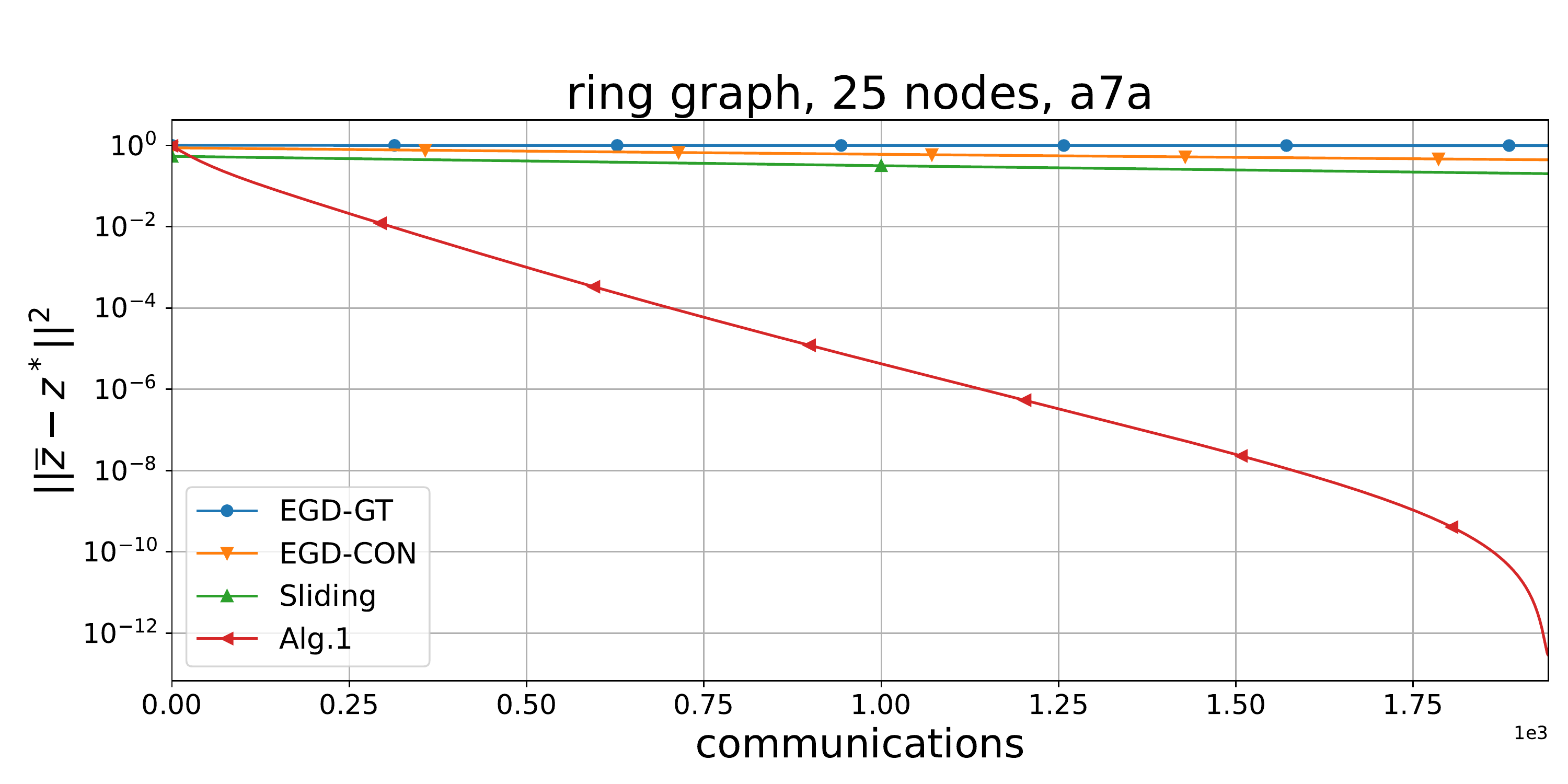}
\includegraphics[width=0.35\textwidth]{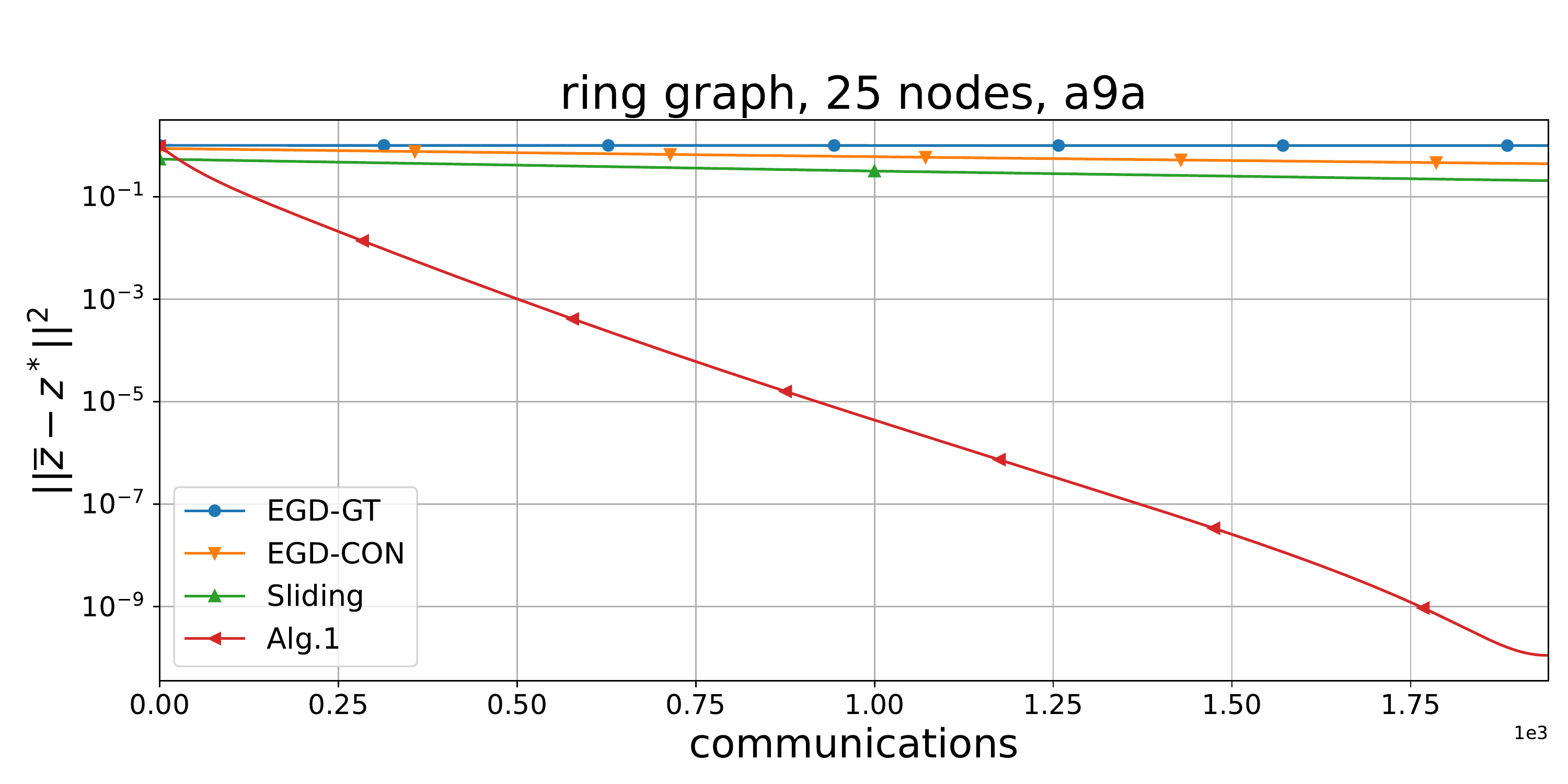}
\\
\includegraphics[width=0.35\textwidth]{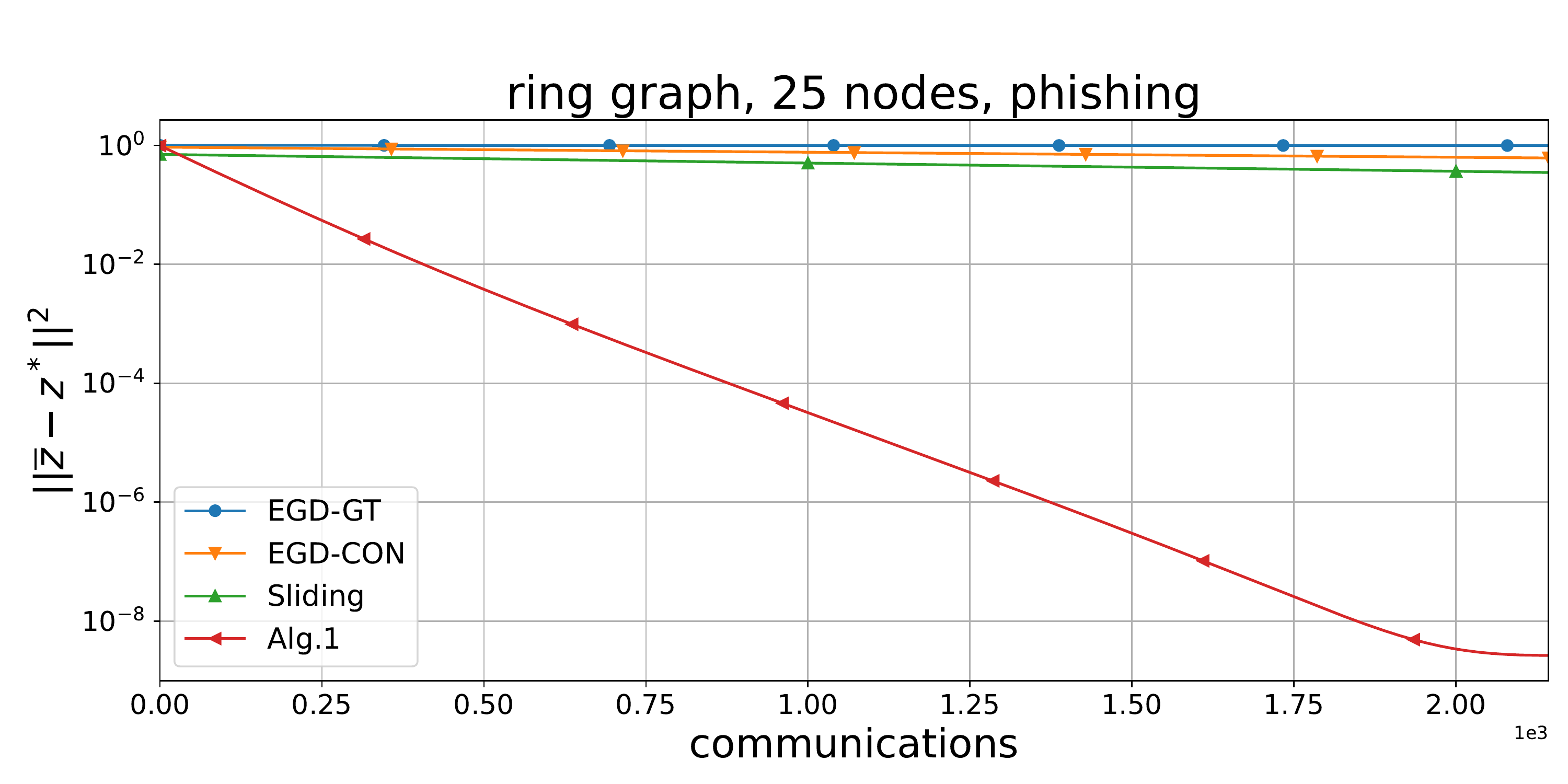}
\end{minipage}
\end{figure}

\begin{figure}[h]
\centering
\captionof{figure}{Comparison communication complexities of Algorithm \ref{alg:vrvi}, \algname{EGD-GT}, \algname{EGD-Con} and \algname{Sliding} on \eqref{eq:regression} with over fixed star networks.}
\vspace{-0.3cm}
\begin{minipage}[][][b]{\textwidth}
\centering
\includegraphics[width=0.35\textwidth]{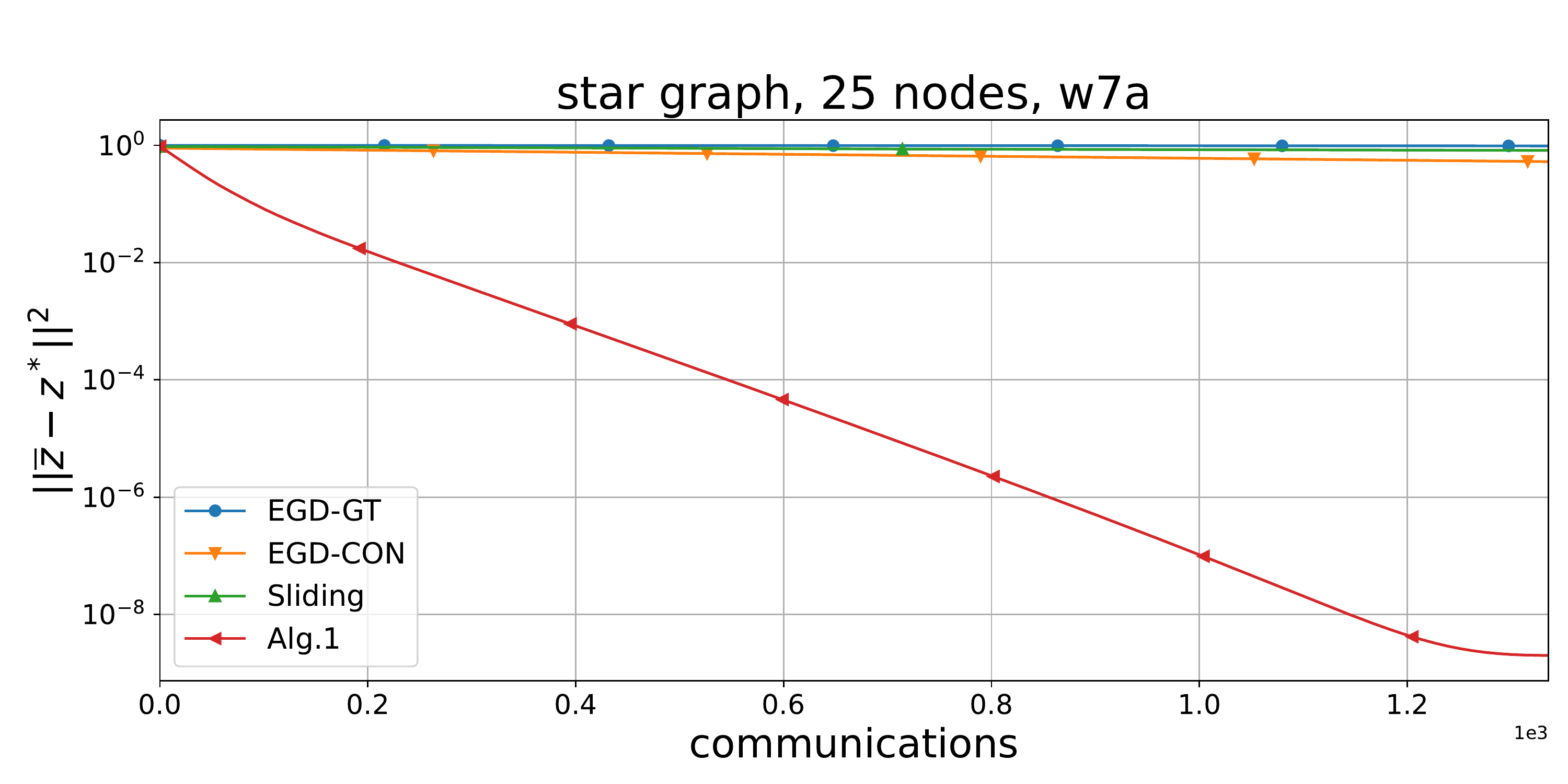}
\includegraphics[width=0.35\textwidth]{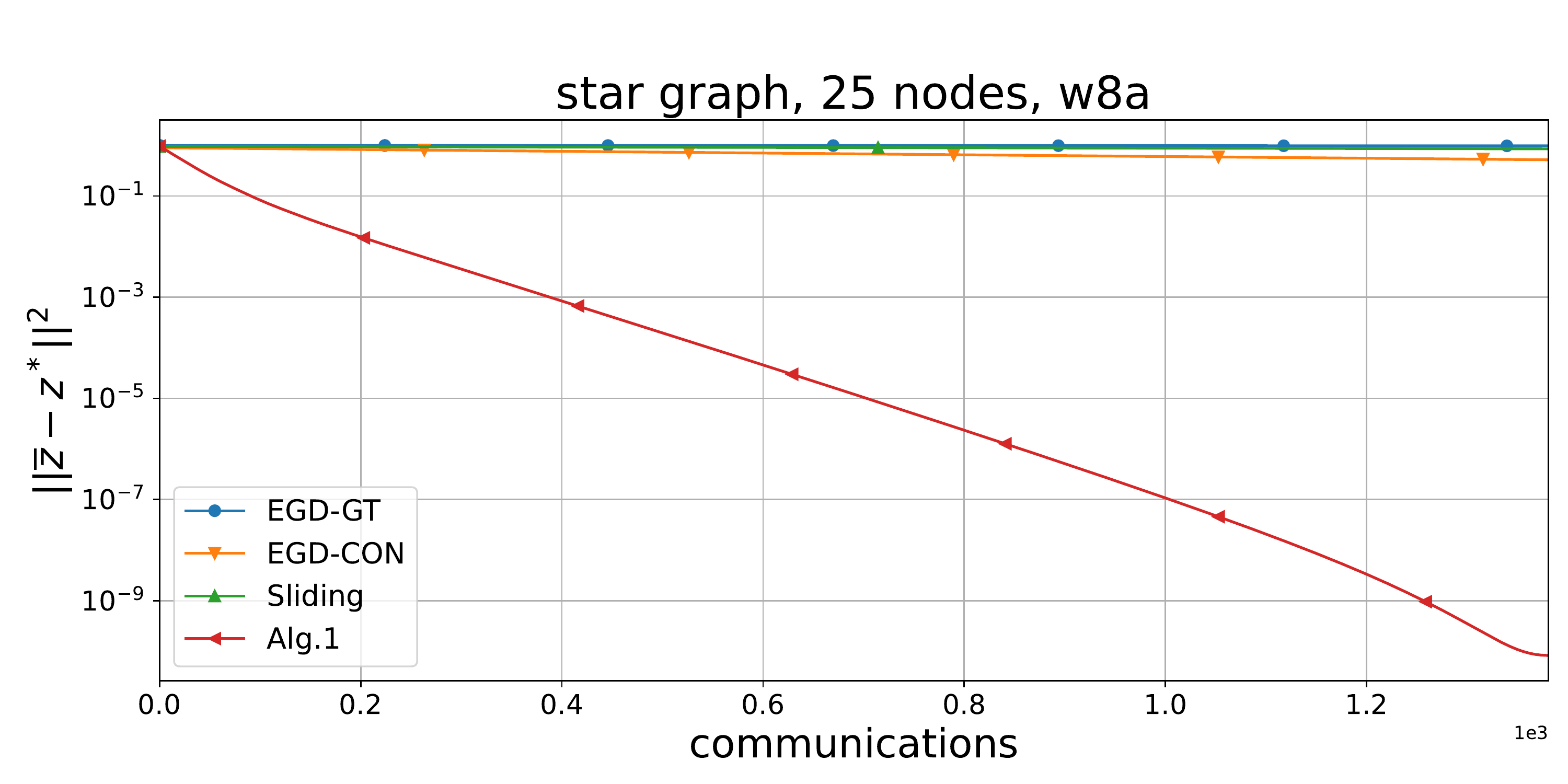}
\\
\includegraphics[width=0.35\textwidth]{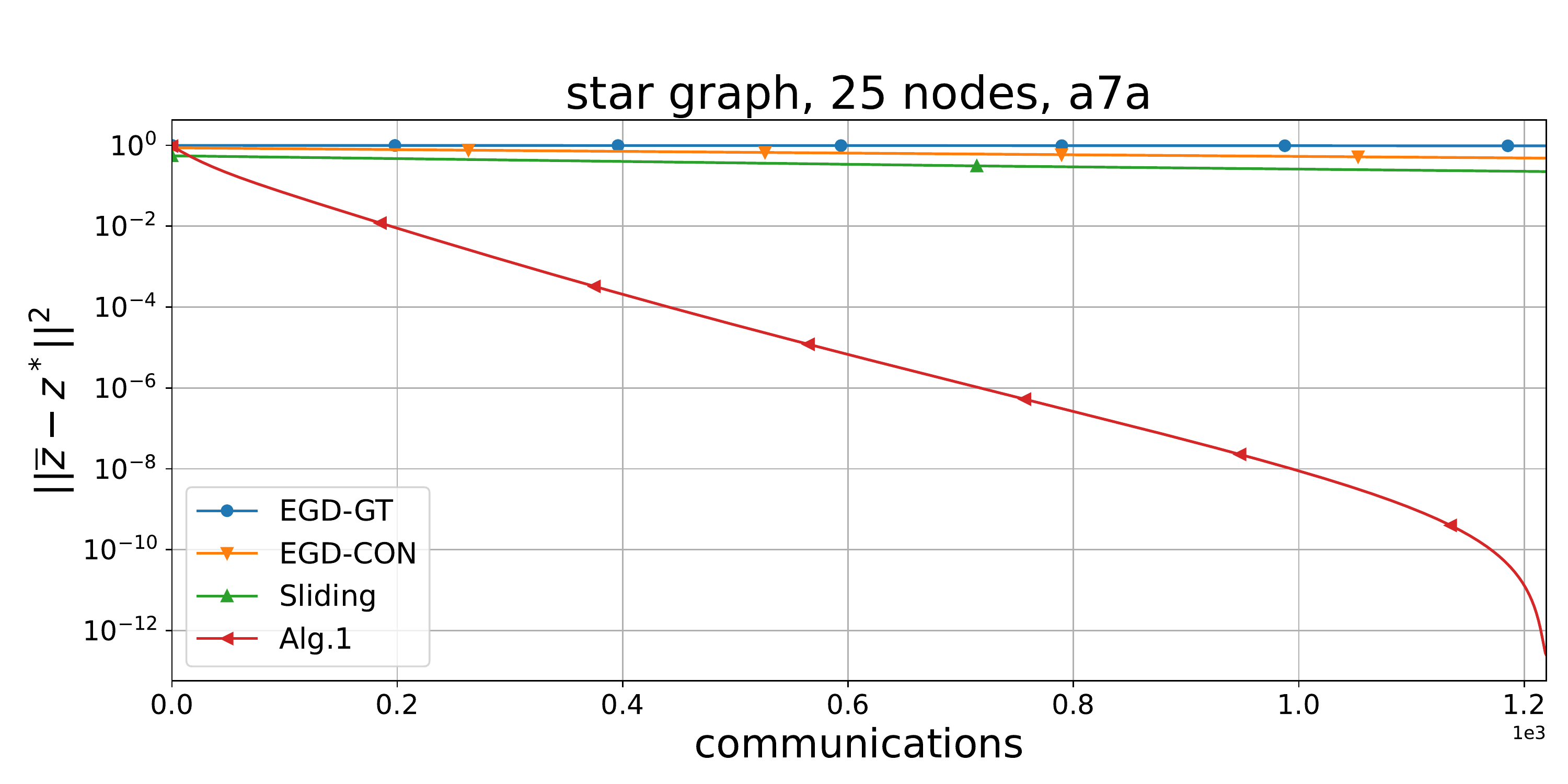}
\includegraphics[width=0.35\textwidth]{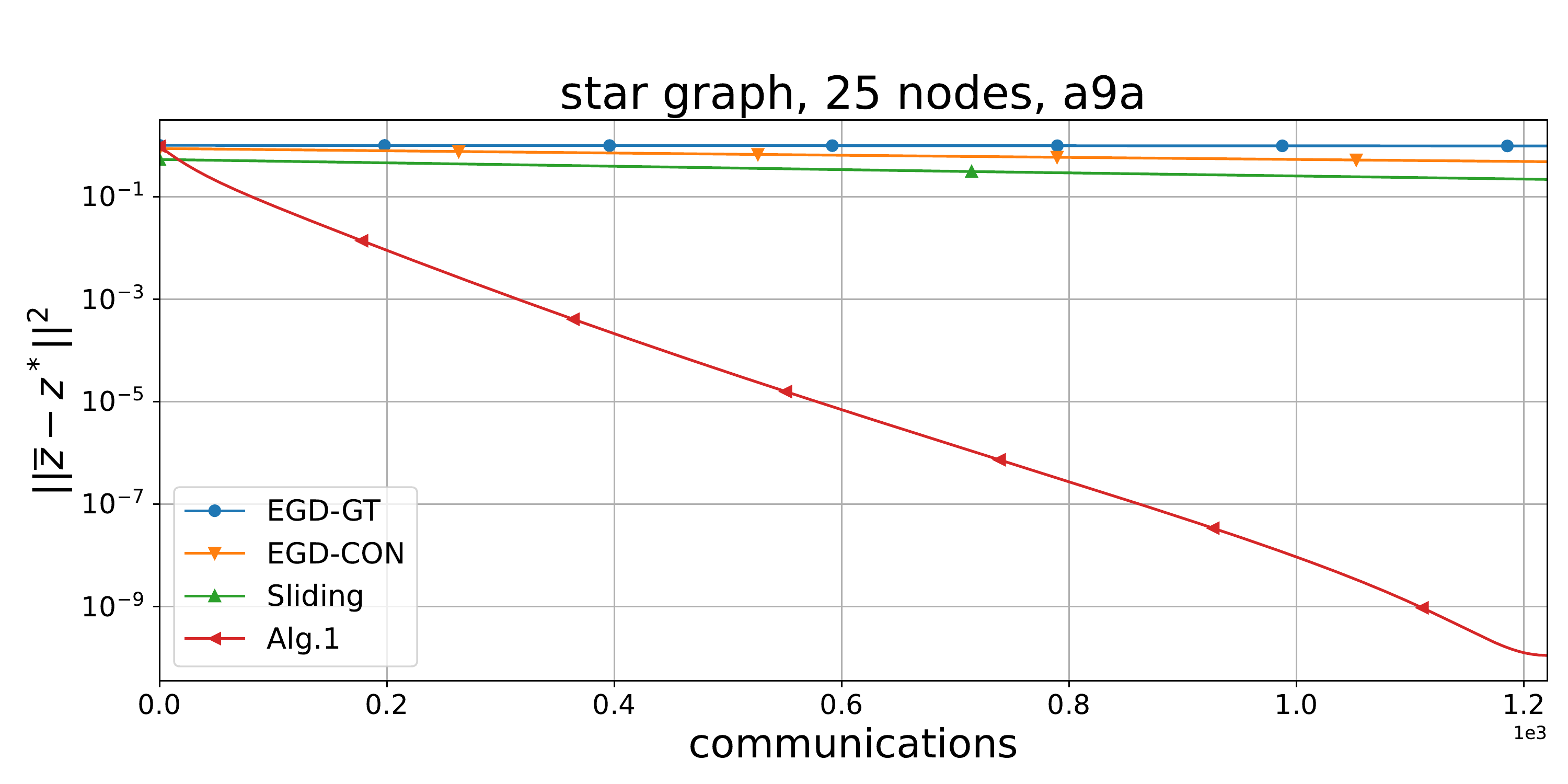}
\\
\includegraphics[width=0.35\textwidth]{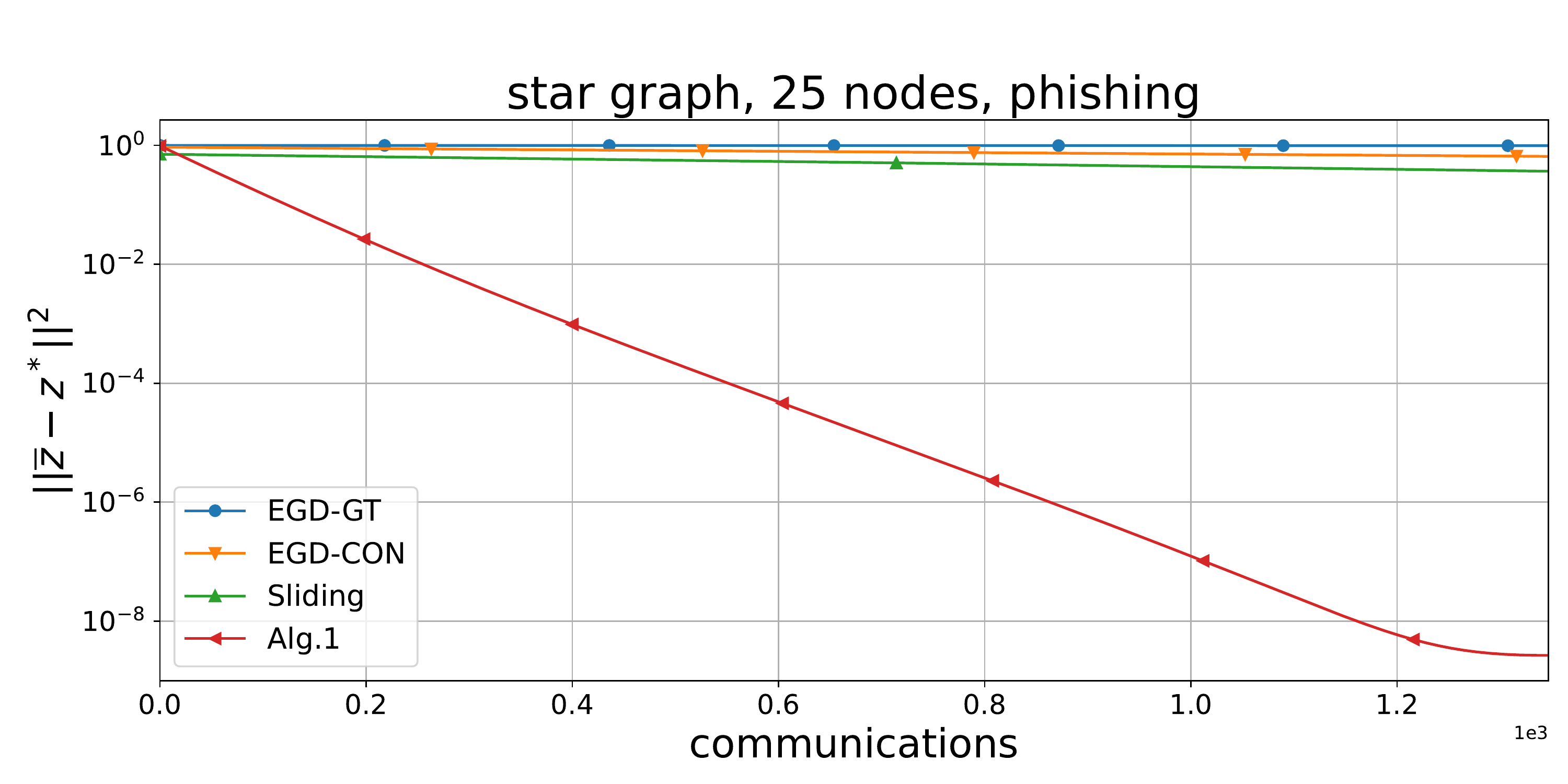}
\end{minipage}
\end{figure}

\newpage

\subsubsection{Time-varying networks} \label{sec:app_dec_tv_exp}

Here we give additional experiments for Section \ref{sec:dec_tv_exp}. 

\begin{figure}[h]
\centering
\captionof{figure}{Comparison communication complexities of Algorithm \ref{2dvi2:alg}, \algname{EGD-GT}, \algname{EGD-Con} and \algname{Sliding} on \eqref{eq:regression} over time-varying grid networks.}
\vspace{-0.3cm}
\begin{minipage}[][][b]{\textwidth}
\centering
\includegraphics[width=0.4\textwidth]{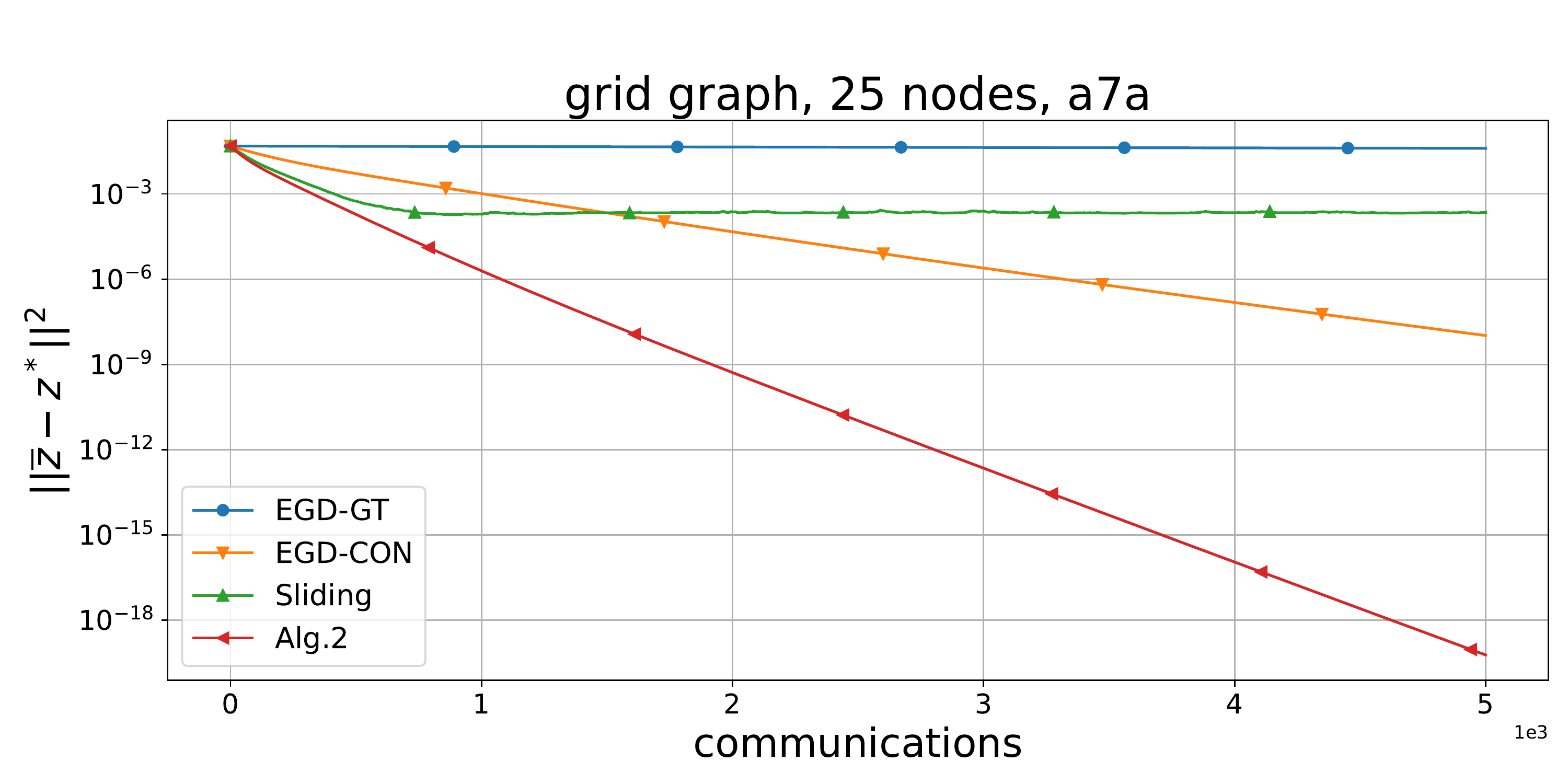}
\includegraphics[width=0.4\textwidth]{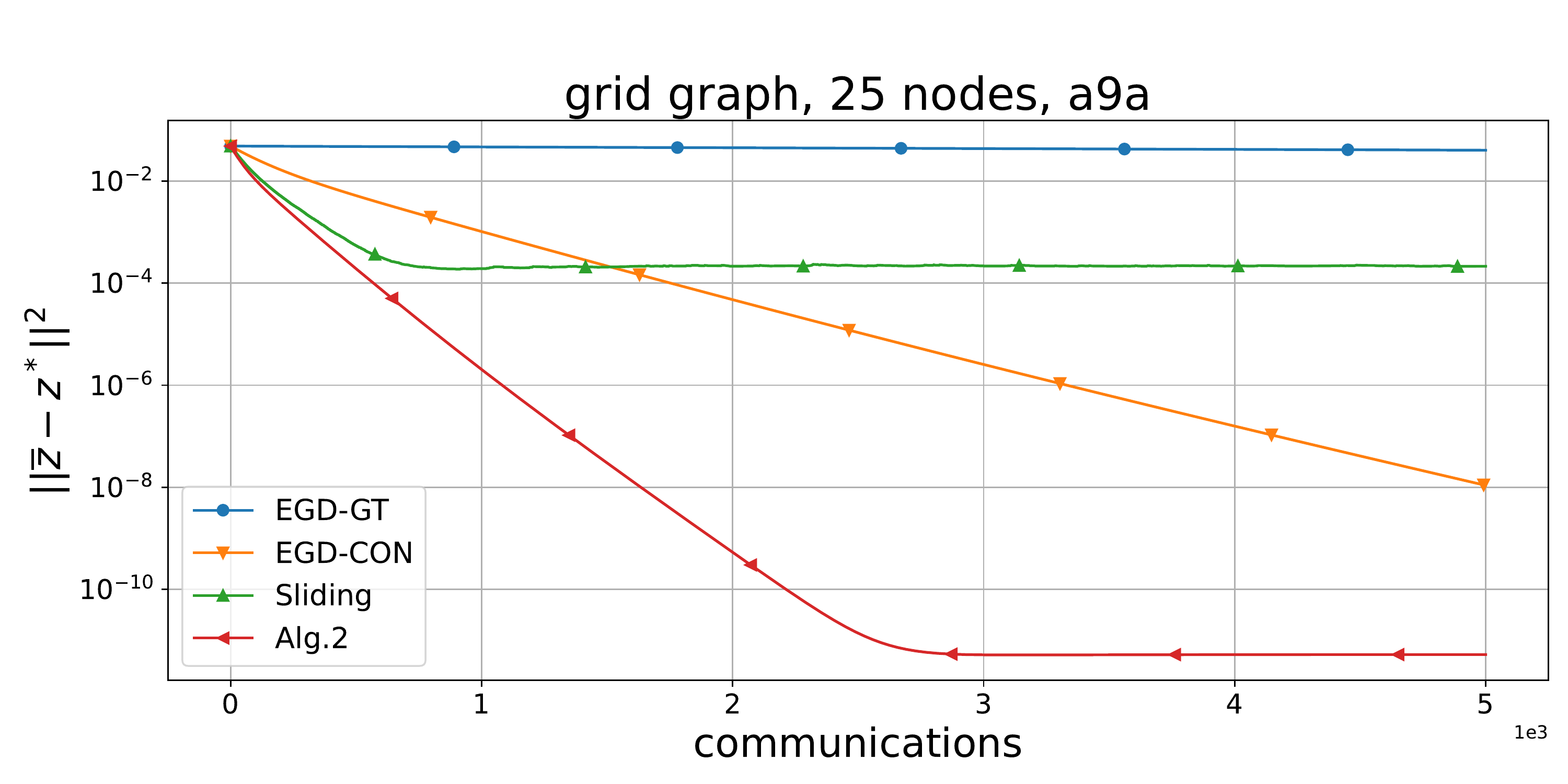}
\\
\includegraphics[width=0.4\textwidth]{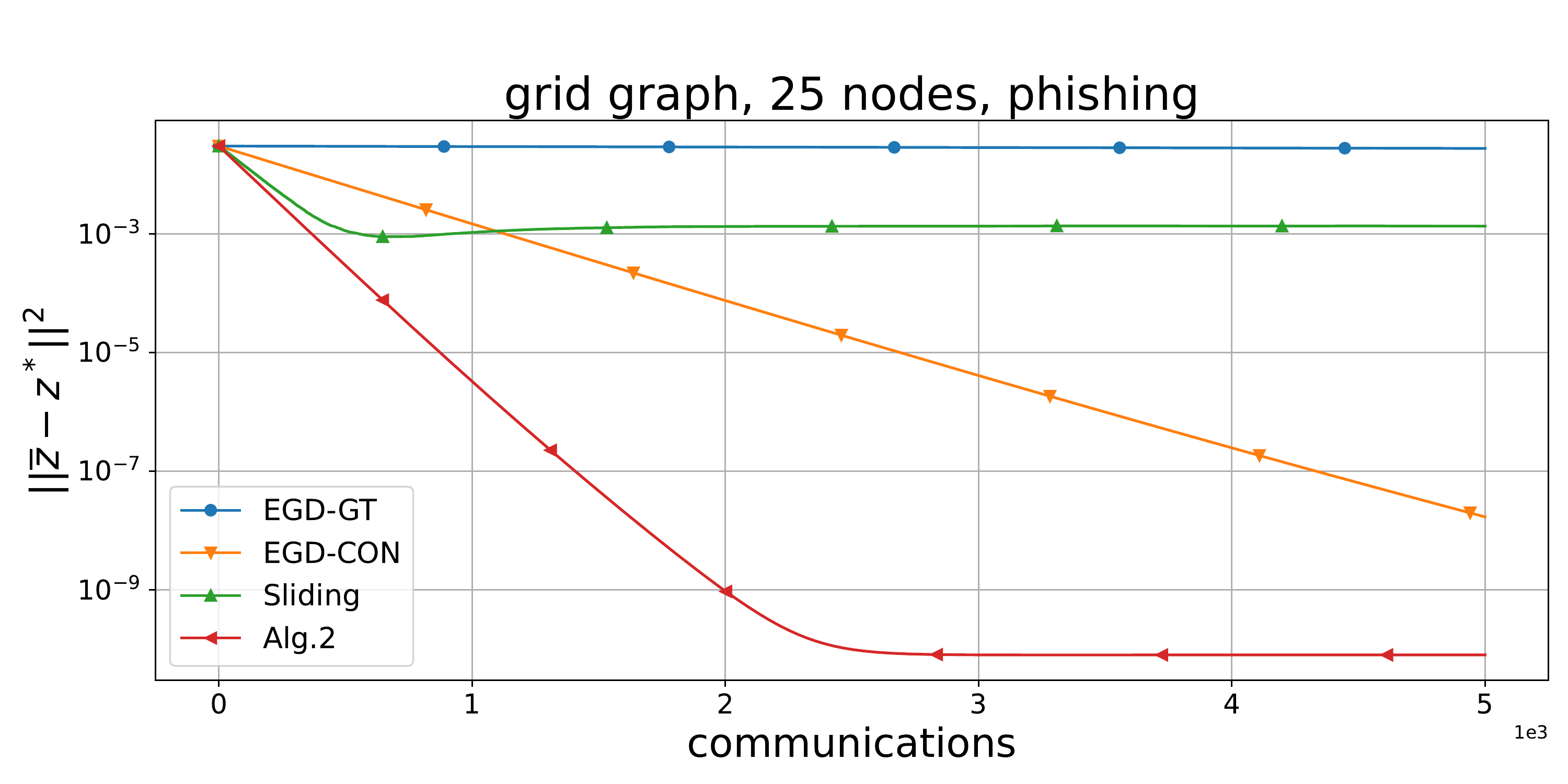}
\end{minipage}
\end{figure}

\begin{figure}[h]
\centering
\captionof{figure}{Comparison communication complexities of Algorithm \ref{alg:vrvi}, \algname{EGD-GT}, \algname{EGD-Con} and \algname{Sliding} on \eqref{eq:regression} with over fixed ring networks.}
\vspace{-0.3cm}
\begin{minipage}[][][b]{\textwidth}
\centering
\includegraphics[width=0.35\textwidth]{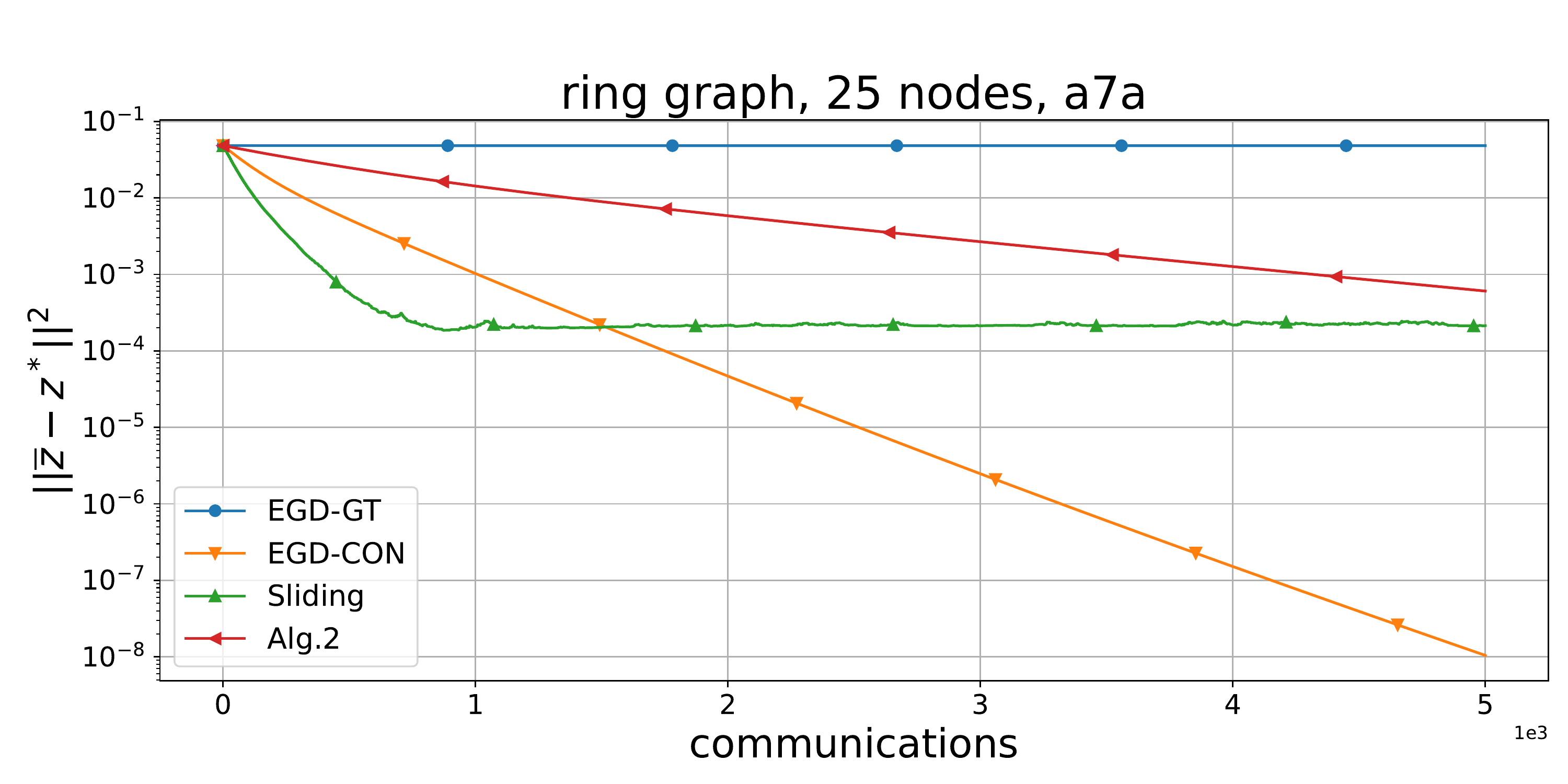}
\includegraphics[width=0.35\textwidth]{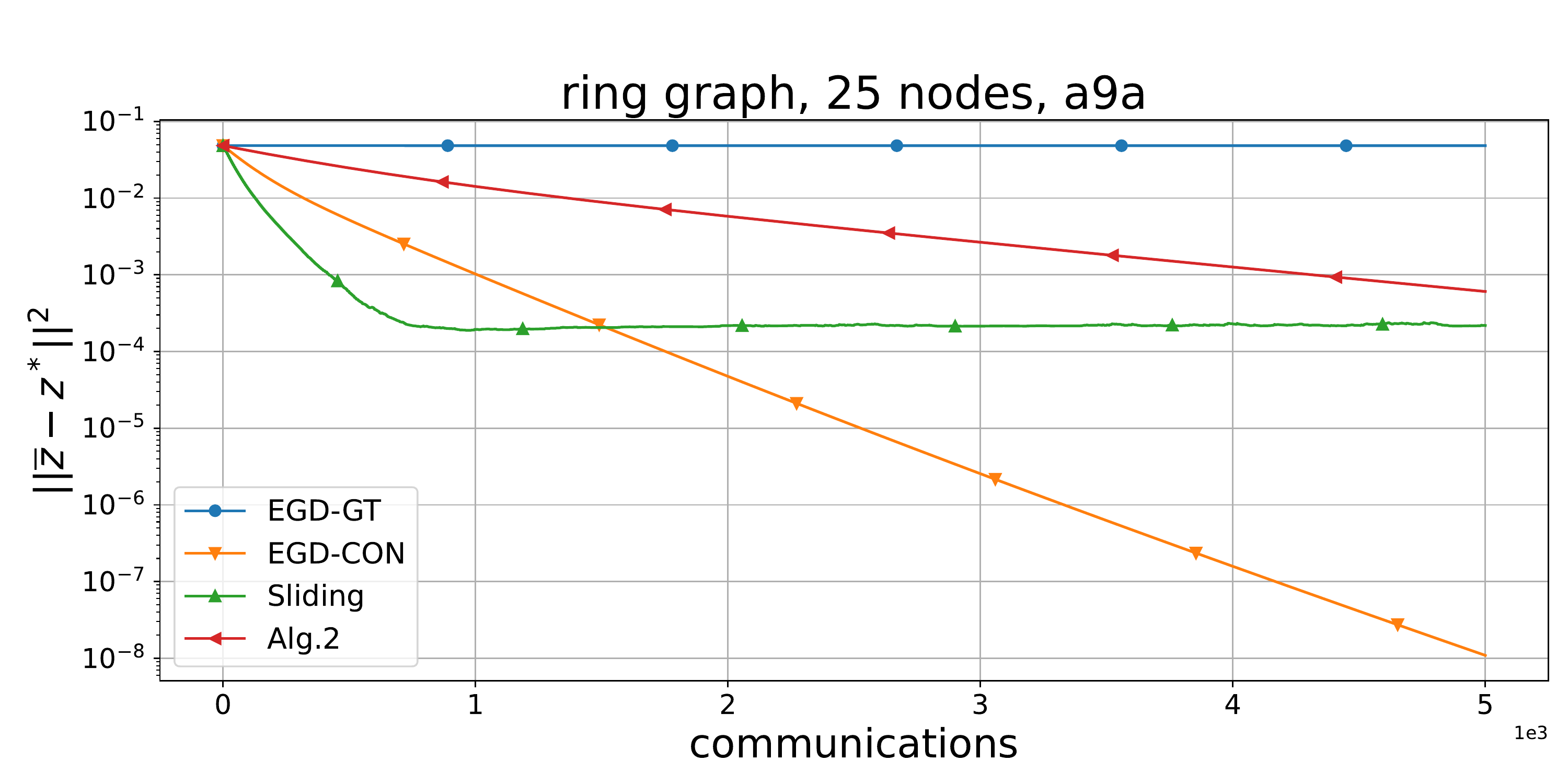}
\\
\includegraphics[width=0.35\textwidth]{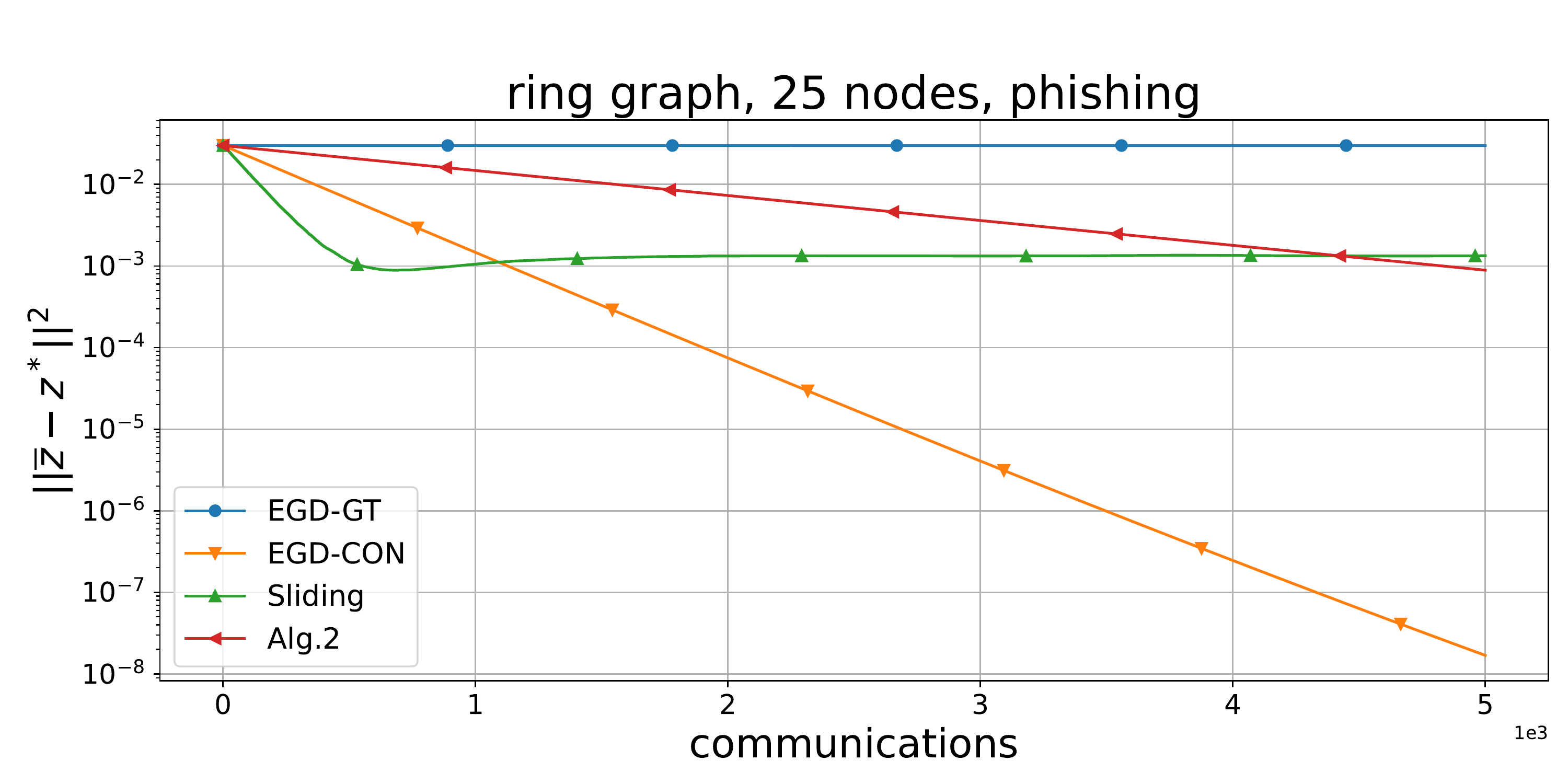}
\end{minipage}
\end{figure}

\newpage

\section{Tables} \label{sec:tables}

\renewcommand{\arraystretch}{2}
\begin{table*}[h]
    \centering
    \small
	\caption{Summary communication and local complexities for finding an $\varepsilon$-solution for strongly monotone \textbf{deterministic} \textbf{decentralized} variational inequality \eqref{eq:VI} on fixed and time-varying networks. Convergence is measured by the distance to the solution. {\em Notation:} $\mu$ = constant of strong monotonicity of the operator $F$, $L$ = Lipschitz constants for all $L_{m,i}$, $\chi$ = characteristic number of the network (see Assumptions \ref{ass:fixed} and \ref{ass:tv}). }
    \label{tab:comparison1}
    \resizebox{\linewidth}{!}{
  \begin{threeparttable}
    \begin{tabular}{|c|c|c|c|c|c|}
    \cline{3-6}
    \multicolumn{2}{c|}{}
     & \textbf{\quad\quad\quad\quad\quad\quad Reference \quad\quad\quad\quad\quad\quad} & \textbf{Communication complexity} & \textbf{Local complexity}  & \textbf{Weaknesses} \\
    \hline
    \multirow{8}{*}{\rotatebox[origin=c]{90}{\textbf{Fixed \quad\quad}}} & \multirow{6}{*}{\rotatebox[origin=c]{90}{\textbf{Upper}\quad\quad}}
     & Mukherjee and Chakraborty \cite{Mukherjee2020:decentralizedminmax} \tnote{{\color{blue}(1,2)}}  & $\mathcal{O} \left( \myred{\chi^{\frac{4}{3}}} \frac{L^{\myred{\frac{4}{3}}}}{\mu^{\myred{\frac{4}{3}}}} \log \frac{1}{\varepsilon} \right)$ &  $\mathcal{O} \left( \myred{\chi^{\frac{4}{3}}} \frac{L^{\myred{\frac{4}{3}}}}{\mu^{\myred{\frac{4}{3}}}} \log \frac{1}{\varepsilon} \right)$ & \makecell{{bad comm. rates} \\ {bad local rates}}
    \\ \cline{3-6}
    && Beznosikov et al. \cite{beznosikov2021distributed}  \tnote{{\color{blue}(1,2)}}  & $\mathcal{O} \left( \sqrt{\chi} \frac{L}{\mu} \log^{\myred{2}} \frac{1}{\varepsilon} \right)$ &  $\mathcal{O} \left(  \frac{L}{\mu} \myred{\log \frac{L+\mu}{\mu}} \log \frac{1}{\varepsilon}\right)$ & \makecell{{multiple gossip} \\ {no linear convergence}}
    \\ \cline{3-6}
    && Beznosikov et al. \cite{beznosikov2020distributed}  \tnote{{\color{blue}(1,3)}} & $\mathcal{O} \left( \sqrt{\chi} \frac{L}{\mu} \log^{\myred{2}} \frac{1}{\varepsilon} \right)$ &  $\mathcal{O} \left(  \frac{L}{\mu}\log \frac{1}{\varepsilon}\right)$ & \makecell{{multiple gossip} \\ {no linear convergence}}
    \\ \cline{3-6}
    && Rogozin et al. \cite{rogozin2021decentralized}  \tnote{{\color{blue}(1,2,4)}} & $\mathcal{O} \left( \sqrt{\chi} \frac{L}{\mu} \log \frac{1}{\varepsilon} \right)$ &  $\mathcal{O} \left( \myred{\sqrt{\chi}} \frac{L}{\mu}\log \frac{1}{\varepsilon}\right)$ & \makecell{{multiple gossip} \\ {bounded gradient}}
    \\ \cline{3-6}
    && \cellcolor{bgcolor2}{This paper}  & \cellcolor{bgcolor2}{$\mathcal{O} \left( \sqrt{\chi}\frac{L}{\mu} \log \frac{1}{\varepsilon} \right)$}  &  \cellcolor{bgcolor2}{$\mathcal{O} \left( \sqrt{\chi}\frac{L}{\mu} \log \frac{1}{\varepsilon} \right)$} & \cellcolor{bgcolor2}{}
    \\ \cline{3-6}
    && \cellcolor{bgcolor2}{This paper}  & \cellcolor{bgcolor2}{$\mathcal{O} \left( \sqrt{\chi} \frac{L}{\mu} \log \frac{1}{\varepsilon} \right)$}  &  \cellcolor{bgcolor2}{$\mathcal{O} \left( \frac{L}{\mu} \log \frac{1}{\varepsilon} \right)$} & \cellcolor{bgcolor2}{multiple gossip}
    \\ \cline{2-6} 
    & \multirow{2}{*}{\rotatebox[origin=c]{90}{\textbf{Lower}}} & Beznosikov et al. \cite{beznosikov2020distributed}  \tnote{{\color{blue}(3)}} & $\mathcal{O} \left( \sqrt{\chi} \frac{L}{\mu} \log \frac{1}{\varepsilon} \right)$  &  $\mathcal{O} \left( \frac{L}{\mu} \log \frac{1}{\varepsilon} \right)$ & 
    \\ \cline{3-6} 
    && \cellcolor{bgcolor2}{This paper} & \cellcolor{bgcolor2}{$\mathcal{O} \left( \sqrt{\chi} \frac{L}{\mu} \log \frac{1}{\varepsilon} \right)$}  &  \cellcolor{bgcolor2}{$\mathcal{O} \left( \frac{L}{\mu} \log \frac{1}{\varepsilon} \right)$} & \cellcolor{bgcolor2}{}
    \\\hline\hline 
    \multirow{6}{*}{\rotatebox[origin=c]{90}{\textbf{Time-varying}\quad}} & \multirow{4}{*}{\rotatebox[origin=c]{90}{\textbf{Upper}\quad \quad}} 
    & Beznosikov et al. \cite{beznosikov2021decentralized} \tnote{{\color{blue}(3)}} & $\mathcal{O} \left( \chi \frac{L}{\mu} \log \frac{1}{\varepsilon} + \myred{\chi \frac{L D }{\mu^2   \sqrt{\varepsilon}}} \right)$ \tnote{{\color{blue}(5)}} & $\mathcal{O} \left( \myred{\chi} \frac{L}{\mu} \log \frac{1}{\varepsilon} + \myred{\chi \frac{L D }{\mu^2   \sqrt{\varepsilon}} }\right)$ & \makecell{{$D$-homogeneity}  \\ {no linear convergence}}
    \\ \cline{3-6}
    && Beznosikov et al. \cite{beznosikov2021}  \tnote{{\color{blue}(1,2)}} & $\mathcal{O} \left( \chi \frac{L}{\mu} \log^{\myred{2}} \frac{1}{\varepsilon} \right)$  &  $\mathcal{O} \left( \frac{L}{\mu} \log \frac{1}{\varepsilon} \right)$ & \makecell{{multiple gossip} \\ {no linear convergence}}
    \\ \cline{3-6}
    && \cellcolor{bgcolor2}{This paper}  & \cellcolor{bgcolor2}{$\mathcal{O} \left( \chi \frac{L}{\mu} \log \frac{1}{\varepsilon} \right)$ \tnote{{\color{blue}(5)}}}  &  \cellcolor{bgcolor2}{$\mathcal{O} \left( \chi\frac{L}{\mu} \log \frac{1}{\varepsilon} \right)$}  & \cellcolor{bgcolor2}{}
    \\ \cline{3-6}
    && \cellcolor{bgcolor2}{This paper}  & \cellcolor{bgcolor2}{$\mathcal{O} \left( \chi \frac{L}{\mu} \log \frac{1}{\varepsilon} \right)$ \tnote{{\color{blue}(5)}}}  &  \cellcolor{bgcolor2}{$\mathcal{O} \left( \frac{L}{\mu} \log \frac{1}{\varepsilon} \right)$} & \cellcolor{bgcolor2}{multiple gossip}
    \\ \cline{2-6}
    & \multirow{2}{*}{\rotatebox[origin=c]{90}{\textbf{Lower}}} 
    & Beznosikov et al. \cite{beznosikov2021} \tnote{{\color{blue}(2)}}  & $\mathcal{O} \left( \chi \frac{L}{\mu} \log \frac{1}{\varepsilon} \right)$  &  $\mathcal{O} \left( \frac{L}{\mu} \log \frac{1}{\varepsilon} \right)$ & 
    \\ \cline{3-6}
    && \cellcolor{bgcolor2}{This paper}  & \cellcolor{bgcolor2}{$\mathcal{O} \left( \chi \frac{L}{\mu} \log \frac{1}{\varepsilon} \right)$ \tnote{{\color{blue}(5)}}}  &  \cellcolor{bgcolor2}{$\mathcal{O} \left( \frac{L}{\mu} \log \frac{1}{\varepsilon} \right)$} & \cellcolor{bgcolor2}{}
    \\\hline 
    \end{tabular}   
    \begin{tablenotes}
    {\small
    \item [{\color{blue}(1)}] for saddle point problems
    \item [{\color{blue}(2)}] deterministic
    \item [{\color{blue}(3)}] stochastic, but not finite sum
    \item [{\color{blue}(4)}] for convex-concave (monotone) case (we reanalyzed for strongly monotone case)
    \item [{\color{blue}(5)}] $B$-connected graphs are also considered. For simplicity in comparison with other works, we put $B = 1$. To get estimates for $B \neq 1$, one need to change $\chi$ to $B \chi$
}
\end{tablenotes}    
    \end{threeparttable}
    }
\end{table*}

\renewcommand{\arraystretch}{2}
\begin{table*}[h]
    \centering
    \small
	\caption{Summary complexities for finding an $\varepsilon$-solution for strongly monotone \textbf{ stochastic (finite-sum)} \textbf{non-distributed} variational inequality \eqref{eq:VI}. Convergence is measured by the distance to the solution. {\em Notation:} $\mu$ = constant of strong monotonicity of the operator $F$, $L$ = Lipschitz constants for all $L_{i}$, $n$ = the size of the local dataset. }
    \label{tab:comparison2}    
  \begin{threeparttable}
    \begin{tabular}{|c|c|c|}
    \hline
    \textbf{Reference } & \textbf{Complexity} & \textbf{Weaknesses} 
    \\\hline
    Palaniappan and Bach \cite{palaniappan2016stochastic} - \algname{SVRG} \tnote{{\color{blue}(1)}} & $\mathcal{O} \left( n + \myred{b}\frac{L^{\myred{2}}}{\mu^{\myred{2}}} \log \frac{1}{\varepsilon} \right)$ & \makecell{{bad complexity}  \\ {batching}} \\\hline
    Palaniappan and Bach \cite{palaniappan2016stochastic} - \algname{Acc-SVRG} \tnote{{\color{blue}(1)}} & $\mathcal{O} \left( n +  \sqrt{\myred{b}n}\frac{L}{\mu} \log^{\myred{2}} \frac{1}{\varepsilon} \right)$ & \makecell{{envelope acceleration}  \\ {batching}}\\\hline
    Chavdarova et al. \cite{chavdarova2019reducing} \tnote{{\color{blue}(2)}}&  $\mathcal{O} \left( n + \myred{b}\frac{L^{\myred{2}}}{\mu^{\myred{2}}} \log \frac{1}{\varepsilon} \right)$ & \makecell{{bad complexity}  \\ {batching}} \\\hline
    Carmon et al. \cite{carmon2019variance}\tnote{{\color{blue}(1)}} & $\mathcal{O} \left( n + \sqrt{\myred{b}n}\frac{L}{\mu} \log \frac{1}{\varepsilon} \right)$ & \makecell{{for games only}  \\ {batching}} \\\hline
    Carmon et al. \cite{carmon2019variance}\tnote{{\color{blue}(1)}} & $\mathcal{\tilde O} \left( n + \sqrt{\myred{b}n}\frac{L}{\mu} \log \frac{1}{\varepsilon} \right)$ & batching \\\hline
    Yang et al. \cite{yang2020global} \tnote{{\color{blue}(1,3)}} & $\mathcal{O}\left( \myred{b^{\frac{1}{3}}} n^{\myred{\frac{2}{3}}}\frac{L^\myred{3}}{\mu^\myred{3}}\log\frac{1}{\varepsilon}\right)$ & \makecell{{bad complexity}  \\ {batching}}  \\\hline
    Alacaoglu and Malitsky \cite{alacaoglu2021stochastic} & $\mathcal{O} \left( n + \sqrt{\myred{b}n}\frac{L}{\mu} \log \frac{1}{\varepsilon} \right)$ & batching  \\\hline
    Alacaoglu et al. \cite{Yura2021} & $\mathcal{O} \left( n + \myred{n}\frac{L}{\mu} \log \frac{1}{\varepsilon} \right)$ & bad rates  \\\hline
    Tominin et al. \cite{vladislav2021accelerated} \tnote{{\color{blue}(1)}} & $\mathcal{O} \left( n +  \sqrt{\myred{b}n}\frac{L}{\mu} \log^{\myred{2}} \frac{1}{\varepsilon} \right)$ & \makecell{{envelope acceleration}  \\ {batching}} \\\hline
    Beznosikov et al. \cite{beznosikov2022stochastic} \tnote{{\color{blue}(2)}}&  $\mathcal{O} \left( n + \myred{b}\frac{L^{\myred{2}}}{\mu^{\myred{2}}} \log \frac{1}{\varepsilon} \right)$ & \makecell{{bad complexity}  \\ {batching}} \\\hline
    \cellcolor{bgcolor2}{This paper} & \cellcolor{bgcolor2}{$\mathcal{O} \left( n + \sqrt{n}\frac{L}{\mu} \log \frac{1}{\varepsilon} \right)$} & \cellcolor{bgcolor2}{} \\\hline
    \end{tabular}   
    \begin{tablenotes}
    {\scriptsize
    \item [{\color{blue}(1)}] for saddle point problems
    \item [{\color{blue}(2)}] under $l$-cocoercivity assumption (in genral case $l = \frac{L^2}{\mu}$)
    \item [{\color{blue}(3)}] under PL - condition
}
\end{tablenotes}    
    \end{threeparttable}
\end{table*}

\begin{landscape}
\renewcommand{\arraystretch}{2}
\begin{table*}[h]
    \centering
    \small
	\caption{Summary communication and local complexities for finding an $\varepsilon$-solution for strongly monotone  \textbf{decentralized} variational inequality \eqref{eq:VI} on fixed and time-varying networks. Convergence is measured by the distance to the solution. {\em Notation:} $\mu$ = constant of strong monotonicity of the operator $F$, $L$ = Lipschitz constants for all $L_{m,i}$, $\chi$ = characteristic number of the network (see Assumptions \ref{ass:fixed} and \ref{ass:tv}). }
    \label{tab:comparison4}    
  \begin{threeparttable}
    \begin{tabular}{|c|c|c|c|c|c|}
    \cline{3-6}
    \multicolumn{2}{c|}{}
     & \textbf{\quad\quad\quad\quad\quad\quad Reference \quad\quad\quad\quad\quad\quad} & \textbf{Communication complexity} & \textbf{Local complexity}  & \textbf{Deter/Stoch?} \\
    \hline
    \multirow{8}{*}{\rotatebox[origin=c]{90}{\textbf{Fixed \quad\quad}}} & \multirow{6}{*}{\rotatebox[origin=c]{90}{\textbf{Upper}\quad\quad}}
     & Mukherjee and Chakraborty \cite{Mukherjee2020:decentralizedminmax} \tnote{{\color{blue}(1)}}  & $\mathcal{O} \left( \chi^{\frac{4}{3}} \frac{L^{\frac{4}{3}}}{\mu^{\frac{4}{3}}} \log \frac{1}{\varepsilon} \right)$ &  $\mathcal{O} \left( \chi^{\frac{4}{3}} \frac{L^{\frac{4}{3}}}{\mu^{\frac{4}{3}}} \log \frac{1}{\varepsilon} \right)$ & {deterministic}
    \\ \cline{3-6}
    && Beznosikov et al. \cite{beznosikov2021distributed}  \tnote{{\color{blue}(1)}}  & $\mathcal{O} \left( \sqrt{\chi} \frac{L}{\mu} \log^2 \frac{1}{\varepsilon} \right)$ &  $\mathcal{O} \left(  \frac{L}{\mu} \log \frac{L+\mu}{\mu} \log \frac{1}{\varepsilon}\right)$ & {deterministic}
    \\ \cline{3-6}
    && Beznosikov et al. \cite{beznosikov2020distributed}  \tnote{{\color{blue}(1)}} & $\mathcal{O} \left( \sqrt{\chi} \frac{L}{\mu} \log^2 \frac{1}{\varepsilon} \right)$ &  $\mathcal{O} \left(  \frac{L}{\mu}\log \frac{1}{\varepsilon} + \frac{\sigma^2}{\mu^2 M \varepsilon}\right)$ & \makecell{{stochastic} \\ {$\sigma^2$-bounded variance}}
    \\ \cline{3-6}
    && Rogozin et al. \cite{rogozin2021decentralized}  \tnote{{\color{blue}(1,2)}} & $\mathcal{O} \left( \sqrt{\chi} \frac{L}{\mu} \log \frac{1}{\varepsilon} \right)$ &  $\mathcal{O} \left( \sqrt{\chi} \frac{L}{\mu}\log \frac{1}{\varepsilon}\right)$ & deterministic
    \\ \cline{3-6}
    && This paper  & $\mathcal{O} \left( \max[\sqrt{n}; \sqrt{\chi}] \frac{L}{\mu} \log \frac{1}{\varepsilon} \right)$  &  $\mathcal{O} \left( \max[\sqrt{n}; \sqrt{\chi}]\frac{L}{\mu} \log \frac{1}{\varepsilon} \right)$ & \makecell{{stochastic} \\ {finite-sum}}
    \\ \cline{3-6}
    && This paper  & $\mathcal{O} \left( \sqrt{\chi} \frac{L}{\mu} \log \frac{1}{\varepsilon} \right)$  &  $\mathcal{O} \left( \sqrt{n}\frac{L}{\mu} \log \frac{1}{\varepsilon} \right)$ & \makecell{{stochastic} \\ {finite-sum}}
    \\ \cline{2-6} 
    & \multirow{2}{*}{\rotatebox[origin=c]{90}{\textbf{Lower}}} & Beznosikov et al. \cite{beznosikov2020distributed}  & $\mathcal{O} \left( \sqrt{\chi} \frac{L}{\mu} \log \frac{1}{\varepsilon} \right)$  &  $\mathcal{O} \left( \frac{L}{\mu} \log \frac{1}{\varepsilon} + \frac{\sigma^2}{\mu^2 M \varepsilon}\right)$ & \makecell{{stochastic} \\ {$\sigma^2$-bounded variance}}
    \\ \cline{3-6} 
    && This paper & $\mathcal{O} \left( \sqrt{\chi} \frac{L}{\mu} \log \frac{1}{\varepsilon} \right)$  &  $\mathcal{O} \left(\sqrt{n} \frac{L}{\mu} \log \frac{1}{\varepsilon} \right)$ & \makecell{{stochastic} \\ {finite-sum}}
    \\\hline\hline 
    \multirow{6}{*}{\rotatebox[origin=c]{90}{\textbf{Time-varying}\quad}} & \multirow{4}{*}{\rotatebox[origin=c]{90}{\textbf{Upper}\quad \quad}} 
    & Beznosikov et al. \cite{beznosikov2021decentralized}  & $\mathcal{O} \left( \chi \frac{L}{\mu} \log \frac{1}{\varepsilon} + \frac{L (\chi D +\sqrt{\chi}\sigma)}{\mu^2   \sqrt{\varepsilon}} + \frac{\sigma^2}{\mu^2 M \varepsilon}\right)$ \tnote{{\color{blue}(3)}} & $\mathcal{O} \left( \chi \frac{L}{\mu} \log \frac{1}{\varepsilon} + \frac{L (\chi D +\sqrt{\chi}\sigma)}{\mu^2   \sqrt{\varepsilon}} + \frac{\sigma^2}{\mu^2 M \varepsilon}\right)$ & \makecell{{$D$-homogeneity}  \\ {$\sigma^2$-bounded variance}}
    \\ \cline{3-6}
    && Beznosikov et al. \cite{beznosikov2021}  \tnote{{\color{blue}(1)}} & $\mathcal{O} \left( \chi \frac{L}{\mu} \log^2 \frac{1}{\varepsilon} \right)$  &  $\mathcal{O} \left( \frac{L}{\mu} \log \frac{1}{\varepsilon} \right)$ & deterministic
    \\ \cline{3-6}
    && This paper  & $\mathcal{O} \left( \chi \frac{L}{\mu} \log \frac{1}{\varepsilon} \right)$ \tnote{{\color{blue}(3)}}  &  $\mathcal{O} \left( \max[\sqrt{n};\chi]\frac{L}{\mu} \log \frac{1}{\varepsilon} \right)$  & \makecell{{stochastic} \\ {finite-sum}}
    \\ \cline{3-6}
    && This paper  & $\mathcal{O} \left( \chi \frac{L}{\mu} \log \frac{1}{\varepsilon} \right)$ \tnote{{\color{blue}(3)}}  &  $\mathcal{O} \left( \sqrt{n}\frac{L}{\mu} \log \frac{1}{\varepsilon} \right)$ & \makecell{{stochastic} \\ {finite-sum}}
    \\ \cline{2-6}
    & \multirow{2}{*}{\rotatebox[origin=c]{90}{\textbf{Lower}}} 
    & Beznosikov et al. \cite{beznosikov2021}  & $\mathcal{O} \left( \chi \frac{L}{\mu} \log \frac{1}{\varepsilon} \right)$  &  $\mathcal{O} \left( \frac{L}{\mu} \log \frac{1}{\varepsilon} \right)$ & \makecell{deterministic}
    \\ \cline{3-6}
    && This paper  & $\mathcal{O} \left( \chi \frac{L}{\mu} \log \frac{1}{\varepsilon} \right)$ \tnote{{\color{blue}(3)}}  &  $\mathcal{O} \left( \sqrt{n}\frac{L}{\mu} \log \frac{1}{\varepsilon} \right)$ & \makecell{{stochastic} \\ {finite-sum}}
    \\\hline 
    \end{tabular}   
    \begin{tablenotes}
    {\scriptsize
    \item [{\color{blue}(1)}] for saddle point problems
    \item [{\color{blue}(2)}] for convex-concave (monotone) case (we reanalyzed for strongly monotone case)
    \item [{\color{blue}(3)}] $B$-connected graphs are also considered. For simplicity in comparison with other works, we put $B = 1$. To get estimates for $B \neq 1$, one need to change $\chi$ to $B \chi$
}
\end{tablenotes}    
    \end{threeparttable}
\end{table*}
\end{landscape}


\end{document}